\documentclass[12pt,reqno]{amsart}
\usepackage{amsbooka}

\makeatletter
\def\chaptermark#1{}

\def\chapter{%
  \if@openright\cleardoublepage\else\clearpage\fi
  \thispagestyle{plain}\global\@topnum\z@
  \@afterindenttrue \secdef\@chapter\@schapter}

\def\@chapter[#1]#2{\refstepcounter{chapter}%
  \ifnum\c@secnumdepth<\z@ \let\@secnumber\@empty
  \else \let\@secnumber\thechapter \fi
  \typeout{\chaptername\space\@secnumber}%
  \def\@toclevel{0}%
  \ifx\chaptername\appendixname \@tocwriteb\tocappendix{chapter}{#2}%
  \else \@tocwriteb\tocchapter{chapter}{#2}\fi
  \chaptermark{#1}%
  \addtocontents{lof}{\protect\addvspace{10\p@}}%
  \addtocontents{lot}{\protect\addvspace{10\p@}}%
  \@makechapterhead{#2}\@afterheading}
\def\@schapter#1{\typeout{#1}%
  \let\@secnumber\@empty
  \def\@toclevel{0}%
  \ifx\chaptername\appendixname \@tocwriteb\tocappendix{chapter}{#1}%
  \else \@tocwriteb\tocchapter{chapter}{#1}\fi
  \chaptermark{#1}%
  \addtocontents{lof}{\protect\addvspace{10\p@}}%
  \addtocontents{lot}{\protect\addvspace{10\p@}}%
  \@makeschapterhead{#1}\@afterheading}
\newcommand\chaptername{Chapter}

\def\@makechapterhead#1{\global\topskip 7.5pc\relax
  \begingroup
  \fontsize{\@xivpt}{18}\bfseries\centering
    \ifnum\c@secnumdepth>\m@ne
      \leavevmode \hskip-\leftskip
      \rlap{\vbox to\z@{\vss
          \centerline{\normalsize\mdseries
              \uppercase\@xp{\chaptername}\enspace\thechapter}
          \vskip 3pc}}\hskip\leftskip\fi
     #1\par \endgroup
  \skip@34\p@ \advance\skip@-\normalbaselineskip
  \vskip\skip@ }
\def\@makeschapterhead#1{\global\topskip 7.5pc\relax
  \begingroup
  \fontsize{\@xivpt}{18}\bfseries\centering
  #1\par \endgroup
  \skip@34\p@ \advance\skip@-\normalbaselineskip
  \vskip\skip@ }
\def\appendix{\par
  \c@chapter\z@ \c@section\z@
  \let\chaptername\appendixname
  \def\thechapter{\@Alph\c@chapter}}

\newcounter{chapter}
\newif\if@openright
\makeatother

\makeatletter
\@addtoreset{section}{chapter}
\makeatother

\usepackage{amsfonts, amsmath, amssymb, xcolor}
\usepackage{amsthm}
\usepackage{mathrsfs}
\usepackage{bm}

\usepackage{booktabs}
\usepackage{hhline}
\usepackage{enumitem}
\usepackage{tabularx}
\usepackage{array}

\usepackage{graphicx}
\usepackage{float}
\usepackage{subfig}
\usepackage{lscape}
\usepackage{rotating}
\usepackage{textcomp}
\usepackage{scalefnt}
\usepackage{a4wide}
\usepackage{lipsum}
\usepackage{longtable}

\usepackage{young}
\usepackage{youngtab}
\usepackage{ytableau}

\usepackage{soul, color} 
\usepackage{tikz}
\usepackage{tikz-cd}
\usetikzlibrary{
  shapes.geometric, shapes,
  arrows, arrows.meta,
  positioning, fit, calc,
  shadows, backgrounds,
  decorations.markings,
  decorations.pathmorphing
}
\pgfdeclarelayer{background}
\pgfsetlayers{background,main}

\usepackage[english]{babel}
\usepackage[colorlinks=true,
  linkcolor=blue, citecolor=magenta, urlcolor=blue, backref=page]{hyperref}
\hypersetup{colorlinks,linkcolor=blue,urlcolor=cyan,citecolor=magenta}

\allowdisplaybreaks[4]
\theoremstyle{plain}
\newtheorem{thm}{Theorem}[chapter]
\newtheorem{lem}[thm]{Lemma}
\newtheorem{prop}[thm]{Proposition}
\newtheorem{cor}[thm]{Corollary}
\newtheorem{rem}[thm]{Remark}

\newtheorem{defn}[thm]{Definition}
\newtheorem{exmp}[thm]{Example}

\newtheorem{mainthm}{Theorem}

\numberwithin{equation}{section}

\makeatletter

\newcommand{\Rmnum}[1]{\expandafter\@slowromancap\romannumeral #1@}
\makeatother

\newcommand{\la}{\lambda}
\newcommand{\id}{\mathrm{id}}
\newcommand{\Cont}{\mathrm{Cont}}

\newcommand{\bla}{\bm{\lambda}}

\newcommand{\ev}{\mathrm{ev}}
\newcommand{\wt}{\mathrm{wt}}
\newcommand{\T}{\mathcal{T}}

\newcommand{\Comp}{\mathrm{Comp}}

\newcommand{\qbinomial}[2]{\genfrac{[}{]}{0pt}{}{#1}{#2}}

\DeclareMathOperator{\rht}{ht}
\DeclareMathOperator{\sgn}{sgn}

\DeclareMathOperator{\sk}{sk}
\newcommand{\notepage}[1]{\hyperref[#1]{p.~\pageref*{#1}}}

\def\K{\mathcal{K}}

\def\R{\mathscr{R}}
\def\O{\mathscr{O}}
\def\E{\mathcal{E}}

\def\P{\mathscr{P}}
\def\q{\mathfrak{q}}
\def\n{\mathrm{n}}

\begin{document}

\title{A skew Murnaghan--Nakayama rule for Hopf dual pairs}
\author{Naihuan Jing}
\address{Department of Mathematics, North Carolina State University, Raleigh, NC 27695, USA}
\email{jing@ncsu.edu}
\author{Ning Liu}
\address{ Beijing International Center for Mathematical Research, Peking University, Beijing 100871, China}
\email{mathliu123@outlook.com}



\begin{abstract}
We develop a uniform skew Murnaghan--Nakayama theory for graded Hopf dual pairs equipped with a nondegenerate Hopf pairing. Using the completed Cauchy element, its grouplike factorization, and the resulting partial contraction operators, we establish a general skew Cauchy identity together with an abstract skew Murnaghan--Nakayama rule. Specializing this framework recovers and extends the classical skew Murnaghan--Nakayama rule for symmetric functions, and yields new skew Murnaghan--Nakayama formulas in several settings, including the dual pairs $(\mathrm{NSym},\mathrm{QSym})$ and $(\Lambda_{(k)},\Lambda^{(k)})$ arising in $k$-Schur theory, as well as the type~$C$ affine Grassmannian context.

As applications, we obtain generating functions for irreducible characters of Ariki--Koike algebras, including their type $A$ and type $B$ specializations, as well as Hecke--Clifford algebras and $\mathfrak q$-rook monoid algebras. We also give ribbon-tableau expansions for skew $(q,t)$-Kostka polynomials and for the entries of the inverse transition matrix, thereby answering a question of Carbonara (1998). Finally, by specializing the auxiliary alphabet $Y$ to sums of powers of primitive roots of unity, we derive a skew plethystic Murnaghan--Nakayama formula together with a Schur expansion for skew modular Schur functions; as a further consequence, we confirm Walker's conjecture (1994) by showing that if the transition from the modular Schur functions to the Schur basis is trivial in the row indexed by $\lambda$, then $\lambda$ must be a $k$-core.
\end{abstract}

\maketitle
\setcounter{tocdepth}{1}
\tableofcontents

\chapter*{Notation}

\begingroup
\renewcommand{\arraystretch}{1.5}
\setlength{\extrarowheight}{1pt}

\begin{longtable}{@{}p{0.29\textwidth}p{0.55\textwidth}p{0.10\textwidth}@{}}
\toprule
Notation & Meaning & Page \\
\midrule
$\P,\ \mathscr S,\ \O$ & The sets of all partitions, strict partitions, and odd partitions. & \notepage{def:partition-basic} \\
$\la\vdash n$ & A partition of $n$. & \notepage{def:partition-basic} \\
$\ell(\la),\ |\la|,\ m_i(\la)$ & The length, weight, and multiplicity of the part $i$ in $\la$. & \notepage{def:partition-basic} \\
$\mathfrak S_n$ & The symmetric group on $n$ letters. & \notepage{def:partition-basic} \\
$\la^t$ & The conjugate partition of $\la$. & \notepage{def:conjugate-partition} \\
\begin{tabular}[t]{@{}l@{}}$a_\la(s),\ l_\la(s),$\\$\Cont(s),\ h_\la(s)$\end{tabular} & Arm-length, leg-length, content, hook length of box $s$. & \notepage{def:box-statistics} \\
$\rht(R)$ & The height of a ribbon $R$. & \notepage{def:ribbon-height} \\
$\Lambda$ & The ring of symmetric functions. & \notepage{def:lambda-ring} \\
$s_\la$ & Schur functions. & \notepage{def:schur-function} \\
$s_{\la/\mu}$ & Skew Schur functions. & \notepage{def:skew-schur-function} \\
$P_\la(X),\ Q_\la(X)$ & Schur $P$- and $Q$-functions, indexed by strict partitions. & \notepage{def:schur-pq-functions} \\
$S_{\la/\mu}(X;t)$ & Big Schur functions. & \notepage{def:big-schur-functions} \\
$P_\la(X;q,t),\ Q_\la(X;q,t)$ & Macdonald $P$- and $Q$-functions. & \notepage{def:macdonald-pq-functions} \\
$J_\la(X;q,t)$ & integral Macdonald polynomials. & \notepage{def:integral-macdonald} \\
$\E$ & The completed Cauchy element $\sum_\lambda f_\lambda\otimes g_\lambda$. & \notepage{d:cauchy-dual} \\
$\q$ & The Hecke deformation parameter, distinguished from the Macdonald parameter $q$. & \notepage{def:q-parameter} \\
$\mathrm{NSym},\ \mathrm{QSym}$ & The Hopf algebras of noncommutative symmetric functions and quasisymmetric functions. & \notepage{ss:NSym-QSym-skewMN} \\
$\Psi_r$ & The first-kind noncommutative power-sum element in $\mathrm{NSym}$. & \notepage{def:first-kind-nsym-powersum} \\
$\mathrm{NCB}^{0}_{r}$ & Tewari nc border strips of size $r$ with empty north-east set. & \notepage{def:nc-border-strip-zero} \\
$\xrightarrow{\mathrm{LH}}_m,\ \xrightarrow{\mathrm{LV}}_m$ & Tewari--van Willigenburg left horizontal and left vertical strip relations for compositions. & \notepage{def:nc-left-strips} \\
$\Lambda_{(k)},\ \Lambda^{(k)}$ & The $k$-bounded subalgebra of symmetric functions and its graded Hopf dual. & \notepage{ss:kSchur-skewMN} \\
$\rightsquigarrow^{(k)}_r$ & The weak $k$-Pieri strip relation on $k$-bounded partitions. & \notepage{def:k-weak-pieri-strip} \\
$\Gamma_{(n)},\ \Gamma^{(n)}$ & The type $C$ affine Grassmannian Hopf-dual pair. & \notepage{ss:SpAffineGr-skewMN} \\
$\mathscr H_{m,n}(\q,\bm u)$ & The Ariki--Koike algebra. & \notepage{sec:ariki-koike-algebras} \\
$\mathscr P_{n,m},\ \mathscr A_m$ & The sets of $m$-multipartitions and finite-support exponent arrays. & \notepage{def:multipartitions-arrays} \\
$\chi^{\bm\lambda}_{\bm\mu}$ & Irreducible character values for Ariki--Koike algebras. & \notepage{def:chi-AK} \\
$\mathscr H_n(\q)$ & The type $A$ Iwahori--Hecke algebra. & \notepage{ss:typeA-hecke} \\
$\phi^\la_\mu$ & Irreducible character values for type $A$ Iwahori--Hecke algebras. & \notepage{def:phi-typeA} \\
$\mathscr B_n(u,\q)$ & The type $B$ Iwahori--Hecke algebra. & \notepage{ss:typeB-hecke} \\
$\varphi^{\bm\lambda}_{\bm\mu}$ & Irreducible character values for type $B$ Iwahori--Hecke algebras. & \notepage{def:varphi-typeB} \\
$\mathscr H_n^c(\q)$ & The Hecke--Clifford algebra. & \notepage{sec:hecke-clifford-algebras} \\
$\zeta^\nu_\mu$ & Irreducible character values for Hecke--Clifford algebras. & \notepage{def:zeta-HC} \\
$\R_n(\q)$ & The $\q$-rook monoid algebra. & \notepage{sec:q-rook-monoid-algebras} \\
$\Upsilon^\la_\mu$ & Irreducible character values for $\q$-rook monoid algebras. & \notepage{def:Upsilon-rook} \\
SRT & Special ribbon tableaux. & \notepage{def:SRT} \\
$K_{\la/\mu,\nu}(q,t)$ & The skew $(q,t)$-Kostka polynomial. & \notepage{ss:skew-qt-kostka} \\
$\K_{\la/\mu,\nu}(q,t)$ & The skew inverse $(q,t)$-Kostka coefficient. & \notepage{ss:inv-qt-kostka} \\
$G(k,m)$ & The Petrie symmetric function. & \notepage{def:petrie-functions} \\
GPR & Good proper $k$-ribbons. & \notepage{def:GPR} \\
$\Omega_k$ & The root-of-unity alphabet $1+\omega_k+\cdots+\omega_k^{k-1}$. & \notepage{def:omega-k} \\
$\mathfrak s^{(k)}_{\la/\mu}$ & The skew modular Schur function. & \notepage{def:skew-modular-schur} \\
$m_{\la/\mu,\nu}$ & The Schur expansion coefficient of $\mathfrak s^{(k)}_{\la/\mu}$. & \notepage{eq:def-skew-m-coeff} \\
GPRT & Good proper $k$-ribbon tableaux. & \notepage{def:GPRT} \\
\bottomrule
\end{longtable}

\endgroup

\chapter{Introduction}
\section{Background}

Decomposition rules in the ring of symmetric functions lie at the heart of Schubert calculus, representation theory, and algebraic combinatorics.
A prototypical example is the classical Murnaghan--Nakayama (MN) rule \cite{Mur37, Nak40a, Nak40b}, which gives the Schur expansion of the product of a power-sum symmetric function with a Schur function:
\begin{align}
    p_r\, s_{\mu}
    =\sum_{\lambda}(-1)^{\rht(\lambda/\mu)}\, s_{\lambda}.
\end{align}
Here the sum ranges over all partitions $\lambda$ for which $\lambda/\mu$ is a ribbon (border strip) of size $r$, and $\rht(\lambda/\mu)$ denotes the height of the ribbon, i.e., the number of rows of $\lambda/\mu$ minus $1$.
Under the standard identification of Schur functions with Schubert classes on the Grassmannian, the MN rule yields a signed rim-hook recursion for intersection numbers against the Newton classes (equivalently, Chern characters) of the tautological bundle.
From the viewpoint of representation theory, it provides an efficient combinatorial procedure for computing irreducible character values of symmetric groups.

Over the past decades, MN-type identities have proliferated well beyond the classical setting.
They appear for Hecke algebras \cite{Ram91}, BMW algebras \cite{HR95}, and Hecke--Clifford algebras \cite{JL23} via $q$-deformations; for Ariki--Koike algebras via cyclotomic deformations \cite{JL25}; for Hall--Littlewood functions via vertex-operator realizations \cite{JL22a}; and for Macdonald polynomials via $(q,t)$-deformations \cite{JL24}, among many others.
More recently, increasing attention has been directed toward plethystic MN rules and their variants \cite{DLT94, EPW14, Wil16, Tur25, CJL25}.

A particularly influential refinement is the \emph{skew} MN rule.
Assaf and McNamara \cite{AM11} obtained an elegant formula for the product of a skew Schur function with a power-sum. In the form recorded in \cite[Theorem~3]{Kon12}, it reads
\begin{equation}\label{e:skew-ps}
    p_r\, s_{\la/\mu}
    =\sum_{\substack{\la^+\supset\la\\
        \la^+/\la\text{ is an }r\text{-ribbon}}}
        (-1)^{\rht(\la^+/\la)}s_{\la^+/\mu}
    -\sum_{\substack{\mu^-\subset\mu\\
        \mu/\mu^-\text{ is an }r\text{-ribbon}}}
        (-1)^{\rht(\mu/\mu^-)}s_{\la/\mu^-}.
\end{equation}
A quantum deformation of \eqref{e:skew-ps}, in which $p_r$ is replaced by a one-parameter family of symmetric functions\footnote{These functions serve as the conjugacy-class parameter functions for Hecke algebras.}, was subsequently obtained in \cite{Kon12}.

Despite their power, these generalizations remain rather dispersed in the literature.
Their proofs typically use different models, such as tableau bijections, raising operators, vertex-operator realizations, or case-specific character formulas, and the resulting identities are often stated separately for the particular family of symmetric functions at hand.
In particular, the common origin of skew identities, plethystic specializations, and character-generating-function formulas is not always visible from the individual proofs.
This motivates a uniform treatment that isolates the structural reason behind MN-type decompositions and makes their skew and plethystic extensions transparent.

\section{Goals and methods}

This paper develops a Hopf-algebraic approach to skew Murnaghan--Nakayama type identities and applies it in several directions.
The starting point is that many skew, plethystic, and character-theoretic variants of the Murnaghan--Nakayama rule arise from the same structure.
We make this structure explicit and use it to derive formulas in a range of settings.

There are two parts to the construction.
First, we formulate a general skew Murnaghan--Nakayama theory for graded Hopf dual pairs equipped with a nondegenerate Hopf pairing.
Second, we specialize this theory in several contexts, obtaining explicit results in symmetric-function theory, character theory, and modular representation-theoretic problems.
The first part is developed in the Hopf-algebraic framework below; the second is summarized in the contribution map that follows it.

\subsection{Hopf-algebraic framework}
Our basic input is a \emph{graded Hopf dual pair} $(H,H^\vee)$ equipped with a nondegenerate Hopf pairing $\langle\cdot,\cdot\rangle$ such that, on each homogeneous degree,
\begin{align}\label{e:int-hopf-pairing}
\langle ab,c\rangle=\langle a\otimes b,\ \Delta^\vee(c)\rangle,
\qquad
\langle a,bc\rangle=\langle \Delta(a),\ b\otimes c\rangle.
\end{align}
Fix dual bases $(f_{\lambda})_{\lambda\in\mathcal I}$ of $H$ and $(g_{\lambda})_{\lambda\in\mathcal I}$ of $H^\vee$.
To formulate an abstract analogue of the skew identity \eqref{e:skew-ps}, we need two ingredients:
a notion of \emph{skew elements} in $H$, and a canonical object playing the role of the Cauchy kernel.

The skew elements arise naturally from the coproduct.
For each $\lambda\in\mathcal I$, we define $f_{\lambda/\mu}\in H$ by expanding $\Delta(f_\lambda)$ in the second tensor factor with respect to the basis $(f_\mu)_{\mu\in\mathcal I}$:
\[
\Delta(f_{\lambda})=\sum_{\mu\in\mathcal I} f_{\lambda/\mu}\otimes f_{\mu}.
\]
To encode the pairing in a form suited to Murnaghan--Nakayama type identities, we introduce the \emph{Cauchy element}
\begin{align}\label{e:int-Cauchy-element}
\E:=\sum_{\lambda\in\mathcal I} f_\lambda\otimes g_\lambda
\ \in\ \widehat{H\otimes H^\vee},
\end{align}
where $\widehat{H\otimes H^\vee}$ denotes the graded completion.
When $H=H^\vee=\Lambda$ is the ring of symmetric functions with its standard Hopf pairing, the element $\E$ specializes to the usual Cauchy kernel.
Thus $\E$ may be regarded as the Hopf-algebraic analogue of the classical Cauchy kernel.

The skew Murnaghan--Nakayama rule on $(H,H^\vee)$ is then obtained from three properties of $\E$:
a grouplike-type factorization of the Cauchy element, a systematic use of partial contractions induced by the pairing, and an orthogonality relation involving the antipode.
Together these properties give a single argument for a number of skew and plethystic variants.

This point of view is close in spirit to the framework of Lam--Lauve--Sottile~\cite{LLS11}, who study skew Littlewood--Richardson rules in the setting of dual Hopf algebras.
Both works start from the same basic input: a graded Hopf pairing, dual homogeneous bases, and skew elements defined via coproduct expansions.
The two settings differ, however, in their main objectives.
The focus of~\cite{LLS11} is on multiplicative structure, whereas here the emphasis is on Murnaghan--Nakayama type identities obtained from the completed Cauchy element and its factorization properties.
Accordingly, while there is no direct implication between the main theorems of~\cite{LLS11} and our skew Murnaghan--Nakayama rule, the two approaches meet in classical specializations at the level of skew Pieri phenomena.
This meeting point is made concrete in Chapters~\ref{s:special-cases-Lambda} and~\ref{s:more}: the former recovers classical skew Pieri rules from the Cauchy-element formalism, while the latter combines the same formalism with straight Pieri inputs in the noncommutative Schur, $k$-Schur, and type $C$ affine Grassmannian settings.

\subsection{Contribution map}
We next indicate how the abstract construction is used in the rest of the paper.
It is useful to separate the results into three levels: recovery of classical identities, Hopf-dual specializations beyond the ordinary symmetric-function setting, and generating or combinatorial formulas obtained by combining the general theory with standard inputs from the relevant applications.
\begin{enumerate}
    \item In the ordinary symmetric-function setting, the framework recovers known skew Pieri and skew Murnaghan--Nakayama identities as specializations. These recoveries serve as checks on the formalism and provide the plethystic language used later.
    \item In nonclassical Hopf-dual settings such as $(\mathrm{NSym},\mathrm{QSym})$, $k$-Schur theory, and type $C$ affine Grassmannian theory, the same completed-Cauchy-element mechanism first gives formal skew MN identities.  When combined with the available straight Murnaghan--Nakayama or Pieri inputs, these identities yield explicit skew combinatorial formulas, including the noncommutative Schur, $k$-Schur, and type $C$ affine Grassmannian specializations in Chapter~\ref{s:more}.
    \item In the application chapters, the framework is combined with standard inputs from representation theory and symmetric functions, such as Frobenius-type character formulas, Macdonald Pieri evaluations, binomial coefficients, and the definitions of Petrie and modular Schur functions. This produces the global character generating series and their refinements, tableau/flag formulas for skew $(q,t)$-Kostka matrices and their inverses, skew plethystic and Petrie-type expansions, and the consequence for Walker's conjecture.
\end{enumerate}
Thus the later chapters should be viewed as parallel applications of the same completed-Cauchy-element formalism, together with the extra input needed in each setting, rather than as a single chain of dependent results.

\section{Main results}
We now state the results in detail. Structurally, the paper consists of one abstract theorem package, one chapter of Hopf-dual specializations beyond $\Lambda$, and three groups of applications.
The abstract core is the skew Cauchy identity together with the general skew Murnaghan--Nakayama rule for Hopf dual pairs.
Chapter~\ref{s:more} applies this core to nonclassical Hopf-dual pairs and extracts concrete skew MN and skew Pieri formulas from known straight inputs.
The subsequent application chapters use this Hopf--Cauchy framework in three directions: generating functions for irreducible characters, skew $(q,t)$-Kostka theory, and plethystic modular Schur theory.
In each direction, the framework is combined with input specific to the setting.
Thus Theorems~\ref{t:B}--\ref{t:F} should be viewed as parallel developments from a common framework, rather than as a single chain of consequences of Theorem~\ref{t:A} alone.

\subsection{General Hopf-algebraic framework}
Let $(H,H^\vee)$ be a graded Hopf dual pair with a nondegenerate Hopf pairing satisfying \eqref{e:int-hopf-pairing}. Their standard structure maps are written as
\[
(H,m,u,\Delta,\varepsilon,S),
\qquad
(H^\vee,m^\vee,u^\vee,\Delta^\vee,\varepsilon^\vee,S^\vee).
\]
All completed tensor products and infinite Cauchy sums in the main statements are understood in the degree-completed sense made precise in Chapter~\ref{s:general-rule}; in particular, the identities are degreewise finite.

\begin{mainthm}[Proposition \ref{p:skew-cauchy-dual}+Theorem \ref{t:skew-MN-dual}]\label{t:A}
  With the notations above, we have the following skew Cauchy identity:
\begin{align}\label{e:int-skew-cauchy-dual}
\E\;\cdot\sum_{\tau\in\mathcal I} f_{\lambda/\tau}\otimes g_{\mu/\tau}
=\sum_{\rho\in\mathcal I} f_{\rho/\mu}\otimes g_{\rho/\lambda}
\end{align}
for any dual bases $(f_{\lambda})_{\lambda\in\mathcal I}$ of $H$ and $(g_{\lambda})_{\lambda\in\mathcal I}$ of $H^\vee$.
Furthermore, if $H^\vee$ is commutative, then the skew Murnaghan--Nakayama rule is obtained as follows:
\begin{align}\label{e:int-skew-MN-dual}
\E\;\cdot\bigl(f_{\lambda/\eta}\otimes 1_{H^\vee}\bigr)
=\sum_{\rho,\mu\in\mathcal I} f_{\rho/\mu}\otimes S^\vee(g_{\eta/\mu})\,g_{\rho/\lambda}
\end{align}
for all $\lambda,\eta\in\mathcal I$.  
\end{mainthm}

The two identities in Theorem~\ref{t:A} have different roles.
The skew Cauchy identity is the basic identity for the completed Cauchy element, while the skew Murnaghan--Nakayama rule is obtained from it by using the antipode and the commutativity hypothesis on the dual side.
In examples, different choices of dual bases and different specializations of the auxiliary alphabet turn these formal identities into the skew Pieri, skew MN, plethystic, and character-theoretic formulas developed below.

Recall the standard Hopf structure on the ring of symmetric functions $\Lambda$ \cite{GR20}.
If we take $(H,H^\vee)=(\Lambda,\Lambda)$, then \eqref{e:int-skew-cauchy-dual} and \eqref{e:int-skew-MN-dual} become (in plethystic form)
\begin{align}
\E\;\cdot\sum_{\tau\in\P} f_{\lambda/\tau}[X]\,g_{\mu/\tau}[Y]
=\sum_{\rho\in\P} f_{\rho/\mu}[X]\, g_{\rho/\lambda}[Y]
\end{align}
and
\begin{align}
\E \cdot f_{\lambda/\eta}[X]
=\sum_{\rho,\mu\in\P} g_{\eta/\mu}[-Y]\,g_{\rho/\lambda}[Y]\,f_{\rho/\mu}[X],
\end{align}
respectively, where $\E$ is precisely the Cauchy kernel.
By choosing different dual bases $(f_{\bullet},g_{\bullet})$ of $\Lambda$ and specializing $Y$, we recover many known identities in skew form, including skew Pieri rules and skew MN rules.
For example, \eqref{e:skew-ps} corresponds to the Hall inner product and the specialization $Y=(\q-1)z$, where $\q$ is an additional parameter. See Table~\ref{tab:Mac} for the Macdonald case and Chapter \ref{s:special-cases-Lambda} for details.

\begin{table}
    \centering
    \renewcommand{\arraystretch}{2}
    \begin{tabular}{ccccc}
    \toprule
          & $e_rP_{\lambda/\eta}(X;q,t)$ & $h_rP_{\lambda/\eta}(X;q,t)$ &
          $g_rP_{\lambda/\eta}(X;q,t)$ & $g_r[(\q-1)z]P_{\lambda/\eta}(X;q,t)$\\
    \midrule
         $Y$ & $\frac{q-1}{1-t}z$ & $\frac{1-q}{1-t}z$ & $z$ & $(\q-1)z$\\
    \bottomrule
    \end{tabular}
    \caption{Various formulas for expanding products involving the skew Macdonald polynomial $P_{\la/\eta}(X;q,t)$. The first three columns correspond to skew Pieri rules, while the last column corresponds to a skew MN rule.}
    \label{tab:Mac}
\end{table}

In addition to $\Lambda$, we also consider other Hopf-dual pairs. These include the dual pair $(\mathrm{NSym},\mathrm{QSym})$ with (noncommutative/quasisymmetric) Schur bases indexed by compositions, the $k$-bounded subalgebra $\Lambda_{(k)}$ and its graded Hopf dual $\Lambda^{(k)}$ with $k$-Schur and dual $k$-Schur bases indexed by $k$-bounded partitions, and the type $C$ affine Grassmannian Hopf-dual pair $(\Gamma_{(n)},\Gamma^{(n)})$ with Schubert bases indexed by affine Grassmannian elements. In Chapter~\ref{s:more}, the formal identities for these pairs are further combined with known straight Murnaghan--Nakayama and Pieri rules to obtain explicit skew combinatorial formulas.

\subsection{Generating functions of irreducible characters}

For $m\ge1$, let $\mathscr P_{n,m}$ denote the set of $m$-multipartitions of $n$, and let
\[
\mathscr A_m:=\Bigl\{\bm\alpha=(\alpha_{i,j})_{1\le i\le m,\ j\ge1}:\ 
\alpha_{i,j}\in\mathbb Z_{\ge0},\ \#\{(i,j):\alpha_{i,j}\ne0\}<\infty\Bigr\}.
\]
For $\bm\alpha\in\mathscr A_m$, write
\[
Z^{\bm\alpha}:=\prod_{i=1}^m\prod_{j\ge1} z_{i,j}^{\,\alpha_{i,j}}.
\]
We also denote by $\P_n$ (resp.\ $\mathscr S_n$, $\O_n$) the set of all partitions (resp.\ strict partitions, odd partitions) of $n$.

As the first application of Theorem~\ref{t:A}, we obtain global generating functions for irreducible characters in these three settings.

\begin{mainthm}[Theorem~\ref{t:gene-AK}+Theorem~\ref{t:gene-HC}+Theorem~\ref{t:gene-rook}]\label{t:B}
With the notation above, we have the following generating functions.
\begin{enumerate}
    \item For the Ariki--Koike algebra $\mathscr H_{m,n}(\q,\bm u)$, let $\chi^{\bm\lambda}_{\bm\alpha}$ be the character coefficients defined in \eqref{eq:Frob-alpha}. Then
    for $N\ge1$, with
    $\mathscr G_{\bm\lambda}^{[N]}(Z;\q,\bm u):=\mathscr G_{\bm\lambda}(Z;\q,\bm u)|_{z_{i,j}=0\ {\rm for}\ j>N}$,
    \begin{align}
        \mathscr G_{\bm\lambda}^{[N]}(Z;\q,\bm u)
        =
        \sum_{\kappa\in\{0,1,\dots,m\}^{\{1,\dots,m\}\times\{1,\dots,N\}}}
        \left(\prod_{i=1}^m\prod_{j=1}^N b^{(i)}_{\kappa_{i,j}}\right)
        \prod_{r=1}^m
        s_{\lambda^{(r)}}\!\bigl[(\q-1)Y_{r,N}(\kappa)\bigr],
    \end{align}
    where the explicit coefficients $b_k^{(i)}$ and alphabets $Y_{r,N}(\kappa)$ are given in Subsection~\ref{ss:plethysm-linear-comb}.
    The full series $\mathscr G_{\bm\lambda}(Z;\q,\bm u)$ is the coefficientwise stable limit of these finite formulas.

    \item For the Hecke--Clifford algebra, one has
    \begin{align}
        \mathscr G^c_{\nu}(Z;\q)
        :=\sum_{\alpha\in\mathscr A_1}
        (\q-1)^{\ell(\alpha)}\zeta^\nu_{\alpha}(\q)\,Z^\alpha
        =
        2^{\frac{-\ell(\nu)+\delta(\nu)}{2}}\,Q_{\nu}\bigl[(\q-1)Z\bigr].
    \end{align}

    \item For the $\q$-rook monoid algebra, one has
    \begin{align}
        \mathscr G^R_{\lambda}(Z;\q)
        :=\sum_{\alpha\in\mathscr A_1}
        (\q-1)^{\ell(\alpha)}\Upsilon^\lambda_{\alpha}(\q)\,Z^\alpha
        =
        \sigma_Z[\q-1]\;s_{\lambda}\bigl[(\q-1)Z\bigr].
    \end{align}
\end{enumerate}
\end{mainthm}

We use the relevant Frobenius-type character formulas as input. The new point is to assemble the resulting character values into single symmetric-function generating series and to derive their refinements from the same skew-MN formalism. In the Ariki--Koike and Hecke--Clifford settings, Jing's vertex operators \cite{J91} lead further to explicit bivariate generating functions, while in the $\q$-rook case a parallel kernel argument gives an analogous two-variable refinement. These identities exhibit phenomena that are much less transparent from individual Murnaghan--Nakayama expansions, including reciprocity in the Ariki--Koike family and self-reciprocity in the Hecke--Clifford case.

The type $A$ and type $B$ Iwahori--Hecke formulas appear naturally as specializations of the Ariki--Koike case. More precisely, the type $A$ formula is recovered at $m=1$ and $u_1=1$ (see Corollary \ref{c:gene-typeA}), while the type $B$ formula arises at $m=2$ and $(u_1,u_2)=(-1,u)$ (see Corollary \ref{c:gene-typeB}). In this sense, the Ariki--Koike generating series provides the common source from which the classical Hecke-theoretic cases emerge. The $\q$-rook formula is parallel in shape, but differs by the additional factor $\sigma_Z[\q-1]$, reflecting the plethystic shift $X\mapsto 1+X$.

\subsection{The skew $(q,t)$-Kostka matrix and its inverse}
For two partitions $\rho\subset\nu$, let $\binom{\nu}{\rho}_{q,t}$ denote the generalized $(q,t)$-binomial coefficient (introduced independently by Lassalle \cite{L98} and Okounkov \cite{O97}).

As the second application of Theorem \ref{t:A}, we obtain combinatorial formulas for the skew $(q,t)$-Kostka polynomial $K_{\la/\mu,\nu}(q,t)$ and its inverse.
The generalized binomial coefficients and the Hall--Littlewood/Macdonald Pieri evaluations used below are standard inputs; the role of the skew-MN rule is to turn them into ribbon-tableau and flag expansions for the skew transition matrices and their inverses.

\begin{mainthm}[Theorem \ref{t:skew-K_iterative}+Corollary \ref{c:skew-K(t)-iterative}]\label{t:C}
We express $K_{\la/\mu,\nu}(q,t)$ in terms of generalized $(q,t)$-binomial coefficients:
\begin{align}
    K_{\la/\mu,\nu}(q,t)
    =t^{n(\nu)}\sum_{(\T,\{\nu\})}\sgn(\T)\prod_{i\ge 1}\binom{\nu^i}{\nu^{i-1}}_{q,t},
\end{align}
summed over all pairs consisting of a special ribbon tableau $\T$ of shape $\la/\mu$ and a partition flag $\{\varnothing=\nu^0\subset\nu^1\subset\cdots\subset\nu\}$ of the same type.
Consequently, at $q=0$ this yields a formula for the skew Kostka polynomial $K_{\la/\mu,\nu}(t)$:
\begin{align}
K_{\la/\mu,\nu}(t)
=\sum_{(\T,\{\nu\})}\sgn(\T)\,
t^{\sum_i n(\nu^i/\nu^{i-1})}
\prod_{i\ge 1}\prod_{j\ge 1}
\qbinomial{(\nu^i)^t_j-(\nu^{i-1})^t_{j+1}}{(\nu^{i-1})^t_j-(\nu^{i-1})^t_{j+1}}_{t},
\end{align}
summed over all such pairs $(\T,\{\nu\})$. Here $\qbinomial{\cdot}{\cdot}_{t}$ is the usual $t$-binomial coefficient.
\end{mainthm}

We refer the reader to Subsection \ref{ss:decomp-Schur} for the definitions of a special ribbon tableau $\T$, its type ${\rm type}(\T)$, and the sign $\sgn(\T)$.
For the skew inverse $(q,t)$-Kostka coefficient $\K_{\la/\mu,\nu}(q,t)$ we have:
\begin{mainthm}[Theorem \ref{t:skew-inv-K_iterative}+Corollary \ref{t:skew-inv-K(t)-iterative}]\label{t:D} 
We have
    \begin{align}
        \K_{\la/\mu,\nu}(q,t)=\frac{1}{c_{\nu}(q,t)}\sum_{(\T,\{\nu\})}\sgn(\T)\prod_{i}Q_{\nu^i/\nu^{i-1}}[1-q;q,t]
    \end{align}
    summed over all pairs consisting of a special ribbon tableau $\T$ of shape $\la/\mu$ and a partition flag $\{\varnothing=\nu^0\subset\nu^1\subset\cdots\subset\nu\}$ with the same type.
    Consequently, at $q=0$ we obtain
     \begin{align}\label{e:int-inv-t-Kostka}
        \K_{\la/\mu,\nu}(t)=\sum_{(\T,T)}\sgn(\T)\psi_{T}(t),
    \end{align}
    summed over all pairs consisting of a special ribbon tableau $\T$ of shape $\la/\mu$ and a semistandard Young tableau (SSYT) $T$ of shape $\nu$ such that ${\rm type}(\T)=\wt(T)$. Here $\psi_T(t)$ is the Pieri-rule coefficient for the Hall--Littlewood polynomial defined in \cite[Ch.\ III (5.9$^{'}$)]{Mac}.
\end{mainthm}

Equation \eqref{e:int-inv-t-Kostka} recovers the well-known formula for inverse Kostka numbers \cite{ER90}. Beyond this recovery, the combinatorial expansions obtained here make it possible to access finer properties of skew $(q,t)$-Kostka matrices and their inverses.  As an illustration, we establish the strict upper unitriangularity of the inverse $t$-Kostka matrix, and we also provide a combinatorial proof of the formula for $\K_{(1^n),\nu}(t)$, answering a question posed by Carbonara \cite{Car98}.

\subsection{Expansion of modular Schur functions and Walker's conjecture}
Let $p_n[h_k]$ denote the plethysm of the power-sum symmetric function $p_n$ and the complete symmetric function $h_k$.
The Petrie symmetric function $G(k,m)$ is defined in \eqref{e:defG}, and the skew modular Schur function $\mathfrak s^{(k)}_{\la/\mu}$ is defined from $G(k,m)$ via the Jacobi--Trudi determinant.
Here we use the theory of Petrie and modular Schur functions, together with root-of-unity specializations of symmetric functions.
The new part is to derive skew plethystic and Petrie-type rules from Theorem~\ref{t:A}, turn them into Schur expansion formulas for skew modular Schur functions, and extract the Walker-conjecture consequence from the resulting plethystic description.

\begin{mainthm}[Theorem \ref{t:skew-PMN}+Theorem \ref{t:Pet}+Theorem \ref{t:combin-m}]\label{t:E} 
Specializing the alphabet $Y$ in Theorem \ref{t:A} to a sum of powers of a primitive root of unity, we have the following combinatorial formulas:
\begin{enumerate}
    \item The skew plethystic MN rule is obtained as follows:
\begin{equation}\label{e:int-skew-PMN}
p_n[h_k]\,s_{\la/\eta}
=\sum_{\rho,\mu}(-1)^{|\eta/\mu|}\,\sgn(\rho/\la)\,\sgn(\eta^{t}/\mu^{t})\,s_{\rho/\mu},
\end{equation}
where the sum ranges over all partitions $\rho,\mu$ such that both $\rho/\la$ and $\eta^{t}/\mu^{t}$ are horizontal $n$-ribbons and the sum of their weights is $k$. The notions of horizontal $n$-ribbons and their weights are referred to in Subsection \ref{ss:skew-PMN}.
\item The skew Pieri-like rule for the Petrie symmetric function is as follows:
   \begin{align}\label{e:int-skew-Pieri-like}
       G(k,m)(X)s_{\la/\eta}(X)=\sum_{\rho,\mu}(-1)^{|\eta/\mu|}\sgn^{'}(\rho/\la)\sgn^{'}(\eta^t/\mu^t)s_{\rho/\mu}(X),
   \end{align}
   where the sum ranges over all partitions $\rho$ and $\mu$ such that $|\rho/\la|+|\eta/\mu|=m$ and both $\rho/\la$ and $\eta^t/\mu^t$ are good proper $k$-ribbons with $\rho_i-\la_i<k$ and $\eta^t_i-\mu^t_i<k$ for all $i$.
\item We have the combinatorial decomposition formula for the modular Schur function $\mathfrak s^{(k)}_{\la/\mu}$ in terms of Schur functions:
    \begin{align}\label{e:int-decomp-modular-schur}
    \mathfrak s^{(k)}_{\la/\mu}
      =\sum_{\nu}\sum_{(\T,\mathscr T)}\sgn(\T)\,\sgn^{'}(\mathscr T)\, s_\nu,
    \end{align}
    where the inner summation runs over all pairs consisting of a special ribbon tableau $\T$ of shape $\la/\mu$ and a good proper $k$-ribbon tableau $\mathscr T$ of shape $\nu$ having the same type. The notions of good proper $k$-ribbons (and their tableaux) and their weights are recalled in Subsection \ref{ss:skew-Pieri-like}.
    \end{enumerate}
\end{mainthm}

Equation \eqref{e:int-skew-PMN} is obtained by applying Theorem~\ref{t:A} at the specialization
$Y=1+\omega_n+\cdots+\omega_n^{n-1}$, while
\eqref{e:int-skew-Pieri-like} and \eqref{e:int-decomp-modular-schur} are obtained from the specialization
$Y=1-(1+\omega_k+\cdots+\omega_k^{k-1})$.
Here $\omega_i$ denotes a primitive $i$-th root of unity.

Transition matrices from modular Schur functions to Schur functions were studied in \cite{Wal94}, motivated by determining the dimensions of weight spaces in modular representation theory of general linear groups.
Several conjectures concerning these matrices were proposed in that context.
The formulas above also fit naturally with questions raised in the literature.
For example, the authors of \cite{CCE+23} decomposed the modular Schur function labeled by one row in terms of Schur functions and asked for combinatorial decomposition formulas for arbitrary shapes; \eqref{e:int-decomp-modular-schur} provides such an answer.

Finally, the same plethystic viewpoint leads to a criterion for when the modular Schur function
$\mathfrak s^{(k)}_\lambda$ collapses to the ordinary Schur function.
When $k$ is an odd prime, we confirm Walker's conjecture from \cite[Conjecture 4.7]{Wal94}.

\begin{mainthm}[Theorem \ref{thm:conj47}]\label{t:F}
Assume that $k$ is an odd prime. If
\[
m_{\lambda,\nu}=\delta_{\lambda,\nu}
\qquad
\text{for all }\nu\vdash |\lambda|,
\]
equivalently if $\mathfrak s^{(k)}_\lambda=s_\lambda$, then $\lambda$ is a $k$-core.
\end{mainthm}

The proof compares the Frobenius character expansion of $s_\lambda\bigl[(1-\Omega_k)X\bigr]$ with that of $s_\lambda(X)$.
This forces the ordinary character $\epsilon^\lambda$ to vanish on all $k$-singular conjugacy classes of the symmetric group.
A defect-zero argument then shows that no hook length of $\lambda$ is divisible by $k$.
In this way, the plethystic description of modular Schur functions feeds directly back into the classical representation theory of $\mathfrak S_d$.

\section{Organization}
The paper is organized as follows.

Chapter~\ref{s:Pre} reviews the combinatorial models, the ring of symmetric functions, plethysm, and Hopf algebras that will be used throughout.
The main abstract results are proved in Chapter~\ref{s:general-rule}, where we introduce the completed Cauchy element, establish its factorization properties, and derive the skew Cauchy identity and the general skew Murnaghan--Nakayama rule for Hopf dual pairs.

Chapter~\ref{s:special-cases-Lambda} specializes this formalism to the symmetric-function setting and provides the language used in the later applications.
In particular, it recovers classical skew Pieri and skew Murnaghan--Nakayama type identities as concrete instances of the general theory.
Chapter~\ref{s:more} records further examples beyond the classical setting, including the dual pair $(\mathrm{NSym},\mathrm{QSym})$, the $k$-bounded/dual $k$-Schur setting, and the type $C$ affine Grassmannian Hopf-dual pair. It also extracts skew MN and skew Pieri formulas from the corresponding straight Murnaghan--Nakayama and Pieri inputs.

The remaining three chapters are devoted to applications.
Chapter~\ref{s:GF-Hecke} develops generating functions for irreducible characters of Ariki--Koike algebras, their type $A$ and type $B$ specializations, Hecke--Clifford algebras, and $\q$-rook monoid algebras.
Chapter~\ref{s:qt-kostka} gives combinatorial formulas for skew $(q,t)$-Kostka matrices and their inverse matrices.
Chapter~\ref{s:modular} establishes plethystic and Petrie-type expansions for modular Schur functions and concludes with a proof of Walker's conjecture.

\medskip

\noindent\textbf{Roadmap.}\quad The logical dependence is not linear after the abstract theorem package. Schematically,
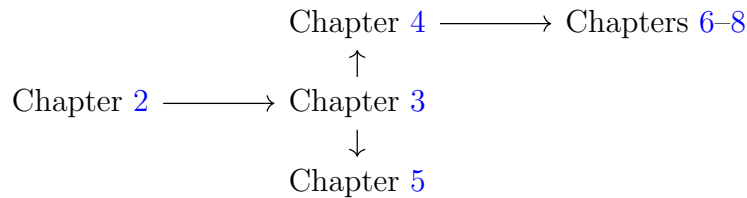
\[
\begin{tikzcd}[column sep=large,row sep=small]
& \text{Chapter \ref{s:special-cases-Lambda}}
  \arrow[r] & \text{Chapters \ref{s:GF-Hecke}--\ref{s:modular}} \\
\text{Chapter \ref{s:Pre}}
  \arrow[r] & \text{Chapter \ref{s:general-rule}}
  \arrow[u]
  \arrow[d] & \\
& \text{Chapter \ref{s:more}} &
\end{tikzcd}
\]
Chapter~\ref{s:special-cases-Lambda} translates the general theory into the symmetric-function and plethystic language used in the three application chapters, whereas Chapter~\ref{s:more} is a separate branch illustrating how the same framework works in further Hopf-dual settings and how known straight MN/Pieri rules can be lifted to skew formulas there.
Thus Chapters~\ref{s:GF-Hecke}--\ref{s:modular} should be read as three parallel applications stemming from the symmetric-function specialization in Chapter~\ref{s:special-cases-Lambda}, each using its own additional input from representation theory, Macdonald theory, or the theory of Petrie and modular Schur functions.

\chapter{Preliminaries}\label{s:Pre}

\section{Partitions and skew diagrams}\label{sec:partitions-skew-diagrams}
The standard reference for partitions and symmetric functions is \cite{Mac}. In this chapter we review the combinatorial objects that will be used throughout this paper.

\phantomsection\label{def:partition-basic}
A {\em composition} is a finite sequence of positive integers; the empty composition is the unique composition of $0$. A {\em weak composition} is a sequence of nonnegative integers with finite support. A {\it partition} is a finite weakly decreasing sequence of positive integers
\[
\la=(\la_1,\la_2,\dots).
\]
We regard the empty sequence as the unique partition of $0$, and extend every partition by $\la_i=0$ for $i>\ell(\la)$ whenever convenient. The entries $\la_i>0$ are the {\it parts} of $\la$. The number of parts is the {\it length} $\ell(\la)$, and the sum of the parts is the {\it weight} $|\la|$. If $|\la|=n$, we write
\[
\la\vdash n,
\qquad
\alpha\models n
\]
if $\alpha$ is a composition of $n$. A partition is {\em strict} if all its parts are distinct. Let
\[
\P_n:=\{\la:\la\vdash n\},
\qquad
\mathscr S_n:=\{\la\vdash n:\la\text{ is strict}\},
\qquad
\P:=\bigcup_{n\ge0}\P_n,
\qquad
\mathscr S:=\bigcup_{n\ge0}\mathscr S_n.
\]
We also write $\O_n$ for the set of odd partitions of $n$, and set $\O:=\bigcup_{n\ge0}\O_n$.
We write $\mathfrak S_n$ for the symmetric group on $n$ letters.
We also use the multiplicity notation
\[
\la=(1^{m_1}2^{m_2}\cdots r^{m_r}\cdots),
\qquad
m_i(\la):=\#\{\,j:\la_j=i\,\}.
\]

\phantomsection\label{def:conjugate-partition}
The {\it (Young/Ferrers) diagram} of $\la$ is the set of unit squares (boxes) arranged in left-justified rows such that the $i$-th row has $\la_i$ boxes, $1\le i\le \ell(\la)$. The {\it conjugate} partition $\la^t$ is obtained by transposing the diagram of $\la$; equivalently,
\[
\la_i^t=\#\{\,j:\la_j\ge i\,\}.
\]

\phantomsection\label{def:box-statistics}
For a box $s=(i,j)\in\la$ with $1\le i\le \ell(\la)$ and $1\le j\le \la_i$, define the {\it arm-length}, {\it leg-length}, {\it content}, and {\it hook-length} by
\[
a_\la(s):=\#\{(i,j')\in\la:\ j<j'\}=\la_i-j,
\]
\[
l_\la(s):=\#\{(i',j)\in\la:\ i<i'\}=\la_j^t-i,
\]
\[
\Cont(s):=j-i,
\qquad
h_\la(s):=a_\la(s)+l_\la(s)+1.
\]
When $\la$ is fixed (or clear from context), we abbreviate $a_\la(s)$ and $l_\la(s)$ by $a(s)$ and $l(s)$, respectively. We also set
\[
\n(\la):=\sum_{i\ge1}(i-1)\la_i
=\sum_{j\ge1}\binom{\la_j^t}{2}.
\]
See Figure~\ref{fig:conjugate} for an example.

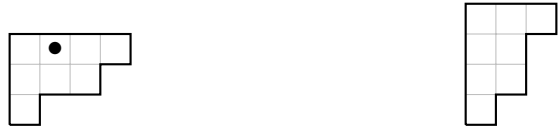
\begin{figure}
\centering
\begin{tikzpicture}[scale = 0.4]
  \begin{scope}
    \clip (0,0) -| (1,1) -| (3,2) -| (4,3) -| (0,3) -| (0,0);
    \draw [color=black!25] (0,0) grid (4,3);
  \end{scope}
  \draw [thick] (0,0) -| (1,1) -| (3,2) -| (4,3) -- (0,3) -- (0,0);
  \node at (1.5,2.5) {$\bullet$};
\end{tikzpicture}
\hspace{10em}
\begin{tikzpicture}[scale = 0.4]
  \begin{scope}
    \clip (0,0) -| (1,1) -| (2,3) -| (3,4) -- (0,4) -- (0,0);
    \draw [color=black!25] (0,0) grid (3,4);
  \end{scope}
  \draw [thick] (0,0) -| (1,1) -| (2,3) -| (3,4) -- (0,4) -- (0,0);
\end{tikzpicture}
\caption{$\la=(4,3,1)$ on the left, and its conjugate partition $\la^t$ on the right. Here $a_\la(\bullet)=2$ and $l_\la(\bullet)=1$.}
\label{fig:conjugate}
\end{figure}

The {\it dominance order} on partitions of the same weight is defined by
\[
\la\trianglerighteq\mu
\qquad\Longleftrightarrow\qquad
\sum_{i=1}^r \la_i\ge \sum_{i=1}^r \mu_i
\quad\text{for all }r\ge1.
\]

We write $\mu\subset\la$ if $\mu_i\le \la_i$ for all $i$. If $\mu\subset\la$, then the set-theoretic difference
\[
\K=\la/\mu
\]
is called a {\it skew diagram}.

\phantomsection\label{def:ribbon-height}
A {\it path} in a skew diagram $\K$ is a sequence $x_0,x_1,\dots,x_n$ of boxes in $\K$ such that $x_{i-1}$ and $x_i$ share a common side for $1\le i\le n$. A subset $\xi\subset\K$ is {\it connected} if any two boxes in $\xi$ are joined by a path lying in $\xi$. The connected components (the maximal connected subsets) of $\K$ are themselves skew diagrams.

A skew diagram $\la/\mu$ is a {\it vertical strip} if each row contains at most one box, and a {\it horizontal strip} if each column contains at most one box. A skew diagram $\K$ is a {\it ribbon} if it is connected and contains no $2\times2$ block. The {\it length} of a ribbon $\K$ is its number of boxes, and its {\it height} is one less than the number of rows it occupies; we write $\rht(\K)$ for the height. The southwest-most (respectively, northeast-most) box is called the {\it starting box} (respectively, {\it ending box}) of the ribbon. 
Figure~\ref{fig:ribbon} gives an example.

\begin{figure}
    \centering
    \begin{tikzpicture}[scale = 0.4]
      \begin{scope}
        \clip (0,0) -| (2,1) -| (3,2) -| (5,3) -| (8,4) -- (0,4) -- (0,0);
        \draw [color=black!25] (0,0) grid (8,4);
      \end{scope}
      \draw [thick] (0,0) -| (2,1) -| (3,2) -| (5,3) -| (8,4) -- (0,4) -- (0,0);
      \draw [thick, rounded corners] (2.5,1.5) -- (2.5,2.5) -- (4.5,2.5) -- (4.5,3.5) -- (7.5,3.5);
      \draw [color=black,fill=black,thick] (7.5,3.5) circle (.6ex);
      \node [draw, circle, fill = white, inner sep = 1.5pt] at (2.5,1.5) { };
    \end{tikzpicture}
    \caption{A ribbon with starting box (white dot) and ending box (black dot). Its length and height are $8$ and $2$, respectively.}
    \label{fig:ribbon}
\end{figure}
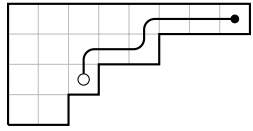

A {\em semi-standard Young tableau} (SSYT) $T$ of shape $\la/\mu$ is a filling of the boxes of $\la/\mu$ by positive integers that are weakly increasing along rows (left to right) and strictly increasing down columns. Its {\em weight} is the weak composition
\[
\wt(T)=(\alpha_1,\alpha_2,\ldots),
\qquad
\alpha_i:=\#\{\text{boxes of }T\text{ filled with }i\}.
\]
We denote by $\operatorname{type}(T)$ the partition obtained by rearranging $\wt(T)$ in weakly decreasing order.

Now let $\la=(\la_1>\cdots>\la_{\ell(\la)}>0)$ and $\mu=(\mu_1>\cdots>\mu_{\ell(\mu)}>0)$ be strict partitions with $\mu\subset\la$. The {\it shifted Young diagram} of $\la$ is
\[
S(\la)=\bigl\{(i,j)\in\mathbb Z_{>0}^2:\ 1\le i\le \ell(\la),\ i\le j\le i+\la_i-1\bigr\},
\]
and similarly for $S(\mu)$. The {\it shifted skew diagram} is
\[
S(\la/\mu):=S(\la)\setminus S(\mu).
\]
The boxes $(i,i)$ (when present) form the {\it main diagonal}.

We use the diagonal-strict unmarked model for shifted tableaux, equivalent to the usual marked model; see \cite[Ch.~III, Sec.~8]{Mac}. A {\it shifted semistandard tableau} $U$ of shape $S(\la/\mu)$ is a filling of $S(\la/\mu)$ by positive integers that is weakly increasing from left to right in each row and from top to bottom in each column, and strictly increasing along every NW--SE diagonal. Consequently, for each $r\ge1$, the set of boxes labeled $r$ is a disjoint union of ribbons. We write $\kappa(U)$ for the total number of these ribbon components, summed over all labels.

\section{Symmetric functions}\label{sec:symmetric-functions}
\phantomsection\label{def:lambda-ring}
Let $X=(x_1,x_2,\dots)$ be a set of variables. Let $\Lambda$ be the ring of symmetric functions in the variables $x_1,x_2,\dots$ over $\mathbb Q$, and let $\Lambda_{\mathbb Z}\subset\Lambda$ be the $\mathbb Z$-subalgebra of symmetric functions with integer coefficients. 

\subsection{ Standard bases and generating functions}
The ring $\Lambda$ has several standard $\mathbb Q$-bases indexed by partitions $\la\in\P$:
\begin{enumerate}
    \item {\em Monomial symmetric functions} $m_\la(X)$:
    \[
    m_\la(X):=\sum_{\alpha}x^\alpha,
    \]
    where the sum runs over distinct rearrangements $\alpha$ of the parts of $\la$.

    \item {\em Elementary symmetric functions} $e_\la(X)$: let
    \[
    e_k(X):=\sum_{i_1<\cdots<i_k}x_{i_1}\cdots x_{i_k},
    \qquad
    e_\la(X):=\prod_j e_{\la_j}(X).
    \]

    \item {\em Complete symmetric functions} $h_\la(X)$: let
    \[
    h_k(X):=\sum_{i_1\le\cdots\le i_k}x_{i_1}\cdots x_{i_k},
    \qquad
    h_\la(X):=\prod_j h_{\la_j}(X).
    \]

    \item {\em Power-sum symmetric functions} $p_\la(X)$: let
    \[
    p_k(X):=\sum_i x_i^k,
    \qquad
    p_\la(X):=\prod_j p_{\la_j}(X).
    \]

    \item {\em Schur functions} $s_\la(X)$:\phantomsection\label{def:schur-function}
    \[
    s_\la(X)=\det_{1\le i,j\le \ell(\la)}\!\bigl(h_{\la_i-i+j}(X)\bigr),
    \]
    where $h_0=1$ and $h_r=0$ for $r<0$.
\end{enumerate}
We adopt the analogous conventions $e_0=1$ and $e_r=0$ for $r<0$.

The generating series for the elementary and complete symmetric functions are
\[
E(z):=\sum_{r\ge0} e_r(X)z^r=\prod_{i\ge1}(1+x_i z),
\qquad
H(z):=\sum_{r\ge0} h_r(X)z^r=\prod_{i\ge1}\frac{1}{1-x_i z}.
\]
In particular,
\[
E(-z)\,H(z)=1.
\]

\subsection{Hall inner product and the involution $\omega$}
The {\it Hall inner product} on $\Lambda$ is the bilinear form characterized by
\[
\langle h_\la,m_\mu\rangle=\delta_{\la,\mu}.
\]
Equivalently,
\[
\langle s_\la,s_\mu\rangle=\delta_{\la,\mu},
\qquad
\langle p_\la,p_\mu\rangle=\delta_{\la,\mu}\,z_\la,
\qquad
z_\la:=\prod_{i\ge1} i^{m_i(\la)}m_i(\la)!.
\]

The standard involution $\omega:\Lambda\to\Lambda$ is determined by
\[
\omega(e_r)=h_r,
\qquad
\omega(h_r)=e_r,
\qquad
\omega(p_r)=(-1)^{r-1}p_r.
\]
It follows that
\[
\omega(s_\la)=s_{\la^t},
\qquad
\omega(s_{\la/\mu})=s_{\la^t/\mu^t}.
\]

\subsection{ Schur functions and skew Schur functions}\label{sec:schur-skew-functions}
\phantomsection\label{def:skew-schur-function}
For a skew shape $\la/\mu$, the tableau definition of the skew Schur function is
\[
s_{\la/\mu}(X):=\sum_T x^{\wt(T)},
\]
where the sum ranges over all SSYTs $T$ of shape $\la/\mu$, and
\[
x^{\wt(T)}:=\prod_i x_i^{\alpha_i}
\qquad
\text{if }
\wt(T)=(\alpha_1,\alpha_2,\dots).
\]
Equivalently, if $m\ge \ell(\la)$ and $\mu$ is padded with zeros to length $m$, then the skew Jacobi--Trudi formula is
\[
s_{\la/\mu}(X)
=
\det_{1\le i,j\le m}\!\bigl(h_{\la_i-\mu_j-i+j}(X)\bigr).
\]
There is also the dual Jacobi--Trudi formula
\[
s_{\la/\mu}(X)
=
\det_{1\le i,j\le \la_1}\!\bigl(e_{\la_i^t-\mu_j^t-i+j}(X)\bigr).
\]

The classical Cauchy identity reads
\begin{equation}\label{e:Cauchy-Schur}
\sum_{\la\in\P} s_\la(X)s_\la(Y)
=
\prod_{i,j\ge1}\frac{1}{1-x_i y_j}
=
\exp\!\left(\sum_{r\ge1}\frac{1}{r}p_r(X)p_r(Y)\right).
\end{equation}
More generally, the skew Cauchy identity is
\begin{equation}\label{e:skew-Cauchy-Schur}
\sum_{\la\in\P} s_{\la/\mu}(X)s_{\la/\nu}(Y)
=
\left(\prod_{i,j\ge1}\frac{1}{1-x_i y_j}\right)
\sum_{\rho\in\P} s_{\nu/\rho}(X)s_{\mu/\rho}(Y).
\end{equation}

\subsection{Schur $P$- and $Q$-functions}\label{sec:schur-pq-functions}
\phantomsection\label{def:schur-pq-functions}
For a shifted semistandard tableau $U$, let
\[
x^U:=\prod_{i\ge1} x_i^{\alpha_i},
\]
where $\alpha_i$ is the number of boxes of $U$ labeled $i$. The {\it skew Schur $Q$-function} is defined by
\[
Q_{\la/\mu}(X):=\sum_U 2^{\kappa(U)}x^U,
\]
where the sum ranges over all shifted semistandard tableaux $U$ of shape $S(\la/\mu)$. The corresponding {\it skew Schur $P$-function} is
\[
P_{\la/\mu}(X):=2^{-(\ell(\la)-\ell(\mu))}Q_{\la/\mu}(X).
\]
For straight shapes one has the Cauchy identity
\[
\sum_{\la\in\mathscr S} Q_\la(X)P_\la(Y)
=
\prod_{i,j\ge1}\frac{1+x_i y_j}{1-x_i y_j}.
\]

Throughout this paper, the unparameterized symbols $P_{\la/\mu}(X)$ and $Q_{\la/\mu}(X)$ refer to the Schur $P$- and $Q$-functions.
By contrast, $P_{\la/\mu}(X;t)$ and $Q_{\la/\mu}(X;t)$ denote Hall--Littlewood functions, while $P_{\la/\mu}(X;q,t)$ and $Q_{\la/\mu}(X;q,t)$ denote Macdonald functions.

\subsection{Hall--Littlewood functions and big Schur functions}\label{sec:hall-littlewood-big-schur}
Let $t$ be an indeterminate and write $\Lambda(t):=\Lambda\otimes_{\mathbb Q}\mathbb Q(t)$. The one-row Hall--Littlewood functions (also called generalized complete symmetric functions) are defined by
\begin{equation}\label{e:qgen}
\sum_{r\ge0} q_r(X;t)z^r
=
\prod_{i\ge1}\frac{1-tx_i z}{1-x_i z}
=
\exp\!\left(\sum_{r\ge1}\frac{1-t^r}{r}p_r(X)z^r\right).
\end{equation}
In particular,
\[
q_r(X;0)=h_r(X).
\]

The $t$-Hall inner product on $\Lambda(t)$ is defined on the power-sum basis by
\[
\langle p_\la,p_\mu\rangle_t
=
\delta_{\la,\mu}\,z_\la(t),
\qquad
z_\la(t):=z_\la\prod_{i=1}^{\ell(\la)}\frac{1}{1-t^{\la_i}}.
\]
Let $\{P_\la(X;t)\}_{\la\in\P}$ and $\{Q_\la(X;t)\}_{\la\in\P}$ be the Hall--Littlewood $P$- and $Q$-functions, normalized so that they form dual bases with respect to $\langle\cdot,\cdot\rangle_t$. Their skew versions are defined by
\[
Q_{\la/\mu}(X;t):=P_\mu^{\perp_t}Q_\la(X;t),
\qquad
P_{\la/\mu}(X;t):=Q_\mu^{\perp_t}P_\la(X;t),
\]
where $\perp_t$ denotes the adjoint with respect to $\langle\cdot,\cdot\rangle_t$. The Hall--Littlewood Cauchy identity is
\[
\sum_{\la\in\P} Q_\la(X;t)P_\la(Y;t)
=
\prod_{i,j\ge1}\frac{1-tx_i y_j}{1-x_i y_j}
=
\exp\!\left(\sum_{r\ge1}\frac{1-t^r}{r}p_r(X)p_r(Y)\right).
\]

\phantomsection\label{def:big-schur-functions}
The {\it big Schur functions} are defined by
\[
S_{\la/\mu}(X;t)
:=
\det_{1\le i,j\le m}\!\bigl(q_{\la_i-\mu_j-i+j}(X;t)\bigr),
\qquad m\ge \ell(\la).
\]
In particular,
\[
S_{\la/\mu}(X;0)=s_{\la/\mu}(X).
\]

\subsection{Macdonald polynomials}\label{sec:macdonald-polynomials}
\phantomsection\label{def:macdonald-pq-functions}
Let $q,t$ be independent indeterminates and $\mathbb Q(q,t)$ the field of rational functions in $q$ and $t$. Set
\[
\Lambda(q,t):=\Lambda\otimes_{\mathbb Q}\mathbb Q(q,t)
=\bigoplus_{n\ge0}\Lambda^n(q,t).
\]
It is well known that the homogeneous Macdonald polynomials
\[
P_\la(X;q,t),\qquad Q_\la(X;q,t)\qquad (\la\in\P)
\]
form bases of $\Lambda(q,t)$, characterized by triangularity and orthogonality.

For a partition $\la$, define
\[
c_\la(q,t):=\prod_{s\in\la}\bigl(1-q^{a(s)}t^{l(s)+1}\bigr),
\qquad
c'_\la(q,t):=\prod_{s\in\la}\bigl(1-q^{a(s)+1}t^{l(s)}\bigr).
\]
\phantomsection\label{def:integral-macdonald}
The {\it integral Macdonald polynomial} is
\[
J_\la(X;q,t):=c_\la(q,t)P_\la(X;q,t)=c'_\la(q,t)Q_\la(X;q,t).
\]

Define
\[
z_\la(q,t):=z_\la\prod_{i=1}^{\ell(\la)}\frac{1-q^{\la_i}}{1-t^{\la_i}}.
\]
The $q,t$-Hall inner product on $\Lambda(q,t)$ is given on the power-sum basis by
\begin{equation}\label{e:innerprod}
\langle p_\la,p_\mu\rangle_{q,t}=\delta_{\la,\mu}\,z_\la(q,t).
\end{equation}
For $f\in\Lambda(q,t)$, let $f$ act by multiplication on $\Lambda(q,t)$, and define its adjoint $f^{\perp_{q,t}}$ by
\[
\langle fg,h\rangle_{q,t}
=
\langle g,f^{\perp_{q,t}}h\rangle_{q,t}
\qquad
\text{for all }g,h\in\Lambda(q,t).
\]
In particular,
\[
p_n^{\perp_{q,t}}
=
\frac{n(1-q^n)}{1-t^n}\,\frac{\partial}{\partial p_n}.
\]

Let
\[
Q_\la(X;q,t)=b_\la(q,t)\,P_\la(X;q,t)
\]
so that $\{P_\la\}$ and $\{Q_\la\}$ are dual bases for $\langle\cdot,\cdot\rangle_{q,t}$. The normalization factor is
\[
b_\la(q,t)=\langle P_\la,P_\la\rangle_{q,t}^{-1}.
\]
If $(a;q)_r:=\prod_{i=0}^{r-1}(1-aq^i)$, then one convenient expression is
\[
b_\la(q,t)
=
\prod_{1\le i\le j\le \ell(\la)}
\frac{(q^{\la_i-\la_j}t^{\,j-i+1};q)_{\la_j-\la_{j+1}}}
{(q^{\la_i-\la_j+1}t^{\,j-i};q)_{\la_j-\la_{j+1}}}.
\]

The skew Macdonald polynomials are defined by
\[
Q_{\la/\mu}(X;q,t):=P_\mu^{\perp_{q,t}}Q_\la(X;q,t),
\qquad
P_{\la/\mu}(X;q,t):=Q_\mu^{\perp_{q,t}}P_\la(X;q,t).
\]
Their Cauchy identity is
\[
\sum_{\la\in\P} P_\la(X;q,t)Q_\la(Y;q,t)
=
\prod_{i,j\ge1}\frac{(t x_i y_j;q)_\infty}{(x_i y_j;q)_\infty}
=
\exp\!\left(\sum_{r\ge1}\frac{1-t^r}{r(1-q^r)}p_r(X)p_r(Y)\right).
\]

The one-row Macdonald functions are defined by
\[
\sum_{r\ge0} g_r(X;q,t)z^r
=
\exp\!\left(\sum_{r\ge1}\frac{1-t^r}{r(1-q^r)}p_r(X)z^r\right).
\]
In particular,
\[
g_r(X;0,t)=q_r(X;t).
\]

\section{Plethysm and virtual alphabets}
We now introduce plethystic notation and record several identities that will be used later.

Let
\[
X=x_1+x_2+\cdots,
\qquad
Y=y_1+y_2+\cdots
\]
be the alphabets associated with the variable sets introduced above.

\begin{defn}
Let $f,g\in\Lambda$ and expand $f=\sum_{\la} c_{\la}\,p_{\la}$. The {\em plethysm} $f[g]$ is defined by
\begin{equation}\label{e:def-plethysm}
f[g]:=\sum_{\la} c_{\la}\,\prod_{j=1}^{\ell(\la)} g\!\bigl[X^{\la_j}\bigr].
\end{equation}
\end{defn}

The following identities hold whenever both sides are defined:
\[
p_k[p_m]=p_{km},
\qquad
p_k[f\pm g]=p_k[f]\pm p_k[g],
\qquad
p_k[fg]=p_k[f]\cdot p_k[g],
\]
\[
(f\pm g)[h]=f[h]\pm g[h],
\qquad
(fg)[h]=f[h]\cdot g[h].
\]

For two alphabets $X$ and $Y$, set
\[
X+Y:=\sum_i x_i+\sum_j y_j,
\qquad
XY:=\sum_{i,j}x_i y_j.
\]
Define the negative alphabet $-X$ by $X+(-X)=\bm 0$, where $\bm 0=0+0+\cdots$. More generally, for a virtual alphabet $A$ and $r\ge1$, plethystic evaluation is determined by the power sums. In particular,
\[
p_r[X+Y]=p_r[X]+p_r[Y],
\qquad
p_r[XY]=p_r[X]\,p_r[Y],
\qquad
p_r[-X]=-p_r[X].
\]
Two especially useful specializations are
\[
p_r[(1-t)X]=(1-t^r)p_r[X],
\qquad
p_r\!\left[\frac{X}{1-q}\right]=\frac{p_r[X]}{1-q^r}.
\]

If $\omega_n$ is a primitive $n$th root of unity, we set
\[
\Omega_n:=1+\omega_n+\omega_n^2+\cdots+\omega_n^{n-1}.
\]
Then
\[
p_r[\Omega_n]
=
1+\omega_n^r+\omega_n^{2r}+\cdots+\omega_n^{(n-1)r}
=
\begin{cases}
0,& n\nmid r,\\[4pt]
n,& n\mid r.
\end{cases}
\]

For any symmetric function $f$, we adopt the convention
\[
f[\bm 0]=
\begin{cases}
0,& \deg f>0,\\
1,& \deg f=0.
\end{cases}
\]
We shall also use the standard shorthand
\[
\sigma_z[A]:=\sum_{r\ge0} h_r[A]\,z^r
=
\exp\!\left(\sum_{n\ge1}\frac{1}{n}p_n[A]\,z^n\right).
\]

\begin{lem}\label{l:plethysm}
Let $X=x_1+x_2+\cdots$ and let $f$ be a homogeneous symmetric function of degree $n$. Then
\[
f[X]=f(x),
\qquad
f[-X]=(-1)^n\,\omega(f)[X],
\]
where $\omega$ is the standard involution on $\Lambda$ determined by $\omega(p_k)=(-1)^{k-1}p_k$.
\end{lem}

\begin{proof}
The first identity is tautological. For the second, note from \eqref{e:def-plethysm} that
\[
p_k[-X]=-(x_1^k+x_2^k+\cdots)=-p_k[X].
\]
Writing
\[
f=\sum_{\la\vdash n} c_\la\,p_\la,
\]
we obtain
\[
f[-X]
=
\sum_{\la\vdash n} c_\la\,(-1)^{\ell(\la)}p_\la[X].
\]
Since
\[
\omega(p_\la)=(-1)^{|\la|-\ell(\la)}p_\la,
\]
we get
\[
f[-X]
=
\sum_{\la\vdash n} c_\la\,(-1)^{|\la|}\,\omega(p_\la)[X]
=
(-1)^n\,\omega(f)[X]. \qedhere
\]
\end{proof}

In particular,
\[
s_{\la/\mu}[-X]=(-1)^{|\la/\mu|}s_{\la^t/\mu^t}[X],
\quad
Q_{\la/\mu}[-X]=(-1)^{|\la/\mu|}Q_{\la/\mu}[X],
\quad
h_r[-X]=(-1)^r e_r[X].
\]
We shall freely use $f[X]=f(x)$, interchanging the two notations when convenient.

Using plethystic notation, the one-row Hall--Littlewood functions satisfy
\[
q_r(X;t)=h_r[(1-t)X],
\]
and hence
\[
\sum_{r\ge0} q_r(X;t)z^r
=
\sigma_z[(1-t)X].
\]
Likewise, the big Schur functions satisfy
\[
S_{\la/\mu}(X;t)=s_{\la/\mu}[(1-t)X].
\]


Another basic identity that will be used repeatedly is the plethystic translation formula
\begin{equation}\label{e:plethystic-translation}
\exp\!\left(\sum_{n\ge1} p_n[Y]\frac{\partial}{\partial p_n[X]}\right)f[X]=f[X+Y].
\end{equation}
Indeed, the differential operator on the left sends each $p_n[X]$ to $p_n[X]+p_n[Y]$, and hence acts on every symmetric function by the substitution $X\mapsto X+Y$.

\section{Hopf algebras}
We conclude this chapter by recalling the Hopf-algebraic language used throughout this paper.

\begin{defn}[Hopf algebra]\label{d:Hopf}
Let $F$ be a field. A {\em Hopf algebra} over $F$ is a vector space $H$ equipped with linear maps
\[
m:H\otimes H\to H,
\qquad
u:F\to H,
\qquad
\Delta:H\to H\otimes H,
\qquad
\varepsilon:H\to F,
\qquad
S:H\to H,
\]
called the multiplication, unit, comultiplication, counit, and antipode, respectively, such that:
\begin{align*}
m\circ(m\otimes \id)&=m\circ(\id\otimes m),\\
m\circ(u\otimes \id)&=\id=m\circ(\id\otimes u),\\
(\Delta\otimes \id)\circ\Delta&=(\id\otimes\Delta)\circ\Delta,\\
(\varepsilon\otimes \id)\circ\Delta&=\id=(\id\otimes\varepsilon)\circ\Delta,\\
m\circ(S\otimes \id)\circ\Delta&=u\circ\varepsilon
=
m\circ(\id\otimes S)\circ\Delta.
\end{align*}
Moreover, $\Delta$ and $\varepsilon$ are required to be algebra homomorphisms.
\end{defn}

A Hopf algebra $H=\bigoplus_{n\ge0}H_n$ is {\it graded connected} if
\[
H_i H_j\subset H_{i+j},
\qquad
\Delta(H_n)\subset \bigoplus_{i+j=n} H_i\otimes H_j,
\qquad
H_0=F\cdot 1.
\]
Given two Hopf algebras $H$ and $H^\vee$, a bilinear form
\[
\langle\cdot,\cdot\rangle:H\times H^\vee\to F
\]
is called a {\it Hopf pairing} if it is compatible with the products and coproducts in the sense that
\[
\langle ab,c\rangle=\langle a\otimes b,\Delta^\vee(c)\rangle,
\qquad
\langle a,bc\rangle=\langle \Delta(a),b\otimes c\rangle.
\]
This is the abstract framework underlying Chapter~3.

The basic concrete example for us is the symmetric-function Hopf algebra. Let $F$ be a field of characteristic zero and set
\[
\Lambda_F:=\Lambda\otimes_{\mathbb Q}F.
\]
Then $\Lambda_F$ is a graded connected, commutative, and cocommutative Hopf algebra with the usual multiplication and unit, and with structure maps determined by
\[
\Delta(f)[X,Y]=f[X+Y],
\qquad
\varepsilon(f)=f[0],
\qquad
S(f)[X]=f[-X].
\]
Equivalently, $\Delta$ is the algebra homomorphism determined by
\[
\Delta(p_r)=p_r\otimes 1+1\otimes p_r
\qquad (r\ge1).
\]
Hence
\[
\Delta(h_r)=\sum_{i=0}^r h_i\otimes h_{r-i},
\qquad
\Delta(e_r)=\sum_{i=0}^r e_i\otimes e_{r-i},
\qquad
\Delta(p_r)=p_r\otimes 1+1\otimes p_r.
\]
The antipode is given by
\[
S(h_r)=(-1)^r e_r,
\qquad
S(e_r)=(-1)^r h_r,
\qquad
S(p_r)=-p_r.
\]

Under the Hall inner product, $\Lambda_F$ is self-dual as a Hopf algebra. In particular, the skew Schur functions are characterized by the coproduct formula
\[
\Delta(s_\la)=\sum_{\mu\subset\la} s_{\la/\mu}\otimes s_\mu
=\sum_{\mu\subset\la} s_\mu\otimes s_{\la/\mu}.
\]
Likewise, after extending scalars to $\Lambda(t)$ or $\Lambda(q,t)$ and using the corresponding Hall pairings, the skew Hall--Littlewood and skew Macdonald functions introduced above may be viewed as Hopf-theoretic skewings. The abstract version of this philosophy, in the setting of general Hopf dual pairs and completed Cauchy elements, will be developed in the next chapter.

\chapter{The skew Murnaghan--Nakayama rule}
\label{s:general-rule}

In this chapter we formulate and prove a general skew Murnaghan--Nakayama rule in a pair of dual graded Hopf algebras.
The discussion is purely Hopf--algebraic: it applies uniformly to any graded connected Hopf algebra equipped with a graded nondegenerate Hopf pairing and a dual pair of homogeneous bases.

\section{Dual Hopf algebras, completions, and Hopf pairings}

Throughout let $F$ be a field.
Let
\[
H=\bigoplus_{n\ge 0} H_n,
\qquad
H^\vee=\bigoplus_{n\ge 0} H^\vee_n
\]
be graded connected Hopf algebras over $F$.
We write their structure maps as
\[
(H,m,u,\Delta,\varepsilon,S),
\qquad
(H^\vee,m^\vee,u^\vee,\Delta^\vee,\varepsilon^\vee,S^\vee).
\]
Thus $m,m^\vee$ are associative products, $\Delta,\Delta^\vee$ are coassociative coproducts, $\varepsilon,\varepsilon^\vee$ are counits, and
$S,S^\vee$ are antipodes satisfying
\begin{equation}\label{e:antipode-axioms-abstract}
m(S\otimes \id)\Delta=u\varepsilon,\qquad m(\id\otimes S)\Delta=u\varepsilon,
\end{equation}
and similarly for $S^\vee$ with $(m^\vee,\Delta^\vee,u^\vee,\varepsilon^\vee)$.

We assume that each graded piece $H_n$ (equivalently $H^\vee_n$) is finite dimensional.
This ensures that dual bases exist degreewise and that the canonical (Cauchy) element below is well defined in a degree completion.
(All sums in this chapter are degreewise finite.)

When infinite homogeneous sums occur, we work in degree-completed tensor products.
For example, for $r,s\ge 0$ we set
\[
\widehat{H^{\otimes r}\otimes (H^\vee)^{\otimes s}}
:=\prod_{n\ge 0}\bigl(H^{\otimes r}\otimes (H^\vee)^{\otimes s}\bigr)_n,
\]
where $\bigl(\cdot\bigr)_n\subset H^{\otimes r}\otimes (H^\vee)^{\otimes s}$ denotes the homogeneous subspace of total degree $n$
(with respect to the given gradings).
All Hopf structure maps extend continuously to these completions.

\begin{defn}[Hopf pairing and dual bases]\label{d:hopf-pairing-dual}
A bilinear form $\langle\cdot,\cdot\rangle: H\times H^\vee\to F$ is a graded nondegenerate \emph{Hopf pairing} if
\[
\langle H_i,H^\vee_j\rangle=0\ \ (i\neq j),
\]
the restriction $H_n\times H^\vee_n\to F$ is nondegenerate for every $n$, and
\begin{equation}\label{e:hopf-pairing-dual}
\langle ab,c\rangle=\langle a\otimes b,\ \Delta^\vee(c)\rangle,
\qquad
\langle a,bc\rangle=\langle \Delta(a),\ b\otimes c\rangle
\quad
(a,b\in H,\ c\in H^\vee).
\end{equation}
Here the pairing on tensor products is the tensor product pairing defined below.

A pair of homogeneous bases $\{f_\lambda\}_{\lambda\in\mathcal I}\subset H$ and
$\{g_\lambda\}_{\lambda\in\mathcal I}\subset H^\vee$ is called a \emph{Hopf-dual pair} if
\[
\langle f_\lambda,g_\mu\rangle=\delta_{\lambda\mu}
\qquad(\lambda,\mu\in\mathcal I),
\]
and $\deg(f_\lambda)=\deg(g_\lambda)$ for all $\lambda$.
(The index set $\mathcal I$ is abstract; in symmetric-function applications one typically takes $\mathcal I=\mathscr P$.)
\end{defn}

\begin{defn}[Tensor product pairing]\label{d:tensor-pairing-dual}
For $r\ge 1$ we equip $H^{\otimes r}$ and $(H^\vee)^{\otimes r}$ with the tensor product pairing
\[
\langle a_1\otimes\cdots\otimes a_r,\ c_1\otimes\cdots\otimes c_r\rangle
:=\prod_{i=1}^r \langle a_i,c_i\rangle,
\]
extended by bilinearity.
This pairing is graded and nondegenerate on each homogeneous component.
\end{defn}

\section{Skew elements}

From now on, fix a Hopf pairing $\langle\cdot,\cdot\rangle$ and a Hopf-dual pair of homogeneous bases
$\{f_\lambda\}_{\lambda\in\mathcal I}\subset H$ and $\{g_\lambda\}_{\lambda\in\mathcal I}\subset H^\vee$.

\begin{defn}[Skew elements]\label{d:skew-elements-dual}
For each $\lambda\in\mathcal I$ define the skew elements $f_{\lambda/\mu}\in H$ and $g_{\lambda/\mu}\in H^\vee$ by the coproduct expansions
\begin{equation}\label{e:skew-via-coprod-dual}
\Delta(f_\lambda)=\sum_{\mu\in\mathcal I} f_{\lambda/\mu}\otimes f_\mu,
\qquad
\Delta^\vee(g_\lambda)=\sum_{\mu\in\mathcal I} g_{\lambda/\mu}\otimes g_\mu.
\end{equation}
(Each sum is finite in each homogeneous degree, hence defines an element of $\widehat{H\otimes H}$ or $\widehat{H^\vee\otimes H^\vee}$.)
\end{defn}

\begin{lem}\label{l:skew-structure-constants-dual}
Define structure constants $\tilde c^{\mu}_{\nu,\tau}\in F$ by
\begin{equation}\label{e:mult-consts-dual}
f_\nu f_\tau=\sum_{\mu\in\mathcal I}\tilde c^{\mu}_{\nu,\tau}\,f_\mu.
\end{equation}
Then the coproduct of $g_\mu$ expands as
\begin{equation}\label{e:coprod-g-dual}
\Delta^\vee(g_\mu)=\sum_{\nu,\tau\in\mathcal I}\tilde c^{\mu}_{\nu,\tau}\,g_\nu\otimes g_\tau,
\end{equation}
and consequently
\begin{equation}\label{e:skew-g-linear-dual}
g_{\mu/\tau}=\sum_{\nu\in\mathcal I}\tilde c^{\mu}_{\nu,\tau}\,g_\nu.
\end{equation}
Likewise, if $\tilde d^{\mu}_{\nu,\tau}$ is defined by
\[
g_\nu g_\tau=\sum_{\mu\in\mathcal I}\tilde d^{\mu}_{\nu,\tau}\,g_\mu,
\]
then
\[
\Delta(f_\mu)=\sum_{\nu,\tau\in\mathcal I}\tilde d^{\mu}_{\nu,\tau}\,f_\nu\otimes f_\tau
\qquad\text{and}\qquad
f_{\mu/\tau}=\sum_{\nu\in\mathcal I}\tilde d^{\mu}_{\nu,\tau}\,f_\nu.
\]
\end{lem}

\begin{proof}
Equip $H\otimes H$ and $H^\vee\otimes H^\vee$ with the tensor product pairing from Definition~\ref{d:tensor-pairing-dual}.
For any $\alpha,\beta\in\mathcal I$, the Hopf pairing axiom \eqref{e:hopf-pairing-dual} gives
\[
\big\langle f_\alpha\otimes f_\beta,\ \Delta^\vee(g_\mu)\big\rangle
=\big\langle f_\alpha f_\beta,\ g_\mu\big\rangle.
\]
Expanding $f_\alpha f_\beta$ via \eqref{e:mult-consts-dual} and using duality yields
\[
\big\langle f_\alpha f_\beta,\ g_\mu\big\rangle
=\Big\langle \sum_{\gamma\in\mathcal I}\tilde c^{\gamma}_{\alpha,\beta} f_\gamma,\ g_\mu\Big\rangle
=\tilde c^{\mu}_{\alpha,\beta}.
\]
Thus
\[
\big\langle f_\alpha\otimes f_\beta,\ \Delta^\vee(g_\mu)\big\rangle=\tilde c^{\mu}_{\alpha,\beta}
\qquad(\alpha,\beta\in\mathcal I).
\]
Since $\{f_\alpha\otimes f_\beta\}$ and $\{g_\alpha\otimes g_\beta\}$ are dual bases for the tensor product pairing,
the coefficients of $\Delta^\vee(g_\mu)$ in the basis $\{g_\nu\otimes g_\tau\}$ are uniquely determined by these matrix entries,
and \eqref{e:coprod-g-dual} follows.

Finally, compare \eqref{e:coprod-g-dual} with the defining skew expansion
\[
\Delta^\vee(g_\mu)=\sum_{\tau\in\mathcal I} g_{\mu/\tau}\otimes g_\tau.
\]
Taking coefficients of $g_\tau$ in the second tensor factor yields \eqref{e:skew-g-linear-dual}.
The remaining statements are proved similarly, exchanging the roles of $H$ and $H^\vee$.
\end{proof}

\begin{lem}[Skew sum rule]\label{l:sumrule-dual}
For all $\mu,\lambda\in\mathcal I$ one has
\[
\Delta(f_{\lambda/\mu})=\sum_{\nu\in\mathcal I} f_{\lambda/\nu}\otimes f_{\nu/\mu},
\qquad
\Delta^\vee(g_{\lambda/\mu})=\sum_{\nu\in\mathcal I} g_{\lambda/\nu}\otimes g_{\nu/\mu}.
\]
\end{lem}

\begin{proof}
We prove the identity for $f_{\lambda/\mu}$; the proof for $g_{\lambda/\mu}$ is analogous.
Fix $\lambda\in\mathcal I$.
By definition \eqref{e:skew-via-coprod-dual},
\begin{equation}\label{e:Delta-f-lambda-dual}
\Delta(f_\lambda)=\sum_{\nu\in\mathcal I} f_{\lambda/\nu}\otimes f_\nu.
\end{equation}
Apply $\id\otimes\Delta$ to \eqref{e:Delta-f-lambda-dual}:
\[
(\id\otimes\Delta)\Delta(f_\lambda)
=\sum_{\nu} f_{\lambda/\nu}\otimes \Delta(f_\nu).
\]
Using again \eqref{e:skew-via-coprod-dual} for $\Delta(f_\nu)$ gives
\begin{equation}\label{e:id-Delta-expansion-dual}
(\id\otimes\Delta)\Delta(f_\lambda)
=\sum_{\nu}\ \sum_{\mu} f_{\lambda/\nu}\otimes f_{\nu/\mu}\otimes f_\mu.
\end{equation}
On the other hand,
\[
(\Delta\otimes\id)\Delta(f_\lambda)
=\sum_{\mu} \Delta(f_{\lambda/\mu})\otimes f_\mu.
\]
Coassociativity $(\Delta\otimes\id)\Delta=(\id\otimes\Delta)\Delta$ implies that the right-hand side equals
\eqref{e:id-Delta-expansion-dual} in $H^{\otimes 3}$.
Comparing coefficients of $f_\mu$ in the third tensor factor yields
\[
\Delta(f_{\lambda/\mu})
=\sum_{\nu\in\mathcal I} f_{\lambda/\nu}\otimes f_{\nu/\mu},
\]
as claimed.
\end{proof}

\section{The Cauchy element and its grouplike factorization}

\begin{defn}[Cauchy element]\label{d:cauchy-dual}
The Cauchy element is the completed sum
\[
\E:=\sum_{\lambda\in\mathcal I} f_\lambda\otimes g_\lambda
\ \in\ \widehat{H\otimes H^\vee}.
\]
We use leg notation in tensor powers. For instance, in $\widehat{H\otimes H\otimes H^\vee}$ we set
\[
\E_{13}:=\sum_{\lambda} f_\lambda\otimes 1\otimes g_\lambda,
\qquad
\E_{23}:=\sum_{\lambda} 1\otimes f_\lambda\otimes g_\lambda,
\]
and in $\widehat{H\otimes H\otimes H^\vee\otimes H^\vee}$ we set
\[
\E_{14}:=\sum_{\lambda} f_\lambda\otimes 1\otimes 1\otimes g_\lambda,
\qquad
\E_{24}:=\sum_{\lambda} 1\otimes f_\lambda\otimes 1\otimes g_\lambda,
\]
etc.
\end{defn}

The following lemma presents the coefficient extraction for $\Delta$ and $\Delta^\vee$.
\begin{lem}\label{l:coprod-coeff-dual}
For every $\rho\in\mathcal I$ one has the coproduct expansions
\[
\Delta(f_\rho)=\sum_{\mu,\nu\in\mathcal I} \langle f_\rho,\ g_\mu g_\nu\rangle\; f_\mu\otimes f_\nu,
\qquad
\Delta^\vee(g_\rho)=\sum_{\mu,\nu\in\mathcal I} \langle f_\mu f_\nu,\ g_\rho\rangle\; g_\mu\otimes g_\nu.
\]
\end{lem}

\begin{proof}
Write $\Delta(f_\rho)=\sum_{\mu,\nu} A^{\rho}_{\mu\nu}\, f_\mu\otimes f_\nu$.
Pair both sides with $g_\mu\otimes g_\nu$ using the tensor product pairing (Definition~\ref{d:tensor-pairing-dual}).
Since $\{f_\mu\otimes f_\nu\}$ and $\{g_\mu\otimes g_\nu\}$ are dual bases for this pairing, we obtain
\[
A^{\rho}_{\mu\nu}
=\big\langle \Delta(f_\rho),\ g_\mu\otimes g_\nu\big\rangle.
\]
By the Hopf pairing axiom $\langle a,bc\rangle=\langle \Delta(a), b\otimes c\rangle$ in \eqref{e:hopf-pairing-dual},
\[
\big\langle \Delta(f_\rho),\ g_\mu\otimes g_\nu\big\rangle
=\langle f_\rho,\ g_\mu g_\nu\rangle.
\]
Substituting $A^{\rho}_{\mu\nu}=\langle f_\rho,\ g_\mu g_\nu\rangle$ proves the first identity.
The second identity is proved similarly, expanding $\Delta^\vee(g_\rho)$ in the basis $\{g_\mu\otimes g_\nu\}$ and using
$\langle ab,c\rangle=\langle a\otimes b,\Delta^\vee(c)\rangle$.
\end{proof}

\begin{prop}[Grouplike identities]\label{p:grouplike-dual}
In $\widehat{H\otimes H\otimes H^\vee}$ and $\widehat{H\otimes H^\vee\otimes H^\vee}$ one has
\[
(\Delta\otimes \id)\E=\E_{13}\E_{23},
\qquad
(\id\otimes \Delta^\vee)\E=\E_{12}\E_{13}.
\]
\end{prop}

\begin{proof}
We prove $(\Delta\otimes \id)\E=\E_{13}\E_{23}$ by comparing coefficients in the basis
$\{f_\mu\otimes f_\nu\otimes g_\rho\}_{\mu,\nu,\rho\in\mathcal I}$.

\smallskip
\noindent\emph{Left-hand side.}
By definition,
\[
(\Delta\otimes \id)\E=\sum_{\rho\in\mathcal I}\Delta(f_\rho)\otimes g_\rho.
\]
Applying Lemma~\ref{l:coprod-coeff-dual} to $\Delta(f_\rho)$ yields
\[
(\Delta\otimes \id)\E
=\sum_{\rho}\ \sum_{\mu,\nu}
\langle f_\rho,\ g_\mu g_\nu\rangle\; f_\mu\otimes f_\nu\otimes g_\rho.
\]
Hence the coefficient of $f_\mu\otimes f_\nu\otimes g_\rho$ on the left-hand side equals
$\langle f_\rho,\ g_\mu g_\nu\rangle$.

\smallskip
\noindent\emph{Right-hand side.}
Compute
\[
\E_{13}\E_{23}
=\Big(\sum_{\mu} f_\mu\otimes 1\otimes g_\mu\Big)\Big(\sum_{\nu} 1\otimes f_\nu\otimes g_\nu\Big)
=\sum_{\mu,\nu} f_\mu\otimes f_\nu\otimes (g_\mu g_\nu).
\]
Expand $g_\mu g_\nu$ in the basis $\{g_\rho\}$:
\[
g_\mu g_\nu=\sum_{\rho}\langle f_\rho,\ g_\mu g_\nu\rangle\; g_\rho,
\]
which is the reconstruction identity in each homogeneous degree.
Substituting gives
\[
\E_{13}\E_{23}
=\sum_{\mu,\nu}\ \sum_{\rho}\langle f_\rho,\ g_\mu g_\nu\rangle\; f_\mu\otimes f_\nu\otimes g_\rho.
\]
Thus the coefficient of $f_\mu\otimes f_\nu\otimes g_\rho$ on the right-hand side is again
$\langle f_\rho,\ g_\mu g_\nu\rangle$, matching the left-hand side.
Therefore $(\Delta\otimes \id)\E=\E_{13}\E_{23}$. The proof of $(\id\otimes \Delta^\vee)\E=\E_{12}\E_{13}$ is analogous, using the second identity in
Lemma~\ref{l:coprod-coeff-dual}.
\end{proof}

For $m\ge 2$ and $1\le i\le m$, let $\Delta_i$ denote the map obtained by applying $\Delta$ in the $i$-th $H$-tensor factor and the identity elsewhere;
likewise, let $\Delta_i^\vee$ denote the map obtained by applying $\Delta^\vee$ in the $i$-th $H^\vee$-tensor factor and the identity elsewhere.
Since $\Delta$ and $\Delta^\vee$ are algebra morphisms, so are $\Delta_i$ and $\Delta_i^\vee$, and hence they extend continuously to completions.

\begin{prop}\label{p:grouplike-exp-rule-dual}
In $\widehat{H\otimes H\otimes H^\vee\otimes H^\vee}$ one has the four-factor decomposition
\begin{equation}\label{e:DeltaDeltaE-factor-dual}
(\Delta\otimes\Delta^\vee)\E
=\E_{13}\E_{14}\E_{23}\E_{24}.
\end{equation}
\end{prop}

\begin{proof}
Start from the first grouplike identity in Proposition~\ref{p:grouplike-dual}:
\[
(\Delta\otimes\id)\E=\E_{13}\E_{23}
\qquad\text{in }\widehat{H\otimes H\otimes H^\vee}.
\]
Apply $\Delta^\vee_3=\id\otimes\id\otimes\Delta^\vee$ to both sides.
Because $\Delta^\vee_3$ is an algebra morphism and extends to the completion, we obtain in
$\widehat{H\otimes H\otimes H^\vee\otimes H^\vee}$:
\[
(\Delta\otimes\Delta^\vee)\E
=\Delta^\vee_3\big((\Delta\otimes\id)\E\big)
=\Delta^\vee_3(\E_{13}\E_{23})
=\Delta^\vee_3(\E_{13})\,\Delta^\vee_3(\E_{23}).
\]
It remains to compute $\Delta^\vee_3(\E_{13})$ and $\Delta^\vee_3(\E_{23})$.
For this, use the second grouplike identity in Proposition~\ref{p:grouplike-dual}:
\[
(\id\otimes\Delta^\vee)\E=\E_{12}\E_{13}
\qquad\text{in }\widehat{H\otimes H^\vee\otimes H^\vee}.
\]
Embed $H\otimes H^\vee\otimes H^\vee$ into $H\otimes H\otimes H^\vee\otimes H^\vee$ by
\[
x\otimes y\otimes z\longmapsto x\otimes 1\otimes y\otimes z.
\]
Under this embedding, the identity becomes precisely
\[
\Delta^\vee_3(\E_{13})=\E_{13}\E_{14}.
\]
Similarly, embedding by
\[
x\otimes y\otimes z\longmapsto 1\otimes x\otimes y\otimes z
\]
gives
\[
\Delta^\vee_3(\E_{23})=\E_{23}\E_{24}.
\]
Therefore
\[
(\Delta\otimes\Delta^\vee)\E
=\big(\E_{13}\E_{14}\big)\big(\E_{23}\E_{24}\big)
=\E_{13}\E_{14}\E_{23}\E_{24},
\]
which is \eqref{e:DeltaDeltaE-factor-dual}.
\end{proof}

\section{Skew Cauchy identity and the skew Murnaghan--Nakayama identity}


\medskip
Fix indices $\mu,\lambda\in\mathcal I$. Define the $F$-linear map (partial contraction)
\[
\Cont_{\mu,\lambda}:\ H\otimes H\otimes H^\vee\otimes H^\vee\longrightarrow H\otimes H^\vee
\]
on pure tensors by
\[
\Cont_{\mu,\lambda}(a\otimes b\otimes c\otimes d)
:=\langle b,\ g_\mu\rangle\ \langle f_\lambda,\ d\rangle\ (a\otimes c),
\]
and extend linearly.
Since the pairing is graded and degreewise finite, $\Cont_{\mu,\lambda}$ respects the degree filtrations and extends continuously to
\[
\Cont_{\mu,\lambda}:\ \widehat{H\otimes H\otimes H^\vee\otimes H^\vee}\longrightarrow \widehat{H\otimes H^\vee}.
\]

\begin{prop}[Skew Cauchy identity]\label{p:skew-cauchy-dual}
For all $\lambda,\mu\in\mathcal I$, one has in $\widehat{H\otimes H^\vee}$:
\begin{equation}\label{e:skew-cauchy-dual}
\E\;\sum_{\tau\in\mathcal I} f_{\lambda/\tau}\otimes g_{\mu/\tau}
=\sum_{\rho\in\mathcal I} f_{\rho/\mu}\otimes g_{\rho/\lambda}.
\end{equation}
\end{prop}

\begin{proof}
By Proposition~\ref{p:grouplike-exp-rule-dual}, we have in $\widehat{H\otimes H\otimes H^\vee\otimes H^\vee}$ the factorization
\[
(\Delta\otimes\Delta^\vee)\E=\E_{13}\E_{14}\E_{23}\E_{24}.
\]
Apply $\Cont_{\mu,\lambda}$ to both sides.

\smallskip
\noindent\emph{Left-hand side.}
Using $\E=\sum_{\rho} f_\rho\otimes g_\rho$ and the skew coproduct expansions
\[
\Delta(f_\rho)=\sum_{\alpha} f_{\rho/\alpha}\otimes f_\alpha,
\qquad
\Delta^\vee(g_\rho)=\sum_{\beta} g_{\rho/\beta}\otimes g_\beta,
\]
we obtain
\[
(\Delta\otimes\Delta^\vee)\E
=\sum_{\rho}\ \sum_{\alpha,\beta}
\bigl(f_{\rho/\alpha}\otimes f_\alpha\otimes g_{\rho/\beta}\otimes g_\beta\bigr).
\]
Hence
\begin{align*}
\Cont_{\mu,\lambda}\bigl((\Delta\otimes\Delta^\vee)\E\bigr)
&=\sum_{\rho}\ \sum_{\alpha,\beta}
\langle f_\alpha,g_\mu\rangle\ \langle f_\lambda,g_\beta\rangle\
\bigl(f_{\rho/\alpha}\otimes g_{\rho/\beta}\bigr)\\
&=\sum_{\rho\in\mathcal I} f_{\rho/\mu}\otimes g_{\rho/\lambda},
\end{align*}
by duality $\langle f_\alpha,g_\mu\rangle=\delta_{\alpha\mu}$ and $\langle f_\lambda,g_\beta\rangle=\delta_{\lambda\beta}$.

\smallskip
\noindent\emph{Right-hand side.}
Write
\[
\E_{13}=\sum_{\alpha} f_\alpha\otimes 1\otimes g_\alpha\otimes 1,\quad
\E_{14}=\sum_{\beta} f_\beta\otimes 1\otimes 1\otimes g_\beta,
\]
\[
\E_{23}=\sum_{\gamma} 1\otimes f_\gamma\otimes g_\gamma\otimes 1,\quad
\E_{24}=\sum_{\delta} 1\otimes f_\delta\otimes 1\otimes g_\delta.
\]
Then
\[
\E_{13}\E_{14}\E_{23}\E_{24}
=\sum_{\alpha,\beta,\gamma,\delta}
(f_\alpha f_\beta)\otimes (f_\gamma f_\delta)\otimes (g_\alpha g_\gamma)\otimes (g_\beta g_\delta).
\]
Applying $\Cont_{\mu,\lambda}$ yields
\begin{equation}\label{e:cont-expand-dual}
\Cont_{\mu,\lambda}(\E_{13}\E_{14}\E_{23}\E_{24})
=\sum_{\alpha,\beta,\gamma,\delta}
\langle f_\gamma f_\delta,\ g_\mu\rangle\ \langle f_\lambda,\ g_\beta g_\delta\rangle\
\bigl(f_\alpha f_\beta\otimes g_\alpha g_\gamma\bigr).
\end{equation}

We now simplify the scalar factors using the Hopf pairing identities \eqref{e:hopf-pairing-dual} and the skew coproduct expansions.
First,
\[
\begin{aligned}
\langle f_\gamma f_\delta,\ g_\mu\rangle
&=\langle f_\gamma\otimes f_\delta,\ \Delta^\vee(g_\mu)\rangle\\
&=\Big\langle f_\gamma\otimes f_\delta,\ \sum_{\tau} g_{\mu/\tau}\otimes g_\tau\Big\rangle\\
&=\sum_{\tau}\langle f_\gamma,g_{\mu/\tau}\rangle\ \langle f_\delta,g_\tau\rangle
=\langle f_\gamma,\ g_{\mu/\delta}\rangle,
\end{aligned}
\]
since $\langle f_\delta,g_\tau\rangle=\delta_{\delta\tau}$.
Similarly,
\[
\langle f_\lambda,\ g_\beta g_\delta\rangle
=\langle \Delta(f_\lambda),\ g_\beta\otimes g_\delta\rangle
=\Big\langle \sum_{\sigma} f_{\lambda/\sigma}\otimes f_\sigma,\ g_\beta\otimes g_\delta\Big\rangle
=\sum_{\sigma}\langle f_{\lambda/\sigma},g_\beta\rangle\ \langle f_\sigma,g_\delta\rangle
=\langle f_{\lambda/\delta},\ g_\beta\rangle.
\]
Substituting these into \eqref{e:cont-expand-dual} gives
\[
\Cont_{\mu,\lambda}(\E_{13}\E_{14}\E_{23}\E_{24})
=\sum_{\alpha,\beta,\gamma,\delta}
\langle f_\gamma,\ g_{\mu/\delta}\rangle\ \langle f_{\lambda/\delta},\ g_\beta\rangle\
\bigl(f_\alpha f_\beta\otimes g_\alpha g_\gamma\bigr).
\]

Now use the dual-basis reconstruction in each homogeneous degree:
for any $x\in H$, $x=\sum_{\beta}\langle x,g_\beta\rangle f_\beta$, and for any $y\in H^\vee$, $y=\sum_{\gamma}\langle f_\gamma,y\rangle g_\gamma$.
Thus
\[
\sum_{\beta}\langle f_{\lambda/\delta},g_\beta\rangle f_\beta=f_{\lambda/\delta},
\qquad
\sum_{\gamma}\langle f_\gamma,g_{\mu/\delta}\rangle g_\gamma=g_{\mu/\delta}.
\]
Therefore
\begin{multline*}
\Cont_{\mu,\lambda}(\E_{13}\E_{14}\E_{23}\E_{24})
=\sum_{\alpha,\delta} (f_\alpha f_{\lambda/\delta})\otimes (g_\alpha g_{\mu/\delta})\\
=\Big(\sum_{\alpha} f_\alpha\otimes g_\alpha\Big)\Big(\sum_{\delta} f_{\lambda/\delta}\otimes g_{\mu/\delta}\Big)
=\E\;\sum_{\delta\in\mathcal I} f_{\lambda/\delta}\otimes g_{\mu/\delta}.
\end{multline*}

Finally, applying $\Cont_{\mu,\lambda}$ to $(\Delta\otimes\Delta^\vee)\E=\E_{13}\E_{14}\E_{23}\E_{24}$ yields
\[
\sum_{\rho\in\mathcal I} f_{\rho/\mu}\otimes g_{\rho/\lambda}
=\E\;\sum_{\tau\in\mathcal I} f_{\lambda/\tau}\otimes g_{\mu/\tau},
\]
which is exactly \eqref{e:skew-cauchy-dual}.
\end{proof}

\begin{lem}[Orthogonality]\label{l:orthogonality-dual}
For all $\eta,\tau\in\mathcal I$ one has
\begin{equation}\label{e:orthogonality-dual}
\sum_{\mu\in\mathcal I} S^\vee(g_{\eta/\mu})\,g_{\mu/\tau}
=\delta_{\eta,\tau}\,1_{H^\vee}.
\end{equation}
\end{lem}

\begin{proof}
By the skew sum rule (Lemma~\ref{l:sumrule-dual}) applied to $g_{\eta/\tau}$, we have
\[
\Delta^\vee(g_{\eta/\tau})=\sum_{\mu\in\mathcal I} g_{\eta/\mu}\otimes g_{\mu/\tau}.
\]
Apply $m^\vee(S^\vee\otimes \id)$ to both sides.
Using the antipode axiom $m^\vee(S^\vee\otimes \id)\Delta^\vee=u^\vee\varepsilon^\vee$, we obtain
\[
\sum_{\mu} S^\vee(g_{\eta/\mu})\,g_{\mu/\tau}
=u^\vee\varepsilon^\vee(g_{\eta/\tau})
=\varepsilon^\vee(g_{\eta/\tau})\cdot 1_{H^\vee}.
\]
It remains to compute $\varepsilon^\vee(g_{\eta/\tau})$.
Apply $(\varepsilon^\vee\otimes\id)$ to the defining skew expansion
\[
\Delta^\vee(g_\eta)=\sum_{\tau\in\mathcal I} g_{\eta/\tau}\otimes g_\tau.
\]
Since $(\varepsilon^\vee\otimes\id)\Delta^\vee=\id$, we get
\[
g_\eta=\sum_{\tau\in\mathcal I} \varepsilon^\vee(g_{\eta/\tau})\,g_\tau.
\]
By linear independence of the basis $\{g_\tau\}$, this forces $\varepsilon^\vee(g_{\eta/\tau})=\delta_{\eta,\tau}$.
Substituting into the previous equation gives \eqref{e:orthogonality-dual}.
\end{proof}

\begin{thm}[Skew Murnaghan--Nakayama rule]\label{t:skew-MN-dual}
We assume that $H^\vee$ is commutative. (Equivalently, $H$ is cocommutative under Hopf duality.) For all $\lambda,\eta\in\mathcal I$, one has in $\widehat{H\otimes H^\vee}$:
\begin{equation}\label{e:skew-MN-dual}
\E\;\bigl(f_{\lambda/\eta}\otimes 1_{H^\vee}\bigr)
=\sum_{\rho,\mu\in\mathcal I} f_{\rho/\mu}\otimes S^\vee(g_{\eta/\mu})\,g_{\rho/\lambda}.
\end{equation}
\end{thm}

\begin{proof}
Start from the skew Cauchy identity (Proposition~\ref{p:skew-cauchy-dual}), which is valid for every $\mu\in\mathcal I$:
\[
\E\;\sum_{\tau\in\mathcal I} f_{\lambda/\tau}\otimes g_{\mu/\tau}
=\sum_{\rho\in\mathcal I} f_{\rho/\mu}\otimes g_{\rho/\lambda}.
\]
Multiply both sides on the left by $1\otimes S^\vee(g_{\eta/\mu})$ and sum over $\mu\in\mathcal I$.
Since $H^\vee$ is commutative, the element $1\otimes S^\vee(g_{\eta/\mu})$ commutes with $\E\in H\otimes H^\vee$, hence
\begin{align*}
\sum_{\mu} (1\otimes S^\vee(g_{\eta/\mu}))\,
\E\;\sum_{\tau} f_{\lambda/\tau}\otimes g_{\mu/\tau}
&=
\E\;\sum_{\tau} f_{\lambda/\tau}\otimes \sum_{\mu} S^\vee(g_{\eta/\mu})\,g_{\mu/\tau}.
\end{align*}
By Lemma~\ref{l:orthogonality-dual}, the inner sum equals $\delta_{\eta,\tau}\,1_{H^\vee}$, so the left-hand side becomes
\[
\E\;\sum_{\tau} f_{\lambda/\tau}\otimes \delta_{\eta,\tau}\,1_{H^\vee}
=\E\;\bigl(f_{\lambda/\eta}\otimes 1_{H^\vee}\bigr).
\]
On the other hand, the right-hand side becomes
\[
\sum_{\mu} (1\otimes S^\vee(g_{\eta/\mu}))\;\sum_{\rho} f_{\rho/\mu}\otimes g_{\rho/\lambda}
=\sum_{\rho,\mu} f_{\rho/\mu}\otimes S^\vee(g_{\eta/\mu})\,g_{\rho/\lambda}.
\]
Equating both sides yields \eqref{e:skew-MN-dual}.
\end{proof}

\begin{cor}[Straight special case]\label{c:straight-MN-dual}
Let $\varnothing\in\mathcal I$ denote the index of the unit elements $f_{\varnothing}=1_H$ and $g_{\varnothing}=1_{H^\vee}$.
Then the straight version of \eqref{e:skew-MN-dual} is
\begin{equation}\label{e:straight-dual}
\E\;\bigl(f_{\lambda}\otimes 1_{H^\vee}\bigr)
=\sum_{\rho\in\mathcal I} f_{\rho}\otimes g_{\rho/\lambda}.
\end{equation}
\end{cor}

\begin{proof}
Take $\eta=\varnothing$ in \eqref{e:skew-MN-dual}.
Since $\Delta^\vee(1_{H^\vee})=1\otimes 1$, the skew elements satisfy $g_{\varnothing/\mu}=0$ unless $\mu=\varnothing$, and
$g_{\varnothing/\varnothing}=1_{H^\vee}$.
Hence $S^\vee(g_{\varnothing/\mu})=\delta_{\mu,\varnothing}\,1_{H^\vee}$ and the sum over $\mu$ collapses to $\mu=\varnothing$:
\[
\E\;(f_{\lambda}\otimes 1_{H^\vee})
=\sum_{\rho} f_{\rho/\varnothing}\otimes g_{\rho/\lambda}
=\sum_{\rho} f_{\rho}\otimes g_{\rho/\lambda},
\]
which is \eqref{e:straight-dual}.
\end{proof}

\chapter{Classical special cases: \texorpdfstring{$H=H^\vee=\Lambda_F$}{}}\label{s:special-cases-Lambda}

\section{Hopf algebra of symmetric function} 
Throughout this chapter, $F$ denotes a field of characteristic zero and we set $\Lambda_F=\Lambda\otimes_{\mathbb Q}F$.
We regard $\Lambda_F$ as a graded vector space $\Lambda_F=\bigoplus_{n\ge0}\Lambda_F^n$. Let us recall the standard Hopf structure on $\Lambda_{F}$.
\begin{defn}[Hopf algebra structure on $\Lambda_F$]\label{d:hopf-structure-LambdaF}
We view $\Lambda_F$ as a graded connected Hopf algebra
\[
(\Lambda_F,m,u,\Delta,\varepsilon,S),
\]
where $\Lambda_F=\bigoplus_{n\ge0}\Lambda_F^n$ is graded by total degree and $\Lambda_F^0=F\cdot 1$.

\noindent\textit{(1) Product and unit.}
The product is the usual multiplication
\[
m:\Lambda_F\otimes \Lambda_F\to \Lambda_F,\qquad m(f\otimes g)=fg,
\]
and the unit map is
\[
u:F\to \Lambda_F,\qquad u(c)=c\cdot 1.
\]

\noindent\textit{(2) Coproduct.}
The coproduct is the unique graded algebra morphism
\[
\Delta:\Lambda_F\to \Lambda_F\otimes \Lambda_F
\]
determined by
\[
\Delta(p_n)=p_n\otimes 1+1\otimes p_n\qquad(n\ge1),
\]
where $\{p_n\}_{n\ge1}$ are the power-sum symmetric functions.

Equivalently, for any alphabet decomposition $X+Y$ and any $f\in\Lambda_F$, one has
\begin{equation}\label{e:coproduct-pleth-def}
(\Delta f)[X,Y]=f[X+Y],
\end{equation}
where if $\Delta f=\sum_i f_i^{(1)}\otimes f_i^{(2)}$ then
\[
(\Delta f)[X,Y]:=\sum_i f_i^{(1)}[X]\;f_i^{(2)}[Y].
\]

\noindent\textit{(3) Counit.}
The counit is the graded algebra morphism
\[
\varepsilon:\Lambda_F\to F,\qquad \varepsilon(f)=f[\mathbf 0],
\]
where $\mathbf 0:=0+0+\cdots$ is the zero alphabet.

\noindent\textit{(4) Antipode.}
The antipode is the unique graded anti-algebra morphism
\[
S:\Lambda_F\to \Lambda_F
\]
determined by
\[
S(p_n)=-p_n\qquad(n\ge1).
\]
Equivalently, for any $f\in\Lambda_F$,
\begin{equation}\label{e:antipode-pleth-def}
(Sf)[X]=f[-X].
\end{equation}

\noindent\textit{(5) Hopf axioms.}
The maps $(m,u,\Delta,\varepsilon,S)$ satisfy the standard Hopf algebra axioms:
$\Delta$ is coassociative, $\varepsilon$ is a counit, and $S$ satisfies
\[
m(S\otimes \id)\Delta=u\varepsilon
\qquad\text{and}\qquad
m(\id\otimes S)\Delta=u\varepsilon.
\]
\end{defn}

If we restrict ourselves to $H=H^\vee=\Lambda_F$, then the condition for Hopf pairing becomes
\begin{equation}\label{e:hopf-pairing-dual-Lambda}
\langle ab,c\rangle=\langle a\otimes b,\ \Delta(c)\rangle,
\qquad
\langle a,bc\rangle=\langle \Delta(a),\ b\otimes c\rangle
\quad
(a,b,c\in \Lambda_F),
\end{equation} 
and the Cauchy element becomes\footnote{The index set becomes the set of partition $\P$}
\[
\E := \sum_{\lambda\in\P} f_\lambda\otimes g_\lambda,
\quad\text{and}\quad
\E[X|Y]:=\sum_{\lambda\in\P} f_\lambda[X]\;g_\lambda[Y].
\]

We use the results in the previous chapter to obtain:
\begin{thm}[Skew Murnaghan--Nakayama for $\Lambda_F$]\label{t:skew-M-N-Lambda}
For $\eta\subseteq\lambda$,
\begin{align}\label{e:skew-M-N-Lambda}
\E[X|Y]\;f_{\lambda/\eta}[X]
=\sum_{\rho,\mu\in\P} g_{\rho/\lambda}[Y]\;g_{\eta/\mu}[-Y]\;f_{\rho/\mu}[X].
\end{align}
\end{thm}

Define the adjoint operator $ f^{\perp} $ (with respect to $ \langle\cdot,\cdot\rangle $) by
\[
\langle f g,h\rangle=\langle g,f^{\perp}h\rangle\qquad\text{for all }g,h,
\]
and set
\[
\E^{\perp}[X|Y]:=\sum_{\la\in\P} f^{\perp}_{\la}\;g_{\la}[Y],
\]
where $f^{\perp}_{\la}$ acts on symmetric functions in the $X$-alphabet.

\begin{cor}
The straight Murnaghan--Nakayama rule and its dual version hold:
\begin{align}
\label{e:straight-Lambda}
\E[X|Y]\;f_{\la}[X]&=\sum_{\rho} g_{\rho/\la}[Y]\;f_{\rho}[X],\\
\label{e:straight-dual-Lambda}
\E^{\perp}[X|Y]\;g_{\rho}[X]&=\sum_{\la} g_{\rho/\la}[Y]\;g_{\la}[X].
\end{align}
\end{cor}

\begin{proof}
Taking $ \eta=\varnothing $ in \eqref{e:skew-M-N-Lambda} gives \eqref{e:straight-Lambda}. Using duality,
\[
g_{\rho/\la}[Y]
=\big\langle \E[X|Y]\,f_{\la}[X],\,g_{\rho}[X]\big\rangle
=\big\langle f_{\la}[X],\,\E^{\perp}[X|Y]\,g_{\rho}[X]\big\rangle,
\]
which implies \eqref{e:straight-dual-Lambda}.
\end{proof}

Fix nonzero scalars $\beta_n\in F^\times$ for $n\ge1$, put
$\beta_\la:=\prod_{j=1}^{\ell(\la)}\beta_{\la_j}$, and equip $\Lambda_F$ with the graded bilinear form
\[
\langle p_{\la},p_{\mu}\rangle_{\beta}
=\delta_{\la,\mu}\,z_{\la}\beta_\la.
\]
This form is nondegenerate and is a Hopf pairing. Indeed, since $\Delta$ is an algebra morphism and $\Delta(p_n)=p_n\otimes 1+1\otimes p_n$, for every partition $\mu$ we have
\[
\Delta(p_\mu)
=\sum_{\nu\sqcup\tau=\mu}
\left(\prod_{r\ge1}\binom{m_r(\mu)}{m_r(\nu)}\right)
p_\nu\otimes p_\tau,
\]
where $\nu\sqcup\tau=\mu$ means that the multiset of parts of $\mu$ is the disjoint union of those of $\nu$ and $\tau$.
Using the tensor product pairing
$\langle a\otimes b,\ c\otimes d\rangle_{\beta}=\langle a,c\rangle_{\beta}\,\langle b,d\rangle_{\beta}$,
both sides of
\[
\langle p_\nu p_\tau,\ p_\mu\rangle_{\beta}
=\langle p_\nu\otimes p_\tau,\ \Delta(p_\mu)\rangle_{\beta},
\]
vanish unless $\nu\sqcup\tau=\mu$. In the remaining case, the identity follows from $\beta_\mu=\beta_\nu\beta_\tau$ and
\[
\left(\prod_{r\ge1}\binom{m_r(\mu)}{m_r(\nu)}\right)z_\nu z_\tau
=z_\mu,
\]
which is immediate from $z_\lambda=\prod_{r\ge1}r^{m_r(\lambda)}m_r(\lambda)!$.
The second identity $\langle p_\nu,\ p_\tau p_\mu\rangle=\langle\Delta(p_\nu),\ p_\tau\otimes p_\mu\rangle$
is proved similarly.
Consequently, the above form is a nondegenerate graded Hopf pairing on $\Lambda_F$. In particular, the Hall inner product corresponds to $\beta_n=1$, while for the $q,t$-Hall inner product \eqref{e:innerprod} one has
$\beta_n=\dfrac{1-q^n}{1-t^n}$.

Moreover, the Cauchy element has the uniform exponential form
\[
\E[X|Y]
=\sum_{\lambda\in\P}\frac{p_\lambda[X]p_\lambda[Y]}{z_\lambda\beta_\lambda}
=\exp\!\left(\sum_{n\ge1}\frac{1}{n\,\beta_n}\,p_n[X]\,p_n[Y]\right).
\]

\bigskip
We now start from the Macdonald specialization of \eqref{e:skew-M-N-Lambda}.

Let $\beta_n=\dfrac{1-q^n}{1-t^n}$ and $ (f_{\la},g_{\la})=(P_{\la}[X;q,t],Q_{\la}[X;q,t]) $ be the Macdonald polynomial. Then 
\begin{align}
    \E[X|Y]=\sum_{\la\in\P}P_{\la}[X;q,t]Q_{\la}[Y;q,t]
    =\exp\left(\sum_{n=1}^{\infty}\frac{1}{n}\frac{1-t^n}{1-q^n}p_n[X]p_n[Y]\right).
\end{align}
In this case, \eqref{e:skew-M-N-Lambda} becomes
\begin{multline}\label{e:ge-skew-Mac}
 \exp\left(\sum_{n=1}^{\infty}\frac{1}{n}\frac{1-t^n}{1-q^n}p_n[X]p_n[Y]\right)P_{\la/\eta}[X;q,t]\\
=\sum_{\rho,\mu}Q_{\rho/\la}[Y;q,t]\;Q_{\eta/\mu}[-Y;q,t]\;P_{\rho/\mu}[X;q,t].   
\end{multline}

\section{Skew Pieri rules.} \label{ss:skew-P}
We obtain skew Pieri-type expansions by specializing the $Y$-alphabet in \eqref{e:ge-skew-Mac} to one-variable (virtual) alphabets.
Here we use the standard plethystic convention that $\frac{1-q}{1-t}$ denotes the virtual alphabet $(1-q)(1+t+t^2+\cdots)$, so that
$p_n\!\left[\frac{1-q}{1-t}\right]=\frac{1-q^n}{1-t^n}$.
Choose $ Y=\frac{q-1}{1-t}z $, $ Y=\frac{1-q}{1-t}z $ and $ Y=z+0+0+\cdots $. Then \eqref{e:ge-skew-Mac} reduces respectively to
\begin{multline}\label{e:e}
   \exp\left(\sum_{n=1}^{\infty}\frac{-1}{n}p_n[X]z^n\right)P_{\la/\eta}[X;q,t]\\
   =\sum_{\rho,\mu}Q_{\rho/\la}\left[\frac{q-1}{1-t}z;q,t\right]Q_{\eta/\mu}\left[\frac{1-q}{1-t}z;q,t\right]P_{\rho/\mu}[X;q,t]\\
=\sum_{\rho,\mu}z^{|\rho/\la|+|\eta/\mu|}Q_{\rho/\la}\left[\frac{q-1}{1-t};q,t\right]Q_{\eta/\mu}\left[\frac{1-q}{1-t};q,t\right]P_{\rho/\mu}[X;q,t],  
\end{multline}

\begin{multline}\label{e:h}
   \exp\left(\sum_{n=1}^{\infty}\frac{1}{n}p_n[X]z^n\right)P_{\la/\eta}[X;q,t]\\
=\sum_{\rho,\mu}z^{|\rho/\la|+|\eta/\mu|}Q_{\rho/\la}\left[\frac{1-q}{1-t};q,t\right]Q_{\eta/\mu}\left[\frac{q-1}{1-t};q,t\right]P_{\rho/\mu}[X;q,t] 
\end{multline}
and
\begin{multline}\label{e:g}
  \exp\left(\sum_{n=1}^{\infty}\frac{1}{n}\frac{1-t^n}{1-q^n}p_n[X]z^n\right)P_{\la/\eta}[X;q,t]\\
=\sum_{\rho,\mu}z^{|\rho/\la|+|\eta/\mu|}Q_{\rho/\la}[1;q,t]Q_{\eta/\mu}[-1;q,t]P_{\rho/\mu}[X;q,t].
\end{multline}
We substitute $ z\rightarrow -z $ in \eqref{e:e}:
\begin{multline}\label{e:e2}
   \exp\left(\sum_{n=1}^{\infty}\frac{(-1)^{n-1}}{n}p_n[X]z^n\right)P_{\la/\eta}[X;q,t]\\
=\sum_{\rho,\mu}(-z)^{|\rho/\la|+|\eta/\mu|}Q_{\rho/\la}\left[\frac{q-1}{1-t};q,t\right]Q_{\eta/\mu}\left[\frac{1-q}{1-t};q,t\right]P_{\rho/\mu}[X;q,t].  
\end{multline}
Recall the generating functions of the elementary symmetric functions, complete symmetric functions, and one-row Macdonald functions:
\begin{align}
\label{e:gen-e}
     \sum_{k\geq 0}e_k(X)z^k=&\exp\left(\sum_{n\geq 1}\frac{(-1)^{n-1}}{n}p_n[X]z^n\right),\\\label{e:gen-h}
    \sum_{k\geq 0}h_k(X)z^k=&\exp\left(\sum_{n\geq 1}\frac{1}{n}p_n[X]z^n\right),\\\label{e:gene-g}
   \sum_{k\geq 0}g_k(X;q,t)z^k=&\exp\left(\sum_{n\geq 1}\frac{1}{n}\frac{1-t^n}{1-q^n}p_n[X]z^n\right). 
\end{align}
Substituting \eqref{e:gen-e} into \eqref{e:e2}, \eqref{e:gen-h} into \eqref{e:h}, and \eqref{e:gene-g} into \eqref{e:g}, respectively, and taking the coefficient of $ z^n $ implies that
\begin{multline}\label{e:eP}
e_n(X)P_{\la/\eta}(X;q,t)\\
=\sum_{\rho,\mu}(-1)^{|\rho/\la|+|\eta/\mu|}Q_{\rho/\la}\left[\frac{q-1}{1-t};q,t\right]Q_{\eta/\mu}\left[\frac{1-q}{1-t};q,t\right]P_{\rho/\mu}[X;q,t],
\end{multline}
\begin{align}\label{e:hP}
h_n(X)P_{\la/\eta}(X;q,t)=&\sum_{\rho,\mu}Q_{\rho/\la}\left[\frac{1-q}{1-t};q,t\right]Q_{\eta/\mu}\left[\frac{q-1}{1-t};q,t\right]P_{\rho/\mu}[X;q,t],\\\label{e:gP}
    g_n(X;q,t)P_{\la/\eta}(X;q,t)=&\sum_{\rho,\mu}Q_{\rho/\la}[1;q,t]Q_{\eta/\mu}[-1;q,t]P_{\rho/\mu}[X;q,t],
\end{align}
where, in all three sums, $ \rho\supset\la $, $ \mu\subset\eta $ and $ |\rho/\la|+|\eta/\mu|=n $.

We remark that these three formulas were first derived in \cite[Theorem 1.2]{War13} using the $ q $-binomial theorem for Macdonald polynomials \cite{LW11}. (The notation there differs; see \cite[(1.7a)–(1.7e)]{War13} for a dictionary.) The corresponding Hall--Littlewood cases can be found in \cite[Theorems 2--4]{KL13}. Next, we treat the Schur and Schur $ Q $ cases. For this, we need the following specializations.

\begin{lem}\label{l:specializations}
    Let $ \mu\subset\la $ be two partitions. Then
   \begin{align}\label{e:s[1]}
       s_{\la/\mu}[1]=&
       \begin{cases}
           1, &\text{if }\la/\mu\text{ is a horizontal strip},\\
           0, &\text{otherwise,}
       \end{cases}\\\label{e:s[-1]}
       s_{\la/\mu}[-1]=&
       \begin{cases}
           (-1)^{|\la/\mu|}, &\text{if }\la/\mu\text{ is a vertical strip},\\
           0, &\text{otherwise,}
       \end{cases}
   \end{align}
   If, in addition, $\la$ and $\mu$ are strict partitions, then
   \begin{align}\label{e:Q[1]}
       Q_{\la/\mu}[1]=&
       \begin{cases}
           2^{a(\la/\mu)}, &\text{if }\la/\mu\text{ is a horizontal strip},\\
           0, &\text{otherwise,}
       \end{cases}\\\label{e:Q[-1]}
       Q_{\la/\mu}[-1]=&
       \begin{cases}
           (-1)^{|\la/\mu|}2^{a(\la/\mu)}, &\text{if }\la/\mu\text{ is a horizontal strip},\\
           0, &\text{otherwise,}
       \end{cases}
   \end{align}
   where $ a(\la/\mu) $ is the number of integers $ i\geq 1 $ such that $ \la/\mu $ has a box in the $ i^{\rm th} $ column but not in the $ (i+1)^{\rm st} $ column.
\end{lem}

\begin{proof}
\textbf{Schur case.}
In one variable $ X=1 $, any SSYT of shape $ \la/\mu $ must be the all-$1$ filling.
Column-strictness forces at most one box per column, i.e.\ $ \la/\mu $ is a horizontal strip; then the all-$1$ tableau is unique, so $ s_{\la/\mu}[1]=1 $, and otherwise it is $0$.
By Lemma~\ref{l:plethysm} and $ \omega(s_{\la/\mu})=s_{\la^t/\mu^t} $,
\[
s_{\la/\mu}[-1]=(-1)^{|\la/\mu|}\,s_{\la^t/\mu^t}[1].
\]
This is nonzero if and only if $ \la^t/\mu^t $ is horizontal, i.e.\ if and only if $ \la/\mu $ is vertical, yielding the stated value.

\textbf{Schur $ Q $ case.}
With the shifted tableau definition, at $ X=1 $ a contributing tableau must be the all-$1$ filling of the shifted skew shape $ S(\la/\mu) $.
In this case, admissibility is equivalent to (after de-shifting) the ordinary skew shape $ \la/\mu $ being a horizontal strip.
When $\la/\mu$ is horizontal, the all-$1$ filling is admissible and its number of ribbon components equals the number of right ends, namely
$ \kappa(\la/\mu)=a(\la/\mu) $.
Hence $ Q_{\la/\mu}[1]=2^{a(\la/\mu)} $; otherwise it vanishes.
Using Lemma~\ref{l:plethysm} and the fact that $ \omega(Q_{\la/\mu})=Q_{\la/\mu} $,
\[
Q_{\la/\mu}[-1]=(-1)^{|\la/\mu|}\,\omega\!\big(Q_{\la/\mu}\big)[1]
= (-1)^{|\la/\mu|}\,Q_{\la/\mu}[1],
\]
and the claim follows from the $ [1] $-case.
\end{proof}

By Lemma \ref{l:specializations}, specializing \eqref{e:eP}–\eqref{e:gP} at $ (q,t)=(0,0) $ gives\footnote{At $(q,t)=(0,0)$, \eqref{e:hP} and \eqref{e:gP} coincide.}
\begin{align}\label{e:es}
    e_n(X)s_{\la/\eta}(X)=&\sum_{\rho,\mu}(-1)^{|\rho/\la|+|\eta/\mu|}s_{\rho/\la}[-1]s_{\eta/\mu}[1]s_{\rho/\mu}(X)
    \ =\ \sum_{\rho,\mu}(-1)^{|\eta/\mu|}s_{\rho/\mu}(X),\\\label{e:hs}
    h_n(X)s_{\la/\eta}(X)=&\sum_{\rho,\mu}s_{\rho/\la}[1]s_{\eta/\mu}[-1]s_{\rho/\mu}(X)
    \ =\ \sum_{\rho,\mu}(-1)^{|\eta/\mu|}s_{\rho/\mu}(X).
\end{align}
The sum in \eqref{e:es} (respectively, \eqref{e:hs}) runs over all $ \rho\supset\la $ and $ \mu\subset\eta $ such that $ |\rho/\la|+|\eta/\mu|=n $, with $ \rho/\la $ a vertical (respectively, horizontal) strip and $ \eta/\mu $ a horizontal (respectively, vertical) strip. 

Equations \eqref{e:hs} and \eqref{e:es} are the skew Pieri rules for Schur functions, first obtained in \cite[Theorem 3.2 and Corollary 3.3]{AM11} by a purely combinatorial method.

For the Schur $Q$ case, let $\la$ and $\eta$ be strict partitions. Specializing \eqref{e:gP} at $(q,t)=(0,-1)$ and applying Lemma~\ref{l:specializations} gives
\begin{multline}\label{e:gSP}
    q_n(X)P_{\la/\eta}(X)=\sum_{\rho,\mu}Q_{\rho/\la}[1]Q_{\eta/\mu}[-1]P_{\rho/\mu}(X)\\
    =\sum_{\rho,\mu}(-1)^{|\eta/\mu|}2^{a(\rho/\la)+a(\eta/\mu)}P_{\rho/\mu}(X).
\end{multline}
Here the sum runs over strict partitions $\rho\supset\la$ and $\mu\subset\eta$ such that $|\rho/\la|+|\eta/\mu|=n$ and both $\rho/\la$ and $\eta/\mu$ are horizontal strips, and $q_n(X):=g_n(X;0,-1)$ is the one-row Schur $Q$-function.

\section{Skew Murnaghan-Nakayama rules.} \label{ss:skew-MN}
\phantomsection\label{def:q-parameter}
Choose $Y=(\q-1)z$ in \eqref{e:ge-skew-Mac}\footnote{Here we view $\q$ as a parameter.}. Then
\begin{multline}\label{e:ge-skew-M-N}
 \exp\left(\sum_{n=1}^{\infty}\frac{1}{n}\frac{1-t^n}{1-q^n}p_n[X](\q^n-1)z^n\right)P_{\la/\eta}[X;q,t]\\
=\sum_{\rho,\mu}z^{|\rho/\la|+|\eta/\mu|}Q_{\rho/\la}[\q-1;q,t]Q_{\eta/\mu}[1-\q;q,t]P_{\rho/\mu}[X;q,t].   
\end{multline}
This is equivalent to
\begin{multline}\label{e:ge-skew-M-N2}
 \exp\left(\sum_{n=1}^{\infty}\frac{1}{n}\frac{1-t^n}{1-q^n}p_n[(\q-1)X]z^n\right)P_{\la/\eta}[X;q,t]\\
=\sum_{\rho,\mu}z^{|\rho/\la|+|\eta/\mu|}Q_{\rho/\la}[\q-1;q,t]Q_{\eta/\mu}[1-\q;q,t]P_{\rho/\mu}[X;q,t].   
\end{multline}
By \eqref{e:gene-g}, we have
\begin{align}
   \exp\left(\sum_{m=1}^{\infty}\frac{1}{m}\frac{1-t^m}{1-q^m}p_m[(\q-1)X]z^m\right)
   =\sum_{n\ge 0} g_n[(\q-1)X;q,t]\,z^n.
\end{align}
For $n\ge1$, the specialization $g_n[(\q-1)X;q,t]$ is divisible by $\q-1$, and we set
\[
\tilde g_n[(\q-1)X;q,t]:=\frac{1}{\q-1}g_n[(\q-1)X;q,t]\qquad(n\ge1).
\]
Equivalently,
\begin{align}\label{e:gene-g2}
\sum_{n\ge1}\tilde g_n[(\q-1)X;q,t]\,z^n
=\frac{\exp\left(\sum_{m=1}^{\infty}\frac{1}{m}\frac{1-t^m}{1-q^m}p_m[(\q-1)X]z^m\right)-1}{\q-1}.
\end{align}
Substituting \eqref{e:gene-g2} into \eqref{e:ge-skew-M-N2} and comparing the coefficient of $z^n$ gives $(n\geq1)$
\begin{align}\label{e:skew-M-N-Mac}
 \tilde{g}_n[(\q-1)X;q,t]P_{\la/\eta}[X;q,t]=\sum_{\rho,\mu}\frac{Q_{\rho/\la}[\q-1;q,t]Q_{\eta/\mu}[1-\q;q,t]}{\q-1}P_{\rho/\mu}[X;q,t]  
\end{align}
summed over all partitions $ \rho\supset\la $ and $ \mu\subset\eta $ such that $ |\rho/\la|+|\eta/\mu|=n $.

We remark that \eqref{e:skew-M-N-Mac} is the skew Murnaghan--Nakayama rule for Macdonald polynomials, which reduces to the straight Murnaghan--Nakayama rule \cite[Corollary 3.2]{JL24} when $\eta=\varnothing$. The normalization by $\frac{1}{\q-1}$ is required for the computation of irreducible characters (see also Chapter \ref{s:GF-Hecke}). Moreover, the special cases for Schur functions and Schur $Q$-functions admit clean combinatorial models. We now introduce these models.

A skew diagram is a {\em generalized ribbon} if it contains no $2\times 2$ blocks of boxes (no connectivity is required). Any generalized ribbon is a union of its connected components, each of which is a ribbon. If the generalized ribbon $\K$ has $m$ connected components $(\xi_1,\xi_2,\ldots,\xi_m)$, we define its weight by
$$
\wt_{\q}(\K)= (\q-1)^{m}\prod\limits_{i=1}^{m}(-1)^{r(\xi_{i})-1}\q^{c(\xi_i)-1}.
$$
Here $r(\xi_i)$ (resp.\ $c(\xi_i)$) denotes the number of rows (resp.\ columns) in the ribbon $\xi_i$. See Figure \ref{fig:generalized-ribbon} for an example.
\begin{figure}
    \centering
\begin{tikzpicture}[scale = 0.4]
  \begin{scope}
    \clip (0,0) -| (1,2) -| (3,3) -| (4,4) -| (5,4)  -| (7,5) -| (9,6) -| (6,5) -| (5,4) -| (2,3) -| (1,2) -| (0,0);
    \draw [color=black!25] (0,0) grid (9,6);
  \end{scope}
   \draw [thick] (0,0) -| (1,2) -| (3,3) -| (4,4) -| (5,4)  -| (7,5) -| (9,6) -| (6,5) -| (5,4) -| (2,3) -| (1,2) -| (0,0);
   \draw[thick] (7,6) -| (0,6) -| (0,2);
  \draw [thick, rounded corners] (0.5,0.5) -- (0.5,1.5);
  \draw [color=black,fill=black,thick] (0.5,1.5) circle (.5ex);
  \node [draw, circle, fill = white, inner sep = 1.5pt] at (0.5,0.5) { };
  \draw [thick, rounded corners] (1.5,2.5) -- (2.5,2.5) -- (2.5,3.5) -- (3.5,3.5);
  \draw [color=black,fill=black,thick] (3.5,3.5) circle (.5ex);
  \node [draw, circle, fill = white, inner sep = 1.5pt] at (1.5,2.5) { };
  \draw [thick, rounded corners] (5.5,4.5) -- (6.5,4.5) -- (6.5,5.5) -- (8.5,5.5);
  \draw [color=black,fill=black,thick] (8.5,5.5) circle (.5ex);
  \node [draw, circle, fill = white, inner sep = 1.5pt] at (5.5,4.5) { };
 \end{tikzpicture}
 \caption{$\la=(9,7,4,3,1,1)$, $\mu=(6,5,2,1)$. This generalized ribbon has three connected components with weight $-\q^5(\q-1)^3$.} \label{fig:generalized-ribbon}
 \end{figure}
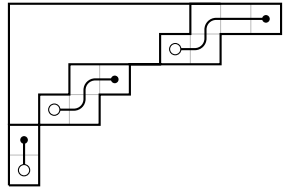

Let $\pi=\la/\mu$ be a skew shape determined by strict partitions $\mu\subset\la$. It is a {\em generalized double strip} if, in the shifted diagram $S(\pi)$, every NW--SE diagonal contains at most two boxes.
Such a $\pi$ decomposes (not necessarily into connected pieces) into two parts: 
$\alpha(\pi)$, consisting of all diagonals of length $2$, and $\beta(\pi)$, consisting of the remaining boxes (those lying on diagonals of length $1$). 
Then $\beta(\pi)$ is a union of its connected components, each of which is a ribbon. Denote by $m(\pi)$ the number of connected components of $\beta(\pi)$.

Let $\tau=(\tau_1,\tau_2,\cdots,\tau_a)$ and $\rho=(\rho_1,\rho_2,\cdots,\rho_b)$ be two compositions. We say that $\rho$ is a {\it refinement} of $\tau$, denoted $\rho\prec\tau$, if there exists $i_0=0< i_1< i_2<\cdots< i_{a-1}< b=i_{a}$ such that
$\tau_j=\rho_{i_{j-1}+1}+\rho_{i_{j-1}+2}+\cdots+\rho_{i_j}$ for $j=1,2,\cdots,a$. 
Equivalently, $\tau$ is a {\it coarsening} of $\rho$. In particular, $\tau\prec\tau$.
Note that a given composition has only finitely many coarsenings.

Given a generalized double strip $\la/\mu$, suppose $\beta(\la/\mu)$ has $m$ connected components denoted by $\gamma^{(1)},\cdots,\gamma^{(m)}$. We define its weight by
\begin{align}
\tilde{\wt}_{\q}(\la/\mu)=(-\q)^{\frac{|\alpha(\la/\mu)|}{2}}2^{\delta_{\ell(\la),\ell(\mu)+2}}\prod_{i=1}^{m}\left(\sum_{\gamma^{(i)}\prec\tau}(-1)^{\ell(\gamma^{(i)})-\ell(\tau)}d_{\tau}\right),
\end{align}
where $d_{\tau}=d_{\tau_1}d_{\tau_2}\cdots$ and $d_k=2(\q-1)\frac{\q^k-(-1)^k}{\q+1}$. Figure \ref{fig:generalized-double} shows an example of a generalized double strip.
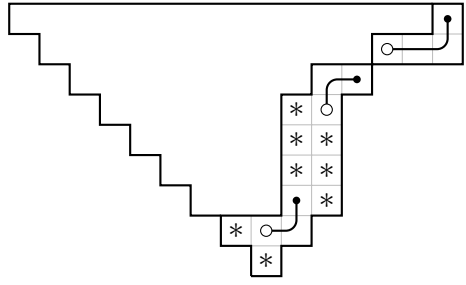
\begin{figure}
    \centering
\begin{tikzpicture}[scale = 0.4]
  \begin{scope}
    \clip (0,0) -| (1,1) -| (2,2) -| (3,6) -| (4,7)  -| (7,9) -| (6,8) -| (4,7) -| (2,6) -| (1,2) -| (1,2) -| (-1,1) -| (0,1) -| (0,0);
    \draw [color=black!25] (0,0) grid (7,9);
  \end{scope}
   \draw [thick] (0,0) -| (1,1) -| (2,2) -| (3,6) -| (4,7)  -| (7,9) -| (6,8) -| (4,7) -| (2,6) -| (1,2) -| (1,2) -| (-1,1) -| (0,1) -| (0,0);
   \draw[thick] (6,9) -| (-8,9) |- (-7,8) |- (-6,7) |- (-5,6) |- (-4,5) |- (-3,4) |- (-2,3) |- (-1,2);
  \draw [thick, rounded corners] (0.5,1.5) -- (1.5,1.5) -- (1.5,2.5);
  \draw [color=black,fill=black,thick] (1.5,2.5) circle (.5ex);
  \node [draw, circle, fill = white, inner sep = 1.5pt] at (0.5,1.5) { };
  \draw [thick, rounded corners] (2.5,5.5) -- (2.5,6.5) -- (3.5,6.5);
  \draw [color=black,fill=black,thick] (3.5,6.5) circle (.5ex);
  \node [draw, circle, fill = white, inner sep = 1.5pt] at (2.5,5.5) { };
  \draw [thick, rounded corners] (4.5,7.5) -- (6.5,7.5) -- (6.5,8.5);
  \draw [color=black,fill=black,thick] (6.5,8.5) circle (.5ex);
  \node [draw, circle, fill = white, inner sep = 1.5pt] at (4.5,7.5) { };
  \node at (0.5,0.5) {$*$}; \node at (-0.5,1.5) {$*$}; \node at (2.5,2.5) {$*$}; \node at (1.5,3.5) {$*$}; \node at (2.5,3.5) {$*$}; \node at (1.5,4.5) {$*$}; \node at (2.5,4.5) {$*$}; \node at (1.5,5.5) {$*$}; 
 \end{tikzpicture}
 \caption{$\la=(15,14,10,8,7,6,5,3,1)$, $\mu=(14,11,8,6,5,4,3)$. The boxes of $\alpha(\la/\mu)$ are marked by $*$. $\beta(\la/\mu)$ has three connected components. The weight of $\la/\mu$ is $(-\q)^4\cdot 2\cdot (d_2d_1-d_3)^2\cdot (d_1d_3-d_4)=16\q^4(\q-1)^6(\q^2-3\q+1)^2$.} \label{fig:generalized-double}
 \end{figure}

The following lemma provides a bridge from plethystic specializations to the combinatorial models above.
\begin{lem}\label{l:h-1}
    Let $\la/\mu$ be a skew diagram. Then
    \begin{align}
        s_{\la/\mu}[\q-1]=&
        \begin{cases}
            \wt_{\q}(\la/\mu), &\text{if $\la/\mu$ is a generalized ribbon},\\
            0, &\text{otherwise.}
        \end{cases}
    \end{align}
    If $\la$ and $\mu$ are strict partitions, then
    \begin{align}
        Q_{\la/\mu}[\q-1]=&
        \begin{cases}
           \tilde{\wt}_{\q}(\la/\mu), &\text{if $\la/\mu$ is a generalized double strip},\\
           0, &\text{otherwise.} 
        \end{cases}
    \end{align}
\end{lem}
\begin{proof}
    The proofs of these two identities can be found in \cite[Lemma 4.1]{JL24} and \cite[Corollary 3.10]{JL23}, respectively.
    For the first identity, see also \cite[Equation (6.7)]{HR95} and \cite[Equation 9]{Las05}.
\end{proof}

The skew quantum rules for Schur functions and Schur $Q$-functions are stated as follows.
\begin{prop}\label{p:skew-M-N-Schur-SchurQ}
Let $n\ge1$ and set
\[
\tilde{h}_n=\frac{1}{\q-1}h_n,
\qquad
\tilde{q}_n=\frac{1}{\q-1}q_n.
\]
For partitions $\eta\subset\la$, we have
\begin{align}\label{e:skew-M-N-Schur}
        \tilde{h}_n[(\q-1)X]s_{\la/\eta}(X)
        =\sum_{\rho,\mu}(-1)^{|\eta/\mu|}\frac{\wt_{\q}(\rho/\la)\wt_{\q}(\eta^t/\mu^t)}{\q-1}s_{\rho/\mu}(X),
\end{align}
where the sum ranges over all partitions $\rho,\mu$ such that both $\rho/\la$ and $\eta^t/\mu^t$ are generalized ribbons and $|\rho/\la|+|\eta/\mu|=n$.

If $\la$ and $\eta$ are strict partitions, then
\begin{align}\label{e:skew-M-N-Schur-Q}
        \tilde{q}_n[(\q-1)X]P_{\la/\eta}(X)
        =\sum_{\rho,\mu}(-1)^{|\eta/\mu|}\frac{\tilde{\wt}_{\q}(\rho/\la)\tilde{\wt}_{\q}(\eta/\mu)}{\q-1}P_{\rho/\mu}(X),
\end{align}
where the sum ranges over all strict partitions $\rho,\mu$ such that both $\rho/\la$ and $\eta/\mu$ are generalized double strips and $|\rho/\la|+|\eta/\mu|=n$.
\end{prop}
\begin{proof}
   These two identities follow by specializing \eqref{e:skew-M-N-Mac} at $(q,t)=(0,0)$ and $(q,t)=(0,-1)$, respectively, and then applying Lemma \ref{l:h-1}.
\end{proof}

We next record the classical limit of these two skew quantum rules as $\q\to1$.
For the Schur $Q$ case, following \cite[Remark~3.12]{JL23}, we single out the generalized double strips that survive in the limit.
A generalized double strip $\theta$ with connected $\beta(\theta)$ will be called a \emph{classical spin strip}.
If $\alpha(\theta)=\varnothing$, it is a shifted border strip; if $\alpha(\theta)\ne\varnothing$, it is a double strip.
For such a strip, set
\[
\pi(\theta):=
\begin{cases}
(-1)^{\rht(\beta(\theta))}, & \text{if $\theta$ is a shifted border strip},\\[1mm]
2(-1)^{\,|\alpha(\theta)|/2+\rht(\beta(\theta))},
& \text{if $\theta$ is a double strip}.
\end{cases}
\]

\begin{cor}\label{c:classical-limit-skew-MN}
Let $\lambda/\eta$ be a skew diagram and let $n\ge1$.
\begin{enumerate}
    \item The Schur specialization of \eqref{e:skew-M-N-Schur} at $\q=1$ gives
    \begin{align}\label{e:classical-skew-MN-Schur}
    p_n(X)s_{\lambda/\eta}(X)
    =
    \sum_{\rho}(-1)^{\rht(\rho/\lambda)}s_{\rho/\eta}(X)
    -
    \sum_{\mu}(-1)^{\rht(\eta/\mu)}s_{\lambda/\mu}(X),
    \end{align}
    where the first sum is over all $\rho\supset\lambda$ such that $\rho/\lambda$ is an $n$-ribbon, and the second sum is over all $\mu\subset\eta$ such that $\eta/\mu$ is an $n$-ribbon.

    \item Suppose that $\lambda$ and $\eta$ are strict partitions and that $n$ is odd.
    The Schur $P/Q$ specialization of \eqref{e:skew-M-N-Schur-Q} at $\q=1$ gives
    \begin{align}\label{e:classical-skew-MN-SchurQ}
    p_n(X)P_{\lambda/\eta}(X)
    =
    \sum_{\rho}\pi(\rho/\lambda)P_{\rho/\eta}(X)
    -
    \sum_{\mu}\pi(\eta/\mu)P_{\lambda/\mu}(X),
    \end{align}
    where the first sum is over all strict partitions $\rho\supset\lambda$ such that $\rho/\lambda$ is a classical spin strip of size $n$, and the second sum is over all strict partitions $\mu\subset\eta$ such that $\eta/\mu$ is a classical spin strip of size $n$.
\end{enumerate}
\end{cor}

\begin{proof}
The left-hand sides follow from the generating functions.  Indeed,
\[
\sum_{n\ge0}h_n[(\q-1)X]z^n
=
\exp\left(\sum_{r\ge1}\frac{\q^r-1}{r}p_r[X]z^r\right),
\]
hence
\[
\lim_{\q\to1}\tilde h_n[(\q-1)X]=p_n[X]\qquad(n\ge1).
\]
Similarly,
\[
\sum_{n\ge0}q_n[(\q-1)X]z^n
=
\exp\left(2\sum_{\substack{r\ge1\\ r\ {\rm odd}}}\frac{\q^r-1}{r}p_r[X]z^r\right),
\]
and therefore
\[
\lim_{\q\to1}\tilde q_n[(\q-1)X]
=
\begin{cases}
2p_n[X], & \text{if $n$ is odd},\\
0, & \text{if $n$ is even}.
\end{cases}
\]

For the Schur formula, the factor $\wt_{\q}(\theta)$ is divisible by $(\q-1)^{m(\theta)}$, where $m(\theta)$ is the number of connected components of the generalized ribbon $\theta$.
Thus, after division by $\q-1$, only the terms in which exactly one of $\rho/\lambda$ and $\eta^t/\mu^t$ is a connected ribbon survive.
If $\rho/\lambda$ is an $n$-ribbon and $\mu=\eta$, the limit gives the coefficient $(-1)^{\rht(\rho/\lambda)}$.
If $\rho=\lambda$ and $\eta/\mu$ is an $n$-ribbon, then
\[
(-1)^n(-1)^{\rht((\eta/\mu)^t)}
=-(-1)^{\rht(\eta/\mu)},
\]
which gives the second sum in \eqref{e:classical-skew-MN-Schur}.

For the Schur $P/Q$ formula, \cite[Remark~3.12]{JL23} gives
\[
\lim_{\q\to1}\frac{\tilde{\wt}_{\q}(\theta)}{\q-1}
=
\begin{cases}
2\pi(\theta), & \text{if $\theta$ is a classical spin strip of odd size},\\
0, & \text{otherwise}.
\end{cases}
\]
Consequently, after division by $\q-1$, only the terms in which exactly one of $\rho/\lambda$ and $\eta/\mu$ is a classical spin strip survive.
Since the left-hand side tends to $2p_n(X)P_{\lambda/\eta}(X)$ for odd $n$, division by $2$ gives \eqref{e:classical-skew-MN-SchurQ}; the minus sign in the second sum comes from $(-1)^{|\eta/\mu|}=-1$.
\end{proof}

\begin{rem}
\begin{enumerate}
    \item Equation \eqref{e:skew-M-N-Schur} was first established in \cite[Theorem~5]{Kon12}. The statement there differs slightly from ours; it can be recovered from our formula by replacing $\q-1$ with $1-\q$. This normalization is compatible with the character-theoretic applications in Chapter~\ref{s:GF-Hecke}. The identity \eqref{e:skew-M-N-Schur} specializes to several known formulas, including the skew Pieri rule for Schur functions \cite{AM11}, the skew Murnaghan--Nakayama rule for Schur functions \cite[Theorem~3]{Kon12}, and the Murnaghan--Nakayama rule for Hecke algebras of type~$A$ \cite{Ram91,Van91,JL22b}; see \cite[p.~526]{Kon12} for details. The limit \eqref{e:classical-skew-MN-Schur} is the usual skew Murnaghan--Nakayama rule, equivalently \eqref{e:skew-ps}.
    \item Equation \eqref{e:skew-M-N-Schur-Q} may be viewed as a skew quantum spin analogue of \eqref{e:skew-M-N-Schur}, extending the Murnaghan--Nakayama rule for Hecke--Clifford algebras \cite{JL23}. Its classical limit \eqref{e:classical-skew-MN-SchurQ} gives the corresponding skew spin version, with the shifted border-strip and double-strip model inherited from \cite[Remark~3.12]{JL23}.
    \item In \cite{LLS11}, Lam--Lauve--Sottile gave a general skew Littlewood--Richardson rule for any pair of dual Hopf algebras. However, evaluating the coefficients that appear in these expansions is computationally difficult, which led Konvalinka to ask in \cite[§7.3]{Kon12} for which special shapes of $ \la,\mu,\sigma,\tau $\footnote{See \cite[Thms.~3.2 and 4.1]{LLS11} for the notation and conventions on $ \la,\mu,\sigma,\tau $.} the coefficients could be computed explicitly. For Schur functions, \eqref{e:skew-M-N-Schur} and \eqref{e:es}--\eqref{e:hs} answer respectively the cases $ (\sigma,\tau)=(\varnothing,\text{hook}) $ and $ (\sigma,\tau)=(\varnothing,\text{one column or one row}) $. For the Schur $P/Q$ case, \eqref{e:skew-M-N-Schur-Q} and \eqref{e:gSP} address respectively the cases $ (\sigma,\tau)=(\varnothing,\text{two rows}) $\footnote{$\tilde{q}_n[(\q-1)X]$ expands into Schur $Q$-functions indexed by two-row partitions.} and $ (\sigma,\tau)=(\varnothing,\text{one row}) $.
\end{enumerate}
\end{rem}

\chapter{More special cases}\label{s:more}

Throughout this chapter, $F$ denotes a field of characteristic zero.
We specialize the general skew Murnaghan--Nakayama rule (Theorem~\ref{t:skew-MN-dual}) to several other Hopf--dual settings arising in algebraic combinatorics.
Each specialization is obtained by choosing a concrete dual Hopf pair $(H,H^\vee)$ together with a Hopf--dual pair of homogeneous bases, and then evaluating the resulting identity in alphabets $X$ and $Y$.
We focus on three families of examples: the dual pair $(\mathrm{NSym},\mathrm{QSym})$ with (noncommutative/quasisymmetric) Schur bases indexed by compositions, the $k$-bounded subalgebra $\Lambda_{(k)}$ and its graded Hopf dual $\Lambda^{(k)}$ with $k$-Schur and dual $k$-Schur bases indexed by $k$-bounded partitions, and the type $C$ affine Grassmannian Hopf--dual pair $(\Gamma_{(n)},\Gamma^{(n)})$ with Schubert bases indexed by affine Grassmannian elements.

\section{Noncommutative Schur functions}
\label{ss:NSym-QSym-skewMN}
The Hopf algebra $\mathrm{NSym}$ was introduced in~\cite{GKL+95}.
It is graded Hopf-dual to $\mathrm{QSym}$; see also~\cite{MR95} and the modern
treatment in~\cite[Ch.~5]{GR20}.

Let
\[
H=\mathrm{NSym},\qquad H^\vee=\mathrm{QSym}
\]
over $F$.  The Hopf algebra $\mathrm{NSym}$ is the free associative algebra
\[
\mathrm{NSym}\cong F\langle H_1,H_2,\dots\rangle,\qquad \deg(H_r)=r,
\]
with product given by concatenation and unit $1$.
Its coproduct is the algebra morphism determined by
\[
\Delta(H_r)=\sum_{i+j=r} H_i\otimes H_j,\qquad H_0:=1,
\]
and the counit is determined by $\varepsilon(H_0)=1$ and $\varepsilon(H_r)=0$ for $r>0$.
The antipode $S$ is the unique linear map satisfying the antipode axioms.

The Hopf algebra $\mathrm{QSym}$ is the commutative Hopf algebra of quasisymmetric functions.
We regard $\mathrm{QSym}$ as the graded Hopf dual of $\mathrm{NSym}$; in particular, $\mathrm{QSym}$ is commutative, so the
commutativity hypothesis in Theorem~\ref{t:skew-MN-dual} holds.

Let the index set be the set $\mathcal I=\Comp$ of finite compositions, including the empty composition $\varnothing$.
Let $\{M_\alpha\}_{\alpha\in\Comp}$ denote the monomial basis of $\mathrm{QSym}$ and write $H_\alpha:=H_{\alpha_1}\cdots H_{\alpha_\ell}\in\mathrm{NSym}$.
The canonical graded Hopf pairing $\langle\cdot,\cdot\rangle_{\mathrm NQ}:\mathrm{NSym}\times\mathrm{QSym}\to F$ is characterized by
\[
\langle H_\alpha,\ M_\beta\rangle_{\mathrm NQ}=\delta_{\alpha\beta}\qquad(\alpha,\beta\in\Comp).
\]

For the present specialization we fix a Hopf-dual pair of homogeneous bases
\[
f_\alpha=\mathbf{s}_\alpha\in\mathrm{NSym},\qquad
g_\alpha=\mathcal{S}_\alpha\in\mathrm{QSym}\qquad(\alpha\in\Comp),
\]
where $\{\mathbf{s}_\alpha\}$ is the \emph{noncommutative Schur} basis of~\cite{BLvW11}
and $\{\mathcal{S}_\alpha\}$ is the \emph{quasisymmetric Schur} basis of~\cite{HLMvW11},
dual under the above pairing:
\[
\langle \mathbf{s}_\alpha,\ \mathcal{S}_\beta\rangle_{\mathrm NQ}=\delta_{\alpha\beta}.
\]

The completed Cauchy element is
\[
\E^{\mathrm{NQ}}
=\sum_{\alpha\in\Comp}H_\alpha\otimes M_\alpha
=\sum_{\alpha\in\Comp}\mathbf{s}_\alpha\otimes \mathcal S_\alpha
\ \in\ \widehat{\mathrm{NSym}\otimes\mathrm{QSym}}.
\]
Fix evaluation maps (algebra morphisms) $\ev_X:\mathrm{NSym}\to A_X$ and $\ev_Y:\mathrm{QSym}\to A_Y$
into commutative $F$-algebras $A_X,A_Y$.
We write $\mathbf{s}_\alpha[X]:=\ev_X(\mathbf{s}_\alpha)$ and $\mathcal{S}_\alpha[Y]:=\ev_Y(\mathcal{S}_\alpha)$, and set
\[
\E^{\mathrm{NQ}}[X\mid Y]
:=\sum_{\alpha\in\Comp}\mathbf{s}_\alpha[X]\ \mathcal{S}_\alpha[Y].
\]

As before, define skew elements by the coproduct expansions
\[
\Delta(\mathbf{s}_\lambda)=\sum_{\eta\in\Comp}\mathbf{s}_{\lambda/\eta}\otimes \mathbf{s}_\eta,
\qquad
\Delta^\vee(\mathcal{S}_\lambda)=\sum_{\eta\in\Comp}\mathcal{S}_{\lambda/\eta}\otimes \mathcal{S}_\eta.
\]

Here $S^\vee:\mathrm{QSym}\to\mathrm{QSym}$ is the Hopf antipode, i.e.\ the convolution inverse of $\id$:
\[
m^\vee(S^\vee\otimes\id)\Delta^\vee=u^\vee\varepsilon^\vee
\quad\text{and}\quad
m^\vee(\id\otimes S^\vee)\Delta^\vee=u^\vee\varepsilon^\vee.
\]
The standard monomial-basis formula~\cite[Thm.~5.1.11]{GR20} is
\[
S^\vee(M_\alpha)=(-1)^{\ell(\alpha)}\sum_{\mathrm{rev}(\alpha)\prec\beta} M_\beta,
\]
where $\mathrm{rev}(\alpha)$ is the reversed composition and $\prec$ is the refinement
relation introduced above; thus $\mathrm{rev}(\alpha)\prec\beta$ means that $\beta$ is a
\emph{coarsening} of $\mathrm{rev}(\alpha)$ (obtained by summing adjacent parts).
Thus $S^\vee(\mathcal{S}_{\eta/\mu})[Y]$ means: first compute $S^\vee(\mathcal{S}_{\eta/\mu})$ inside $\mathrm{QSym}$,
then apply the evaluation $\ev_Y$.

\begin{thm}[Skew MN rule in $(\mathrm{NSym},\mathrm{QSym})$]\label{t:skew-MN-rule-NQ}
For $(H,H^\vee)=(\mathrm{NSym}, \mathrm{QSym})$, we have the following skew MN rule for noncommutative Schur functions:
\begin{equation}\label{e:skew-MN-rule-NQ}
\E^{\mathrm{NQ}}[X|Y]\ \mathbf{s}_{\lambda/\eta}[X]
=\sum_{\rho,\mu\in\Comp}S^\vee\!\bigl(\mathcal{S}_{\eta/\mu}\bigr)[Y] \mathcal{S}_{\rho/\lambda}[Y]\mathbf{s}_{\rho/\mu}[X].    
\end{equation}
\end{thm}

We next record the skew noncommutative Schur Murnaghan--Nakayama rule obtained by combining
Theorem~\ref{t:skew-MN-rule-NQ} with Tewari's Murnaghan--Nakayama rule for noncommutative
Schur functions~\cite{Tew16}.
Let $\mathbf{r}_\gamma$ denote the noncommutative ribbon Schur function indexed by a composition
$\gamma$.  The first-kind noncommutative power-sum element is, as in~\cite[(3)]{Tew16},
\phantomsection\label{def:first-kind-nsym-powersum}
\[
\Psi_r:=\sum_{a=0}^{r-1}(-1)^a\,\mathbf{r}_{(1^a,r-a)}\qquad(r\ge1),
\]
where $(1^0,r)$ is interpreted as the one-part composition $(r)$.
The elements $\Psi_r$ are primitive in $\mathrm{NSym}$, hence $S(\Psi_r)=-\Psi_r$.

We recall the necessary composition-shape terminology from~\cite[Sec.~5]{Tew16}.
Draw the reverse composition diagram of a composition with rows numbered from top to bottom and columns
from left to right.  For compositions $\alpha,\beta$, write $\alpha\lessdot_c\beta$ if either
\[
\beta=(1,\alpha_1,\ldots,\alpha_\ell),
\]
or
\[
\beta=(\alpha_1,\ldots,\alpha_{p-1},\alpha_p+1,\alpha_{p+1},\ldots,\alpha_\ell)
\quad\text{and}\quad
\alpha_j\ne\alpha_p\ \text{for all }j<p.
\]
Let $<_c$ be the transitive closure of this relation.
If $\alpha<_c\beta$, the skew reverse composition shape $\beta/\!/\alpha$ is obtained by drawing
$\alpha$ in the bottom-left corner of $\beta$ and taking the set difference.
This double-slash notation is only used for reverse composition shapes; the single slash
$\mathcal S_{\beta/\alpha}$ continues to denote the Hopf skew element defined by the coproduct.

For a skew reverse composition shape $\Theta$, set
\[
\operatorname{supp}(\Theta):=\{j\ge1\mid \Theta\text{ has a box in column }j\}.
\]
The shape $\Theta$ is called an interval shape if $\operatorname{supp}(\Theta)$ is an interval in
$\mathbb Z_{\ge1}$.  An interval shape $\beta/\!/\alpha$ is an \emph{nc border strip} if it satisfies:
\begin{enumerate}
    \item if $\beta/\!/\alpha$ has boxes in positions $(i,1)$ and $(i,2)$, then the box $(i,1)$ is the bottommost box in column $1$;
    \item if $\beta/\!/\alpha$ has boxes in positions $(i,j)$ and $(i,j+1)$ with $j\ge2$, then the box $(i,j)$ is the topmost box in column $j$.
\end{enumerate}
Its height is
\[
\operatorname{ht}_{\mathrm{nc}}(\beta/\!/\alpha)
:=\#\{\text{rows meeting }\beta/\!/\alpha\}-1.
\]
For an interval shape, say that column $j+1$ is north-east of column $j$ if column $j+1$ is nonempty and
for every pair of boxes $(i_1,j+1),(i_2,j)$ in the shape one has $i_1<i_2$.
Let
\[
\operatorname{NE}(\beta/\!/\alpha)
:=\{j\mid \text{column }j+1\text{ is north-east of column }j\}.
\]
\phantomsection\label{def:nc-border-strip-zero}
For $\alpha,\beta\in\Comp$, write
\[
\beta/\!/\alpha\in\mathrm{NCB}^{0}_{r}
\]
if $\alpha<_c\beta$, the shape $\beta/\!/\alpha$ is an nc border strip of size $r$, and
$\operatorname{NE}(\beta/\!/\alpha)=\varnothing$.
The superscript $0$ records this last condition.

\begin{cor}[Skew noncommutative Schur Murnaghan--Nakayama rule]\label{c:skew-ncSchur-MN}
For $r\ge1$ and $\lambda,\eta\in\Comp$, one has in $\mathrm{NSym}$
\begin{equation}\label{e:skew-ncSchur-MN}
\Psi_r\,\mathbf{s}_{\lambda/\eta}
=
\sum_{\rho:\,\rho/\!/\lambda\in\mathrm{NCB}^{0}_{r}}
(-1)^{\operatorname{ht}_{\mathrm{nc}}(\rho/\!/\lambda)}
\mathbf{s}_{\rho/\eta}
-
\sum_{\mu:\,\eta/\!/\mu\in\mathrm{NCB}^{0}_{r}}
(-1)^{\operatorname{ht}_{\mathrm{nc}}(\eta/\!/\mu)}
\mathbf{s}_{\lambda/\mu}.
\end{equation}
\end{cor}

\begin{proof}
Let
\[
\vartheta_r:\mathrm{QSym}\longrightarrow F,\qquad
\vartheta_r(F):=\langle \Psi_r,F\rangle_{\mathrm NQ}.
\]
Apply $\id\otimes\vartheta_r$ to Theorem~\ref{t:skew-MN-dual} in the present
$(\mathrm{NSym},\mathrm{QSym})$ specialization:
\[
\E^{\mathrm{NQ}}\bigl(\mathbf{s}_{\lambda/\eta}\otimes1\bigr)
=
\sum_{\rho,\mu\in\Comp}
\mathbf{s}_{\rho/\mu}\otimes
S^\vee(\mathcal S_{\eta/\mu})\,\mathcal S_{\rho/\lambda}.
\]
Since
\[
\Psi_r=\sum_{\alpha\in\Comp}\langle\Psi_r,\mathcal S_\alpha\rangle_{\mathrm NQ}\,\mathbf{s}_\alpha,
\]
the left-hand side becomes $\Psi_r\,\mathbf{s}_{\lambda/\eta}$.

For the right-hand side, put $a=\mathcal S_{\eta/\mu}$ and $b=\mathcal S_{\rho/\lambda}$.
Using the Hopf pairing and the primitiveness of $\Psi_r$,
\[
\begin{aligned}
\vartheta_r(S^\vee(a)b)
&=\langle \Psi_r,S^\vee(a)b\rangle_{\mathrm NQ}  \\
&=\langle \Delta(\Psi_r),S^\vee(a)\otimes b\rangle_{\mathrm NQ} \\
&=\langle \Psi_r,S^\vee(a)\rangle_{\mathrm NQ}\,\varepsilon^\vee(b)
  +\varepsilon^\vee(S^\vee(a))\,\langle\Psi_r,b\rangle_{\mathrm NQ} \\
&=-\delta_{\rho,\lambda}\,\langle\Psi_r,\mathcal S_{\eta/\mu}\rangle_{\mathrm NQ}
  +\delta_{\eta,\mu}\,\langle\Psi_r,\mathcal S_{\rho/\lambda}\rangle_{\mathrm NQ}.
\end{aligned}
\]
Here we used $\langle\Psi_r,S^\vee(a)\rangle_{\mathrm NQ}
=\langle S(\Psi_r),a\rangle_{\mathrm NQ}
=-\langle\Psi_r,a\rangle_{\mathrm NQ}$ and
$\varepsilon^\vee(\mathcal S_{\alpha/\beta})=\delta_{\alpha,\beta}$.
Thus
\[
\Psi_r\,\mathbf{s}_{\lambda/\eta}
=
\sum_{\rho\in\Comp}\langle\Psi_r,\mathcal S_{\rho/\lambda}\rangle_{\mathrm NQ}
\mathbf{s}_{\rho/\eta}
-
\sum_{\mu\in\Comp}\langle\Psi_r,\mathcal S_{\eta/\mu}\rangle_{\mathrm NQ}
\mathbf{s}_{\lambda/\mu}.
\]

It remains only to identify the coefficients.
Tewari's rule~\cite[Thm.~5.16]{Tew16}, translated to the present notation, says that
\[
\Psi_r\,\mathbf{s}_\alpha
=
\sum_{\beta:\,\beta/\!/\alpha\in\mathrm{NCB}^{0}_{r}}
(-1)^{\operatorname{ht}_{\mathrm{nc}}(\beta/\!/\alpha)}
\mathbf{s}_\beta .
\]
On the other hand,
\[
\langle\Psi_r\mathbf{s}_\alpha,\mathcal S_\beta\rangle_{\mathrm NQ}
=\langle\Psi_r\otimes\mathbf{s}_\alpha,\Delta^\vee(\mathcal S_\beta)\rangle_{\mathrm NQ}
=\langle\Psi_r,\mathcal S_{\beta/\alpha}\rangle_{\mathrm NQ}.
\]
Therefore
\[
\langle\Psi_r,\mathcal S_{\beta/\alpha}\rangle_{\mathrm NQ}
=
\begin{cases}
(-1)^{\operatorname{ht}_{\mathrm{nc}}(\beta/\!/\alpha)},&
\text{if }\beta/\!/\alpha\in\mathrm{NCB}^{0}_{r},\\
0,&\text{otherwise.}
\end{cases}
\]
Substituting this coefficient formula with $(\beta,\alpha)=(\rho,\lambda)$ and
$(\beta,\alpha)=(\eta,\mu)$ gives \eqref{e:skew-ncSchur-MN}.
\end{proof}

We now turn to Pieri-type formulas.  Lam--Lauve--Sottile~\cite[Sec.~5]{LLS11} applied
their Hopf-algebraic skew Littlewood--Richardson formula to the noncommutative ribbon
Schur basis \(\mathbf r_\alpha\) (denoted \(R_\alpha\) in their paper).  This is different from the noncommutative Schur basis
\(\mathbf{s}_\alpha\) used here.  For the basis \(\mathbf{s}_\alpha\), Tewari--van Willigenburg
\cite{TvW18} recalled both the right and left straight Pieri rules
\cite[Thms.~3.11--3.12]{TvW18}, and proved skew right Pieri rules
\cite[Thm.~4.9]{TvW18}.  We record the corresponding skew left Pieri rules, obtained
from the present Cauchy-kernel formalism and the straight left Pieri rule.

We use the left-strip terminology of~\cite{TvW18}.  For a composition \(\alpha\), set
\(t_0(\alpha)=\alpha\).  For \(i\ge1\), define \(t_i\) by
\[
t_1(\alpha)=(1,\alpha_1,\ldots,\alpha_{\ell(\alpha)}),
\]
and, for \(i\ge2\), let \(t_i(\alpha)\) be obtained by adding \(1\) to the leftmost part of
\(\alpha\) equal to \(i-1\); if no such part exists, set \(t_i(\alpha)=0\).
Here $0$ is a formal absorbing value, so $t_j(0)=0$ for every $j\ge1$.
\phantomsection\label{def:nc-left-strips}
For compositions \(\alpha,\beta\) and \(m\ge0\), write
\[
\alpha\xrightarrow{\mathrm{LH}}_m\beta
\]
if either \(m=0\) and \(\alpha=\beta\), or \(m>0\) and
\(\beta=t_{i_m}\cdots t_{i_1}(\alpha)\) for some \(1\le i_1<\cdots<i_m\).  Similarly, write
\[
\alpha\xrightarrow{\mathrm{LV}}_m\beta
\]
if either \(m=0\) and \(\alpha=\beta\), or \(m>0\) and
\(\beta=t_{i_m}\cdots t_{i_1}(\alpha)\) for some \(i_1\ge\cdots\ge i_m\ge1\).
Thus \(\xrightarrow{\mathrm{LH}}_m\) and \(\xrightarrow{\mathrm{LV}}_m\) mean adding an
\(m\)-left horizontal strip and an \(m\)-left vertical strip, respectively, in the sense of
Tewari--van Willigenburg.

\begin{cor}[Skew left Pieri rules for noncommutative Schur functions]\label{c:skew-left-Pieri-ncSchur}
For \(n\ge1\) and \(\lambda,\eta\in\Comp\), one has in \(\mathrm{NSym}\)
\begin{equation}\label{e:skew-left-Pieri-ncSchur-h}
\mathbf{s}_{(n)}\,\mathbf{s}_{\lambda/\eta}
=
\sum_{\substack{a+b=n\\a,b\ge0}}(-1)^b
\sum_{\substack{\lambda\xrightarrow{\mathrm{LH}}_a\rho\\
\mu\xrightarrow{\mathrm{LV}}_b\eta}}
\mathbf{s}_{\rho/\mu},
\end{equation}
and
\begin{equation}\label{e:skew-left-Pieri-ncSchur-e}
\mathbf{s}_{(1^n)}\,\mathbf{s}_{\lambda/\eta}
=
\sum_{\substack{a+b=n\\a,b\ge0}}(-1)^b
\sum_{\substack{\lambda\xrightarrow{\mathrm{LV}}_a\rho\\
\mu\xrightarrow{\mathrm{LH}}_b\eta}}
\mathbf{s}_{\rho/\mu}.
\end{equation}
Here \(\mathbf{s}_{(0)}=\mathbf{s}_{(1^0)}=1\), and the relations with subscript \(0\) mean equality.
\end{cor}

\begin{proof}
The straight left Pieri rules of Tewari--van Willigenburg~\cite[Thm.~3.12]{TvW18}, in the
present notation, are
\[
\mathbf{s}_{(a)}\mathbf{s}_\alpha
=\sum_{\alpha\xrightarrow{\mathrm{LH}}_a\beta}\mathbf{s}_\beta,
\qquad
\mathbf{s}_{(1^a)}\mathbf{s}_\alpha
=\sum_{\alpha\xrightarrow{\mathrm{LV}}_a\beta}\mathbf{s}_\beta.
\]
Equivalently, by the Hopf pairing,
\begin{align}\label{e:nc-left-Pieri-coeff-h}
\langle \mathbf{s}_{(a)},\mathcal S_{\beta/\alpha}\rangle_{\mathrm NQ}
&=
\begin{cases}
1, & \alpha\xrightarrow{\mathrm{LH}}_a\beta,\\
0, & \text{otherwise,}
\end{cases}\\
\label{e:nc-left-Pieri-coeff-e}
\langle \mathbf{s}_{(1^a)},\mathcal S_{\beta/\alpha}\rangle_{\mathrm NQ}
&=
\begin{cases}
1, & \alpha\xrightarrow{\mathrm{LV}}_a\beta,\\
0, & \text{otherwise.}
\end{cases}
\end{align}

Apply the linear functional \(F\mapsto\langle \mathbf{s}_{(n)},F\rangle_{\mathrm NQ}\) to the
second tensor factor in the tensor form underlying \eqref{e:skew-MN-rule-NQ}.  The left-hand
side becomes \(\mathbf{s}_{(n)}\mathbf{s}_{\lambda/\eta}\).  For the coefficient of
\(\mathbf{s}_{\rho/\mu}\) on the right-hand side, we use the coproduct and antipode formulas
for one-row and one-column noncommutative Schur functions recorded immediately before
\cite[Thm.~4.9]{TvW18}.  In the present notation, \(\mathbf{s}_{(n)}=H_n\) and
\(\mathbf{s}_{(1^n)}\) is the degree-\(n\) noncommutative elementary generator, so these formulas are
\[
\Delta(\mathbf{s}_{(n)})=\sum_{a+b=n}\mathbf{s}_{(b)}\otimes \mathbf{s}_{(a)},
\qquad
S(\mathbf{s}_{(b)})=(-1)^b\mathbf{s}_{(1^b)}.
\]
Thus
\[
\begin{aligned}
&\left\langle \mathbf{s}_{(n)},
S^\vee(\mathcal S_{\eta/\mu})\,\mathcal S_{\rho/\lambda}
\right\rangle_{\mathrm NQ}  \\
&\qquad =
\sum_{a+b=n}
\left\langle S(\mathbf{s}_{(b)}),\mathcal S_{\eta/\mu}\right\rangle_{\mathrm NQ}
\left\langle \mathbf{s}_{(a)},\mathcal S_{\rho/\lambda}\right\rangle_{\mathrm NQ} \\
&\qquad =
\sum_{a+b=n}(-1)^b
\left\langle \mathbf{s}_{(1^b)},\mathcal S_{\eta/\mu}\right\rangle_{\mathrm NQ}
\left\langle \mathbf{s}_{(a)},\mathcal S_{\rho/\lambda}\right\rangle_{\mathrm NQ}.
\end{aligned}
\]
Substituting \eqref{e:nc-left-Pieri-coeff-h}--\eqref{e:nc-left-Pieri-coeff-e} gives
\eqref{e:skew-left-Pieri-ncSchur-h}.

The proof of \eqref{e:skew-left-Pieri-ncSchur-e} is identical, using
\[
\Delta(\mathbf{s}_{(1^n)})=\sum_{a+b=n}\mathbf{s}_{(1^b)}\otimes \mathbf{s}_{(1^a)},
\qquad
S(\mathbf{s}_{(1^b)})=(-1)^b\mathbf{s}_{(b)}.
\]
Then the first pairing is evaluated by \eqref{e:nc-left-Pieri-coeff-h}, and the second by
\eqref{e:nc-left-Pieri-coeff-e}.
\end{proof}

\section{\texorpdfstring{$k$-Schur functions}{k-Schur functions}}
\label{ss:kSchur-skewMN}
Recall that $\Lambda_F$ is the (commutative) Hopf algebra of symmetric functions over $F$ with the standard coproduct
characterized by $\Delta(h_r)=\sum_{i+j=r}h_i\otimes h_j$.
Fix $k\ge 1$.  Following the standard $k$-Schur convention, set
\[
H=\Lambda_{(k)}:=F[h_1,\dots,h_k]\subseteq \Lambda_F,
\]
a graded connected Hopf subalgebra with Hopf structure induced from $\Lambda_F$.
Its graded Hopf dual with respect to the Hall inner product is the quotient
\[
H^\vee=\Lambda^{(k)}
:=\Lambda_F/\langle p_j:j>k\rangle
=\Lambda_F/\langle m_\nu:\nu_1>k\rangle.
\]
Both $H$ and $H^\vee$ are commutative Hopf algebras, hence satisfy the commutativity hypothesis in Theorem~\ref{t:skew-MN-dual}.

Let $\mathcal I=\mathscr P^{(k)}$ be the set of $k$-bounded partitions (i.e.\ $\lambda_1\le k$).
Let $\langle\cdot,\cdot\rangle_{\mathrm{Hall}}$ denote the induced Hall pairing between
$\Lambda_{(k)}$ and $\Lambda^{(k)}$.  Let $s^{(k)}_{\lambda}$ denote the {\em $k$-Schur
function} and $\widetilde{F}^{(k)}_\lambda$ the {\em dual $k$-Schur function}; these form
dual bases under the Hall pairing.  See~\cite{LM07, LLM+10} for the $k$-Schur setting and
\cite[p.~1591]{BSZ11}, \cite[Sec.~6]{LLS11} for the duality convention used here.

Set
\[
f_\lambda=s^{(k)}_\lambda\in \Lambda_{(k)},\qquad
g_\lambda=\widetilde{F}^{(k)}_\lambda\in \Lambda^{(k)}\qquad(\lambda\in\mathscr P^{(k)}),
\]
characterized by
\[
\langle s^{(k)}_\lambda,\ \widetilde{F}^{(k)}_\mu\rangle_{\mathrm{Hall}}=\delta_{\lambda\mu}.
\]

The Cauchy element is
\[
\E^{(k)}=\sum_{\lambda\in\mathscr P^{(k)}} s^{(k)}_\lambda\otimes \widetilde{F}^{(k)}_\lambda
\ \in\ \widehat{\Lambda_{(k)}\otimes \Lambda^{(k)}},
\]
and after evaluation in alphabets $X$ and $Y$ we write
\[
\E^{(k)}[X| Y]
:=\sum_{\lambda\in\mathscr P^{(k)}} s^{(k)}_\lambda[X]\ \widetilde{F}^{(k)}_\lambda[Y].
\]

Define skew $k$-Schur and skew dual $k$-Schur elements by
\[
\Delta\bigl(s^{(k)}_\lambda\bigr)=\sum_{\mu\in\mathscr P^{(k)}} s^{(k)}_{\lambda/\mu}\otimes s^{(k)}_\mu,
\qquad
\Delta^\vee\bigl(\widetilde{F}^{(k)}_\lambda\bigr)=\sum_{\mu\in\mathscr P^{(k)}} \widetilde{F}^{(k)}_{\lambda/\mu}\otimes \widetilde{F}^{(k)}_\mu.
\]

In the symmetric-function Hopf algebra $\Lambda$, the antipode $S$ is given plethystically by
\[
S(F)[Y]=F[-Y]\qquad(F\in\Lambda),
\]
equivalently $S(p_r)=-p_r$ and $S(h_r)=(-1)^r e_r$.
Since $S(h_r)=(-1)^r e_r\in F[h_1,\ldots,h_k]$ for $1\le r\le k$, the subalgebra
$\Lambda_{(k)}=F[h_1,\ldots,h_k]$ is stable under the usual antipode of $\Lambda$.
Hence the antipode on the graded Hopf dual $\Lambda^{(k)}$ is the one induced by this restricted antipode.
Under the standard symmetric-function realization of representatives of $\Lambda^{(k)}$ by dual
$k$-Schur functions \cite{LM07, LLM+10}, this means that if $g\in\Lambda^{(k)}$ is represented by a
symmetric function $G$, then $S^\vee(g)$ is represented by $G[-Y]$ modulo the quotient defining
$\Lambda^{(k)}$. We write this as
\[
S^\vee(g)[Y]=g[-Y]\qquad(g\in \Lambda^{(k)}),
\]
i.e.\ apply the alphabet substitution $Y\mapsto -Y$ to a representative in the ambient symmetric-function
ring and interpret the result in $\Lambda^{(k)}$.

\begin{thm}[Skew MN rule in $(\Lambda_{(k)},\Lambda^{(k)})$]\label{t:skew-MN-rule-k-Schur}
    For $(H,H^\vee)=(\Lambda_{(k)},\Lambda^{(k)})$, the following skew MN rule holds:
    \begin{equation}\label{e:skew-MN-rule-k-Schur}
        \E^{(k)}[X|Y]\ s^{(k)}_{\lambda/\eta}[X]
=\sum_{\rho,\mu\in\mathscr P^{(k)}} \widetilde{F}^{(k)}_{\eta/\mu}[-Y] \widetilde{F}^{(k)}_{\rho/\lambda}[Y]s^{(k)}_{\rho/\mu}[X].
    \end{equation}
\end{thm}

We next record the skew $k$-Schur Murnaghan--Nakayama rule obtained by combining
\eqref{e:skew-MN-rule-k-Schur} with the $k$-Schur Murnaghan--Nakayama rule
in~\cite{BSZ11}.  Let $\mathcal C_{k+1}$ denote the set of $(k+1)$-cores, i.e.,
partitions with no removable $(k+1)$-border strip. Let
\[
\mathrm{core}_{k+1}:\mathscr P^{(k)}\to\mathcal C_{k+1}
\]
denote the usual bijection from $k$-bounded partitions to $(k+1)$-cores~\cite{LM05}.
The same notation and the associated $k$-conjugation convention are recalled
in~\cite[Sec.~2.1]{BSZ11}.  We write
\[
\lambda^{(k)}:=\mathrm{core}_{k+1}^{-1}\bigl(\mathrm{core}_{k+1}(\lambda)^t\bigr)
\]
for the $k$-conjugate of $\lambda$.
For $\alpha,\beta\in\mathscr P^{(k)}$ and $1\le r\le k$, write
\[
\beta/\alpha\in\mathrm{KR}^{(k)}_r
\]
if $\beta/\alpha$ is a BSZ $k$-ribbon of size $r$, in the sense of~\cite[Def.~1.1]{BSZ11}.
For convenience, we recall that this condition is characterized by the containment, size,
core-ribbon, $k$-connectedness, and height-statistic conditions
\[
\alpha\subseteq\beta,\qquad
\alpha^{(k)}\subseteq\beta^{(k)},\qquad
|\beta/\alpha|=r,
\]
the skew core
$\mathrm{core}_{k+1}(\beta)/\mathrm{core}_{k+1}(\alpha)$ is a $k$-connected ribbon, and
\[
\operatorname{ht}^{(k)}(\beta/\alpha)+
\operatorname{ht}^{(k)}\bigl(\beta^{(k)}/\alpha^{(k)}\bigr)=r-1.
\]
Here $\operatorname{ht}^{(k)}$ denotes the BSZ height statistic, distinguished from the usual
ribbon height: $\operatorname{ht}^{(k)}(\beta/\alpha)$ is the number of vertical dominoes in
$\beta/\alpha$.

\begin{cor}[Skew $k$-Schur Murnaghan--Nakayama rule]\label{c:skew-kSchur-MN}
For $1\le r\le k$ and $\lambda,\eta\in\mathscr P^{(k)}$, one has
\begin{equation}\label{e:skew-kSchur-MN}
p_r[X]\,s^{(k)}_{\lambda/\eta}[X]
=
\sum_{\rho:\,\rho/\lambda\in\mathrm{KR}^{(k)}_r}
(-1)^{\operatorname{ht}^{(k)}(\rho/\lambda)}
s^{(k)}_{\rho/\eta}[X]
-
\sum_{\mu:\,\eta/\mu\in\mathrm{KR}^{(k)}_r}
(-1)^{\operatorname{ht}^{(k)}(\eta/\mu)}
s^{(k)}_{\lambda/\mu}[X].
\end{equation}
\end{cor}

\begin{proof}
Let $Y_r(z)$ be the virtual alphabet determined by
\[
p_m[Y_r(z)]=r z\,\delta_{m,r}\qquad(m\ge1).
\]
For any homogeneous $g\in\Lambda^{(k)}$, expanding $g$ in the quotient power-sum basis
$\{p_\nu:\nu_1\le k\}$ gives
\[
[z]\,g[Y_r(z)]=\langle p_r,g\rangle_{\mathrm{Hall}},
\qquad
[z]\,g[-Y_r(z)]=-\langle p_r,g\rangle_{\mathrm{Hall}}.
\]
Moreover,
\[
[z]\,\E^{(k)}[X|Y_r(z)]
=\sum_{\nu\in\mathscr P^{(k)}}s^{(k)}_\nu[X]\,
\langle p_r,\widetilde F^{(k)}_\nu\rangle_{\mathrm{Hall}}
=p_r[X],
\]
and
\[
[z^0]\,\widetilde F^{(k)}_{\alpha/\beta}[Y_r(z)]
=\delta_{\alpha,\beta}.
\]
Taking the coefficient of $z$ in \eqref{e:skew-MN-rule-k-Schur} with $Y=Y_r(z)$ gives
\[
p_r[X]\,s^{(k)}_{\lambda/\eta}[X]
=
\sum_{\rho\in\mathscr P^{(k)}}
\langle p_r,\widetilde F^{(k)}_{\rho/\lambda}\rangle_{\mathrm{Hall}}\,
s^{(k)}_{\rho/\eta}[X]
-
\sum_{\mu\in\mathscr P^{(k)}}
\langle p_r,\widetilde F^{(k)}_{\eta/\mu}\rangle_{\mathrm{Hall}}\,
s^{(k)}_{\lambda/\mu}[X].
\]
By the definition of skew dual $k$-Schur functions,
\[
\langle p_r,\widetilde F^{(k)}_{\rho/\lambda}\rangle_{\mathrm{Hall}}
=
\langle p_r\,s^{(k)}_\lambda,\widetilde F^{(k)}_\rho\rangle_{\mathrm{Hall}}.
\]
The $k$-Schur Murnaghan--Nakayama rule~\cite[Thm.~1.2 and Cor.~1.4]{BSZ11} states that
\[
p_r\,s^{(k)}_\lambda
=
\sum_{\substack{\rho\in\mathscr P^{(k)}\\
\rho/\lambda\in\mathrm{KR}^{(k)}_r}}
(-1)^{\operatorname{ht}^{(k)}(\rho/\lambda)}s^{(k)}_\rho .
\]
Thus, for every $\gamma\in\mathscr P^{(k)}$,
\[
\langle p_r,\widetilde F^{(k)}_{\gamma/\lambda}\rangle_{\mathrm{Hall}}
=
\begin{cases}
(-1)^{\operatorname{ht}^{(k)}(\gamma/\lambda)},&
\gamma/\lambda\in\mathrm{KR}^{(k)}_r,\\
0,&\text{otherwise}.
\end{cases}
\]
The same argument with $(\lambda,\gamma)$ replaced by $(\mu,\eta)$ gives the coefficients in the
second sum. This proves \eqref{e:skew-kSchur-MN}.
\end{proof}

We finally record the parallel Pieri specialization.  This uses the straight \(k\)-Pieri
rule rather than the straight \(k\)-Schur Murnaghan--Nakayama rule.

\phantomsection\label{def:k-weak-pieri-strip}
For $\alpha,\beta\in\mathscr P^{(k)}$ and $0\le r\le k$, write
\[
\alpha\rightsquigarrow^{(k)}_r\beta
\]
if $r=0$ and $\alpha=\beta$, or if $r>0$, $\alpha\subseteq\beta$, $|\beta/\alpha|=r$,
$\beta/\alpha$ is a horizontal strip, and $\beta^{(k)}/\alpha^{(k)}$ is a vertical strip.
Thus $\rightsquigarrow^{(k)}_r$ is the weak $k$-Pieri strip relation in the bounded-partition
notation of Lapointe--Morse~\cite[Thm.~29]{LM07}.  With this notation, the straight
$k$-Pieri rules of~\cite[Thms.~29 and 33]{LM07} are
\begin{align}
h_a[X]\,s^{(k)}_\alpha[X]
&=\sum_{\beta:\,\alpha\rightsquigarrow^{(k)}_a\beta}s^{(k)}_\beta[X],
\label{e:straight-k-Pieri-h}\\
e_b[X]\,s^{(k)}_\alpha[X]
&=\sum_{\beta:\,\alpha^{(k)}\rightsquigarrow^{(k)}_b\beta^{(k)}}s^{(k)}_\beta[X],
\label{e:straight-k-Pieri-e}
\end{align}
for $0\le a,b\le k$.  Under the affine Grassmannian realization, \eqref{e:straight-k-Pieri-h}
is the homology Pieri rule of Lam--Lapointe--Morse--Shimozono~\cite[Thm.~4.12]{LLM+10},
with the above relation corresponding to the weak-strip condition.

\begin{cor}[Skew $k$-Schur Pieri rule]\label{c:skew-kSchur-Pieri}
For $1\le r\le k$ and $\lambda,\eta\in\mathscr P^{(k)}$, one has
\begin{equation}\label{e:skew-kSchur-Pieri}
h_r[X]\,s^{(k)}_{\lambda/\eta}[X]
=
\sum_{a+b=r}(-1)^b
\sum_{\substack{\rho,\nu\in\mathscr P^{(k)}\\
\lambda\rightsquigarrow^{(k)}_a\rho\\
\nu^{(k)}\rightsquigarrow^{(k)}_b\eta^{(k)}}}
s^{(k)}_{\rho/\nu}[X].
\end{equation}
\end{cor}

\begin{proof}
Let $\pi_r:\Lambda^{(k)}\to F$ be the linear functional
\[
\pi_r(g)=\langle h_r,\ g\rangle_{\mathrm{Hall}}.
\]
Applying $\operatorname{id}\otimes\pi_r$ to the tensor form underlying
\eqref{e:skew-MN-rule-k-Schur} gives
\[
h_r[X]\,s^{(k)}_{\lambda/\eta}[X]
=
\sum_{\rho,\nu\in\mathscr P^{(k)}}
\left\langle h_r,\,
S^\vee\bigl(\widetilde F^{(k)}_{\eta/\nu}\bigr)\widetilde F^{(k)}_{\rho/\lambda}
\right\rangle_{\mathrm{Hall}}
s^{(k)}_{\rho/\nu}[X].
\]
Since $\Delta(h_r)=\sum_{a+b=r}h_b\otimes h_a$ and
$S(h_b)=(-1)^b e_b$, the coefficient in the last display is
\[
\sum_{a+b=r}(-1)^b
\left\langle e_b,\widetilde F^{(k)}_{\eta/\nu}\right\rangle_{\mathrm{Hall}}
\left\langle h_a,\widetilde F^{(k)}_{\rho/\lambda}\right\rangle_{\mathrm{Hall}}.
\]
By the definition of the skew dual $k$-Schur functions, the second pairing is the
coefficient of $s^{(k)}_\rho$ in $h_a s^{(k)}_\lambda$, while the first pairing is the
coefficient of $s^{(k)}_\eta$ in $e_b s^{(k)}_\nu$.  The two straight Pieri rules
\eqref{e:straight-k-Pieri-h}--\eqref{e:straight-k-Pieri-e} therefore give
\[
\left\langle h_a,\widetilde F^{(k)}_{\rho/\lambda}\right\rangle_{\mathrm{Hall}}
=
\begin{cases}
1,& \lambda\rightsquigarrow^{(k)}_a\rho,\\
0,& \text{otherwise},
\end{cases}
\qquad
\left\langle e_b,\widetilde F^{(k)}_{\eta/\nu}\right\rangle_{\mathrm{Hall}}
=
\begin{cases}
1,& \nu^{(k)}\rightsquigarrow^{(k)}_b\eta^{(k)},\\
0,& \text{otherwise}.
\end{cases}
\]
Substitution gives \eqref{e:skew-kSchur-Pieri}.
\end{proof}

\begin{rem}
After translating notation, with $\rightsquigarrow^{(k)}_r$ read as the bounded-partition
form of the weak-strip relation, \eqref{e:skew-kSchur-Pieri} recovers the skew
$k$-Pieri rule of Lam--Lauve--Sottile~\cite[Thm.~6.1]{LLS11}.  Here the skew formula is
derived directly from the completed Cauchy identity \eqref{e:skew-MN-rule-k-Schur}
together with the straight Pieri input \eqref{e:straight-k-Pieri-h}--\eqref{e:straight-k-Pieri-e}.
\end{rem}

\section{\texorpdfstring{Affine Grassmannian of the symplectic group (type $C$)}{Affine Grassmannian of the symplectic group (type C)}}
\label{ss:SpAffineGr-skewMN}

Fix $n\ge 1$.  Following Lam--Schilling--Shimozono~\cite[Sec.~1.3 and Thm.~1.3]{LSS10}, let
\[
\Gamma_*:=\mathbb Z[P_1,P_2,\ldots]\otimes_{\mathbb Z}F,\qquad
\Gamma^*:=\mathbb Z[Q_1,Q_2,\ldots]\otimes_{\mathbb Z}F
\]
be the dual Hopf algebras of Schur \(P\)- and \(Q\)-functions, with their standard Hopf
pairing, after extension of scalars to $F$.  Set
\[
H=\Gamma_{(n)}:=F[P_1,\ldots,P_{2n}]\subseteq\Gamma_*,\qquad
H^\vee=\Gamma^{(n)},
\]
where \(\Gamma^{(n)}\) is the quotient of \(\Gamma^*\) by the annihilator of
\(\Gamma_{(n)}\) under the standard pairing.  Thus \(\Gamma^{(n)}\) is the graded Hopf dual
of the type \(C\) affine Grassmannian homology Hopf algebra \(\Gamma_{(n)}\).  Since
\(\Gamma^*\) is commutative, so is \(\Gamma^{(n)}\), and Theorem~\ref{t:skew-MN-dual} applies.

Let $\mathcal I=\widetilde C_n^{\,0}$ be the index set of type $C$ affine Grassmannian elements (equivalently, the Schubert index set).
Fix a graded nondegenerate Hopf pairing $\langle\cdot,\cdot\rangle_{Sp}:\Gamma_{(n)}\times \Gamma^{(n)}\to F$ and dual homogeneous bases
\[
f_w=P^{(n)}_w\in \Gamma_{(n)},\qquad g_w=Q^{(n)}_w\in \Gamma^{(n)}\qquad(w\in\widetilde C_n^{\,0}),
\]
characterized by
\[
\langle P^{(n)}_u,\ Q^{(n)}_v\rangle_{Sp}=\delta_{uv}.
\]
Here $\{P^{(n)}_w\}_{w\in \widetilde C_n^{\,0}}$ and
$\{Q^{(n)}_w\}_{w\in \widetilde C_n^{\,0}}$ are the dual Schubert bases of
\cite[Thm.~1.3]{LSS10}. 

The Cauchy element is
\[
\E^{Sp(n)}
=\sum_{w\in\widetilde C_n^{\,0}} P^{(n)}_w\otimes Q^{(n)}_w
\ \in\ \widehat{\Gamma_{(n)}\otimes\Gamma^{(n)}},
\]
and its evaluated Cauchy kernel is
\[
\E^{Sp(n)}[X| Y]
:=\sum_{w\in\widetilde C_n^{\,0}} P^{(n)}_w[X]\ Q^{(n)}_w[Y].
\]

Define skew elements by the coproduct expansions
\[
\Delta\bigl(P^{(n)}_w\bigr)=\sum_{v\in\widetilde C_n^{\,0}} P^{(n)}_{w/v}\otimes P^{(n)}_v,
\qquad
\Delta^\vee\bigl(Q^{(n)}_w\bigr)=\sum_{v\in\widetilde C_n^{\,0}} Q^{(n)}_{w/v}\otimes Q^{(n)}_v.
\]

By the Hopf-dual realization of \(\Gamma_{(n)}\) and \(\Gamma^{(n)}\)
in~\cite[Sec.~1.3]{LSS10}, these Hopf algebras are obtained from the Schur \(P\)- and
\(Q\)-function Hopf algebras above by Hopf subalgebra and graded-dual constructions.
Their antipodes are therefore induced by the usual antipode of $\Lambda$.
Accordingly, for $g\in \Gamma^{(n)}$ the induced antipode is computed plethystically by
\[
S^\vee(g)[Y]=g[-Y],
\]
i.e.\ apply the alphabet substitution $Y\mapsto -Y$ in the ambient symmetric-function realization and interpret the
result in $\Gamma^{(n)}$.

\begin{thm}[Skew MN rule in type $C$ affine Grassmannian]\label{t:skew-MN-rule-C}
    Let $(H,H^\vee)=(\Gamma_{(n)},\Gamma^{(n)})$. Then we have the following skew MN rule for the affine Grassmannian of the symplectic group:
    \begin{equation}\label{e:skew-MN-rule-C}
        \E^{Sp(n)}[X| Y]\ P^{(n)}_{w/v}[X]
=\sum_{x,y\in\widetilde C_n^{\,0}} Q^{(n)}_{v/y}[-Y] Q^{(n)}_{x/w}[Y]P^{(n)}_{x/y}[X].
    \end{equation}
\end{thm}

We next extract the one-variable specialization of \eqref{e:skew-MN-rule-C}.  This gives a skew
version of the Lam--Schilling--Shimozono Pieri rule.
Let \(I_{\mathrm{af}}=\{0,1,\ldots,n\}\) be the affine type \(C\) index set.  For
\(u\in\widetilde C_n\), let \(\operatorname{Supp}(u)\subseteq I_{\mathrm{af}}\) be the support of any
reduced word for \(u\), and let \(\operatorname{cc}(u)\) be the number of connected components of
\(\operatorname{Supp}(u)\) in the affine Dynkin diagram.
For \(1\le a\le 2n\), the special affine Grassmannian element \(\rho_a\in\widetilde C_n^{\,0}\)
is
\[
\rho_a=
\begin{cases}
s_{a-1}s_{a-2}\cdots s_1s_0, & 1\le a\le n,\\
\bigl(s_{2n+1-a}s_{2n+2-a}\cdots s_{n-1}\bigr)s_ns_{n-1}\cdots s_1s_0,
& n+1\le a\le 2n.
\end{cases}
\]
In the second line the parenthesized product is understood to be empty when \(a=n+1\).
Following~\cite[Sec.~1.5]{LSS10}, let \(\mathcal Z^{(n)}\) be the Bruhat order ideal in \(\widetilde C_n\)
generated by the conjugates of \(\rho_{2n}\); its elements are the type \(C\) Pieri factors.  Put
\[
\mathcal Z^{(n)}_a:=\{u\in\mathcal Z^{(n)}:\ell(u)=a\}.
\]
Define the straight type \(C\) Pieri coefficients \(\Pi_a^C(x,w)\), for
\(x,w\in\widetilde C_n^{\,0}\), by
\[
\Pi_0^C(x,w):=\delta_{x,w},
\]
and, for \(a>0\), by
\[
\Pi_a^C(x,w):=
\sum_{\substack{u\in\mathcal Z^{(n)}_a\\ x=uw\\
\ell(x)=\ell(u)+\ell(w)}}
2^{\operatorname{cc}(u)-1}.
\]
Under the identification of~\cite[Thm.~1.3]{LSS10}, the special Schubert basis element
satisfies $P_{\rho_a}^{(n)}=P_a$ for $1\le a\le 2n$.  Consequently, the
Lam--Schilling--Shimozono homology Pieri rule~\cite[Thm.~1.4]{LSS10}
is
\begin{equation}\label{e:type-C-straight-Pieri-coeff}
P_{\rho_a}^{(n)}[X]\,P_w^{(n)}[X]
=
\sum_{x\in\widetilde C_n^{\,0}}\Pi_a^C(x,w)\,P_x^{(n)}[X]
\qquad(1\le a\le 2n).
\end{equation}

\begin{cor}[Skew Pieri rule in type \(C\) affine Grassmannian]\label{c:skew-Pieri-type-C}
For \(1\le r\le 2n\) and \(w,v\in\widetilde C_n^{\,0}\), one has
\begin{equation}\label{e:skew-Pieri-type-C}
P_{\rho_r}^{(n)}[X]\,P_{w/v}^{(n)}[X]
=
\sum_{\substack{a+b=r\\ a,b\ge0}}
(-1)^b\,\epsilon(a,b)
\sum_{x,y\in\widetilde C_n^{\,0}}
\Pi_a^C(x,w)\Pi_b^C(v,y)\,P_{x/y}^{(n)}[X],
\end{equation}
where
\[
\epsilon(a,b)=
\begin{cases}
1, & a=0\text{ or }b=0,\\
2, & a>0\text{ and }b>0.
\end{cases}
\]
\end{cor}

\begin{proof}
We first record the one-variable Cauchy specialization.  By the defining factorization
formula for \(Q_w^{(n)}[Y]\) in~\cite[(1.2)]{LSS10}, one has
\[
Q_u^{(n)}[z]=2^{\operatorname{cc}(u)}z^{\ell(u)}
\qquad(u\in\mathcal Z^{(n)}),
\]
while \(Q_u^{(n)}[z]=0\) for \(u\notin\mathcal Z^{(n)}\).  Therefore
\[
\E^{Sp(n)}[X|z]
=1+\sum_{a=1}^{2n}z^a
\sum_{\substack{u\in\mathcal Z^{(n)}_a\\ u\in\widetilde C_n^{\,0}}}
2^{\operatorname{cc}(u)}P_u^{(n)}[X].
\]
Taking \(w=e\) in the straight Pieri rule \eqref{e:type-C-straight-Pieri-coeff} gives
\[
P_{\rho_a}^{(n)}[X]
=
\sum_{\substack{u\in\mathcal Z^{(n)}_a\\ u\in\widetilde C_n^{\,0}}}
2^{\operatorname{cc}(u)-1}P_u^{(n)}[X],
\]
and hence
\begin{equation}\label{e:type-C-one-variable-Cauchy}
\E^{Sp(n)}[X|z]
=1+2\sum_{a=1}^{2n}P_{\rho_a}^{(n)}[X]z^a.
\end{equation}

Next, the straight specialization of the skew Cauchy identity gives
\[
\E^{Sp(n)}[X|z]\,P_w^{(n)}[X]
=
\sum_{x\in\widetilde C_n^{\,0}}Q_{x/w}^{(n)}[z]\,P_x^{(n)}[X].
\]
Combining this with \eqref{e:type-C-one-variable-Cauchy} and the straight Pieri rule
\eqref{e:type-C-straight-Pieri-coeff}, we obtain
\begin{equation}\label{e:type-C-skew-Q-one-var}
Q_{x/w}^{(n)}[z]=\sum_{a\ge0}\gamma_a(x,w)z^a,
\qquad
\gamma_0(x,w)=\Pi_0^C(x,w),\quad
\gamma_a(x,w)=2\Pi_a^C(x,w)\quad(a>0).
\end{equation}
Since \(Q_{v/y}^{(n)}\) lies in the Schur \(Q\)-function Hopf algebra, equivalently in the
subalgebra generated by odd power sums, the plethystic antipode acts on its homogeneous
degree \(b\) component by the scalar \((-1)^b\).  Hence \eqref{e:type-C-skew-Q-one-var}
also gives
\[
Q_{v/y}^{(n)}[-z]=\sum_{b\ge0}(-1)^b\gamma_b(v,y)z^b.
\]
Now put \(Y=z\) in \eqref{e:skew-MN-rule-C}.  The coefficient of \(z^r\) on the left-hand
side is, by \eqref{e:type-C-one-variable-Cauchy},
\[
2P_{\rho_r}^{(n)}[X]\,P_{w/v}^{(n)}[X].
\]
The coefficient of \(z^r\) on the right-hand side is
\[
\sum_{\substack{a+b=r\\a,b\ge0}}(-1)^b
\sum_{x,y\in\widetilde C_n^{\,0}}
\gamma_a(x,w)\gamma_b(v,y)P_{x/y}^{(n)}[X].
\]
Dividing by \(2\) and substituting
\(\gamma_0=\Pi_0^C\), \(\gamma_a=2\Pi_a^C\) for \(a>0\), gives exactly
\eqref{e:skew-Pieri-type-C}.
\end{proof}

\begin{rem}
For \(1\le r\le 2n-1\), after translating notation, \eqref{e:skew-Pieri-type-C}
recovers the type \(C\) skew Pieri rule of Lam--Lauve--Sottile~\cite[Thm.~7.1]{LLS11}.
The endpoint \(r=2n\) follows here from the Lam--Schilling--Shimozono homology Pieri
rule~\cite[Thm.~1.4]{LSS10} together with the one-variable Cauchy specialization
\eqref{e:type-C-one-variable-Cauchy}.  In the present formulation, the coefficients
\(\Pi_a^C\) package the Pieri factors and their multiplicities, while the factor
\(\epsilon(a,b)\) comes from the normalization of \eqref{e:type-C-one-variable-Cauchy}.
\end{rem}

\begin{rem}
The three examples treated in this chapter are not exhaustive.  The same completed-Cauchy
formalism can be applied, at least at the level of Hopf-dual skew identities, to other
graded Hopf dual pairs.  For instance:
\begin{enumerate}
    \item The peak algebra and the Hopf algebra of peak quasisymmetric functions form
    a natural dual pair related to Schur \(Q\)-theory and \(q=0\) Hecke--Clifford
    algebras; see Stembridge's enriched \(P\)-partitions and the Hopf-algebraic
    treatment in~\cite{Ste97,BHT04}.
    \item Tensor products such as \(\Lambda^{\otimes m}\), equipped with the tensor
    Hall pairing, provide colored symmetric-function models relevant to wreath
    products and cyclotomic Hecke algebras~\cite{AK94,Sho00}.
    \item A further filtered analogue is furnished by \(K\)-\(k\)-Schur functions and
    affine stable Grothendieck polynomials in \(K\)-theoretic affine Grassmannian
    Schubert calculus~\cite{LSS10K}.
\end{enumerate}
To turn these formal skew identities into explicit combinatorial formulas, one must
additionally identify the appropriate straight Pieri or Murnaghan--Nakayama input in
each setting.
\end{rem}

\chapter{Generating functions for irreducible characters}\label{s:GF-Hecke}
In this chapter we derive generating functions for the irreducible characters of Ariki--Koike algebras and Hecke--Clifford algebras, building on the results of the preceding chapters. Such generating functions provide a global perspective on the structure of irreducible characters, as will be illustrated by the application below.

\section{Ariki--Koike algebras and specializations}\label{sec:ariki-koike-algebras}
\phantomsection\label{def:multipartitions-arrays}
In this section, we need more notations. A {\em $m$-multipartition} of $n$ is an ordered tuple $\bm{\la}=(\la^{(1)},\la^{(2)},\cdots,\la^{(m)})$ of partitions $\la^{(i)}$ such that $|\bm{\la}|=\sum_{i=1}^{m}|\la^{(i)}|=n$. We denote by $\mathscr{P}_{n,m}$ the set of all $m$-multipartitions of $n$. Define $\ell(\bm{\la})=\sum_{i\geq1}^ml(\la^{(i)})$. 
For $a\ge0$ define
\[
C_{a,m}:=\Bigl\{c=(c_1,\dots,c_m)\in\mathbb Z_{\ge0}^m:\ \sum_{r=1}^m c_r=a\Bigr\},\qquad
\kappa(c):=\max\{r:\ c_r\ne0\}\ (a\ge1).
\]
Introduce auxiliary variables $Z^{(i)}:=z_{i,1}+z_{i,2}+\cdots$ $(1\le i\le m)$ and let
\begin{align}\label{e:def-A}
\mathscr A_m:=\Bigl\{\bm\alpha=(\alpha_{i,j})_{1\le i\le m,\ j\ge1}:\ 
\alpha_{i,j}\in\mathbb Z_{\ge0},\ \#\{(i,j):\alpha_{i,j}\ne0\}<\infty\Bigr\}.
\end{align}
For $\bm\alpha\in\mathscr A_m$ define the monomial
\[
Z^{\bm\alpha}:=\prod_{i=1}^m\prod_{j\ge1} z_{i,j}^{\,\alpha_{i,j}},\qquad
|\bm\alpha|:=\sum_{i,j}\alpha_{i,j},\qquad
\ell(\bm\alpha):=\#\{(i,j):\alpha_{i,j}>0\}.
\]

\subsection{Generating functions for Ariki--Koike algebras}\label{ss:plethysm-linear-comb}

Let $\mathscr W_{m,n}$ be the complex reflection group of type $G(m,1,n)$ (cf. \cite{ST54}), generated by
$s_1,s_2,\dots,s_n$ with relations
\begin{align}
\begin{split}
&s_{1}^{m}=1, \quad s_{2}^2=\cdots=s^2_{n}=1,\\
&s_1s_2s_1s_2=s_2s_1s_2s_1,\\
&s_is_j=s_js_i, \quad\text{if $|i-j|>1$},\\
&s_is_{i+1}s_i=s_{i+1}s_is_{i+1}, \quad\text{for $2\leq i<n$}.
\end{split}
\end{align}
Fix $\mathbb K=\mathbb C(\q,u_1,\dots,u_m)$. As the Hecke-deformation of $\mathscr W_{m,n}$, the Ariki--Koike algebra  \cite{AK94,BM93,C87} (also called the cyclotomic Hecke algebra of type $G(m,1,n)$)
$\mathscr H_{m,n}(\q,\bm u)$ is the unital $\mathbb K$-algebra generated by $g_1,g_2,\dots,g_n$ subject to
\begin{align}
\begin{split}
&(g_1-u_1)\cdots(g_1-u_m)=0,\\
&g_1g_2g_1g_2=g_2g_1g_2g_1,\\
&g_i^2=(\q-1)g_i+\q, \quad \text{for $2\leq i\leq n$},\\
&g_ig_j=g_jg_i, \quad \text{for $|i-j| > 1$},\\
&g_ig_{i+1}g_i = g_{i+1}g_ig_{i+1},\quad \text{for $2 \leq i < n$}.
\end{split}
\end{align}

It is known \cite{AK94} that $\mathscr H_{m,n}(\q,\bm u)$ is a free $\mathbb K$-module with
\[
\operatorname{rank}_{\mathbb K}\mathscr H_{m,n}(\q,\bm u)=n!m^n.
\]
Moreover, $\mathscr H_{m,n}(\q,\bm u)$ specializes to $\mathscr W_{m,n}$ at
\[
\q=1,\qquad
\bm u=\bm\omega=(1,\omega_m,\dots,\omega_m^{m-1}),
\]
where $u_i=\omega_m^{i-1}$ and $\omega_m$ is a primitive $m$th root of unity.
For a survey on the representation theory of $\mathscr{H}_{m,n}(q,\bm{u})$ and $q$-Schur algebras, see \cite{Mat04}.

Let $A$ be the matrix of degree $m$ whose $(a,b)$-entry is $u^a_b$  for $0 \leq a\leq m-1$ and $1\leq b \leq m$. Its determinant $\Delta=\det A=\prod_{i>j}(u_i-u_j)$ is the Vandermonde determinant. Let $B=A^*$ be the adjoint matrix of $A$. We write $B=(v_{b,a}(\bm{u}))$, where $v_{b,a}(\bm{u})$ is a polynomial in $\bm{u}$. For $1 \leq b \leq m$, we define $F_b(X)$ as the generating polynomial of $v_{b, a}(\bm u)$ as follows:
\begin{align*}
F_b(X)=\sum_{0\leq a\leq m-1}v_{b,a}(\bm{u})X^a.
\end{align*}

Following Shoji \cite{Sho00}, define $\mathscr H^\natural$ as the algebra generated by $g_2,\dots,g_n$ and commuting elements
$\xi_1,\dots,\xi_n$ with
\begin{align*}
&g_i^2=(\q-1)g_i+\q, \quad 2 \leq i \leq n;\quad (\xi_i-u_1)\cdots(\xi_i-u_m) = 0,\quad 1 \leq i \leq n;\\
&g_ig_{i+1}g_i = g_{i+1}g_ig_{i+1},\quad 2 \leq i < n;\quad g_ig_j = g_jg_i,\quad |i-j| \geq 2;\\
&\xi_i\xi_j = \xi_j\xi_i,\quad 1 \leq i, j \leq n;\quad g_j\xi_i = \xi_ig_j ,\quad i \neq j-1, j;\\
&g_j\xi_j
= \xi_{j-1}g_j + \Delta^{-2}\sum_{a<b}(u_a-u_b)(\q-\q^{-1})F_a(\xi_{j-1})F_b(\xi_j),
\quad 2\leq j\leq n;\\
&g_j\xi_{j-1}
= \xi_{j}g_j-\Delta^{-2}\sum_{a<b}(u_a-u_b)(\q-\q^{-1})F_a(\xi_{j-1})F_b(\xi_j),
\quad 2\leq j\leq n.
\end{align*}
By \cite[Th.\ 3.7]{Sho00}, $\mathscr H^\natural\simeq \mathscr H_{m,n}(\q,\bm u)$. We henceforth identify
$\mathscr H^\natural$ with $\mathscr H_{m,n}(\q,\bm u)$.

For $1\le i\le m$ and $r\ge1$, define in $\mathscr H^\natural_{m,r}$ the $i$-cycle element
\[
g_{r,i}:=\xi_r^{\,i-1}g_rg_{r-1}\cdots g_2,\qquad (g_{1,i}:=\xi_1^{\,i-1}).
\]
For a partition $\mu=(\mu_1,\dots,\mu_\ell)$ of $r$, using the canonical block-embedding
$\iota_\mu:\mathscr H^\natural_{m,\mu_1}\otimes\cdots\otimes\mathscr H^\natural_{m,\mu_\ell}\hookrightarrow \mathscr H^\natural_{m,r}$,
set
\[
g(\mu,i):=\iota_\mu\bigl(g_{\mu_1,i}\otimes\cdots\otimes g_{\mu_\ell,i}\bigr)\in\mathscr H^\natural_{m,r}.
\]
For a multipartition $\bm\mu=(\mu^{(1)},\dots,\mu^{(m)})\in\mathscr P_{n,m}$ define the {\em standard element}
\[
g(\bm\mu):=g(\mu^{(1)},1)\,g(\mu^{(2)},2)\cdots g(\mu^{(m)},m)\in\mathscr H^\natural_{m,n}.
\]
By \cite[Prop.\ 7.5]{Sho00}, irreducible characters are fully determined by their values on $\{g(\bm\mu)\}$.
Thus for $\bm\lambda,\bm\mu\in\mathscr P_{n,m}$ we write
\phantomsection\label{def:chi-AK}
\[
\chi^{\bm\lambda}_{\bm\mu}:=\chi^{\bm\lambda}\!\bigl(g(\bm\mu)\bigr).
\]

Fix $m\ge1$. For each $r=1,\dots,m$ let $X^{(r)}:=x^{(r)}_1+x^{(r)}_2+\cdots$ and set
$$s_{\bm\lambda}(X):=\prod_{r=1}^m s_{\lambda^{(r)}}(X^{(r)}).$$
Let $q_a(\cdot\,;\q^{-1})$ be the one-row Hall--Littlewood functions with parameter $\q^{-1}$. Put
\[
H_r(w):=\sum_{a\ge0} q_a\!\bigl(X^{(r)};\q^{-1}\bigr)\,w^a,\qquad
P_0(w):=1,\qquad P_k(w):=\prod_{r=1}^k H_r(w).
\]
Thus we have the plethystic forms:
\begin{align}\label{e:plethysm-form}
H_r(w)=\sigma_{w}\big[(1-\q^{-1})X^{(r)}\big],\qquad P_k(w)=\sigma_w\big[(1-\q^{-1})X^{\leq k}\big],
\end{align}
where $X^{\le k}:=X^{(1)}+\cdots+X^{(k)}$ and $\sigma_w[A]:=\sum_{a\ge0}h_a[wA]$.

For $1\le i\le m$ set
\[
q^{(i)}_0(X;\q,\bm u):=\frac1{\q-1},\qquad
q^{(i)}_a(X;\q,\bm u):=\frac{\q^{a}}{\q-1}
\sum_{c\in C_{a,m}} u_{\kappa(c)}^{\,i-1}\prod_{r=1}^m q_{c_r}\!\bigl(X^{(r)};\q^{-1}\bigr)\quad(a\ge1).
\]
For $\bm\mu\in\mathscr P_{n,m}$ define
\[
q_{\bm\mu}(X;\q,\bm u):=\prod_{i=1}^m\ \prod_{j=1}^{\ell(\mu^{(i)})} q^{(i)}_{\mu^{(i)}_j}(X;\q,\bm u).
\]
Shoji's Frobenius formula states that for each $\bm\mu\in\mathscr P_{n,m}$,
\begin{equation}\label{eq:Frob-self}
q_{\bm\mu}(X;\q,\bm u)=\sum_{\bm\lambda\in\mathscr P_{n,m}} \chi^{\bm\lambda}_{\bm\mu}\, s_{\bm\lambda}(X).
\end{equation}

We introduce the modifications by
\[
\widehat q^{(i)}_a(X;\q,\bm u):=(\q-1)\,\q^{-a}\,q^{(i)}_a(X;\q,\bm u),\qquad
\widehat q_{\bm\alpha}:=\prod_{i=1}^m\ \prod_{j\ge1}\widehat q^{(i)}_{\alpha_{i,j}},
\]
so that $\widehat q^{(i)}_0=1$ and, for every finite-support exponent array $\bm\alpha$, i.e., $\bm \alpha\in\mathscr A_m$,
\begin{equation}\label{eq:qhat-factor-self}
\widehat q_{\bm\alpha}(X;\q,\bm u)=(\q-1)^{\ell(\bm\alpha)}\,\q^{-|\bm\alpha|}\,q_{\bm\alpha}(X;\q,\bm u),
\end{equation}
where
\[
q_{\bm\alpha}:=\prod_{\alpha_{i,j}>0}q^{(i)}_{\alpha_{i,j}}.
\]

For $\bm\alpha\in\mathscr A_m$, let $\alpha^{(i)}=(\alpha_{i,1},\alpha_{i,2},\ldots)$ be the $i$th component.
Reorder the positive entries of $\alpha^{(i)}$ in weakly decreasing order to obtain a partition, denoted $\alpha^{(i)\downarrow}$, and set
\[
\bm\alpha^\downarrow:=\bigl(\alpha^{(1)\downarrow},\ldots,\alpha^{(m)\downarrow}\bigr)\in\mathscr P_{*,m}:=\bigcup_{n\ge0}\mathscr P_{n,m}.
\]
Since both $q_{\bm\alpha}$ and $\widehat q_{\bm\alpha}$ are invariant under permuting the indices $j$ within each fixed $i$, we have
\begin{equation}\label{eq:qalpha-invariance}
q_{\bm\alpha}(X;\q,\bm u)=q_{\bm\alpha^\downarrow}(X;\q,\bm u),\qquad
\widehat q_{\bm\alpha}(X;\q,\bm u)=\widehat q_{\bm\alpha^\downarrow}(X;\q,\bm u).
\end{equation}
We make the convention $\chi^{\bla}_{\bm \alpha}=0$ if $|\bla|\neq|\bm \alpha|$.

For each $\bm\alpha\in\mathscr A_m$, define $\chi^{\bm\lambda}_{\bm\alpha}\in\mathbb K$ by the Schur expansion
\begin{equation}\label{eq:Frob-alpha}
q_{\bm\alpha}(X;\q,\bm u)=\sum_{\bm\lambda\in\mathscr P_{*,m}}\chi^{\bm\lambda}_{\bm\alpha}\, s_{\bm\lambda}(X).
\end{equation}
By \eqref{eq:qalpha-invariance}, the coefficient $\chi^{\bm\lambda}_{\bm\alpha}$ depends only on $\bm\alpha^\downarrow$; equivalently,
\[
\chi^{\bm\lambda}_{\bm\alpha}=\chi^{\bm\lambda}_{\bm\alpha^\downarrow}\qquad(\bm\alpha\in\mathscr A_m).
\]
If $\bm\mu\in\mathscr P_{n,m}$ is a multipartition, we view it as an element of $\mathscr A_m$ via $\alpha_{i,j}=\mu^{(i)}_j$
(with $\mu^{(i)}_j=0$ for $j>\ell(\mu^{(i)})$). Then \eqref{eq:Frob-alpha} reduces to Shoji's Frobenius formula
\eqref{eq:Frob-self}, and in particular $\chi^{\bm\lambda}_{\bm\mu}$ coincides with the character value
$\chi^{\bm\lambda}(g(\bm\mu))$ defined above.

We introduce the generating function of $\chi^{\bla}_{\bm \alpha}$ by
\begin{equation}\label{eq:G-def-Zalpha}
\mathscr G_{\bm{\lambda}}(Z;\q,\bm u)
:=\sum_{\bm\alpha\in\mathscr A_m}
(\q-1)^{\ell(\bm\alpha)}\,\chi^{\bm\lambda}_{\bm\alpha}\,Z^{\bm\alpha}.
\end{equation}
In particular, when $\bm\mu\in\mathscr P_{n,m}$, the coefficient of $Z^{\bm\mu}$ in $\mathscr G_{\bm\lambda}$ is the normalized character value
$(\q-1)^{\ell(\bm\mu)}\chi^{\bm\lambda}_{\bm\mu}$.

Define $\widehat Q^{(i)}(z):=\sum_{a\ge0}\widehat q^{(i)}_a z^a$ and the kernel
\begin{equation}\label{eq:kernel-def-Zalpha}
\mathcal K(X,Z;\q,\bm u):=\prod_{i=1}^m\ \prod_{j\ge1}\widehat Q^{(i)}(z_{i,j})
=\sum_{\bm\alpha\in\mathscr A_m} \widehat q_{\bm\alpha}(X;\q,\bm u)\, Z^{\bm\alpha}.
\end{equation}
Using \eqref{eq:Frob-alpha} and \eqref{eq:qhat-factor-self} we obtain the Schur expansion
\[
\mathcal K(X,Z;\q,\bm u)=\sum_{\bm\lambda\in\mathscr P_{*,m}} \q^{-|\bla|}\,s_{\bm\lambda}(X)\,\mathscr G_{\bm\lambda}(Z;\q,\bm u).
\]

\begin{lem}
For $1\le i\le m$ one has
\begin{equation}\label{eq:Qhat-lin-self}
\widehat Q^{(i)}(z)=1+\sum_{k=1}^m u_k^{\,i-1}\bigl(P_k(z)-P_{k-1}(z)\bigr)
=\sum_{k=0}^m b_k^{(i)}\,P_k(z),
\end{equation}
where
\begin{align}\label{e:def-b}
b^{(i)}_k=
\begin{cases}
    1-u_1^{i-1}, &\text{if $k=0$}\\
    u_k^{\,i-1}-u_{k+1}^{\,i-1}, &\text{if $1\le k\le m-1$}\\
    u_m^{i-1}, &\text{if $k=m$}.
\end{cases}
\end{align}
\end{lem}
\begin{proof}
For $a\ge1$,
\[
\widehat q^{(i)}_a=\sum_{c\in C_{a,m}} u_{\kappa(c)}^{\,i-1}\prod_{r=1}^m q_{c_r}\!\bigl(X^{(r)};\q^{-1}\bigr),
\]
and summing over $a$ with fixed last nonzero position $\kappa(c)=k$ gives $P_k(z)-P_{k-1}(z)$.
\end{proof}

For $N\ge1$, let
\[
\mathscr G_{\bm{\lambda}}^{[N]}(Z;\q,\bm u)
:=\mathscr G_{\bm{\lambda}}(Z;\q,\bm u)\big|_{z_{i,j}=0\ {\rm for}\ j>N}.
\]
For each finite choice
\[
\kappa=(\kappa_{i,j})_{1\le i\le m,\ 1\le j\le N}
\in\{0,1,\dots,m\}^{\{1,\dots,m\}\times\{1,\dots,N\}},
\]
define alphabets
\[
Y_{r,N}(\kappa):=
\sum_{\substack{1\le i\le m,\ 1\le j\le N\\ \kappa_{i,j}\ge r}} z_{i,j}
\qquad (r=1,\dots,m).
\]
\begin{thm}\label{t:gene-AK}
For every multipartition $\bm\lambda$ and every $N\ge1$, we have
\begin{align}\label{eq:gene-AK-finite}
\mathscr G_{\bm{\lambda}}^{[N]}(Z;\q,\bm u)
=
\sum_{\kappa\in\{0,1,\dots,m\}^{\{1,\dots,m\}\times\{1,\dots,N\}}}
\left(\prod_{i=1}^m\ \prod_{j=1}^N b^{(i)}_{\kappa_{i,j}}\right)
\ \prod_{r=1}^m
s_{\lambda^{(r)}}\!\Bigl[(\q-1)\,Y_{r,N}(\kappa)\Bigr],
\end{align}
where $b_k^{(i)}$ are as in \eqref{e:def-b}. Consequently,
$\mathscr G_{\bm{\lambda}}(Z;\q,\bm u)$ is the coefficientwise stable limit of
the finite sums in \eqref{eq:gene-AK-finite} as $N\to\infty$.
\end{thm}

\begin{proof}
Let
\[
\mathcal K^{[N]}(X,Z;\q,\bm u)
:=\prod_{i=1}^m\ \prod_{j=1}^N\widehat Q^{(i)}(z_{i,j})
=\mathcal K(X,Z;\q,\bm u)\big|_{z_{i,j}=0\ {\rm for}\ j>N}.
\]
Then
\[
\mathcal K^{[N]}(X,Z;\q,\bm u)
=\sum_{\bm\lambda\in\mathscr P_{*,m}}
\q^{-|\bla|}\,s_{\bm\lambda}(X)\,
\mathscr G_{\bm{\lambda}}^{[N]}(Z;\q,\bm u).
\]
Insert \eqref{eq:Qhat-lin-self} into this finite product:
\[
\mathcal K^{[N]}(X,Z;\q,\bm u)
=\prod_{i=1}^m\prod_{j=1}^N
\left(\sum_{k=0}^m b_k^{(i)}\,P_k(z_{i,j})\right)
=
\sum_{\kappa}
\left(\prod_{i=1}^m\prod_{j=1}^N b^{(i)}_{\kappa_{i,j}}\right)
\prod_{i=1}^m\prod_{j=1}^N P_{\kappa_{i,j}}(z_{i,j}),
\]
where the sum is over
$\kappa\in\{0,1,\dots,m\}^{\{1,\dots,m\}\times\{1,\dots,N\}}$.
Since $P_k(w)=\prod_{r=1}^k \sigma_w\!\bigl[(1-\q^{-1})X^{(r)}\bigr]$ by \eqref{e:plethysm-form}, plethystic multiplicativity gives
\[
\prod_{i=1}^m\prod_{j=1}^N P_{\kappa_{i,j}}(z_{i,j})
=\prod_{r=1}^m \sigma_{Y_{r,N}(\kappa)}\!\bigl[(1-\q^{-1})X^{(r)}\bigr].
\]
Recall that \eqref{e:skew-M-N-Mac} was obtained from the skew Cauchy identity \eqref{e:ge-skew-Mac} by specializing the auxiliary alphabet, normalizing, and extracting homogeneous components. Returning to the latter and specializing $(q,t,\la,\eta)=(0,0,\varnothing,\varnothing)$ gives, for any alphabet $B$,
\[
\exp\left(\sum_{n\ge1}\frac{p_n[X^{(r)}]p_n[B]}{n}\right)
=\sum_{\lambda\in\mathscr P}s_\lambda[X^{(r)}]s_\lambda[B].
\]
Taking $B=(1-\q^{-1})Y_{r,N}(\kappa)$, we obtain
\[
\begin{aligned}
\sigma_{Y_{r,N}(\kappa)}\!\bigl[(1-\q^{-1})X^{(r)}\bigr]
&=\sum_{a\geq 0}h_a\bigl[ (1-\q^{-1})Y_{r,N}(\kappa) X^{(r)} \bigr]\\
&=\sum_{a\geq 0}\sum_{\lambda^{(r)}\vdash a}
s_{\lambda^{(r)}}[X^{(r)}]\,
s_{\lambda^{(r)}}\!\Bigl[(1-\q^{-1})Y_{r,N}(\kappa)\Bigr].
\end{aligned}
\]
Multiplying over $r$ yields
\[
\begin{aligned}
\prod_{i=1}^m\prod_{j=1}^N P_{\kappa_{i,j}}(z_{i,j})
&=\sum_{\bm\lambda} s_{\bm\lambda}(X)\,
\prod_{r=1}^m
s_{\lambda^{(r)}}\!\Bigl[(1-\q^{-1})Y_{r,N}(\kappa)\Bigr]\\
&=\sum_{\bm\lambda} \q^{-|\bla|}s_{\bm\lambda}(X)\,
\prod_{r=1}^m
s_{\lambda^{(r)}}\!\Bigl[(\q-1)Y_{r,N}(\kappa)\Bigr].
\end{aligned}
\]
Substitute back and compare the coefficient of $s_{\bm\lambda}(X)$ with the Schur expansion
\[
\mathcal K^{[N]}=\sum_{\bm\lambda} \q^{-|\bla|}s_{\bm\lambda}(X)\mathscr G_{\bm\lambda}^{[N]}(Z;\q,\bm u).
\]
This gives \eqref{eq:gene-AK-finite}. The stable-limit statement follows because every monomial in the variables $Z$ involves only finitely many indices $j$.
\end{proof}

Theorem \ref{t:gene-AK} presents generating functions in the conjugacy-class parameter (the ``lower" index). Next we give generating functions that involve both the conjugacy-class parameter and the representation-theory partition (the ``upper" index). Before that, we need to make some preparations: recall the vertex operator realization of Hall--Littlewood functions and introduce some notations.

Let us recall the vertex operator realization of the Hall--Littlewood function. Introduce the map $H_n: \Lambda_{\mathbb Q(t)}\longrightarrow\Lambda_{\mathbb Q(t)}$ by
\begin{align}
    H(u)=\exp\left(\sum_{n=1}^{\infty}\frac{1-t^n}{n}p_nu^n\right)\exp\left(-\sum_{n=1}^{\infty}\frac{\partial}{\partial p_n}u^{-n}\right)=\sum_{n\in\mathbb Z}H_nu^n.
\end{align}
Then for a partition $\la$, according to \cite[Section 3]{J91}, $H_{\la_1}H_{\la_2}\cdots\,1$ is exactly the Hall--Littlewood function $Q_{\la}(X;t)$ labeled by $\la$. The following lemma gives the generating function of Hall--Littlewood functions.
\begin{lem}\label{l:gene-HL}
    Let $U=u_1+u_2+\cdots$. Denote $U^{\la}:=u_1^{\la_1}u_2^{\la_2}\cdots u_{\ell(\la)}^{\la_{\ell(\la)}}$ for a partition $\la$. Then 
    \begin{align}\label{e:gene-HL}
        \sum_{\la\in \P}Q_{\la}(X;t)U^{\la}=\prod_{i<j}\frac{u_i-u_j}{u_i-tu_j}\prod_{i,j}\frac{1-x_itu_j}{1-x_iu_j}.
    \end{align}
\end{lem}
\begin{proof}
    By the vertex operator realization of Hall--Littlewood functions,
    \begin{multline}
        Q_{\la}(X;t)=[U^{\la}]\prod_{j=1}^{\infty}\exp\left(\sum_{n=1}^{\infty}\frac{1-t^n}{n}p_nu_j^n\right)\exp\left(-\sum_{n=1}^{\infty}\frac{\partial}{\partial p_n}u_j^{-n}\right)\cdot1\\
        =[U^{\la}]\prod_{i<j}\frac{u_i-u_j}{u_i-tu_j}\prod_{j=1}^{\infty}\exp\left(\sum_{n=1}^{\infty}\frac{1-t^n}{n}p_nu_j^n\right)\\
        =[U^{\la}]\prod_{i<j}\frac{u_i-u_j}{u_i-tu_j}\prod_{j=1}^{\infty}\prod_{i=1}^{\infty}\frac{1-x_itu_j}{1-x_iu_j}.
    \end{multline}
    We use the Baker--Campbell--Hausdorff formula (see e.g. \cite[Theorem 5.1]{Hal15}) in the second identity. This completes the proof.  
\end{proof}
Specializing \eqref{e:gene-HL} at $t=0$ gives 
\begin{align}\label{e:gene-S}
     \sum_{\la\in \P}s_{\la}(X)U^{\la}=&\prod_{i<j}\left(1-\frac{u_j}{u_i}\right)\prod_{i,j}\frac{1}{1-x_iu_j}.
\end{align}

We introduce
$m$ independent auxiliary alphabets
\[
U^{(r)}:=u_{r,1}+u_{r,2}+\cdots\qquad (1\le r\le m),
\]
For a partition $\la$ we write
\[
(U^{(r)})^\la:=u_{r,1}^{\la_1}u_{r,2}^{\la_2}\cdots u_{r,\ell(\la)}^{\la_{\ell(\la)}},
\]
and for a multipartition $\bm\lambda=(\lambda^{(1)},\ldots,\lambda^{(m)})$ set
\[
U^{\bm\lambda}:=\prod_{r=1}^m (U^{(r)})^{\lambda^{(r)}}.
\]
Define the bivariate generating function
\begin{equation}\label{eq:AK-bivar-def}
\mathscr F(Z,U;\q,\bm u)
:=\sum_{\bm\lambda\in\mathscr P_{*,m}}\mathscr G_{\bm\lambda}(Z;\q,\bm u)\,U^{\bm\lambda}
=\sum_{\bm\lambda\in\mathscr P_{*,m}}\ \sum_{\bm\alpha\in\mathscr A_m}
(\q-1)^{\ell(\bm\alpha)}\chi^{\bm\lambda}_{\bm\alpha}\,Z^{\bm\alpha}U^{\bm\lambda}.
\end{equation}

For $1\le r\le m$ set
\begin{equation}\label{eq:Delta-Phi-def}
\begin{aligned}
\Delta\!\bigl(U^{(r)}\bigr)&:=\prod_{a<b}\left(1-\frac{u_{r,b}}{u_{r,a}}\right),\\
\Phi_r(z)&:=\prod_{a\ge1}\frac{1-z\,u_{r,a}}{1-\q z\,u_{r,a}}
=\sigma_{z}\bigl[(\q-1)U^{(r)}\bigr],
\end{aligned}
\end{equation}
and set
\[
\Phi_{\le k}(z):=\prod_{r=1}^k\Phi_r(z),
\qquad
\Phi_{\le0}(z):=1.
\]
(Here $\sigma_z[A]=\sum_{d\ge0}h_d[zA]$ is plethystic; in particular,
$\sigma_z[(\q-1)B]=\sigma_z[\q B]/\sigma_z[B]$, hence the ratio in \eqref{eq:Delta-Phi-def}.)

\begin{thm}\label{t:AK-bivariate}
With the above notation, one has the closed product formula
\begin{equation}\label{eq:AK-bivar-formula}
\mathscr F(Z,U;\q,\bm u)
=
\left(\prod_{r=1}^m \Delta\!\bigl(U^{(r)}\bigr)\right)
\ \prod_{i=1}^m\ \prod_{j\ge1}
\left(
\sum_{k=0}^m b_k^{(i)}\,\Phi_{\le k}(z_{i,j})
\right),
\end{equation}
where the coefficients $b_k^{(i)}$ are given by \eqref{e:def-b}.
The products over $j\ge1$ are understood coefficientwise, as stable limits of
the corresponding finite products over $1\le j\le N$.
Equivalently,
\[
\mathscr F(Z,U;\q,\bm u)
=
\left(\prod_{r=1}^m \prod_{a<b}\left(1-\frac{u_{r,b}}{u_{r,a}}\right)\right)
\prod_{i=1}^m\prod_{j\ge1}
\left(
\sum_{k=0}^m b_k^{(i)}
\prod_{r=1}^k\ \prod_{a\ge1}\frac{1-z_{i,j}u_{r,a}}{1-\q z_{i,j}u_{r,a}}
\right).
\]
\end{thm}

\begin{proof}
For $N\ge1$, let
\[
\mathscr F^{[N]}(Z,U;\q,\bm u)
:=\mathscr F(Z,U;\q,\bm u)\big|_{z_{i,j}=0\ {\rm for}\ j>N}.
\]
Starting from the finite formula \eqref{eq:gene-AK-finite}, multiply both sides by $U^{\bm\lambda}$ and sum over all
$\bm\lambda=(\lambda^{(1)},\ldots,\lambda^{(m)})\in\mathscr P_{*,m}$. For each fixed $r$ we use the
Schur generating function (the specialization $t=0$ of Lemma~\ref{l:gene-HL}): for any alphabet $A$,
\begin{equation}\label{eq:Schur-monomial-gf-pleth}
\sum_{\la\in\P}s_\la[A]\,(U^{(r)})^\la
=
\Delta\!\bigl(U^{(r)}\bigr)\ \prod_{a\ge1}\sigma_{u_{r,a}}[A].
\end{equation}
Applying \eqref{eq:Schur-monomial-gf-pleth} with $A=(\q-1)Y_{r,N}(\kappa)$ yields, for each finite $\kappa$,
\[
\sum_{\lambda^{(r)}\in\P} s_{\lambda^{(r)}}\!\bigl[(\q-1)Y_{r,N}(\kappa)\bigr]\,(U^{(r)})^{\lambda^{(r)}}
=
\Delta\!\bigl(U^{(r)}\bigr)\ \prod_{a\ge1}\ \prod_{\substack{1\le i\le m,\ 1\le j\le N\\ \kappa_{i,j}\ge r}}
\frac{1-z_{i,j}u_{r,a}}{1-\q z_{i,j}u_{r,a}}.
\]
Multiplying over $r=1,\ldots,m$ gives
\[
\sum_{\bm\lambda\in\mathscr P_{*,m}}
\prod_{r=1}^m s_{\lambda^{(r)}}\!\bigl[(\q-1)Y_{r,N}(\kappa)\bigr]\ U^{\bm\lambda}
=
\left(\prod_{r=1}^m \Delta\!\bigl(U^{(r)}\bigr)\right)
\prod_{i=1}^m\prod_{j=1}^N\Phi_{\le \kappa_{i,j}}(z_{i,j}),
\]
because $z_{i,j}$ appears in $Y_{r,N}(\kappa)$ precisely for $1\le r\le \kappa_{i,j}$.
Substituting this identity into \eqref{eq:gene-AK-finite} and summing over all finite $\kappa$ we obtain
\[
\mathscr F^{[N]}(Z,U;\q,\bm u)
=
\left(\prod_{r=1}^m \Delta\!\bigl(U^{(r)}\bigr)\right)
\sum_{\kappa}
\left(\prod_{i=1}^m\prod_{j=1}^N b^{(i)}_{\kappa_{i,j}}\right)
\prod_{i=1}^m\prod_{j=1}^N\Phi_{\le \kappa_{i,j}}(z_{i,j}).
\]
Since $\kappa_{i,j}\in\{0,1,\ldots,m\}$ and the summand factorizes over the pairs $(i,j)$, the $\kappa$-sum
splits as a product:
\[
\sum_{\kappa}\prod_{i=1}^m\prod_{j=1}^N
\Bigl(b^{(i)}_{\kappa_{i,j}}\Phi_{\le\kappa_{i,j}}(z_{i,j})\Bigr)
=
\prod_{i=1}^m\prod_{j=1}^N
\left(\sum_{k=0}^m b_k^{(i)}\Phi_{\le k}(z_{i,j})\right).
\]
This proves the finite truncation of \eqref{eq:AK-bivar-formula}. Letting $N\to\infty$ coefficientwise gives
\eqref{eq:AK-bivar-formula}.
\end{proof}

\subsection{Iwahori--Hecke algebra in type $A$}\label{ss:typeA-hecke}

In this subsection we specialize the Ariki--Koike algebra to the Iwahori--Hecke algebra in type $A$.

The Iwahori--Hecke algebra in type $A$ (denoted by $\mathscr{H}_{n}(\q)$) is the unital associative algebra
over $\mathbb{C}(q)$ generated by generators $T_{1}, T_{2},\ldots, T_{n-1}$ subject to the relations
\begin{align}\label{e:Hecke1}
T_{i}T_{j}&=T_{j}T_{i},  ~~~~\text{if}~ |i-j|>1,\\ \label{e:Hecke2}
T_{i}T_{i+1}T_{i}&=T_{i+1}T_{i}T_{i+1},\\ \label{e:Hecke3}
T_{i}^{2}&=(\q-1)T_{i}+\q.
\end{align}
It is clear that $\mathscr{H}_{m,n}(q,\bm{u})$ specializes to $\mathscr{H}_{n}(q)$ at $m=1$ and $\bm{u}=1$. In this case, multipartitions are just partitions; we write $\lambda,\mu\vdash n$.

For $\mu=(\mu_1,\dots,\mu_\ell)\vdash n$ the standard element $g(\mu)$ defined in the previous subsection
is the block-embedding product of these ``cycle elements''. The set of all irreducible modules of $\mathscr H_n(\q)$ is parametrized by $\P_n$. Let $\phi^{\la}$ denote the irreducible character of $\mathscr H_n(\q)$ associated with the partition $\la$. We continue to write
\phantomsection\label{def:phi-typeA}
$\phi^\lambda_\mu:=\phi^\lambda\bigl(g(\mu)\bigr)$ (see \cite{Ram91}).

\medskip

We also simplify the auxiliary alphabets as follows.  Since $m=1$, the $Z$-variables reduce to a single family
\[
Z:=Z^{(1)}=z_1+z_2+\cdots \qquad (z_j:=z_{1,j}),
\]
and the exponent arrays $\bm\alpha\in\mathscr A_1$ may be identified with finite-support sequences
$\alpha=(\alpha_1,\alpha_2,\dots)$, with monomials $Z^\alpha=\prod_{j\ge1}z_j^{\alpha_j}$ and
$\ell(\alpha)=\#\{j:\alpha_j>0\}$.
We keep the convention $\chi^\lambda_\alpha=0$ unless $|\lambda|=|\alpha|$.

\begin{cor}\label{c:gene-typeA}
Assume $m=1$ and $u_1=1$.  Then for every partition $\lambda$ one has the closed plethystic formula
\begin{equation}\label{eq:gene-typeA}
\mathscr G_\lambda(Z;\q)
=\sum_{\alpha\in\mathscr A_1}(\q-1)^{\ell(\alpha)}\phi^\lambda_\alpha\,Z^\alpha
= s_\lambda\!\bigl[(\q-1)\,Z\bigr].
\end{equation}
Therefore, for every partition $\mu\vdash n$,
\begin{equation}\label{eq:coeff-typeA}
(\q-1)^{\ell(\mu)}\,\phi^\lambda_\mu
=\bigl[Z^\mu\bigr]\ s_\lambda\!\bigl[(\q-1)\,Z\bigr].
\end{equation}
Here we use $[u]v$ to denote the coefficient of $u$ in the expansion of $v$.
\end{cor}

\begin{proof}
When $m=1$ and $u_1=1$, the coefficients $b_k^{(i)}$ in \eqref{e:def-b} reduce to
\[
b^{(1)}_0=1-u_1^{0}=0,\qquad b^{(1)}_1=u_1^{0}=1.
\]
Hence, in the finite formula \eqref{eq:gene-AK-finite}, the sum collapses to the unique choice
$\kappa_{1,j}=1$ for $1\le j\le N$. For this choice we have
$Y_{1,N}(\kappa)=z_1+\cdots+z_N$, and the product of $b$-factors equals $1$.
Passing to the coefficientwise stable limit gives \eqref{eq:gene-typeA}, and \eqref{eq:coeff-typeA} follows by taking
the coefficient of $Z^\mu$.
\end{proof}

Next we specialize the bivariate generating function \eqref{eq:AK-bivar-def}.
To avoid notational collision with the cyclotomic parameter $u_1=1$, we rename the auxiliary alphabet
$U^{(1)}$ as
\[
V:=v_1+v_2+\cdots,
\qquad
V^\lambda:=v_1^{\lambda_1}v_2^{\lambda_2}\cdots v_{\ell(\lambda)}^{\lambda_{\ell(\lambda)}}.
\]
Set
\[
\Delta(V):=\prod_{a<b}\Bigl(1-\frac{v_b}{v_a}\Bigr),
\qquad
\Phi(z):=\prod_{a\ge1}\frac{1-zv_a}{1-\q z\,v_a}
=\sigma_z\bigl[(\q-1)V\bigr].
\]

\begin{cor}\label{c:bivar-typeA}
Assume $m=1$ and $u_1=1$.  Then the bivariate generating function
\[
\mathscr F(Z,V;\q)
=\sum_{\lambda\in\mathscr P}\mathscr G_\lambda(Z;\q)\,V^\lambda
=\sum_{\lambda\in\mathscr P}\ \sum_{\alpha\in\mathscr A_1}(\q-1)^{\ell(\alpha)}\phi^\lambda_\alpha\,Z^\alpha V^\lambda
\]
admits the closed product form
\begin{equation}\label{eq:bivar-typeA}
\mathscr F(Z,V;\q)
=\Delta(V)\ \prod_{j\ge1}\Phi(z_j)
=\left(\prod_{a<b}\Bigl(1-\frac{v_b}{v_a}\Bigr)\right)\ \prod_{j\ge1}\ \prod_{a\ge1}\frac{1-z_jv_a}{1-\q z_jv_a}.
\end{equation}
Namely,
\begin{equation}\label{eq:Cauchy-typeA}
\sum_{\lambda\in\mathscr P} s_\lambda\!\bigl[(\q-1)Z\bigr]\ V^\lambda
=\Delta(V)\ \prod_{j\ge1}\ \prod_{a\ge1}\frac{1-z_jv_a}{1-\q z_jv_a}.
\end{equation}
\end{cor}

\begin{proof}
This follows immediately from Theorem~\ref{t:AK-bivariate} by the same specialization
$b^{(1)}_0=0$, $b^{(1)}_1=1$, which reduces \eqref{eq:AK-bivar-formula} to \eqref{eq:bivar-typeA}.
Alternatively, summing \eqref{eq:gene-typeA} over $\lambda$ against $V^\lambda$ and applying the Schur monomial
generating function \eqref{eq:Schur-monomial-gf-pleth} with the alphabet $A=(\q-1)Z$ yields \eqref{eq:Cauchy-typeA}.
\end{proof}

\subsection{Iwahori--Hecke algebra in type $B$}\label{ss:typeB-hecke}

The Iwahori--Hecke algebra $\mathscr{B}_n(u,\q)$ of type $B$ is the associative algebra over
$\mathbb C(u,\q)$ with generators $g_1,g_2,\ldots,g_n$ and relations
\begin{align*}
g_ig_j &= g_jg_i, \quad|i-j|>1,\\
g_ig_{i+1}g_i &= g_{i+1}g_ig_{i+1},\quad 2\le i\le n-1,\\
g_1g_2g_1g_2 &= g_2g_1g_2g_1,\\
g_1^2 &= (u-1)g_1+u,\\
g_i^2 &= (\q-1)g_i+\q,\quad 2\le i\le n.
\end{align*}
Note that $(g_1+1)(g_1-u)=0$ is equivalent to $g_1^2=(u-1)g_1+u$.
Hence $\mathscr H_{m,n}(\q,\bm u)$ specializes to $\mathscr B_n(u,\q)$ at
\[
m=2,\qquad \bm u=(u_1,u_2)=(-1,u).
\]
In this case multipartitions are bipartitions; we write
\[
\bm\lambda=(\lambda^{(1)},\lambda^{(2)})=(\lambda^+,\lambda^-),\qquad
\bm\mu=(\mu^{(1)},\mu^{(2)})=(\mu^+,\mu^-).
\]
The standard elements are $g(\bm\mu)=g(\mu^+,1)\,g(\mu^-,2)$, $\bm \mu\in\P_{n,2}$. The set of all irreducible modules of $\mathscr B_n(u,\q)$ is parametrized by $\P_{n,2}$. Let $\varphi^{\bm \la}$ ($\bm \la\in\P_{n,2}$) denote the irreducible character of $\mathscr B_n(u,\q)$ associated with the bipartition $\bm \la$. We continue to write
\phantomsection\label{def:varphi-typeB}
$\varphi^{\bm \lambda}_{\bm \mu}:=\varphi^{\bm \lambda}\bigl(g(\bm \mu)\bigr)$. 

\medskip

We also rename the two $Z$-alphabets as
\[
Z^+:=Z^{(1)}=z^+_1+z^+_2+\cdots,\qquad
Z^-:=Z^{(2)}=z^-_1+z^-_2+\cdots
\quad (z^+_j:=z_{1,j},\quad z^-_j:=z_{2,j}),
\]
so that for $\bm\alpha=(\alpha^+,\alpha^-)\in\mathscr A_2$ we have
$Z^{\bm\alpha}=\prod_{j\ge1}(z^+_j)^{\alpha^+_j}\prod_{j\ge1}(z^-_j)^{\alpha^-_j}$ and
$\ell(\bm\alpha)=\ell(\alpha^+)+\ell(\alpha^-)$.
We keep the convention $\varphi^{\bm\lambda}_{\bm\alpha}=0$ unless $|\bm\lambda|=|\bm\alpha|$.

\begin{cor}\label{c:gene-typeB}
Assume $m=2$ and $(u_1,u_2)=(-1,u)$. For $N\ge1$, set
\[
\mathscr G_{\bm\lambda}^{[N]}(Z^+,Z^-;\q,u)
:=\mathscr G_{\bm\lambda}(Z^+,Z^-;\q,u)\big|_{z_j^\pm=0\ {\rm for}\ j>N},
\qquad
Z^+_{\le N}:=\sum_{j=1}^N z^+_j.
\]
Then for every bipartition $\bm\lambda=(\lambda^+,\lambda^-)$ and every $N\ge1$ one has
\begin{align}\label{eq:gene-typeB}
\mathscr G_{\bm\lambda}^{[N]}(Z^+,Z^-;\q,u)
&=
\sum_{\kappa\in\{0,1,2\}^{\{1,\dots,N\}}}
\left(\prod_{j=1}^N c_{\kappa_j}\right)\,
s_{\lambda^+}\!\Bigl[(\q-1)\,Y_{1,N}(\kappa)\Bigr]\,
s_{\lambda^-}\!\Bigl[(\q-1)\,Y_{2,N}(\kappa)\Bigr],
\end{align}
where
\[
c_0=2,\qquad c_1=-(u+1),\qquad c_2=u,
\]
and the (plethystic) alphabets $Y_{1,N}(\kappa),Y_{2,N}(\kappa)$ are given by
\[
\begin{aligned}
Z^-_{\ge1,N}(\kappa)&:=\sum_{\substack{1\le j\le N\\ \kappa_j\ge1}} z^-_j,
&
Z^-_{=2,N}(\kappa)&:=\sum_{\substack{1\le j\le N\\ \kappa_j=2}} z^-_j,\\
Y_{1,N}(\kappa)&:=Z^+_{\le N}+ Z^-_{\ge1,N}(\kappa),
&
Y_{2,N}(\kappa)&:=Z^+_{\le N}+ Z^-_{=2,N}(\kappa).
\end{aligned}
\]
Consequently, $\mathscr G_{\bm\lambda}(Z^+,Z^-;\q,u)$ is the coefficientwise stable limit of
\eqref{eq:gene-typeB} as $N\to\infty$, and for every bipartition $\bm\mu=(\mu^+,\mu^-)\vdash n$,
\begin{equation}\label{eq:coeff-typeB}
(\q-1)^{\ell(\bm\mu)}\,\varphi^{\bm\lambda}_{\bm\mu}
=\bigl[(Z^+)^{\mu^+}(Z^-)^{\mu^-}\bigr]\ \mathscr G_{\bm\lambda}(Z^+,Z^-;\q,u).
\end{equation}
\end{cor}

\begin{proof}
We specialize the coefficients $b_k^{(i)}$ in \eqref{e:def-b} (with $m=2$).
For $i=1$ (so $i-1=0$) we get
\[
b^{(1)}_0=1-u_1^0=0,\qquad b^{(1)}_1=u_1^0-u_2^0=0,\qquad b^{(1)}_2=u_2^0=1,
\]
so, in the finite formula \eqref{eq:gene-AK-finite}, the $\kappa_{1,j}$-choices collapse to
$\kappa_{1,j}=2$ for $1\le j\le N$.
For $i=2$ (so $i-1=1$) we have
\[
b^{(2)}_0=1-u_1=2,\qquad b^{(2)}_1=u_1-u_2=-(u+1),\qquad b^{(2)}_2=u_2=u,
\]
which we denote by $c_0,c_1,c_2$.
Writing $\kappa_j:=\kappa_{2,j}\in\{0,1,2\}$ and observing that $Y_{r,N}(\kappa)$ contains all
$z^+_j$ with $1\le j\le N$ (since $\kappa_{1,j}=2\ge r$ for $r=1,2$), we obtain exactly
\eqref{eq:gene-typeB}. The stable-limit statement follows by letting $N\to\infty$ coefficientwise, and
\eqref{eq:coeff-typeB} is immediate from the definition \eqref{eq:G-def-Zalpha}.
\end{proof}

\medskip

Next we specialize the bivariate generating function \eqref{eq:AK-bivar-def}.
To avoid notational collision with the Hecke parameter $u$, we rename the auxiliary alphabets
\[
V^+:=v^+_1+v^+_2+\cdots,\qquad V^-:=v^-_1+v^-_2+\cdots,
\qquad (V^\pm)^\la:=(v^\pm_1)^{\la_1}(v^\pm_2)^{\la_2}\cdots.
\]
For a bipartition $\bm\lambda=(\lambda^+,\lambda^-)$ set $V^{\bm\lambda}:=(V^+)^{\lambda^+}(V^-)^{\lambda^-}$, and define
\[
\Delta(V^\pm):=\prod_{a<b}\left(1-\frac{v^\pm_b}{v^\pm_a}\right),\qquad
\Phi_\pm(z):=\prod_{a\ge1}\frac{1-z\,v^\pm_a}{1-\q z\,v^\pm_a}
=\sigma_z\bigl[(\q-1)V^\pm\bigr].
\]
We also write $\Phi_{\pm,\le2}(z):=\Phi_+(z)\Phi_-(z)$ for brevity.

\begin{cor}\label{c:bivar-typeB}
Assume $m=2$ and $(u_1,u_2)=(-1,u)$.  Then the bivariate generating function
\[
\mathscr F(Z^+,Z^-,V^+,V^-;\q,u)
=\sum_{\bm\lambda\in\mathscr P_{*,2}}\mathscr G_{\bm\lambda}(Z^+,Z^-;\q,u)\,V^{\bm\lambda}
\]
admits the closed product form
\begin{multline}\label{eq:bivar-typeB}
\mathscr F(Z^+,Z^-,V^+,V^-;\q,u)\\
=
\Delta(V^+)\Delta(V^-)\,
\prod_{j\ge1}\Phi_+(z^+_j)\Phi_-(z^+_j)\,
\prod_{j\ge1}\Bigl(2-(u+1)\Phi_+(z^-_j)+u\,\Phi_+(z^-_j)\Phi_-(z^-_j)\Bigr).
\end{multline}
The products over $j\ge1$ are understood coefficientwise, as stable limits of
the finite products over $1\le j\le N$.
Equivalently,
\begin{multline}\label{eq:Cauchy-typeB}
\sum_{\lambda^+,\lambda^-\in\mathscr P}
\mathscr G_{(\lambda^+,\lambda^-)}(Z^+,Z^-;\q,u)\ (V^+)^{\lambda^+}(V^-)^{\lambda^-}\\
=
\left(\prod_{a<b}\left(1-\frac{v^+_b}{v^+_a}\right)\right)
\left(\prod_{a<b}\left(1-\frac{v^-_b}{v^-_a}\right)\right)
\prod_{j\ge1}\left(\prod_{a\ge1}\frac{1-z^+_j v^+_a}{1-\q z^+_j v^+_a}\right)
\left(\prod_{a\ge1}\frac{1-z^+_j v^-_a}{1-\q z^+_j v^-_a}\right)\\
\cdot
\prod_{j\ge1}\left(
2-(u+1)\prod_{a\ge1}\frac{1-z^-_j v^+_a}{1-\q z^-_j v^+_a}
+u\left(\prod_{a\ge1}\frac{1-z^-_j v^+_a}{1-\q z^-_j v^+_a}\right)
\left(\prod_{a\ge1}\frac{1-z^-_j v^-_a}{1-\q z^-_j v^-_a}\right)
\right).
\end{multline}
\end{cor}

\begin{proof}
This is a direct specialization of Theorem~\ref{t:AK-bivariate}.
Indeed, for $i=1$ we have $b_2^{(1)}=1$ and $b_0^{(1)}=b_1^{(1)}=0$, hence
\[
\sum_{k=0}^2 b_k^{(1)}\Phi_{\le k}(z)=\Phi_{\le2}(z)=\Phi_+(z)\Phi_-(z).
\]
For $i=2$ we have $(b_0^{(2)},b_1^{(2)},b_2^{(2)})=(2,-(u+1),u)$, hence
\[
\sum_{k=0}^2 b_k^{(2)}\Phi_{\le k}(z)=2-(u+1)\Phi_+(z)+u\,\Phi_+(z)\Phi_-(z).
\]
Substituting these into \eqref{eq:AK-bivar-formula} (with $m=2$) yields \eqref{eq:bivar-typeB},
and \eqref{eq:Cauchy-typeB} is just the same identity written out as a product.
\end{proof}

\section{Hecke--Clifford algebras}\label{sec:hecke-clifford-algebras}
As a deformation of the Sergeev algebra $\mathfrak{H}^c_n$, the Hecke--Clifford algebra $\mathscr{H}^c_n(\q)$ was defined by Olshanski \cite{Ol92} as the semidirect product of the Hecke algebra $\mathcal H_n$ with the Clifford algebra ${\rm Cl}_n$. It is the associative superalgebra over the field $\mathbb{C}(\q^{1/2})$ with even generators $T_1,\dots,T_{n-1}$ and odd generators $c_1,\dots,c_n$ subject to the following relations\footnote{Here $\q$ is also viewed as a parameter.}:
\begin{align*}
(T_i-\q)(T_i+1)=0, \quad &1\le i\le n-1,\\
T_iT_j=T_jT_i, \quad &1\le i,j\le n-1,\ |i-j|>1,\\
T_iT_{i+1}T_i=T_{i+1}T_iT_{i+1}, \quad &1\le i\le n-2,\\
c_i^2=1,\ \ c_ic_j=-c_jc_i, \quad &1\le i\ne j\le n,\\
T_ic_j=c_jT_i, \quad &j\ne i,i+1,\ 1\le i\le n-1,\ 1\le j\le n,\\
T_ic_i=c_{i+1}T_i, \quad &1\le i\le n-1,
\end{align*}

Let $\sigma=s_{i_1}s_{i_2}\cdots s_{i_r}\in \mathfrak{S}_n$ be a reduced expression. Define
\[
T_{\sigma}:=T_{i_1}T_{i_2}\cdots T_{i_r}.
\]
Then $\{T_{\sigma}\mid \sigma\in\mathfrak{S}_n\}$ forms a linear basis of $\mathscr{H}_{n}(\q)$. (Here $T_{\sigma}$ is independent of the choice of reduced expression of $\sigma$.)

For a composition $\mu=(\mu_1,\mu_2,\dots,\mu_l)$ of $n$, let
\[
\sigma_{\mu}:=(1\,2\,\cdots\,\mu_1)(\mu_1+1\,\cdots\,\mu_1+\mu_2)\cdots
(n-\mu_l+1\,\cdots\,n)\in\mathfrak{S}_n,
\]
and set $T_{\sigma_{\mu}}:=T_{\sigma_{\mu_1}}T_{\sigma_{\mu_2}}\cdots T_{\sigma_{\mu_l}}$, where each $\sigma_{\mu_i}$ denotes the corresponding cycle on its block of letters.

It is well known that the irreducible characters $\zeta^{\la}$ of $\mathscr{H}^c_n(\q)$ are parametrized by the strict partitions of $n$. If $\psi^{\nu}:\mathscr {H}^c_{n}(\q)\to {\rm End}(W)$ is an irreducible representation of $\mathscr{H}^c_n(\q)$ labeled by the strict partition $\nu$, then
\phantomsection\label{def:zeta-HC}
\begin{align}
    \zeta^{\nu}_{\mu}(\q):={\rm Tr}_{W}(T_{\sigma_{\mu}}).
\end{align}
Denote by $\O_n$ the set of partitions all of whose parts are odd (such partitions are called \emph{odd partitions}) and $\O=\bigsqcup_{n\geq 0}\O_n$. It is known that $\zeta^{\nu}$ is determined by the values $\zeta^{\nu}_{\mu}(\q)$ for $\mu\in\O_n$; see \cite{WW13}.

The Frobenius-type formula for $\mathscr H^c_n(\q)$ can be found in  \cite{WW13} and reads
\begin{align}
    \tilde{q}_{\mu}[(\q-1)X]&=\sum_{\nu\in\mathscr{S}_n}2^{-\frac{\ell(\nu)+\delta(\nu)}{2}}\zeta^{\nu}_{\mu}(\q)\,Q_{\nu}(X),
\end{align}
where $\tilde{q}_{\mu}\bigl[(\q-1)X\bigr]=\tilde{q}_{\mu_1}\bigl[(\q-1)X\bigr]\tilde{q}_{\mu_2}\bigl[(\q-1)X\bigr]\cdots$, and
\[
\delta(\nu)=
\begin{cases}
0, & \text{if $\ell(\nu)$ is even},\\
1, & \text{if $\ell(\nu)$ is odd}.
\end{cases}
\]
Here recall $\tilde{q}_n(\cdot)=\frac{1}{\q-1}q_n(\cdot)$ for $n\ge1$, and $q_{n}(\cdot)$ is the one-row Schur $Q$-function.

For $\nu\in\mathscr S$, we introduce the generating function for $\zeta^{\nu}$ by
\[
\mathscr G^c_{\nu}(Z;\q):=\sum_{\alpha\in\mathscr A_1}(\q-1)^{\ell(\alpha)}\zeta^{\nu}_{\alpha}(\q)Z^{\alpha}.
\]
Here $Z^{\alpha}:=z_1^{\alpha_1}z_2^{\alpha_2}\cdots $ for the alphabet $Z=z_1+z_2+\cdots$. Recall that $\mathscr A_1$ denotes the set of all non-negative compositions with finite supports (cf. \eqref{e:def-A}).

For $\alpha=(\alpha_1,\alpha_2,\ldots)\in\mathscr A_1$, set
\[
q_\alpha:=\prod_{\alpha_j>0} q_{\alpha_j},\qquad 
\tilde q_\alpha:=\prod_{\alpha_j>0} \tilde q_{\alpha_j}.
\]
We keep $q_0=1$, and the symbol $\tilde q_0$ is not used in these finite-support products.
Then
\begin{equation}\label{eq:qtilde-factor-HC}
(\q-1)^{\ell(\alpha)}\,\tilde q_\alpha = q_\alpha .
\end{equation}

Define the (formal) kernel
\begin{equation}\label{eq:kernel-HC-def}
\mathcal K^c(X,Z;\q)
:=\sum_{\alpha\in\mathscr A_1}(\q-1)^{\ell(\alpha)}\,\tilde q_\alpha[(\q-1)X]\;Z^\alpha.
\end{equation}
Using \eqref{eq:qtilde-factor-HC}, we may also write
\begin{equation}\label{eq:kernel-HC-def2}
\mathcal K^c(X,Z;\q)
=\sum_{\alpha\in\mathscr A_1} q_\alpha[(\q-1)X]\;Z^\alpha
=\prod_{j\ge1}\left(\sum_{a\ge0} q_a[(\q-1)X]\;z_j^{a}\right).
\end{equation}

For $\alpha\in\mathscr A_1$ we define $\zeta^\nu_\alpha(\q)$ by the $Q$-expansion
\begin{equation}\label{eq:Frob-alpha-HC}
\tilde q_\alpha[(\q-1)X]
=\sum_{\nu\in\mathscr S}2^{-\frac{\ell(\nu)+\delta(\nu)}{2}}\,
\zeta^\nu_\alpha(\q)\,Q_\nu(X),
\end{equation}
so that for $\alpha=\mu\in\O$ this specializes to the Frobenius-type formula in \cite{WW13}.
Substituting \eqref{eq:Frob-alpha-HC} into \eqref{eq:kernel-HC-def} and exchanging sums gives
\begin{equation}\label{eq:kernel-HC-Qexp}
\mathcal K^c(X,Z;\q)
=\sum_{\nu\in\mathscr S}2^{-\frac{\ell(\nu)+\delta(\nu)}{2}}\,
Q_\nu(X)\,\mathscr G^c_\nu(Z;\q).
\end{equation}

\begin{thm}\label{t:gene-HC}
One has
\begin{align}\label{e:gene-HC}
\mathscr G^c_{\nu}(Z;\q)=2^{\frac{-\ell(\nu)+\delta(\nu)}{2}}\,Q_{\nu}\bigl[(\q-1)Z\bigr].
\end{align}
\end{thm}

\begin{proof}

For an arbitrary alphabet $A$ introduce the generating series
\[
\Omega_A(z):=\sum_{n\ge0} q_n[(\q-1)A]\;z^n
=1+(\q-1)\sum_{n\ge1}\tilde q_n[(\q-1)A]\;z^n.
\]
Using \eqref{e:skew-M-N-Mac} for $n\ge1$ in the specialization $(q,t,\eta)=(0,-1,\varnothing)$, and adding the constant term, we obtain for every strict partition
$\lambda$:
\begin{multline}\label{eq:Omega-action}
\Omega_X(z)\,P_\lambda[X]
=P_\lambda[X]+\sum_{n\ge1}(\q-1)\tilde q_n[(\q-1)X]\,z^n\,P_\lambda[X] \\
=\sum_{n\ge0}\ \sum_{\rho\in\mathscr S_{n+|\lambda|}} Q_{\rho/\lambda}[\q-1]\;z^n\,P_\rho[X]\\
=\sum_{\rho\in\mathscr S} Q_{\rho/\lambda}\!\bigl[(\q-1)z\bigr]\;P_\rho[X].
\end{multline}
Now truncate $Z$ by $Z_{\le N}:=z_1+\cdots+z_N$ and set
\[
\mathcal K^c_N(X,Z;\q):=\prod_{j=1}^N\Omega_X(z_j).
\]
Applying \eqref{eq:Omega-action} repeatedly to $P_\varnothing[X]=1$,
one proves by induction on $N$ that
\begin{equation}\label{eq:kernel-trunc}
\mathcal K^c_N(X,Z;\q)\cdot 1
=\sum_{\nu\in\mathscr S} Q_\nu\!\bigl[(\q-1)Z_{\le N}\bigr]\;P_\nu[X].
\end{equation}
Passing to the stable limit $N\to\infty$ gives
\begin{equation}\label{eq:kernel-Pexp}
\mathcal K^c(X,Z;\q)
=\sum_{\nu\in\mathscr S} Q_\nu\!\bigl[(\q-1)Z\bigr]\;P_\nu[X].
\end{equation}
Finally, using $Q_\nu=2^{\ell(\nu)}P_\nu$ for strict partitions $\nu$, we rewrite \eqref{eq:kernel-Pexp} as
\begin{equation}\label{eq:kernel-Qexp2}
\mathcal K^c(X,Z;\q)
=\sum_{\nu\in\mathscr S}2^{-\ell(\nu)}\,Q_\nu(X)\,Q_\nu\!\bigl[(\q-1)Z\bigr].
\end{equation}
Comparing the coefficients of $Q_\nu(X)$ in \eqref{eq:kernel-HC-Qexp} and \eqref{eq:kernel-Qexp2}, we get
\[
2^{-\frac{\ell(\nu)+\delta(\nu)}{2}}\;\mathscr G^c_\nu(Z;\q)
=2^{-\ell(\nu)}\,Q_\nu\!\bigl[(\q-1)Z\bigr].
\]
Therefore
\[
\mathscr G^c_\nu(Z;\q)=2^{\frac{-\ell(\nu)+\delta(\nu)}{2}}\,Q_\nu\!\bigl[(\q-1)Z\bigr],
\]
which is exactly \eqref{e:gene-HC}.
\end{proof}

Now we introduce the bi-variables generating function of $\zeta^{\nu}$ by
\begin{align}
    \mathscr F^c(Z,U;\q):=\sum_{\nu\in\mathscr S} 2^{\frac{\ell(\nu)-\delta(\nu)}{2}}\mathscr G^c_{\nu}(Z;\q)U^{\nu}
\end{align}
for an arbitrary alphabet $U=u_1+u_2+\cdots$. Here $U^{\nu}:=u_1^{\nu_1}u_2^{\nu_2}\cdots$.

\begin{thm}\label{t:bigene-HC}
    We have
    \begin{align}\label{e:bigene-HC}
       \mathscr F^c(Z,U;\q)=\prod_{i<j}\frac{u_i-u_j}{u_i+u_j}\prod_{i,j}\frac{1+\q z_iu_j}{1-\q z_iu_j}\frac{1- z_iu_j}{1+ z_iu_j}.
    \end{align}
\end{thm}
\begin{proof}
    Specializing \eqref{e:gene-HL} at $t=-1$ gives 
\begin{align}\label{e:gene-SQ}
      \sum_{\la\in \mathscr{S}}Q_{\la}(X)U^{\la}=&\prod_{i<j}\frac{u_i-u_j}{u_i+u_j}\prod_{i,j}\frac{1+x_iu_j}{1-x_iu_j}.
\end{align}
    Substituting $X\rightarrow (\q-1)Z$ in \eqref{e:gene-SQ} implies that
    \begin{align}\label{e:gene-SQ-h}
        \sum_{\la\in \mathscr S}Q_{\la}[(\q-1)Z]U^{\la}=&\prod_{i<j}\frac{u_i-u_j}{u_i+u_j}\prod_{i,j}\frac{1+\q z_iu_j}{1-\q z_iu_j}\frac{1-z_iu_j}{1+z_iu_j}.
    \end{align}
    Then \eqref{e:bigene-HC} is obtained by plugging \eqref{e:gene-HC} into \eqref{e:gene-SQ-h}.
\end{proof}

\section{\texorpdfstring{$\q$-rook monoid algebras}{q-rook monoid algebras}}\label{sec:q-rook-monoid-algebras}

As before, we let $\q$ be an indeterminate. For $n\geq2$, the $\q$-rook monoid $\R_n(\q)$ is defined as the associative $\mathbb{C}(\q)$-algebra with generators $1, T_1,\cdots, T_{n-1}$, and $ P_1, \cdots , P_n$ subject to the relations\footnote{Here we adopt Halverson's presentation of $R_{n}(\q)$ \cite{Hal04} for our purposes.} (cf. \cite{Hal04}):
\begin{align*}
(T_i-\q)(T_i+1)=0, \quad &1\leq i\leq n-1,\\
T_iT_j=T_jT_i, \quad &1\leq i,j\leq n-1,\ \lvert i-j\rvert>1,\\
T_iT_{i+1}T_i=T_{i+1}T_iT_{i+1}, \quad &1\leq i\leq n-2,\\
T_iP_j=P_jT_i=\q P_j, \quad &1\leq i< j\leq n,\\
T_iP_j=P_jT_i, \quad &1\leq j< i\leq n-1,\\
P^2_{i}=P_i, \quad &1\leq i\leq n,\\
P_{i+1}=\q P_iT^{-1}_{i}P_i, \quad &1\leq i \leq n-1.
\end{align*}

Define $\R_0(\q) = \mathbb{C}(\q)$, and $\R_1(\q)$ is the associative $\mathbb{C}(\q)$-algebra spanned by $1$ and
$P_1$ subject to $P^2_1 = P_1$. The subalgebra of $\R_n(\q)$ generated by $T_1,\cdots, T_{n-1}$ is isomorphic
to the Iwahori-Hecke algebra $\mathscr H_n(\q)$ in type $A$. 

Let $s_i\in\mathfrak S_n$ be the transposition of $i$ and $i+1$. We define $\gamma_1=T_{\gamma_1}=1$ and, for $2\leq t\leq n$,
\begin{align*}
    \gamma_t&=s_1s_2\cdots s_{t-1},\\
    T_{\gamma_t}&=T_1\cdots T_{t-1}.
\end{align*}
Set $P_0:=1$. For a composition $\mu=(\mu_1,\ldots,\mu_\ell)\models k$, where $0\leq k\leq n$, put
\[
b_i:=n-k+\sum_{a<i}\mu_a,\qquad
T_{\gamma_{\mu_i}}^{[b_i]}:=T_{b_i+1}T_{b_i+2}\cdots T_{b_i+\mu_i-1}
\quad(1\le i\le\ell),
\]
with the last product interpreted as $1$ when $\mu_i=1$, and define
\[
T_\mu:=P_{n-k}\,
T_{\gamma_{\mu_1}}^{[b_1]}\cdots
T_{\gamma_{\mu_\ell}}^{[b_\ell]}\in\R_n(\q).
\]
Equivalently, this is the image of
$P_{n-k}\otimes T_{\gamma_{\mu_1}}\otimes\cdots\otimes T_{\gamma_{\mu_\ell}}$
under the standard block embedding
$\R_{n-k}(\q)\otimes\R_{\mu_1}(\q)\otimes\cdots\otimes\R_{\mu_\ell}(\q)\subset\R_n(\q)$.
It is called the {\it standard element}; in particular, $T_\mu=1$ when $\mu=(1^n)$, while the empty composition gives $T_\varnothing=P_n$.

It is known that irreducible characters of $\R_n(\q)$ are parameterized by partitions of $j\leq n, j=0, 1, \ldots, n$. Here a partition of $0$ means $\varnothing$. The irreducible
$\R_n(\q)$-characters are completely determined by their values on the set of standard elements $T_{\mu}, \mu\vdash k, 0 \leq k \leq n.$ 

The irreducible module $V^{\la}$ indexed by $\la\vdash k\leq n$ have been studied
in \cite{G02, M57, Sol02}, and the seminormal bases of $V^{\la}$ were found in \cite{Hal04}. \phantomsection\label{def:Upsilon-rook}Define $\Upsilon^{\la}(T_{\mu})$ as the value of the irreducible characters $\Upsilon^{\la}$ of $\R_n(\q)$ indexed by $\la\vdash k (\leq n)$ at the standard element $T_{\mu}$. In the following we sometimes write $\Upsilon^{\la}_{\mu}(\q)$ for $\Upsilon^{\la}(T_{\mu})$ for brevity.

Dieng--Halverson--Poladian \cite{DHP03} used the Schur--Weyl
duality to prove the following Frobenius-type formula in the ring of symmetric functions ($\mu\vdash n$)
\begin{align}\label{chara formula}
\frac{\q^{|\mu|}}{(\q-1)^{\ell(\mu)}}\hat{q}_{\mu}(X;\q^{-1})=\sum_{j=0}^{n}\sum_{\la\vdash j}\Upsilon^{\la}(T_{\mu})s_{\la}(X)
\end{align}
where $\hat q_{\mu}(X;t)=q_{\mu}[1+X;t]$ in the plethystic form and $q_{\mu}(X;t)$ is the product of the one-row Hall--Littlewood functions in parameter $t$. A Murnaghan--Nakayama rule for the irreducible character of the $\q$-rook monoid algebra can be found in \cite{DHP03, JWL26}.

As in the previous subsections, we extend the lower index from partitions to arbitrary weak compositions.
Recall from \eqref{e:def-A} that
\[
\mathscr A_1=\{\,\alpha=(\alpha_1,\alpha_2,\ldots):\ \alpha_j\in\mathbb Z_{\ge0},\ 
\#\{j:\alpha_j\neq0\}<\infty\,\}.
\]
For \(\alpha\in\mathscr A_1\), set
\[
Z:=z_1+z_2+\cdots,\qquad
Z^\alpha:=\prod_{j\ge1} z_j^{\alpha_j},\qquad
|\alpha|:=\sum_{j\ge1}\alpha_j,\qquad
\ell(\alpha):=\#\{j:\alpha_j>0\}.
\]
Let \(\alpha^\downarrow\) be the partition obtained by rearranging the positive entries of \(\alpha\) in weakly decreasing order.

For \(a\ge0\) define
\[
\hat q_a(X;t):=q_a[1+X;t],
\]
and for \(\alpha\in\mathscr A_1\) set
\[
\hat q_\alpha(X;t):=\prod_{j\ge1}\hat q_{\alpha_j}(X;t).
\]
Since the product is invariant under permuting the indices \(j\), one has
\[
\hat q_\alpha(X;t)=\hat q_{\alpha^\downarrow}(X;t).
\]

For every \(\alpha\in\mathscr A_1\), define \(\Upsilon^\lambda_\alpha(\q)\in\mathbb C(\q)\) by the Schur expansion
\begin{equation}\label{eq:Frob-alpha-rook}
\q^{|\alpha|}\hat q_\alpha(X;\q^{-1})
=
\sum_{\lambda\in\mathscr P}
(\q-1)^{\ell(\alpha)}\,\Upsilon^\lambda_\alpha(\q)\,s_\lambda(X).
\end{equation}
When \(\alpha=\mu\in\mathscr P_n\), this is exactly \eqref{chara formula} for the algebra \(\R_n(\q)\); in particular,
\[
\Upsilon^\lambda_\mu(\q)=\Upsilon^\lambda(T_\mu).
\]
Moreover,
\[
\Upsilon^\lambda_\alpha(\q)=\Upsilon^\lambda_{\alpha^\downarrow}(\q).
\]
Since \(\hat q_\alpha(X;\q^{-1})\) has total degree at most \(|\alpha|\), the sum in \eqref{eq:Frob-alpha-rook} is finite and
\[
\Upsilon^\lambda_\alpha(\q)=0\qquad\text{unless}\qquad |\lambda|\le |\alpha|.
\]

For a fixed partition \(\lambda\), define the generating function
\begin{equation}\label{eq:G-rook-def}
\mathscr G^R_\lambda(Z;\q)
:=
\sum_{\alpha\in\mathscr A_1}
(\q-1)^{\ell(\alpha)}\,\Upsilon^\lambda_\alpha(\q)\,Z^\alpha.
\end{equation}
Thus, for every partition \(\mu\), the coefficient of \(Z^\mu\) in \(\mathscr G^R_\lambda(Z;\q)\) is the normalized character value
\((\q-1)^{\ell(\mu)}\Upsilon^\lambda_\mu(\q)\).

Define
\[
\bar q_a(X;\q):=\q^a\hat q_a(X;\q^{-1}),\qquad
\bar q_\alpha(X;\q):=\prod_{j\ge1}\bar q_{\alpha_j}(X;\q)
=\q^{|\alpha|}\hat q_\alpha(X;\q^{-1}),
\]
and introduce the kernel
\begin{equation}\label{eq:kernel-rook-def}
\mathcal K^R(X,Z;\q)
:=
\sum_{\alpha\in\mathscr A_1}\bar q_\alpha(X;\q)\,Z^\alpha
=
\prod_{j\ge1}\bar Q(z_j),
\qquad
\bar Q(z):=\sum_{a\ge0}\bar q_a(X;\q)\,z^a.
\end{equation}
Using \eqref{eq:Frob-alpha-rook}, we may rewrite \(\mathcal K^R\) as
\begin{equation}\label{eq:kernel-rook-Schur}
\mathcal K^R(X,Z;\q)
=
\sum_{\lambda\in\mathscr P}s_\lambda(X)\,\mathscr G^R_\lambda(Z;\q).
\end{equation}

\begin{lem}\label{l:Qbar-rook}
One has
\begin{equation}\label{eq:Qbar-rook}
\bar Q(z)
=
\sigma_z\bigl[(\q-1)(1+X)\bigr]
=
\frac{1-z}{1-\q z}\,\sigma_z\bigl[(\q-1)X\bigr]
=
\frac{1-z}{1-\q z}\prod_{i\ge1}\frac{1-zx_i}{1-\q zx_i}.
\end{equation}
\end{lem}

\begin{proof}
By definition and the generating function of the one-row Hall--Littlewood functions,
\[
\bar Q(z)
=
\sum_{a\ge0}\q^a q_a[1+X;\q^{-1}]\,z^a
=
\sum_{a\ge0}q_a[1+X;\q^{-1}](\q z)^a
=
\sigma_{\q z}\bigl[(1-\q^{-1})(1+X)\bigr].
\]
Since \((\q z)(1-\q^{-1})=z(\q-1)\), this equals \(\sigma_z[(\q-1)(1+X)]\). The second equality in
\eqref{eq:Qbar-rook} follows from plethystic additivity:
\[
(\q-1)(1+X)=(\q-1)+(\q-1)X,
\]
and the last equality is the explicit product expression for \(\sigma_z[(\q-1)X]\).
\end{proof}

\begin{thm}\label{t:gene-rook}
For every partition \(\lambda\), one has
\begin{equation}\label{eq:gene-rook}
\mathscr G^R_\lambda(Z;\q)
=
\sigma_Z[\q-1]\;s_\lambda\bigl[(\q-1)Z\bigr]
=
\left(\prod_{j\ge1}\frac{1-z_j}{1-\q z_j}\right)\,
s_\lambda\bigl[(\q-1)Z\bigr].
\end{equation}
Consequently, for every partition \(\mu\),
\begin{equation}\label{eq:coeff-rook}
(\q-1)^{\ell(\mu)}\,\Upsilon^\lambda_\mu(\q)
=
\bigl[Z^\mu\bigr]
\left(
\left(\prod_{j\ge1}\frac{1-z_j}{1-\q z_j}\right)\,
s_\lambda\bigl[(\q-1)Z\bigr]
\right).
\end{equation}
\end{thm}

\begin{proof}
By Lemma~\ref{l:Qbar-rook},
\[
\mathcal K^R(X,Z;\q)
=
\prod_{j\ge1}\sigma_{z_j}\bigl[(\q-1)(1+X)\bigr]
=
\sigma_Z[\q-1]\prod_{j\ge1}\sigma_{z_j}\bigl[(\q-1)X\bigr].
\]
Using plethystic multiplicativity, we get
\[
\prod_{j\ge1}\sigma_{z_j}\bigl[(\q-1)X\bigr]
=
\sigma_Z\bigl[(\q-1)X\bigr].
\]
Using the same Schur specialization of the parent skew Cauchy identity \eqref{e:ge-skew-Mac} underlying \eqref{e:skew-M-N-Mac}, now with the second alphabet equal to $(\q-1)Z$, we obtain
\[
\sigma_Z\bigl[(\q-1)X\bigr]=\sum_{n\geq 0}h_n[(\q-1)ZX]
=
\sum_{\lambda\in\mathscr P}s_\lambda(X)\,s_\lambda\bigl[(\q-1)Z\bigr].
\]
Therefore
\[
\mathcal K^R(X,Z;\q)
=
\sum_{\lambda\in\mathscr P}
s_\lambda(X)\,
\left(
\sigma_Z[\q-1]\;s_\lambda\bigl[(\q-1)Z\bigr]
\right).
\]
Comparing the coefficient of \(s_\lambda(X)\) with \eqref{eq:kernel-rook-Schur} gives \eqref{eq:gene-rook}. The coefficient extraction formula \eqref{eq:coeff-rook} follows immediately from \eqref{eq:G-rook-def}.
\end{proof}

We remark that the \(\q\)-rook generating function differs from the type \(A\) formula \eqref{eq:gene-typeA} by the extra scalar factor
\(\sigma_Z[\q-1]\), which arises exactly from the plethystic shift \(X\mapsto 1+X\).
\smallskip

Next we introduce the bivariate generating function. Let
\[
U:=u_1+u_2+\cdots,\qquad
U^\lambda:=u_1^{\lambda_1}u_2^{\lambda_2}\cdots
\qquad(\lambda\in\mathscr P),
\]
and define
\begin{equation}\label{eq:rook-bivar-def}
\mathscr F^R(Z,U;\q)
:=
\sum_{\lambda\in\mathscr P}\mathscr G^R_\lambda(Z;\q)\,U^\lambda
=
\sum_{\lambda\in\mathscr P}\ \sum_{\alpha\in\mathscr A_1}
(\q-1)^{\ell(\alpha)}\,\Upsilon^\lambda_\alpha(\q)\,Z^\alpha U^\lambda.
\end{equation}
Set
\[
\Delta(U):=\prod_{a<b}\left(1-\frac{u_b}{u_a}\right).
\]

\begin{cor}\label{c:rook-bivar}
One has
\begin{multline}\label{eq:rook-bivar}
\mathscr F^R(Z,U;\q)
=
\Delta(U)\,\sigma_Z\bigl[(\q-1)(1+U)\bigr]\\
=
\left(\prod_{a<b}\left(1-\frac{u_b}{u_a}\right)\right)
\prod_{j\ge1}
\left(
\frac{1-z_j}{1-\q z_j}\prod_{a\ge1}\frac{1-z_ju_a}{1-\q z_ju_a}
\right).
\end{multline}
\end{cor}

\begin{proof}
Substituting \eqref{eq:gene-rook} into \eqref{eq:rook-bivar-def}, we obtain
\[
\mathscr F^R(Z,U;\q)
=
\sigma_Z[\q-1]\sum_{\lambda\in\mathscr P}s_\lambda\bigl[(\q-1)Z\bigr]\,U^\lambda.
\]
Applying \eqref{eq:Schur-monomial-gf-pleth} with \(A=(\q-1)Z\), we get
\[
\sum_{\lambda\in\mathscr P}s_\lambda\bigl[(\q-1)Z\bigr]\,U^\lambda
=
\Delta(U)\prod_{a\ge1}\sigma_{u_a}\bigl[(\q-1)Z\bigr].
\]
Since
\[
\sigma_{u_a}\bigl[(\q-1)Z\bigr]
=
\prod_{j\ge1}\frac{1-z_ju_a}{1-\q z_ju_a},
\]
it follows that
\[
\mathscr F^R(Z,U;\q)
=
\left(\prod_{a<b}\left(1-\frac{u_b}{u_a}\right)\right)
\left(\prod_{j\ge1}\frac{1-z_j}{1-\q z_j}\right)
\left(\prod_{j\ge1}\prod_{a\ge1}\frac{1-z_ju_a}{1-\q z_ju_a}\right),
\]
which is exactly the explicit product form in \eqref{eq:rook-bivar}. Finally,
\[
\sigma_Z\bigl[(\q-1)(1+U)\bigr]
=
\prod_{j\ge1}
\left(
\frac{1-z_j}{1-\q z_j}\prod_{a\ge1}\frac{1-z_ju_a}{1-\q z_ju_a}
\right),
\]
so the plethystic form follows as well.
\end{proof}

\section{Some symmetries of irreducible characters}\label{s:symmetry-characters}

The generating functions obtained in the previous sections also encode a natural
\(\q\leftrightarrow \q^{-1}\) symmetry of irreducible character values.  For Ariki--Koike algebras
this symmetry is accompanied by componentwise transposition of the upper multipartition,
whereas for Hecke--Clifford algebras one obtains a genuine self-reciprocity.
We now make this precise.

For a multipartition \(\bm\lambda=(\lambda^{(1)},\ldots,\lambda^{(m)})\), write
\[
\bm\lambda^t:=\bigl(\lambda^{(1)t},\ldots,\lambda^{(m)t}\bigr).
\]

\begin{thm}\label{t:sym-AK}
Let \(\bm\lambda\in\mathscr P_{n,m}\). Then
\begin{equation}\label{eq:AK-gf-sym}
\mathscr G_{\bm\lambda}(Z;\q^{-1},\bm u)
=
(-\q)^{-n}\,\mathscr G_{\bm\lambda^t}(Z;\q,\bm u).
\end{equation}
Equivalently, for every \(\bm\alpha\in\mathscr A_m\) with \(|\bm\alpha|=n\),
\begin{equation}\label{eq:AK-char-sym}
\chi^{\bm\lambda}_{\bm\alpha}(\q,\bm u)
=
(-\q)^{\,n-\ell(\bm\alpha)}\,
\chi^{\bm\lambda^t}_{\bm\alpha}(\q^{-1},\bm u).
\end{equation}
In particular, for every \(\bm\mu\in\mathscr P_{n,m}\),
\begin{equation}\label{eq:AK-char-sym-mu}
\chi^{\bm\lambda}_{\bm\mu}(\q,\bm u)
=
(-\q)^{\,n-\ell(\bm\mu)}\,
\chi^{\bm\lambda^t}_{\bm\mu}(\q^{-1},\bm u).
\end{equation}
\end{thm}

\begin{proof}
Fix $N\ge1$. By the finite formula \eqref{eq:gene-AK-finite},
\[
\mathscr G_{\bm\lambda}^{[N]}(Z;\q^{-1},\bm u)
=
\sum_{\kappa\in\{0,1,\dots,m\}^{\{1,\dots,m\}\times\{1,\dots,N\}}}
\left(\prod_{i=1}^m\prod_{j=1}^N b^{(i)}_{\kappa_{i,j}}\right)
\prod_{r=1}^m
s_{\lambda^{(r)}}\!\Bigl[(\q^{-1}-1)Y_{r,N}(\kappa)\Bigr].
\]
The coefficients \(b_k^{(i)}\) are independent of \(\q\), so only the Schur factors need to be transformed.
For an ordinary partition \(\lambda\), a monomial \(a\), and an alphabet \(A\), we use the standard plethystic identities
\[
s_\lambda[aA]=a^{|\lambda|}s_\lambda[A],
\qquad
s_\lambda[-A]=(-1)^{|\lambda|}s_{\lambda^t}[A].
\]
Applying these with \(a=\q^{-1}\) and \(A=(1-\q)Y_{r,N}(\kappa)\), we get
\begin{align*}
s_{\lambda^{(r)}}\!\Bigl[(\q^{-1}-1)Y_{r,N}(\kappa)\Bigr]
&=
s_{\lambda^{(r)}}\!\Bigl[\q^{-1}(1-\q)Y_{r,N}(\kappa)\Bigr]
=
\q^{-|\lambda^{(r)}|} s_{\lambda^{(r)}}\!\Bigl[(1-\q)Y_{r,N}(\kappa)\Bigr]\\
&=
\q^{-|\lambda^{(r)}|} s_{\lambda^{(r)}}\!\Bigl[-(\q-1)Y_{r,N}(\kappa)\Bigr]=
(-\q)^{-|\lambda^{(r)}|}\,
s_{\lambda^{(r)t}}\!\Bigl[(\q-1)Y_{r,N}(\kappa)\Bigr].
\end{align*}
Multiplying over \(r=1,\ldots,m\) and using \(\sum_{r=1}^m |\lambda^{(r)}|=|\bm\lambda|=n\), we obtain
\[
\prod_{r=1}^m
s_{\lambda^{(r)}}\!\Bigl[(\q^{-1}-1)Y_{r,N}(\kappa)\Bigr]
=
(-\q)^{-n}
\prod_{r=1}^m
s_{\lambda^{(r)t}}\!\Bigl[(\q-1)Y_{r,N}(\kappa)\Bigr].
\]
Substituting this back into the finite \(\kappa\)-sum and applying \eqref{eq:gene-AK-finite} once again gives
\[
\mathscr G_{\bm\lambda}^{[N]}(Z;\q^{-1},\bm u)
=
(-\q)^{-n}\,
\mathscr G_{\bm\lambda^t}^{[N]}(Z;\q,\bm u).
\]
Letting $N\to\infty$ coefficientwise gives \eqref{eq:AK-gf-sym}.

Now let \(\bm\alpha\in\mathscr A_m\) with \(|\bm\alpha|=n\).  Taking the coefficient of \(Z^{\bm\alpha}\) in
\eqref{eq:AK-gf-sym} and using the definition \eqref{eq:G-def-Zalpha}, we get
\[
(\q^{-1}-1)^{\ell(\bm\alpha)}\chi^{\bm\lambda}_{\bm\alpha}(\q^{-1},\bm u)
=
(-\q)^{-n}
(\q-1)^{\ell(\bm\alpha)}\chi^{\bm\lambda^t}_{\bm\alpha}(\q,\bm u).
\]
Since
\[
\q^{-1}-1=-\q^{-1}(\q-1),
\]
we have
\[
(-1)^{\ell(\bm\alpha)}\q^{-\ell(\bm\alpha)}(\q-1)^{\ell(\bm\alpha)}
\chi^{\bm\lambda}_{\bm\alpha}(\q^{-1},\bm u)
=
(-1)^n\q^{-n}(\q-1)^{\ell(\bm\alpha)}
\chi^{\bm\lambda^t}_{\bm\alpha}(\q,\bm u).
\]
After canceling \((\q-1)^{\ell(\bm\alpha)}\) and rearranging, this becomes
\[
\chi^{\bm\lambda}_{\bm\alpha}(\q,\bm u)
=
(-\q)^{\,n-\ell(\bm\alpha)}\,
\chi^{\bm\lambda^t}_{\bm\alpha}(\q^{-1},\bm u),
\]
which is \eqref{eq:AK-char-sym}. The specialization \eqref{eq:AK-char-sym-mu} is immediate.
\end{proof}

Theorem~\ref{t:sym-AK} shows the reciprocity of the irreducible characters in the full
Ariki--Koike setting.  The specializations to type \(A\) and type \(B\) are obtained simply by choosing
\(m=1\) and \(m=2\), respectively.

\begin{cor}\label{c:sym-typeA}
Let \(\lambda,\mu\in\P_n\). Then
\begin{equation}\label{eq:typeA-char-sym}
\phi^\lambda_{\mu}(\q)
=
(-\q)^{\,n-\ell(\mu)}\,\phi^{\lambda^t}_{\mu}(\q^{-1}).
\end{equation}
Moreover, \(\phi^\lambda_\mu(\q)\) is a polynomial in \(\q\) of degree at most \(n-\ell(\mu)\).
\end{cor}

\begin{proof}
The identity \eqref{eq:typeA-char-sym} is the specialization of Theorem~\ref{t:sym-AK} at
\(m=1\) and \(u_1=1\), where \(\chi^\lambda_\mu=\phi^\lambda_\mu\).

For the degree bound, it is known that \(\phi^\lambda_\mu(\q)\) is a polynomial in \(\q\). Applying
\eqref{eq:typeA-char-sym} to \(\lambda^t\) in place of \(\lambda\), we also have
\[
\phi^{\lambda^t}_{\mu}(\q)
=
(-\q)^{\,n-\ell(\mu)}\,\phi^{\lambda}_{\mu}(\q^{-1}).
\]
Hence \(\phi^{\lambda^t}_{\mu}(\q)\in\mathbb C[\q]\), and therefore \(\phi^{\lambda^t}_{\mu}(\q^{-1})\) is a Laurent polynomial
in \(\q\) involving only nonpositive powers. Multiplying by \((-\q)^{\,n-\ell(\mu)}\), we see from
\eqref{eq:typeA-char-sym} that \(\phi^\lambda_\mu(\q)\) is a polynomial of degree at most
\(n-\ell(\mu)\).
\end{proof}



\begin{cor}\label{c:sym-typeB}
Let \(\bm\lambda=(\lambda^+,\lambda^-)\in\mathscr P_{n,2}\) and
\(\bm\mu=(\mu^+,\mu^-)\in\mathscr P_{n,2}\). Then
\begin{equation}\label{eq:typeB-char-sym}
\varphi^{(\lambda^+,\lambda^-)}_{(\mu^+,\mu^-)}(\q,u)
=
(-\q)^{\,n-\ell(\mu^+)-\ell(\mu^-)}
\varphi^{(\lambda^{+t},\lambda^{-t})}_{(\mu^+,\mu^-)}(\q^{-1},u).
\end{equation}
\end{cor}

\begin{proof}
This is the specialization of Theorem~\ref{t:sym-AK} at
\[
m=2,\qquad (u_1,u_2)=(-1,u),
\]
for which \(\chi^{\bm\lambda}_{\bm\mu}=\varphi^{\bm\lambda}_{\bm\mu}\) and
\(\bm\lambda^t=(\lambda^{+t},\lambda^{-t})\).
\end{proof}

The type \(A\) and type \(B\) reciprocity formulas are simply the one-component and two-component
shadows of the general Ariki--Koike symmetry.  The transpose on the upper index comes from the Schur
identity \(s_\lambda[-A]=(-1)^{|\lambda|}s_{\lambda^t}[A]\).  For Hecke--Clifford algebras the same
substitution \(A\mapsto -A\) behaves differently, because Schur \(Q\)-functions lie in the subring
generated by the odd power sums; consequently, no transpose appears.

\begin{thm}\label{t:sym-HC}
Let \(\nu\in\mathscr S_n\). Then
\begin{equation}\label{eq:HC-gf-sym}
\mathscr G^c_\nu(Z;\q^{-1})
=
(-1)^n\,\q^{-n}\,\mathscr G^c_\nu(Z;\q).
\end{equation}
Consequently, for every \(\rho\in\O_n\),
\begin{equation}\label{eq:HC-char-sym}
\zeta^\nu_{\rho}(\q)
=
\q^{\,n-\ell(\rho)}\,\zeta^\nu_{\rho}(\q^{-1}).
\end{equation}
Moreover, \(\zeta^\nu_\rho(\q)\) is a polynomial in \(\q\) of degree at most \(n-\ell(\rho)\). More precisely,
if
\[
\zeta^\nu_\rho(\q)=\sum_{i=0}^{n-\ell(\rho)} a_i \q^i
\]
(with the convention that some of the top coefficients \(a_i\) may be zero), then
\[
a_i=a_{\,n-\ell(\rho)-i}\qquad (0\le i\le n-\ell(\rho)).
\]
\end{thm}

\begin{proof}
By Theorem~\ref{t:gene-HC},
\[
\mathscr G^c_\nu(Z;\q^{-1})
=
2^{\frac{-\ell(\nu)+\delta(\nu)}{2}}\,
Q_\nu\!\bigl[(\q^{-1}-1)Z\bigr].
\]
Since \(|\nu|=n\), homogeneity gives
\[
Q_\nu\!\bigl[(\q^{-1}-1)Z\bigr]
=
\q^{-n}Q_\nu\!\bigl[(1-\q)Z\bigr]
=
\q^{-n}Q_\nu\!\bigl[-(\q-1)Z\bigr].
\]
Applying the fact
\[
Q_\nu[-A]=(-1)^nQ_\nu[A],
\]
we obtain
\[
Q_\nu\!\bigl[-(\q-1)Z\bigr]
=
(-1)^n Q_\nu\!\bigl[(\q-1)Z\bigr].
\]
Therefore
\[
\mathscr G^c_\nu(Z;\q^{-1})
=
(-1)^n\q^{-n}\,
2^{\frac{-\ell(\nu)+\delta(\nu)}{2}}\,
Q_\nu\!\bigl[(\q-1)Z\bigr]
=
(-1)^n\q^{-n}\,\mathscr G^c_\nu(Z;\q),
\]
which is \eqref{eq:HC-gf-sym}.

Now let \(\rho\in\O_n\). Taking the coefficient of \(Z^\rho\) in \eqref{eq:HC-gf-sym} and using the definition
of \(\mathscr G^c_\nu\), we get
\[
(\q^{-1}-1)^{\ell(\rho)}\zeta^\nu_\rho(\q^{-1})
=
(-1)^n\q^{-n}(\q-1)^{\ell(\rho)}\zeta^\nu_\rho(\q).
\]
Since
\[
\q^{-1}-1=-\q^{-1}(\q-1),
\]
this becomes
\[
(-1)^{\ell(\rho)}\q^{-\ell(\rho)}(\q-1)^{\ell(\rho)}\zeta^\nu_\rho(\q^{-1})
=
(-1)^n\q^{-n}(\q-1)^{\ell(\rho)}\zeta^\nu_\rho(\q).
\]
Because \(\rho\in\O_n\), all parts of \(\rho\) are odd, hence
\[
\ell(\rho)\equiv n \pmod 2.
\]
Thus \((-1)^{\ell(\rho)}=(-1)^n\), and after canceling \((\q-1)^{\ell(\rho)}\) we obtain
\[
\q^{-\ell(\rho)}\zeta^\nu_\rho(\q^{-1})
=
\q^{-n}\zeta^\nu_\rho(\q),
\]
which is exactly \eqref{eq:HC-char-sym}.

It remains to prove the degree bound and the palindromicity statement. It is known that
\(\zeta^\nu_\rho(\q)\) is a polynomial in \(\q\). Hence \(\zeta^\nu_\rho(\q^{-1})\) is a Laurent polynomial involving
only nonpositive powers of \(\q\), and \eqref{eq:HC-char-sym} shows that
\(\zeta^\nu_\rho(\q)\) has degree at most \(n-\ell(\rho)\).

Now write
\[
\zeta^\nu_\rho(\q)=\sum_{i=0}^{n-\ell(\rho)} a_i \q^i.
\]
Substituting this into \eqref{eq:HC-char-sym}, we get
\[
\sum_{i=0}^{n-\ell(\rho)} a_i \q^i
=
\q^{\,n-\ell(\rho)}
\sum_{i=0}^{n-\ell(\rho)} a_i \q^{-i}
=
\sum_{i=0}^{n-\ell(\rho)} a_i \q^{\,n-\ell(\rho)-i}.
\]
Comparing coefficients of \(\q^i\) on both sides yields
\[
a_i=a_{\,n-\ell(\rho)-i}
\qquad (0\le i\le n-\ell(\rho)),
\]
as claimed.
\end{proof}

\begin{rem}
The symmetry relations established above illustrate particularly clearly the strength of the generating-function approach.  Once the irreducible character values are assembled into a single kernel, identities such as the \(\q\leftrightarrow \q^{-1}\) symmetry are no longer proved by checking character values one by one; instead, they emerge from a simple transformation rule at the level of symmetric functions, together with standard plethystic identities.  In this sense, the generating function provides a conceptual framework that both unifies the character formulas and makes structural phenomena---such as reciprocity, degree bounds, and palindromicity---transparent.
\end{rem}

\begin{rem}
It is also worth emphasizing that the method developed here is not restricted to the particular families of algebras considered in this chapter.  The essential ingredients are rather general: one needs a Frobenius-type formula expressing character values in terms of symmetric functions, a vertex operator realization of the symmetric functions used to give the generating function of the symmetric function, and an explicit description of the corresponding symmetric-function kernel.  Whenever these ingredients are available, the same strategy should apply.  From this perspective, it is natural to expect that analogous generating functions and symmetry results can be developed for other diagrammatic or Hecke-type algebras as well; two particularly promising examples are provided by the BMW algebras and the mirabolic Hecke algebras, where much of the necessary representation-theoretic and symmetric-function input is already available in the literature \cite{HR95,JN15} and \cite{Wan26}.
\end{rem}

\chapter{Skew \texorpdfstring{$(q,t)$}{(q,t)}-Kostka matrices and their inverse}\label{s:qt-kostka}
This chapter is divided into two parts. In the first part we apply our skew Murnaghan--Nakayama framework to the study of skew $(q,t)$-Kostka polynomials.
We first recall the big Schur functions and a ribbon-removal decomposition of skew Schur functions, which in turn yields an analogous decomposition for the dual basis $\Phi_{\la/\mu}(X;q)$.
Combining this decomposition with an explicit formula for the adjoint operators $d_k^{\perp_{q,t}}$ acting on integral Macdonald polynomials $J_\nu(X;q,t)$, we obtain iterative formulas and a tableau/flag expansion for $K_{\la/\mu,\nu}(q,t)$.
Then, we specialize to $q=0$ to obtain explicit expressions for the skew Kostka--Foulkes polynomials.

In the second part, parallel to the first part, we present combinatorial formulas of the skew inverse $(q,t)$-Kostka coefficient and its specializations, which recovers the classical combinatorial formula for the inverse Kostka number due to E\u gecio\u glu--Remmel \cite{ER90}.

\section{A combinatorial decomposition of Schur functions}\label{ss:decomp-Schur}

Recall the generalized complete symmetric function $q_n[X;t]$ and the big Schur function $S_{\la/\mu}[X;t]$ by
\begin{align}\label{e:q-h}
    q_n[X;t]=h_n[(1-t)X],\qquad S_{\la/\mu}[X;t]=s_{\la/\mu}[(1-t)X].
\end{align}



Recall the inner product $\langle\cdot,\cdot\rangle_{q,t}$ from \eqref{e:innerprod}. Let $(\Phi_{\la}(X;q,t))$ be the basis of $\Lambda(q,t)$ dual to $(S_{\la}(X;t))$ with respect to $\langle\cdot,\cdot\rangle_{q,t}$, i.e.,
\begin{align}
    \langle \Phi_{\la}(X;q,t), S_{\mu}(X;t) \rangle_{q,t}=\delta_{\la\mu}.
\end{align}

\begin{lem}\label{l:Phi-s}
The symmetric functions $\Phi_{\la}(X;q,t)$ depend only on $q$. More precisely, in plethystic notation,
\begin{align}\label{e:Phi-s}
    \Phi_{\la}(X;q,t)=s_{\la}\!\left[\frac{X}{1-q}\right].
\end{align}
\end{lem}

\begin{proof}
Let $Y=y_1+y_2+\cdots$. The Schur Cauchy identity reads
\begin{align*}
    \sum_{\la}s_{\la}[X]\,s_{\la}[Y]
    =\prod_{n\ge 1}\exp\!\left(\frac{1}{n}p_n[X]p_n[Y]\right).
\end{align*}
Substituting $X\mapsto (1-t)X$ and $Y\mapsto \frac{Y}{1-q}$ gives
\begin{align*}
    \sum_{\la} S_{\la}[X;t]\; s_{\la}\!\left[\frac{Y}{1-q}\right]
    =\prod_{n\ge 1}\exp\!\left(\frac{1}{n}\frac{1-t^n}{1-q^n}p_n[X]p_n[Y]\right).
\end{align*}
Now use the standard characterization (cf.\ \cite[p.~310, (2.7)]{Mac}):
\[
\langle u_{\la},v_{\mu}\rangle_{q,t}=\delta_{\la\mu}\ \text{for all }\la,\mu
\quad \Longleftrightarrow \quad
\sum_{\la}u_{\la}[X]\,v_{\la}[Y]
=\prod_{n\ge 1}\exp\!\left(\frac{1}{n}\frac{1-t^n}{1-q^n}p_n[X]p_n[Y]\right),
\]
which yields \eqref{e:Phi-s}.
\end{proof}

By Lemma~\ref{l:Phi-s}, we henceforth write $\Phi_{\la}(X;q)$, omitting $t$. Define the skew version by
\begin{align}
    \Phi_{\la/\mu}(X;q):=s_{\la/\mu}\!\left[\frac{X}{1-q}\right].
\end{align}

\begin{rem}
By \eqref{e:Phi-s}, $\Phi_{\la/\mu}(X;q)$ admits a Jacobi--Trudi type formula
\begin{align}
    \Phi_{\la/\mu}(X;q)=\det_{1\le i,j\le m}\bigl(d_{\la_i-\mu_j-i+j}(X;q)\bigr),
\end{align}
where $d_n(X;q):=h_n\!\left[\frac{X}{1-q}\right]$, equivalently
\begin{align}\label{e:gen-d}
    \sum_{n\ge 0} d_n(X;q)\,z^n
    =\exp\!\left(\sum_{n\ge 1}\frac{1}{n(1-q^n)}\,p_n(X)\,z^n\right).
\end{align}
\end{rem}

For a skew diagram $\la/\mu$, we call a ribbon $\xi$ a \emph{special ribbon} of $\la/\mu$ if $\xi$ starts at the southwest-most box of $\la/\mu$; see Figure~\ref{fig:s-ribbon}.

\begin{figure}
    \centering
\begin{tikzpicture}[scale = 0.35]
  \begin{scope}
    \clip (0,0) -| (2,2) -| (4,3) -| (9,4) -| (10,5) -| (3,4) -| (1,2) -| (0,0);
    \draw [color=black!25] (0,0) grid (10,5);
  \end{scope}
   \draw [thick] (0,0) -| (2,2) -| (4,3) -| (9,4) -| (10,5)  -| (3,4) -| (1,2) -| (0,0);
  \draw [thick, rounded corners] (0.5,0.5) -- (1.5,0.5) |- (3.5,2.5);
  \draw [color=black,fill=black,thick] (3.5,2.5) circle (.5ex);
  \node [draw, circle, fill = white, inner sep = 1.5pt] at (0.5,0.5) { };
\end{tikzpicture}
\hspace{10em}
\begin{tikzpicture}[scale = 0.35]
  \begin{scope}
    \clip (0,0) -| (2,2) -| (4,3) -| (9,4) -| (10,5) -| (3,4) -| (1,2) -| (0,0);
    \draw [color=black!25] (0,0) grid (10,5);
  \end{scope}
   \draw [thick] (0,0) -| (2,2) -| (4,3) -| (9,4) -| (10,5)  -| (3,4) -| (1,2) -| (0,0);
  \draw [thick, rounded corners] (1.5,0.5) -- (1.5,2.5) -- (3.5,2.5) -- (3.5,3.5) -- (8.5,3.5);
  \draw [color=black,fill=black,thick] (8.5,3.5) circle (.5ex);
  \node [draw, circle, fill = white, inner sep = 1.5pt] at (1.5,0.5) { };
\end{tikzpicture}
\caption{The left ribbon is a special ribbon, while the right one is not.}\label{fig:s-ribbon}
\end{figure}
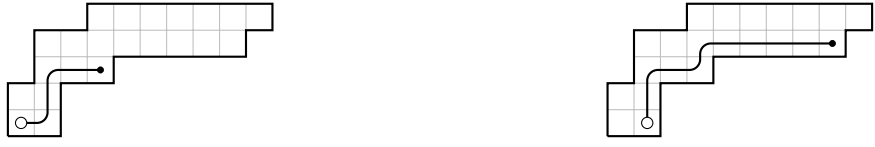

Let $\mu=(\mu_1,\dots,\mu_m)\subset\la$ with $\ell(\la)=m$ (padding $\mu$ with zeros if necessary). Expanding the Jacobi--Trudi determinant along the last column gives
\begin{align}\label{e:JT-expand}
    s_{\la/\mu}
    =\sum_{i=1}^m (-1)^{m-i}\,h_{\la_i-\mu_m-i+m}\,s_{\la^{(i)}/\mu^{(m)}},
\end{align}
where
\[
\la^{(i)}=(\la_1,\dots,\la_{i-1},\,\la_{i+1}-1,\dots,\la_m-1),
\qquad
\mu^{(m)}=(\mu_1,\dots,\mu_{m-1}).
\]
It follows from \cite[Section~5]{ER90} that $\la^{(i)}/\mu^{(m)}$ is obtained from $\la/\mu$ by removing a special ribbon of length $\la_i-\mu_m-i+m$ and height $m-i$. Therefore, $s_{\la/\mu}$ admits the following ribbon-removal decomposition:
\begin{align}\label{e:SqS}
    s_{\la/\mu}
    =\sum_{\xi} (-1)^{\rht((\la/\mu)/\xi)}\,h_{|\la|-|\mu|-|\xi|}\,s_{\xi},
\end{align}
summed over all skew diagrams $\xi$ obtained from $\la/\mu$ by removing a special ribbon (so that $(\la/\mu)/\xi$ is exactly the removed ribbon).

\phantomsection\label{def:SRT}
For a skew shape $\la/\mu$ and a composition $\nu$, define a \emph{special ribbon tableau} (SRT) $\T$ of shape $\la/\mu$ and type $\nu$ to be a filling of the boxes of $\la/\mu$ with positive integers such that
\begin{enumerate}
    \item entries are weakly increasing along rows (left to right) and down columns;
    \item for each $i\ge 1$, the skew diagram $\T^{(i)}/\T^{(i-1)}$ is a ribbon starting at the southwest-most box of $\T^{(i)}$;
    \item $\nu=(|\T^{(1)}|,\ |\T^{(2)}|-|\T^{(1)}|,\ |\T^{(3)}|-|\T^{(2)}|,\dots)$, equivalently $\nu_i=|\T^{(i)}/\T^{(i-1)}|$.
\end{enumerate}
Here $\T^{(i)}$ denotes the subdiagram consisting of boxes labeled by integers at most $i$ (and we set $\T^{(0)}=\varnothing$). We write ${\rm type}(\T)=\nu$. Define
\[
\sgn(\T)=\prod_{i\ge 1}(-1)^{\rht(\T^{(i)}/\T^{(i-1)})}.
\]

Iterating \eqref{e:SqS} yields the following expansion in the complete symmetric basis:
\begin{equation}\label{e:decom-skew-s}
    s_{\la/\mu}(X)
    =\sum_{\nu\models |\la|-|\mu|}\ \sum_{\T}\ \sgn(\T)\,h_{\nu}(X),
\end{equation}
where the inner sum runs over all SRTs $\T$ of shape $\la/\mu$ and type $\nu$.

\section{\texorpdfstring{A combinatorial interpretation for $K_{\la/\mu,\nu}(q,t)$}{A combinatorial interpretation for K(lambda/mu,nu)(q,t)}}\label{ss:skew-qt-kostka}

The $(q,t)$-Kostka polynomials $K_{\la,\nu}(q,t)$ are defined as the transition coefficients from the integral Macdonald polynomials $J_{\nu}(X;q,t)$ to the big Schur basis $S_{\la}(X;t)$:
\begin{align}
    J_{\nu}(X;q,t)=\sum_{\la}K_{\la,\nu}(q,t)\,S_{\la}(X;t),
\end{align}
equivalently,
\[
K_{\la,\nu}(q,t)=\left\langle \Phi_{\la}(X;q),\, J_{\nu}(X;q,t) \right\rangle_{q,t}.
\]
Define the \emph{skew $(q,t)$-Kostka polynomials} by (cf.\ \cite[p.~366, Example~10]{Mac})
\begin{align}\label{e:def-skew-K}
    K_{\la/\mu,\nu}(q,t):=\left\langle \Phi_{\la/\mu}(X;q),\, J_{\nu}(X;q,t) \right\rangle_{q,t}.
\end{align}
Clearly $K_{\la/\varnothing,\nu}(q,t)=K_{\la,\nu}(q,t)$. Moreover,
\begin{align}\label{e:LR-skew-K}
    K_{\la/\mu,\nu}(q,t)=\sum_{\rho}c^{\la}_{\mu,\rho}\,K_{\rho,\nu}(q,t),
\end{align}
where $c^{\la}_{\mu,\rho}$ is the Littlewood--Richardson coefficient.

Before giving a combinatorial formula for $K_{\la/\mu,\nu}(q,t)$, we need two ingredients.

\begin{lem}\label{l:decom-Phi}
Let $\mu\subset\la$. Then
\begin{align}\label{e:decom-Phi1}
    \Phi_{\la/\mu}(X;q)
    =\sum_{\xi}(-1)^{\rht((\la/\mu)/\xi)}\,d_{|\la|-|\mu|-|\xi|}(X;q)\,\Phi_{\xi}(X;q),
\end{align}
summed over all skew diagrams $\xi$ obtained from $\la/\mu$ by removing a special ribbon (and where, for a skew shape $\xi$, we write $\Phi_{\xi}(X;q):=s_{\xi}\!\left[\frac{X}{1-q}\right]$).
Consequently,
\begin{equation}\label{e:decom-Phi2}
    \Phi_{\la/\mu}(X;q)
    =\sum_{\nu\models |\la|-|\mu|}\ \sum_{\T}\ \sgn(\T)\,d_{\nu}(X;q),
\end{equation}
where the inner sum runs over all SRTs $\T$ of shape $\la/\mu$ and type $\nu$.
\end{lem}

\begin{proof}
This follows by substituting $X\mapsto \frac{X}{1-q}$ into \eqref{e:SqS} and \eqref{e:decom-skew-s}, using $d_k(X;q)=h_k\!\left[\frac{X}{1-q}\right]$.
\end{proof}

\begin{prop}\label{p:d*J}
Let $k\ge 0$. Then
\begin{align}\label{e:d*J}
    d_k^{\perp_{q,t}}\,J_{\nu}(X;q,t)
    =\sum_{\rho\subset_k \nu}\frac{c^{\prime}_{\nu}(q,t)}{c^{\prime}_{\rho}(q,t)}\,
    Q_{\nu/\rho}\!\left[\frac{1}{1-t};q,t\right]\,J_{\rho}(X;q,t),
\end{align}
where $\rho\subset_k \nu$ means $\rho\subset\nu$ and $|\nu/\rho|=k$.
\end{prop}

\begin{proof}
From \eqref{e:gen-d} and $p_n^{\perp_{q,t}}=\frac{n(1-q^n)}{1-t^n}\frac{\partial}{\partial p_n}$, we obtain
\begin{align}\label{e:gene-dual-d}
    \sum_{n\ge 0} d_n^{\perp_{q,t}}\,z^{-n}
    =\exp\!\left(\sum_{n\ge 1}\frac{1}{1-t^n}\frac{\partial}{\partial p_n}\,z^{-n}\right).
\end{align}
Taking $(f_{\la},g_{\la})=(P_{\la}(X;q,t),Q_{\la}(X;q,t))$ in \eqref{e:straight-dual-Lambda} yields
\begin{align}\label{e:Mac-shift}
    \exp\!\left(\sum_{n\ge 1}p_n(Y)\frac{\partial}{\partial p_n(X)}\right)Q_{\nu}(X;q,t)
    =\sum_{\rho\subset\nu}Q_{\nu/\rho}[Y;q,t]\;Q_{\rho}(X;q,t).
\end{align}
Specializing $Y=\frac{1}{(1-t)z}$ gives $p_n(Y)=\frac{z^{-n}}{1-t^n}$, hence
\begin{align}\label{e:partialQ}
   \exp\!\left(\sum_{n\ge 1}\frac{z^{-n}}{1-t^n}\frac{\partial}{\partial p_n(X)}\right)Q_{\nu}(X;q,t)
   =\sum_{\rho\subset\nu}z^{-|\nu/\rho|}\,
   Q_{\nu/\rho}\!\left[\frac{1}{1-t};q,t\right]\,Q_{\rho}(X;q,t).
\end{align}
Comparing the coefficient of $z^{-k}$ in \eqref{e:gene-dual-d} and \eqref{e:partialQ} yields
\[
d_k^{\perp_{q,t}}Q_{\nu}
=\sum_{\rho\subset_k \nu} Q_{\nu/\rho}\!\left[\frac{1}{1-t};q,t\right]\,Q_{\rho}.
\]
Finally, using $J_{\lambda}(X;q,t)=c'_{\lambda}(q,t)\,Q_{\lambda}(X;q,t)$ gives \eqref{e:d*J}.
\end{proof}

In \cite[Eq.~(6.15)]{JL24} we proved that
\[
\frac{c^{\prime}_{\nu}(q,t)}{c^{\prime}_{\rho}(q,t)}\,
Q_{\nu/\rho}\!\left[\frac{1}{1-t};q,t\right]
=t^{n(\nu)-n(\rho)}\binom{\nu}{\rho}_{q,t},
\]
where $n(\nu)=\sum_i (i-1)\nu_i$ and $\binom{\nu}{\rho}_{q,t}$ is the generalized $(q,t)$-binomial coefficient (introduced independently by Lassalle \cite{L98} and Okounkov \cite{O97}). Therefore \eqref{e:d*J} can be rewritten as
\begin{align}\label{e:d*J2}
    d_k^{\perp_{q,t}}\,J_{\nu}(X;q,t)
    =\sum_{\rho\subset_k \nu} t^{n(\nu)-n(\rho)}\binom{\nu}{\rho}_{q,t}\,J_{\rho}(X;q,t).
\end{align}

A \emph{flag of partitions} of shape $\nu$ and type $\tau$ is a chain
\[
\varnothing=\nu^0\subset\nu^1\subset\cdots\subset\nu^N=\nu
\]
such that $|\nu^i/\nu^{i-1}|=\tau_i$ for all $i\ge 1$.

We are now ready to give combinatorial formulas for the skew $(q,t)$-Kostka polynomials $K_{\la/\mu,\nu}(q,t)$.

\begin{thm}\label{t:skew-K_iterative}
Let $\mu\subset\la$ and $\nu$ be partitions with $|\la/\mu|=|\nu|$. Then
\begin{align}\label{e:skew-K-iterative}
 K_{\la/\mu,\nu}(q,t)
 =\sum_{\xi,\rho}(-1)^{\rht((\la/\mu)/\xi)}\,
 t^{n(\nu)-n(\rho)}\binom{\nu}{\rho}_{q,t}\,
 K_{\xi,\rho}(q,t),
\end{align}
summed over all skew diagrams $\xi$ obtained from $\la/\mu$ by removing a special ribbon and all $\rho\subset\nu$ with $|\rho|=|\xi|$. Consequently,
\begin{align}\label{e:skew-K-iterative2}
    K_{\la/\mu,\nu}(q,t)
    =t^{n(\nu)}\sum_{(\T,\{\nu\})}\sgn(\T)\prod_{i\ge 1}\binom{\nu^i}{\nu^{i-1}}_{q,t},
\end{align}
summed over all pairs consisting of an SRT $\T$ of shape $\la/\mu$ and a flag $\{\nu\}$ of shape $\nu$, with the same type.
\end{thm}

\begin{proof}
Using Lemma~\ref{l:decom-Phi} and \eqref{e:d*J2}, we compute
\begin{align*}
K_{\la/\mu,\nu}(q,t)
&=\left\langle \Phi_{\la/\mu},\,J_{\nu}\right\rangle_{q,t}\\
&=\sum_{\xi}(-1)^{\rht((\la/\mu)/\xi)}
\left\langle d_{|\la|-|\mu|-|\xi|}\,\Phi_{\xi},\,J_{\nu}\right\rangle_{q,t}\\
&=\sum_{\xi}(-1)^{\rht((\la/\mu)/\xi)}
\left\langle \Phi_{\xi},\,d_{|\la|-|\mu|-|\xi|}^{\perp_{q,t}}J_{\nu}\right\rangle_{q,t}\\
&=\sum_{\xi}(-1)^{\rht((\la/\mu)/\xi)}\!
\sum_{\rho\subset_{|\la|-|\mu|-|\xi|}\nu}\!
t^{n(\nu)-n(\rho)}\binom{\nu}{\rho}_{q,t}
\left\langle \Phi_{\xi},\,J_{\rho}\right\rangle_{q,t}\\
&=\sum_{\xi}(-1)^{\rht((\la/\mu)/\xi)}\!
\sum_{\rho\subset_{|\la|-|\mu|-|\xi|}\nu}\!
t^{n(\nu)-n(\rho)}\binom{\nu}{\rho}_{q,t}\,K_{\xi,\rho}(q,t),
\end{align*}
which is \eqref{e:skew-K-iterative}. Iterating \eqref{e:skew-K-iterative} yields \eqref{e:skew-K-iterative2}.
\end{proof}

For $q=0$, the specialization $Q_{\nu/\rho}\!\left[\frac{1}{1-t};0,t\right]$ admits a closed form (see \cite[Theorem~3.1]{Las05} or \cite[Eq.~(4.3)]{WZ12}):
\begin{align}\label{e:q=0}
Q_{\nu/\rho}\!\left[\frac{1}{1-t};0,t\right]
=t^{n(\nu/\rho)}\prod_{j\ge 1}
\qbinomial{\nu^t_j-\rho^t_{j+1}}{\rho^t_j-\rho^t_{j+1}}_{t},
\end{align}
where
\[
n(\nu/\rho)=\sum_i \binom{\nu^t_i-\rho^t_i}{2},
\qquad
\qbinomial{a}{b}_t
=\prod_{r=0}^{b-1}\frac{1-t^{a-r}}{1-t^{b-r}}.
\]
Using \eqref{e:q=0}, we obtain explicit formulas for the skew Kostka--Foulkes polynomials
\[
K_{\la/\mu,\nu}(t):=K_{\la/\mu,\nu}(0,t).
\]

\begin{cor}\label{c:skew-K(t)-iterative}
Under the assumptions of Theorem~\ref{t:skew-K_iterative}, we have:
\begin{enumerate}
\item The iterative formula
\begin{align}\label{e:skew-K(t)-iterative}
K_{\la/\mu,\nu}(t)
=\sum_{\xi,\rho}(-1)^{\rht((\la/\mu)/\xi)}\,
t^{n(\nu/\rho)}
\prod_{j\ge 1}\qbinomial{\nu^t_j-\rho^t_{j+1}}{\rho^t_j-\rho^t_{j+1}}_{t}\,
K_{\xi,\rho}(t),
\end{align}
summed over all skew diagrams $\xi$ obtained by removing a special ribbon from $\la/\mu$ and all partitions $\rho\subset\nu$ with $|\rho|=|\xi|$.

\item Moreover,
\begin{align}\label{e:skew-K(t)-iterative2}
K_{\la/\mu,\nu}(t)
=\sum_{(\T,\{\nu\})}\sgn(\T)\,
t^{\sum_i n(\nu^i/\nu^{i-1})}
\prod_{i\ge 1}\prod_{j\ge 1}
\qbinomial{(\nu^i)^t_j-(\nu^{i-1})^t_{j+1}}{(\nu^{i-1})^t_j-(\nu^{i-1})^t_{j+1}}_{t},
\end{align}
summed over all pairs consisting of an SRT $\T$ of shape $\la/\mu$ and a flag $\{\nu\}$ of shape $\nu$, with the same type.
\end{enumerate}
\end{cor}

\begin{rem}
Theorem~\ref{t:skew-K_iterative} and Corollary~\ref{c:skew-K(t)-iterative} extend \cite[Cor.~6.2, Cor.~6.3]{JL24} to the skew setting.
\end{rem}







\section{Skew inverse \texorpdfstring{$(q,t)$}{(q,t)}-Kostka coefficients}\label{ss:inv-qt-kostka}

The skew inverse $(q,t)$-Kostka coefficient (denoted by $\K_{\la/\mu,\nu}(q,t)$) is defined by
\begin{align}
    S_{\la/\mu}(X;t)=\sum_{\nu}\K_{\la/\mu,\nu}(q,t)J_{\nu}(X;q,t),
\end{align}
which is equivalent to 
\begin{align}
    \K_{\la/\mu,\nu}(q,t)
    &=\frac{1}{c_{\nu}(q,t)c^{'}_{\nu}(q,t)}
    \left\langle S_{\la/\mu}(X;t), J_{\nu}(X;q,t)\right\rangle_{q,t}\\
    &=\frac{1}{c_{\nu}(q,t)}
    \left\langle S_{\la/\mu}(X;t), Q_{\nu}(X;q,t)\right\rangle_{q,t}. 
\end{align}
We write $\K_{\la,\nu}(q,t)$ for $\K_{\la/\varnothing,\nu}(q,t)$ in short.

Recall that $S_{\la/\mu}(X;t)=s_{\la/\mu}[(1-t)X]$. Similar to $\Phi_{\la/\mu}$, $S_{\la/\mu}$ also has the following combinatorial decompositions
    \begin{align}\label{e:decom-S1}
        S_{\la/\mu}(X;t)=\sum_{\xi}(-1)^{\rht((\la/\mu)/\xi)}q_{|\la|-|\mu|-|\xi|}(X;t)S_{\xi}(X;t)
    \end{align}
    summed over all skew diagrams $\xi$ obtained from $\la/\mu$ by removing a special ribbon of $\la/\mu$, where $q_n(X;t)$ is the generalized complete symmetric function. Moreover,
   \begin{equation}\label{e:decom-S2}
    S_{\la/\mu}(X;t)=\sum_{\nu\models |\la|-|\mu|}\sum_{\T} \sgn(\T)q_{\nu}(X;t)
\end{equation}
the inner summation is over all SRTs of shape $\la/\mu$ and type $\nu$. 

\begin{prop}\label{p:q*Q}
    Let $k\geq 0$ and $\nu$ be a partition. Then
    \begin{align}
        q^{\perp_{q,t}}_kQ_{\nu}(X;q,t)=\sum_{\rho\subset_{k}\nu}Q_{\nu/\rho}[1-q;q,t]Q_{\rho}(X;q,t).
    \end{align}
\end{prop}
\begin{proof}
    It follows from the generating function of $q_n$ that
    \begin{align}\label{e:gen-q*}
        \exp\left(\sum_{n\geq 1} (1-q^n)\frac{\partial}{\partial p_n}z^{-n}\right)=\sum_{n\geq0}q^{\perp_{q,t}}_nz^{-n}.
    \end{align}
    Taking specialization $Y=\frac{1-q}{z}$ in \eqref{e:Mac-shift} gives
    \begin{align}\label{e:partialQ2}
       \exp\left(\sum_{n=1}^{+\infty}\frac{1-q^{n}}{z^n}\frac{\partial}{\partial p_n(X)}\right)Q_{\nu}(X;q,t)=\sum_{\rho\subset\nu}z^{-|\nu/\rho|}Q_{\nu/\rho}\left[1-q;q,t\right]Q_{\rho}(X;q,t). 
    \end{align}
    The proof is finished by \eqref{e:gen-q*} and taking the coefficient of $z^{-k}$ on both sides of \eqref{e:partialQ2}.
\end{proof}

\begin{thm}\label{t:skew-inv-K_iterative}
    Let $\mu\subset\la$ and $\nu$ be partitions with $|\la/\mu|=|\nu|$. Then
    \begin{align}\label{e:skew-inv-K-iterative}
     \K_{\la/\mu,\nu}(q,t)=\sum_{\xi,\rho}(-1)^{\rht((\la/\mu)/\xi)}\frac{c_{\rho}(q,t)}{c_{\nu}(q,t)}Q_{\nu/\rho}[1-q;q,t]\K_{\xi,\rho}(q,t)   
    \end{align}
    summed over all skew diagrams $\xi$ obtained by removing a special ribbon from $\la/\mu$ and partitions $\rho\subset\nu$ with $|\rho|=|\xi|$. Consequently, we have
    \begin{align}\label{e:skew-inv-K-iterative2}
        \K_{\la/\mu,\nu}(q,t)=\frac{1}{c_{\nu}(q,t)}\sum_{(\T,\{\nu\})}\sgn(\T)\prod_{i}Q_{\nu^i/\nu^{i-1}}[1-q;q,t]
    \end{align}
    summed over all pairs of SRTs $\T$ of shape $\la/\mu$ and flags of partition $\{\nu\}$ of shape $\nu$ with the same type. 
\end{thm}
\begin{proof}
    \eqref{e:skew-inv-K-iterative} holds by the following calculations:
    \begin{align*}
    &\K_{\la/\mu,\nu}(q,t)\\
    =&\frac{1}{c_{\nu}(q,t)c^{'}_{\nu}(q,t)}
    \left\langle S_{\la/\mu}(X;t), J_{\nu}(X;q,t) \right\rangle_{q,t}\\
    =&\frac{1}{c_{\nu}(q,t)}
    \left\langle \sum_{\xi}(-1)^{\rht((\la/\mu)/\xi)}
    q_{|\la|-|\mu|-|\xi|}(X;t)S_{\xi}(X;t), Q_{\nu}(X;q,t) \right\rangle_{q,t}\\
    &\hspace{7.5cm}\text{(by \eqref{e:decom-S1})}\\
    =&\frac{1}{c_{\nu}(q,t)}\sum_{\xi}(-1)^{\rht((\la/\mu)/\xi)}
    \left\langle S_{\xi}(X;t),
    q_{|\la|-|\mu|-|\xi|}^{\perp_{q,t}}(X;t)Q_{\nu}(X;q,t) \right\rangle_{q,t} \\
    =&\frac{1}{c_{\nu}(q,t)}\sum_{\xi}(-1)^{\rht((\la/\mu)/\xi)}
    \left\langle S_{\xi}(X;t),
    \sum_{\rho\subset_{|\la|-|\mu|-|\xi|}\nu}
    Q_{\nu/\rho}[1-q;q,t]Q_{\rho}(X;q,t) \right\rangle_{q,t}\\
    &\hspace{7.5cm}\text{(by Prop. \ref{p:q*Q})}\\
    =&\sum_{\xi,\rho}(-1)^{\rht((\la/\mu)/\xi)}
    \frac{c_{\rho}(q,t)}{c_{\nu}(q,t)}
    Q_{\nu/\rho}[1-q;q,t]\K_{\xi,\rho}(q,t)   
    \end{align*}
    summed over all skew diagrams $\xi$ obtained by removing a special ribbon from $\la/\mu$ and partitions $\rho\subset\nu$ with $|\rho|=|\xi|$.
    Equation \eqref{e:skew-inv-K-iterative2} is obtained by \eqref{e:decom-S2} and using Prop. \ref{p:q*Q} repeatedly.  
\end{proof}

\begin{exmp}\label{exmp:(2,1)(1)(2)}
  We consider the partitions $\la=(2,1)$, $\mu=(1)$, and $\nu=(2)$. The skew diagram $\la/\mu$ consists of two disconnected boxes: one at position $(1,2)$ (the second box in the first row) and one at position $(2,1)$ (the first box in the second row).
  
  According to the definition, a special ribbon $R=(\la/\mu)/\xi$ must start at the southwest-most box of $\la/\mu$, which is the box at $(2,1)$. Since $\la/\mu$ is disconnected, $R$ cannot extend to the box at $(1,2)$. Thus, the unique special ribbon $R$ consists of the single box at $(2,1)$. We have $|R|=1$, $\rht(R)=0$, and the remaining shape $\xi$ is the single box at $(1,2)$, which is isomorphic to the partition $(1)$.
  
  We apply Theorem \ref{t:skew-inv-K_iterative}. We sum over partitions $\rho \subset \nu=(2)$ with $|\rho|=|\xi|=1$. The only such partition is $\rho=(1)$. The recurrence relation \eqref{e:skew-inv-K-iterative} yields a single term:
  \begin{equation*}
    \K_{(2,1)/(1),(2)}(q,t) = (-1)^0 \frac{c_{(1)}(q,t)}{c_{(2)}(q,t)} Q_{(2)/(1)}[1-q;q,t] \K_{(1),(1)}(q,t).
  \end{equation*}
  We compute the components as follows:
  \begin{itemize}
    \item The Macdonald denominators are $c_{(1)}(q,t) = 1-t$ and $c_{(2)}(q,t) = (1-qt)(1-t)$.
    \item By the branching rule, $Q_{(2)/(1)} = Q_{(1)}$. Its plethystic evaluation is given by $Q_{(1)}[1-q;q,t] = 1-t$.
    \item The base case of the inverse $(q,t)$-Kostka coefficient is $\K_{(1),(1)}(q,t) = 1$.
  \end{itemize}
  Substituting these values, we obtain:
  \begin{align*}
    \K_{(2,1)/(1),(2)}(q,t) &= \frac{1-t}{(1-qt)(1-t)} \cdot (1-t) \cdot 1 = \frac{1-t}{1-qt}.
  \end{align*}
  This example demonstrates that $\K_{\la/\mu,\nu}(q,t)$ is, in general, a rational function of $q$ and $t$, rather than a polynomial.
\end{exmp}

The preceding example shows that the skew inverse $(q,t)$-Kostka coefficient
$\K_{\lambda/\mu,\nu}(q,t)$ is not, in general, a polynomial in $q$ and $t$.
Nevertheless, after a natural normalization one obtains a polynomial modification.

\begin{prop}\label{p:modified-skew-inv-qt-K}
For partitions $\mu\subset\lambda$ and $\nu$, define
\[
\widetilde{\K}_{\lambda/\mu,\nu}(q,t)
:=c_{\nu}(q,t)c'_{\nu}(q,t)\K_{\lambda/\mu,\nu}(q,t).
\]
Then
\[
\widetilde{\K}_{\lambda/\mu,\nu}(q,t)\in \mathbb Z[q,t].
\]
\end{prop}

\begin{proof}
Let $\langle\cdot,\cdot\rangle_{0,0}$ denote the ordinary Hall inner product on $\Lambda$,
characterized by
\[
\langle p_{\alpha},p_{\beta}\rangle_{0,0}=\delta_{\alpha,\beta}\,z_{\alpha}.
\]
We first claim that for any $f,g\in \Lambda(q,t)$,
\begin{equation}\label{e:bridge-inner-products}
\langle f[(1-t)X],g(X)\rangle_{q,t}
=
\langle f(X),g[(1-q)X]\rangle_{0,0}.
\end{equation}
Since both sides are bilinear and diagonal on the power-sum basis, it suffices to verify
\eqref{e:bridge-inner-products} for $f=p_{\alpha}$ and $g=p_{\beta}$. In this case,
\[
p_{\alpha}[(1-t)X]
=
\prod_{i=1}^{\ell(\alpha)}(1-t^{\alpha_i})\,p_{\alpha}(X),
\qquad
p_{\beta}[(1-q)X]
=
\prod_{i=1}^{\ell(\beta)}(1-q^{\beta_i})\,p_{\beta}(X).
\]
Hence
\begin{align*}
\langle p_{\alpha}[(1-t)X],p_{\beta}(X)\rangle_{q,t}
&=
\delta_{\alpha,\beta}\,
z_{\alpha}(q,t)\prod_{i=1}^{\ell(\alpha)}(1-t^{\alpha_i})\\
&=
\delta_{\alpha,\beta}\,
z_{\alpha}\prod_{i=1}^{\ell(\alpha)}\frac{1-q^{\alpha_i}}{1-t^{\alpha_i}}
\prod_{i=1}^{\ell(\alpha)}(1-t^{\alpha_i})\\
&=
\delta_{\alpha,\beta}\,
z_{\alpha}\prod_{i=1}^{\ell(\alpha)}(1-q^{\alpha_i})\\
&=
\langle p_{\alpha}(X),p_{\beta}[(1-q)X]\rangle_{0,0}.
\end{align*}
This proves \eqref{e:bridge-inner-products}.

Now, by the definition of $\K_{\lambda/\mu,\nu}(q,t)$ and the identity
$S_{\lambda/\mu}(X;t)=s_{\lambda/\mu}[(1-t)X]$, we obtain
\begin{align*}
\widetilde{\K}_{\lambda/\mu,\nu}(q,t)
&=
c_{\nu}(q,t)c'_{\nu}(q,t)\K_{\lambda/\mu,\nu}(q,t)\\
&=
\left\langle S_{\lambda/\mu}(X;t),J_{\nu}(X;q,t)\right\rangle_{q,t}\\
&=
\left\langle s_{\lambda/\mu}(X),J_{\nu}[(1-q)X;q,t]\right\rangle_{0,0},
\end{align*}
where the last equality follows from \eqref{e:bridge-inner-products}.

Since $J_{\nu}(X;q,t)$ is the integral Macdonald polynomial, we have
\[
J_{\nu}(X;q,t)\in \Lambda_{\mathbb Z}\otimes_{\mathbb Z}\mathbb Z[q,t]
\qquad\text{(see \cite[Ch.~VI, \S 8]{Mac})}.
\]
Moreover, the plethystic substitution $X\mapsto (1-q)X$ preserves
$\Lambda_{\mathbb Z}\otimes_{\mathbb Z}\mathbb Z[q,t]$.
Indeed, $\Lambda_{\mathbb Z}$ is generated by the complete symmetric functions, and by
\eqref{e:q-h} with $t=q$ we have
\[
h_r[(1-q)X]=q_r(X;q)\in \Lambda_{\mathbb Z}\otimes_{\mathbb Z}\mathbb Z[q]
\qquad (r\ge 0).
\]
Therefore
\[
J_{\nu}[(1-q)X;q,t]\in \Lambda_{\mathbb Z}\otimes_{\mathbb Z}\mathbb Z[q,t].
\]

Now expand
\[
s_{\lambda/\mu}(X)=\sum_{\rho}c^{\lambda}_{\mu,\rho}\,s_{\rho}(X),
\qquad
J_{\nu}[(1-q)X;q,t]=\sum_{\rho}a_{\nu\rho}(q,t)\,s_{\rho}(X),
\]
where $c^{\lambda}_{\mu,\rho}\in\mathbb Z_{\ge0}$ and
$a_{\nu\rho}(q,t)\in\mathbb Z[q,t]$.
Since the Schur basis is orthonormal for $\langle\cdot,\cdot\rangle_{0,0}$, it follows that
\[
\widetilde{\K}_{\lambda/\mu,\nu}(q,t)
=
\sum_{\rho}c^{\lambda}_{\mu,\rho}\,a_{\nu\rho}(q,t)\in \mathbb Z[q,t].
\]
This proves the proposition.
\end{proof}

\begin{rem}\label{r:modified-skew-inv-qt-K}
The normalization in Proposition~\ref{p:modified-skew-inv-qt-K} is universal, but it is
not minimal in general.
Indeed, Example~\ref{exmp:(2,1)(1)(2)} gives
\[
\K_{(2,1)/(1),(2)}(q,t)=\frac{1-t}{1-qt}.
\]
Hence
\[
c_{(2)}(q,t)\K_{(2,1)/(1),(2)}(q,t)
=
(1-t)^2\in\mathbb Z[q,t],
\]
whereas
\[
c'_{(2)}(q,t)\K_{(2,1)/(1),(2)}(q,t)
=
\frac{(1-q)(1-q^2)(1-t)}{1-qt}\notin \mathbb Z[q,t].
\]
Thus $c'_{\nu}(q,t)$ alone does not clear denominators in general.
Note that $c_{\nu}(q,t)$ alone is not sufficient either. For instance, we choose $\la=(2)$, $\mu=\varnothing$, and $\nu=(1,1)$. Then 
\[
\K_{(2),(1,1)}(q,t)=-\frac{q}{1-qt}.
\]
Consequently,
\[
c_{(1,1)}(q,t)\K_{(2),(1,1)}(q,t)
=
-\frac{q(1-t)(1-t^2)}{1-qt}\notin \mathbb Z[q,t].
\]
It would be interesting to find the minimal denominator $M(\la/\mu,\nu;q,t)$ such that $$M(\la/\mu,\nu;q,t)\K_{\la/\mu,\nu}(q,t)\in\mathbb Z[q,t].$$ Such $M(\la/\mu,\nu;q,t)$ is highly likely to rely on both $\la/\mu$ and $\nu$ simultaneously. 
\end{rem}



As we can see, Theorem \ref{t:skew-inv-K_iterative} involves the plethysms $Q_{\la/\mu}[1-q;q,t]$. Now we will give an explicit combinatorial formula for the coefficient $Q_{\la/\mu}[1-q;q,t]$. To proceed this, we need more notations.

The infinite (finite) $q$-shifted factorials in base $q$ are defined by
\begin{align*}
(x;q)_{\infty}=\prod_{k=0}^{\infty}(1-xq^k), \quad (x;q)_{n}=\prod_{k=0}^{n-1}(1-xq^k).
\end{align*} 
For a skew diagram
$\la/\mu$, denote\footnote{The factorized form of $\sk_{\la/\mu}(q,t)$ can be found in \cite[(1.8)]{War13}, \cite[p.~173, Remark~2]{Rai06}, and \cite[Prop.~3.2]{War05}.},
\begin{align*}
f(u)&=\frac{(tu;q)_{\infty}}{(qu;q)_{\infty}}, \quad \varphi_{\la/\mu}(q,t)=\prod_{1\leq i\leq j\leq \ell(\la)}\frac{f(q^{\la_i-\la_j}t^{j-i})f(q^{\mu_i-\mu_{j+1}}t^{j-i})}{f(q^{\la_i-\mu_j}t^{j-i})f(q^{\mu_i-\la_{j+1}}t^{j-i})},\\
&\sk_{\la/\mu}(q,t)=Q_{\la/\mu}\left(\frac{1-q/t}{1-t};q,t\right)=t^{n(\la)-n(\mu)}\prod_{i,j=1}^{\ell(\la)}\frac{(qt^{j-i-1};q)_{\la_i-\mu_j}(qt^{j-i};q)_{\mu_i-\mu_j}}
{(qt^{j-i-1};q)_{\mu_i-\mu_j}(qt^{j-i};q)_{\la_i-\mu_j}},\\
&\psi'_{\la/\mu}(q,t)=\prod_{(i,j)}\frac{(1-q^{\mu_i-\mu_j}t^{j-i-1})(1-q^{\la_i-\la_j}t^{j-i+1})}{(1-q^{\mu_i-\mu_j}t^{j-i})(1-q^{\la_i-\la_j}t^{j-i})},
\end{align*}
where the product is taken over all pairs $(i,j)$ such that $i<j$ and $\la_i=\mu_i$, $\la_j=\mu_j+1$. Note that
$\varphi_{\la/\mu}(q,t)$ (resp. $\psi'_{\la/\mu}(q,t)$) is zero unless $\la/\mu$ is a horizontal (resp. vertical) strip.

It is well known that \cite[p. 346, (7.14)]{Mac}
\begin{align}\label{e:Q(1)}
Q_{\la/\nu}[1;q,t]=
\begin{cases}
\varphi_{\la/\nu}(q,t)&\text{if $\la/\nu$ is a horizontal strip},\\
0&\text{otherwise},
\end{cases}
\end{align}
whereas we have \cite[p. 524]{War13}
\begin{align}\label{e:Q(-1)}
Q_{\nu/\mu}[-1;q,t]=\sum_{\eta}(-1)^{|\nu/\eta|}t^{|\eta/\mu|}\psi'_{\nu/\eta}(q,t)\sk_{\eta/\mu}(q,t),
\end{align}
summed over all partitions $\eta$ such that $\mu\subset\eta\subset\nu$ and $\nu/\eta$ is a vertical strip.

Summarizing \eqref{e:Q(1)} and \eqref{e:Q(-1)}, we have
\begin{align}\label{e:Q(1-q)}
\begin{split}
    &Q_{\la/\mu}[1-q;q,t]\\
    =&\sum_{\mu\subset\nu\subset\la}Q_{\la/\nu}[1;q,t]Q_{\nu/\mu}[-q;q,t]\quad \text{(by sum rule)}\\
    =&\sum_{\mu\subset\nu\subset^{\rm h}\la} \varphi_{\la/\nu}(q,t)q^{|\nu/\mu|}\sum_{\mu\subset\eta\subset^{\rm v}\nu}(-1)^{|\nu/\eta|}t^{|\eta/\mu|}\psi^{'}_{\nu/\eta}(q,t)\sk_{\eta/\mu}(q,t)\quad \text{(by \eqref{e:Q(1)} and \eqref{e:Q(-1)})}\\
    =&\sum_{\mu\subset\eta\subset^{\rm v}\nu\subset^{\rm h}\la}(-1)^{|\nu/\eta|}t^{|\eta/\mu|}q^{|\nu/\mu|}\varphi_{\la/\nu}(q,t)\psi^{'}_{\nu/\eta}(q,t)\sk_{\eta/\mu}(q,t)
    \end{split}
\end{align}
where $\eta\subset^{\rm v}\nu$ (resp. $\nu\subset^{\rm h}\la$) means $\eta\subset\nu$ (resp. $\nu\subset\la$) and $\nu/\eta$ (resp. $\la/\nu$) is a vertical strip (resp. horizontal strip).

\section{\texorpdfstring{Skew inverse $t$-Kostka coefficients}{Skew inverse t-Kostka coefficients}} The skew inverse $t$-Kostka coefficient $\K_{\la/\mu,\nu}(t)$ is defined by $\K_{\la/\mu,\nu}(t):=\K_{\la/\mu,\nu}(0,t)$, i.e.,
\begin{align}
   S_{\la/\mu}(X;t)=\sum_{\nu}\K_{\la/\mu,\nu}(t)Q_{\nu}(X;t).
\end{align}
Here $Q_{\nu}(X;t)$ is the Hall-Littlewood $Q$-function. Let $q=0$. Then \eqref{e:skew-inv-K-iterative} and \eqref{e:skew-inv-K-iterative2} reduce to, respectively, 
\begin{itemize}
    \item 
   \begin{align}\label{e:skew-inv-K(t)-iterative1}
     \K_{\la/\mu,\nu}(t)=\sum_{\xi,\rho}(-1)^{\rht((\la/\mu)/\xi)}\frac{b_{\rho}(t)}{b_{\nu}(t)}Q_{\nu/\rho}[1;t]\K_{\xi,\rho}(t)   
    \end{align}
    summed over all skew diagrams $\xi$ obtained by removing a special ribbon from $\la/\mu$ and partitions $\rho\subset\nu$ with $|\rho|=|\xi|$.
    \item  \begin{align}\label{e:skew-inv-K(t)-iterative2}
        \K_{\la/\mu,\nu}(t)=\frac{1}{b_{\nu}(t)}\sum_{(\T,\{\nu\})}\sgn(\T)\prod_{i}Q_{\nu^i/\nu^{i-1}}[1;t]
    \end{align}
    summed over all pairs of SRTs $\T$ of shape $\la/\mu$ and flags of partition $\{\nu\}$ of shape $\nu$ with the same type.
\end{itemize}  

We have the following facts
\begin{enumerate}
    \item According to \cite[Ch. III (5.14)]{Mac}
\begin{align}
Q_{\nu/\rho}[1;t]=
\begin{cases}
\varphi_{\nu/\rho}(t)&\text{if $\nu/\rho$ is a horizontal strip},\\
0&\text{otherwise},
\end{cases}
\end{align}
where $\varphi_{\nu/\rho}(t)=\prod_{i\in I}(1-t^{m_i(\nu)})$ in which $I$ is the set of integers $i\geq1$ such that $\nu^t_i-\rho^t_{i}=1$ and $\nu^t_{i+1}-\rho^t_{i+1}=0$.
\item Furthermore, for a flag of partition $\{\nu\}$ of shape $\nu$ and type $\tau$,
\begin{align}
    \prod_{i}Q_{\nu^i/\nu^{i-1}}[1;t]=
    \begin{cases}
\varphi_{T}(t)&\text{if $\{\nu\}$ is a SSYT of shape $\nu$ and weight $\tau$},\\
0&\text{otherwise},
\end{cases}
\end{align}
where $\varphi_{T}(t):=\prod_{i}\varphi_{\nu^i/\nu^{i-1}}(t)$.
\item It follows from \cite[Ch III (5.12)]{Mac} that
\begin{align}
    \varphi_{\nu/\rho}(t)\frac{b_{\rho}(t)}{b_{\nu}(t)}=\psi_{\nu/\rho}(t)
\end{align}
where $\psi_{\nu/\rho}(t)=\prod_{j\in J}(1-t^{m_j(\rho)})$ in which $J$ is the set of integers $j\geq1$ such that $\nu^t_j=\rho^t_{j}$ and $\nu^t_{j+1}-\rho^t_{j+1}=1$. Moreover,
\begin{align}\label{e:def-psi}
    \frac{\varphi_{T}(t)}{b_{\nu}(t)}=\psi_{T}(t)
\end{align}
where $T=(\varnothing=\nu^0\subset\cdots\subset\nu^N=\nu)$ is a SSYT of shape $\nu$ and $\psi_{T}(t):=\prod_{i}^N \psi_{\nu^i/\nu^{i-1}}(t)$
\end{enumerate}
With \eqref{e:skew-inv-K(t)-iterative1}, \eqref{e:skew-inv-K(t)-iterative2} and these facts, we have the combinatorial formulas as follows.
\begin{thm}\label{t:skew-inv-K(t)-iterative}
    For partitions $\mu\subset\la$ and $\nu$ with $|\la/\mu|=|\nu|$,
    \begin{align}\label{e:skew-inv-K(t)-iterative3}
     \K_{\la/\mu,\nu}(t)=\sum_{\xi,\rho}(-1)^{\rht((\la/\mu)/\xi)}\psi_{\nu/\rho}(t)\K_{\xi,\rho}(t)  
    \end{align}
    summed over all skew diagrams $\xi$ obtained by removing a special ribbon from $\la/\mu$ and partitions $\rho\subset^{\rm h}\nu$ with $|\rho|=|\xi|$. And
     \begin{align}\label{e:skew-inv-K(t)-iterative4}
        \K_{\la/\mu,\nu}(t)=\sum_{(\T,T)}\sgn(\T)\psi_{T}(t)
    \end{align}
    summed over all pairs of SRTs $\T$ of shape $\la/\mu$ and SSYTs $T$ of shape $\nu$ such that ${\rm type}(\T)=\wt(T)$.
\end{thm}

\begin{rem}
Wheeler and Zinn-Justin \cite{WZ18} gave a different combinatorial formula for $\K_{\la/\mu,\nu}(t)$ in terms of lattice tilings (analogous to Knutson--Tao puzzles \cite{KT03}). In particular, Carbonara \cite{Car98} interpreted the straight inverse $t$-Kostka coefficients using special tournament matrices.
\end{rem}

\begin{lem}\label{lem:special-horizontal-dominance}
Let $\xi,\rho,\lambda,\nu$ be partitions with $|\xi|=|\rho|$.
\begin{enumerate}
    \item If $\xi\le\rho$, $\lambda/\xi$ is a special ribbon, and $\nu/\rho$ is a horizontal strip with
    $|\lambda/\xi|=|\nu/\rho|$, then $\lambda\le\nu$.
    \item If $\lambda/\xi$ is a special ribbon and $\lambda/\rho$ is a horizontal strip, then $\rho\le\xi$.
\end{enumerate}
\end{lem}

\begin{proof}
We pad all partitions with zeros. Let $a$ be the top row met by the special ribbon $\lambda/\xi$.
Then
\[
\xi_j=\lambda_j\quad(j<a),
\qquad
\xi_j\le\lambda_{j+1}\quad(j\ge a).
\]
We also use the standard characterization of horizontal strips: if $\beta/\alpha$ is a horizontal strip, then
\[
\beta_{j+1}\le \alpha_j\qquad(j\ge1).
\]

For (1), if $r<a$, then
\[
\sum_{j\le r}\lambda_j=\sum_{j\le r}\xi_j\le \sum_{j\le r}\rho_j\le \sum_{j\le r}\nu_j.
\]
If $r\ge a$, then the horizontal-strip condition and $\xi\le\rho$ give
\[
\sum_{j>r}\nu_j\le \sum_{j\ge r}\rho_j\le \sum_{j\ge r}\xi_j
\le \sum_{j>r}\lambda_j.
\]
Since $|\lambda|=|\nu|$, this is equivalent to
$\sum_{j\le r}\lambda_j\le\sum_{j\le r}\nu_j$.
Thus $\lambda\le\nu$.

For (2), if $r<a$, then
\[
\sum_{j\le r}\rho_j\le \sum_{j\le r}\lambda_j=\sum_{j\le r}\xi_j.
\]
If $r\ge a$, then the horizontal-strip condition for $\lambda/\rho$ gives
\[
\sum_{j>r}\rho_j\ge \sum_{j>r}\lambda_{j+1}\ge \sum_{j>r}\xi_j.
\]
Since $|\rho|=|\xi|$, this is equivalent to
$\sum_{j\le r}\rho_j\le\sum_{j\le r}\xi_j$.
Hence $\rho\le\xi$.
\end{proof}

It is well known that the $t$-Kostka matrix $(K_{\la,\nu}(t))_{\la,\nu}$ is strictly upper unitriangular, i.e., $K_{\la,\la}=1$ and $K_{\la,\nu}(t)=0$ unless $\la\ge\nu$ in the dominance order, i.e., $\sum_{i=1}^j\la_i\ge\sum_{i=1}^j\nu_i$ for all $j\geq1$. Thus, its inverse matrix is also strictly upper unitriangular. As an application of our formulas, now we give a combinatorial proof of this by Theorem \ref{t:skew-inv-K(t)-iterative}.

\begin{cor}\label{c:K(t)_properties}
    The inverse $t$-Kostka polynomials satisfy the following properties:
    \begin{enumerate}
    \item Integrality: $\K_{\la/\mu,\nu}(t)\in\mathbb Z[t]$.
        \item Support condition: $\K_{\la,\nu}(t) = 0$ unless $\la \le \nu$.
        \item Unit triangularity: $\K_{\la,\la}(t) = 1$.
    \end{enumerate}
    Consequently, with partitions ordered by dominance, the matrix $(\K_{\la,\nu}(t))_{\la,\nu}$ is unitriangular in the triangular direction opposite to the Kostka--Foulkes matrix under the notation convention used here\footnote{Note that we adopt the different notation of $\K_{\la,\nu}(t)$ from that in other literature which exchanges the positions of $\la$ and $\nu$.}.
\end{cor}

\begin{proof}
    We proceed by induction on the size $|\la|$. The base case $\la=\varnothing$ is trivial as $\K_{\varnothing,\varnothing}(t)=1$. Assume the properties hold for all shapes smaller than $|\la|$.

For the integrality, it is clear by using the recurrence \eqref{e:skew-inv-K(t)-iterative3}.
     
    For the support condition, the summation ranges over partitions $\rho \subset \nu$ such that $\nu/\rho$ is a horizontal strip and partitions $\xi$ obtained by removing a special ribbon $R$ from $\la$ (so $\la = \xi \cup R$). By the inductive hypothesis, $\K_{\xi,\rho}(t) \neq 0$ implies $\xi \le \rho$. Lemma~\ref{lem:special-horizontal-dominance}(1) then gives $\la\le\nu$. If $\la \not\le \nu$, every term in the summation vanishes.
    
    For unit triangularity, set $\nu = \la$. We examine the recurrence \eqref{e:skew-inv-K(t)-iterative3}:
    \begin{equation*}
        \K_{\la,\la}(t) = \sum_{\xi,\rho} (-1)^{\rht(R)} \psi_{\la/\rho}(t) \K_{\xi,\rho}(t),
    \end{equation*}
    where $R = \la/\xi$ is a special ribbon and $\la/\rho$ is a horizontal strip. For a non-vanishing contribution, we require $\xi \le \rho$ (by the support condition on $\K_{\xi,\rho}(t)$). Lemma~\ref{lem:special-horizontal-dominance}(2) gives $\rho\leq\xi$. Thus, we are forced to choose $\rho = \xi$, which implies $R$ is also a horizontal strip with $(-1)^{\rht(R)}=1$ and $\la/\rho$ is also a special ribbon starting from the first column with $\psi_{\la/\rho}(t)=1$. Since a horizontal special ribbon in the straight shape $\la$ must start at the southwest-most box, it is forced to be the entire bottom row of $\la$; hence the choice of $\xi=\rho$ is unique.
    Thus, the sum reduces to a single term with coefficient $1$:
        $\K_{\la,\la}(t)=\K_{\xi,\xi}(t)=1$ (by the inductive hypothesis).
\end{proof}

\begin{exmp}
    It follows from the support condition and the unit triangularity of $\K_{\la,\nu}(t)$ that
    \begin{enumerate}
        \item $\K_{(n),\nu}(t)=\delta_{\nu,(n)}$;
        \item $\K_{\la,(1^n)}(t)=\delta_{\la,(1^n)}$.
    \end{enumerate}
\end{exmp}
Now we explore more formulas for $\K_{\la,\nu}(t)$ in special cases by using our combinatorial formula in Theorem \ref{t:skew-inv-K(t)-iterative}.  

\begin{cor}\label{c:K-lambda-n-closed}
Let $\lambda\vdash n$. Then
\[
\K_{\lambda,(n)}(t)=
\begin{cases}
(-t)^{\,n-\lambda_1}, & \text{if $\lambda=(\la_1,1^{n-\la_1})$},\\
0, & \text{otherwise}.
\end{cases}
\]
\end{cor}

\begin{proof}
We argue it by induction on $n$. The initial step $n=1$ is clear. We assume the statement holds for all smaller sizes. Using the iterative formula  \eqref{e:skew-inv-K(t)-iterative3},
we obtain
\begin{equation}\label{eq:K-lambda-n-rec}
\K_{\lambda,(n)}(t)
=\sum_{\xi}\;(-1)^{\rht(\lambda/\xi)}\,
\psi_{(n)/\rho}(t)\,\K_{\xi,\rho}(t),
\end{equation}
where the sum runs over all partitions $\xi$ obtained from $\lambda$ by removing
a special ribbon $\lambda/\xi$, and $\rho$ runs over all $\rho\subset^h (n)$ with
$|\rho|=|\xi|$.
Since $\nu=(n)$ is a single row, the condition $\rho\subset^h (n)$ forces
$\rho=(|\rho|)$, i.e. $\rho=(m)$ with $m=|\xi|$.
Moreover, for $\rho=(m)$ one checks directly from the definition of $\psi$ that
\begin{equation}\label{eq:psi-row}
\psi_{(n)/(m)}(t)=
\begin{cases}
1-t, & m\ge 1,\\
1, & m=0.
\end{cases}
\end{equation}
Therefore \eqref{eq:K-lambda-n-rec} simplifies to
\begin{equation}\label{eq:K-lambda-n-rec-2}
\K_{\lambda,(n)}(t)
=\sum_{\xi:\,|\xi|\ge 1}(-1)^{\rht(\lambda/\xi)}(1-t)\,\K_{\xi,(|\xi|)}(t)
\;+\;\delta_{\xi,\varnothing}\cdot(-1)^{\rht(\lambda)}.
\end{equation}
Here the last term is present exactly when $\lambda$ itself is a hook shape.

\smallskip
\noindent\emph{Case A: $\lambda$ is not a hook.} In this case, 
\begin{equation}\label{eq:K-lambda-n-not-hook}
\K_{\lambda,(n)}(t)
=(1-t)\sum_{\xi:\,|\xi|\ge 1}(-1)^{\rht(\lambda/\xi)}\K_{\xi,(|\xi|)}(t).
\end{equation}
By the induction hypothesis, $\K_{\xi,(|\xi|)}(t)=0$ unless $\xi$ is a hook. This together with the fact that $R=\la/\xi$ is a special ribbon and $\la$ is not a hook, $R$ must starts from the southwest-most box and ends at the rightmost box of either the second row or the first row. We denote these two ribbons by $R'=\la/\xi'$ and $R''=\la/\xi''$. Then the height difference of these two special ribbon is exactly $1$. Moreover, $\xi'$ and $\xi''$ have the same leg length. By inductive hypothesis, $\K_{\xi',(|\xi'|)}(t)=\K_{\xi'',(|\xi''|)}(t)$. Plugging these into \eqref{eq:K-lambda-n-not-hook} implies $\K_{\la,(n)}(t)=0$.

\smallskip
\noindent\emph{Case B: $\lambda$ is a hook, write $\lambda=(n-b,1^b)$ with $b=n-\lambda_1$.}
In this case $\lambda$ \emph{is} a ribbon, so the $\xi=\varnothing$ term contributes
$(-1)^{\rht(\lambda)}=(-1)^b$ in \eqref{eq:K-lambda-n-rec-2}.
For $\xi\neq\varnothing$, the induction hypothesis implies that $\K_{\xi,(|\xi|)}(t)$ vanishes
unless $\xi$ is again a hook. A direct inspection of special ribbon removals from a hook shows that among all $\xi$
arising from $\lambda$ by removing a special ribbon, the only nonzero contributions in
\eqref{eq:K-lambda-n-rec-2} come from those $\xi$ which are hooks
$\xi=(n-b,1^{b-1}), (n-b,1^{b-2}),\dots,(n-b)$, i.e. obtained by removing a vertical strip
from the leg.
For such a removal of height $r$ (so $\rht(\lambda/\xi)=r-1$), we have $|\xi|=n-r$ and
by induction $\K_{\xi,(n-r)}(t)=(-t)^{b-(r)}$.
Substituting into \eqref{eq:K-lambda-n-rec-2} yields
\[
\K_{(n-b,1^b),(n)}(t)
=(1-t)\sum_{r=1}^{b}(-1)^{r-1}(-t)^{\,b-r}\;+\;(-1)^b=(-t)^b.
\]
Hence $\K_{\lambda,(n)}(t)=(-t)^b=(-t)^{n-\lambda_1}$, completing the induction. This finishes the proof.
\end{proof}

We highlight that Carbonara \cite[Prop. 3]{Car98} presented a closed formula for $\K_{(1^n),\nu}(t)$ proven by Macdonald by using algebraic method. At Appendix B of this paper, he asked for a combinatorial proof of this formula. Here we provide a combinatorial proof of this by our formulas.

Write $[r]_t:=\frac{1-t^r}{1-t}$ and $[r]_t!:=\prod_{i=1}^r [i]_t$. 

\begin{cor}\label{c:(1^n)}
    Let $\nu\vdash n$ with $\ell(\nu)=k$. Then
    \begin{align}
        \K_{(1^n),\nu}(t)=(-t)^{n-k}t^{n(\nu)-\binom{k}{2}}\qbinomial{k}{m(\nu)}_t,
    \end{align}
    where $\qbinomial{k}{m(\nu)}_t:=\qbinomial{k}{m_1(\nu),m_{2}(\nu),\cdots,m_n(\nu)}_t=\frac{[k]_t!}{[m_1(\nu)]_t!\cdots[m_n(\nu)]_t!}$
    is the $t$-multinomial coefficient.
\end{cor}
\begin{proof}
    We proceed by induction on $n$. The base case $n=0$ is trivial. Assume the formula holds for all partitions of size less than $n$.
    
    Consider the left-hand side of \eqref{e:skew-inv-K(t)-iterative3} with $\lambda=(1^n)$ and $\mu=\emptyset$. The special ribbons of $(1^n)$ are vertical strips of the form $(1^r)$ for $1 \leq r \leq n$, with $\rht((1^r))=r-1$. Removing such a ribbon leaves the shape $\xi=(1^{n-r})$. Applying Theorem \ref{t:skew-inv-K(t)-iterative}, we express $\K_{(1^n),\nu}(t)$ as a sum over $r$ and partitions $\rho \subset^{\rm h} \nu$ with $|\nu/\rho|=r$:
    \begin{align}\label{e:proof-recursion}
        \K_{(1^n),\nu}(t)=\sum_{r=1}^n \sum_{\substack{\rho \subset^{\rm h} \nu \\ |\nu/\rho|=r}} (-1)^{r-1}\psi_{\nu/\rho}(t)\K_{(1^{n-r}),\rho}(t).
    \end{align}
    Let $k=\ell(\nu)$ and $k_\rho=\ell(\rho)$. By the inductive hypothesis, for each $\rho \vdash (n-r)$, we have:
    \begin{align*}
        \K_{(1^{n-r}),\rho}(t) = (-t)^{(n-r)-k_\rho}t^{n(\rho)-\binom{k_\rho}{2}}\qbinomial{k_\rho}{m(\rho)}_t.
    \end{align*}
    Substituting this into \eqref{e:proof-recursion}, we rewrite the term $(-t)^{(n-r)-k_\rho}$ as $(-1)^{n-r-k_\rho}t^{n-r-k_\rho}$. The powers of $-1$ combine as $(-1)^{r-1}(-1)^{n-r-k_\rho} = (-1)^{n-k_\rho-1}$. Thus, the RHS becomes:
    \begin{align*}
         \sum_{r, \rho} (-1)^{n-k_\rho-1} t^{n-r-k_\rho} t^{n(\rho)-\binom{k_\rho}{2}} \psi_{\nu/\rho}(t) \qbinomial{k_\rho}{m(\rho)}_t.
    \end{align*}
    We claim that this sum equals the desired expression $(-t)^{n-k}t^{n(\nu)-\binom{k}{2}}\qbinomial{k}{m(\nu)}_t$. As the proof of this claim is fairly technical, using a transfer-matrix argument, we postpone it to Appendix \ref{A:identity} for the reader’s convenience.  
Thus, the inductive step is closed.
\end{proof}


\section{Skew inverse Kostka coefficients}
The skew inverse Kostka coefficient $\K_{\la/\mu,\nu}$ is defined by
\[
\K_{\la/\mu,\nu}:=\K_{\la/\mu,\nu}(0,1)=\K_{\la/\mu,\nu}(1).
\]
In particular, the inverse Kostka coefficient $\K_{\la,\nu}$ is defined by $\K_{\la,\nu}:=\K_{\la/\varnothing,\nu}$, i.e.,
\begin{align}
    m_{\nu}(X)=\sum_{\la}\K_{\la,\nu}s_{\la}(X).
\end{align}

As a specialization of Theorem \ref{t:skew-inv-K(t)-iterative}, we have the following combinatorial formulas for $\K_{\la/\mu,\nu}$.

\begin{thm}\label{t:skew-inv-K(1)-iterative}
    For partitions $\mu\subset\la$ and $\nu=(\nu_1,\cdots,\nu_m)$ with $|\la/\mu|=|\nu|$,
    \begin{align}\label{e:skew-inv-K(1)-iterative1}
     \K_{\la/\mu,\nu}=\sum_{\xi}(-1)^{\rht((\la/\mu)/\xi)}\K_{\xi,\tilde{\nu}}
    \end{align}
    summed over all skew diagrams $\xi$ obtained by removing a special ribbon from $\la/\mu$ such that $|\la/\mu|-|\xi|$ equals any part of $\nu$, and $\tilde{\nu}$ is the partition obtained from $\nu$ by deleting that part. Furthermore,
     \begin{align}\label{e:skew-inv-K(1)-iterative2}
        \K_{\la/\mu,\nu}=\sum_{\T}\sgn(\T)
    \end{align}
    summed over all SRTs $\T$ of shape $\la/\mu$ such that the type of $\T$ rearranges to $\nu$.
\end{thm}

\begin{proof}
    We derive this result by considering the limit $t\rightarrow 1$ in Theorem \ref{t:skew-inv-K(t)-iterative}. Recall that the coefficient involves $\psi_{\nu/\rho}(t)$, which vanishes at $t=1$ unless the set $J$ in its definition is empty.

    The condition $J = \varnothing$ implies that for all $j \ge 1$, if the $(j+1)$-th column of the skew diagram grows (i.e., $\nu^t_{j+1} > \rho^t_{j+1}$), then the $j$-th column must also have grown (i.e., $\nu^t_j > \rho^t_j$). Since $\nu/\rho$ is a horizontal strip, the growth in any column is at most $1$. Consequently, the set of columns where $\nu/\rho$ adds cells must be of the form $\{1, 2, \dots, k\}$ for some $k$. Geometrically, this means $\nu$ is obtained from $\rho$ by adding a single row of length $k=|\nu|-|\rho|$. Conversely, this implies $\rho$ is obtained from $\nu$ by deleting some row (i.e., some part $\nu_i$) of length $k$. Since $\psi_{\nu/\rho}(1)=1$ for any such $\rho$, the recurrence sum in \eqref{e:skew-inv-K(t)-iterative3} runs over all sub-partitions $\rho$ obtained by deleting any single part of $\nu$.

Repeatedly applying \eqref{e:skew-inv-K(1)-iterative1}, or equivalently taking $t \to 1$ in \eqref{e:skew-inv-K(t)-iterative4}, we observe that the factor $\psi_T(1) = \prod \psi_{\nu^i/\nu^{i-1}}(1)$ is non-zero (and equal to 1) if and only if each step adds a row. This corresponds to the SSYTs T of shape $\nu$ such that ${\rm wt}(T)$ rearranges to $\nu$. The condition ${\rm type}(\T) = \wt(T)$ implies ${\rm type}(\T)$ also rearranges to $\nu$. This completes the proof.
\end{proof}

We remark that E\u gecio\u glu--Remmel first obtained \eqref{e:skew-inv-K(1)-iterative1}--\eqref{e:skew-inv-K(1)-iterative2} in \cite{ER90}. Further combinatorial descriptions and recurrences were developed in \cite{D03,PR17}, motivated in part by expressing mod-$p$ Steenrod operations on Chern classes in the Schubert basis.

\chapter{Modular Schur functions}\label{s:modular}

Building on the skew Murnaghan--Nakayama expansions developed in the previous chapters, we now turn to a plethystic specialization of Schur functions that arises naturally in modular representation theory.

Throughout this chapter we fix a positive integer $k$.

\section{Skew plethystic Murnaghan-Nakayama rule}\label{ss:skew-PMN}
The plethystic Murnaghan--Nakayama rule was introduced in \cite{DLT94} via Muir’s rule, giving an expansion of the product $p_n[h_k]\,s_{\la}$ in the Schur basis. Subsequent proofs include a character-theoretic argument \cite{EPW14}, a direct combinatorial proof \cite{Wil16}, as well as abacus and vertex-operator proofs \cite{Tur25,CJL25}. In this subsection we extend the rule to skew shapes, i.e., we expand the product $p_n[h_k]\,s_{\la/\eta}$ in terms of skew Schur functions by means of our skew Murnaghan--Nakayama framework. We begin with some combinatorial notions.

For a skew diagram $\rho/\la$, we call $\rho/\la$ a {\em horizontal $n$-ribbon of weight $k$} if there exist partitions
\[
\la=\tau^{(0)}\subset\tau^{(1)}\subset\cdots\subset\tau^{(k)}=\rho
\]
such that (i) each difference $\tau^{(i)}/\tau^{(i-1)}$ is an $n$-ribbon (a ribbon of $n$ boxes), and (ii) the ending box of each $\tau^{(i)}/\tau^{(i-1)}$ is the topmost box in its column of $\rho/\la$. If $\rho/\la$ is a horizontal $n$-ribbon of weight $k$, define
\[
\sgn(\rho/\la):=\prod_{i=1}^{k}(-1)^{\rht\big(\tau^{(i)}/\tau^{(i-1)}\big)}.
\]
Figure~\ref{fig:n-ribbon-of-k} shows a horizontal $4$-ribbon of weight $5$.

\begin{figure}
    \centering
\begin{tikzpicture}[scale = 0.4]
  \begin{scope}
    \clip (0,0) -| (1,2) -| (3,3) -| (4,4) -| (8,7)  -| (4,5) -| (3,4) -| (2,3) -| (0,0);
    \draw [color=black!25] (0,0) grid (8,7);
  \end{scope}
  \draw [thick] (0,0) -| (1,2) -| (3,3) -| (4,4) -| (8,7)  -| (4,5) -| (3,4) -| (2,3) -| (0,0);
  \draw[thick] (4,7) -- (0,7) -- (0,3);

  \draw [thick, rounded corners] (0.5,0.5) -- (0.5,2.5) -- (1.5,2.5);
  \draw [color=black,fill=black,thick] (1.5,2.5) circle (.5ex);
  \node [draw, circle, fill = white, inner sep = 1.5pt] at (0.5,0.5) { };

  \draw [thick, rounded corners] (2.5,2.5) -- (2.5,3.5) -- (3.5,3.5) -- (3.5,4.5);
  \draw [color=black,fill=black,thick] (3.5,4.5) circle (.5ex);
  \node [draw, circle, fill = white, inner sep = 1.5pt] at (2.5,2.5) { };

  \draw [thick, rounded corners] (4.5,4.5) -- (4.5,6.5) -- (5.5,6.5);
  \draw [color=black,fill=black,thick] (5.5,6.5) circle (.5ex);
  \node [draw, circle, fill = white, inner sep = 1.5pt] at (4.5,4.5) { };

  \draw [thick, rounded corners] (5.5,4.5) -- (5.5,5.5) -- (6.5,5.5) -- (6.5,6.5);
  \draw [color=black,fill=black,thick] (6.5,6.5) circle (.5ex);
  \node [draw, circle, fill = white, inner sep = 1.5pt] at (5.5,4.5) { };

  \draw [thick, rounded corners] (6.5,4.5) -- (7.5,4.5) -- (7.5,6.5);
  \draw [color=black,fill=black,thick] (7.5,6.5) circle (.5ex);
  \node [draw, circle, fill = white, inner sep = 1.5pt] at (6.5,4.5) { };
\end{tikzpicture}
\caption{$\rho=(8,8,8,4,3,1,1)$, $\la=(4,4,3,2)$. Here $\rho/\la$ is a horizontal $4$-ribbon of weight $5$, and $\sgn(\rho/\la)=(-1)^2\!\cdot(-1)^2\!\cdot(-1)^2\!\cdot(-1)^2\!\cdot(-1)^2=1$.}
\label{fig:n-ribbon-of-k}
\end{figure}
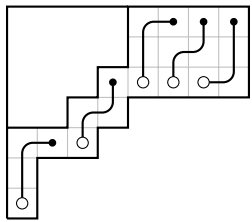

If $\rho/\la$ is a horizontal $n$-ribbon of weight $k$, then $k$ is determined by $|\rho/\la|=nk$; we will sometimes omit $k$ from the terminology.

With this notation, the plethystic Murnaghan--Nakayama rule states that
\begin{equation}\label{e:PMN}
p_n[h_k]\,s_{\la}=\sum_{\rho}\sgn(\rho/\la)\,s_{\rho},
\end{equation}
where the sum ranges over all partitions $\rho$ such that $\rho/\la$ is a horizontal $n$-ribbon of weight $k$. This simultaneously generalizes the classical Murnaghan--Nakayama rule ($k=1$) and the Pieri rule ($n=1$).

The next lemma is the key specialization needed for the skew version of \eqref{e:PMN}.

\begin{lem}\label{l:spec-omega}
Let $n\ge1$ and let $\omega_n$ be a primitive $n$th root of unity. Then
\begin{equation}\label{e:spec-omega}
s_{\rho/\la}[\Omega_n]=
\begin{cases}
\sgn(\rho/\la), & \text{if $\rho/\la$ is a horizontal $n$-ribbon},\\
0, & \text{otherwise},
\end{cases}
\end{equation}
where $\Omega_n:=1+\omega_n+\omega_n^2+\cdots+\omega_n^{n-1}$.
\end{lem}

\begin{proof}
Consider the generating series for $p_n[h_k]$:
\[
\begin{aligned}
\sum_{k\ge0}p_n[h_k]\,z^{nk}
&=\sum_{k\ge0}h_k(x_1^n,x_2^n,\dots)\,z^{nk}
=\prod_{i\ge1}\frac{1}{1-x_i^n z^n}\\
&=\prod_{i\ge1}\prod_{j=0}^{n-1}\frac{1}{1-\omega_n^{j}x_i z}
=\sum_{m\ge0}h_m[\Omega_n X]\,z^{m}.
\end{aligned}
\]
Hence $p_n[h_k]=h_{kn}[\Omega_n X]$. Specializing \eqref{e:ge-skew-Mac} at $(q,t,Y,\eta)=(0,0,z\Omega_n,\varnothing)$ gives
\[
\exp\!\left(\sum_{m\ge1}\frac{1}{m}\,p_m[X]p_m[z\Omega_n]\right)s_{\la}
=\sum_{\rho\supset\la} z^{|\rho/\la|}\,s_{\rho/\la}[\Omega_n]\,s_{\rho}.
\]
Since $\sum_{m\ge0}h_m[\Omega_n X]\,z^m=\exp\!\left(\sum_{r\ge1}\frac{1}{r}p_r[\Omega_n X]\,z^r\right)$, we obtain
\[
p_n[h_k]\,s_{\la}=h_{kn}[\Omega_n X]\,s_{\la}
=\sum_{\rho\vdash|\la|+kn}s_{\rho/\la}[\Omega_n]\,s_{\rho}.
\]
Comparing the coefficient of $s_{\rho}$ with \eqref{e:PMN} yields \eqref{e:spec-omega}.
\end{proof}

We can now state the skew plethystic Murnaghan--Nakayama rule.

\begin{thm}\label{t:skew-PMN}
Let $k\ge0$ and $n\ge1$, and let $\la/\eta$ be a skew diagram. Then
\begin{equation}\label{e:skew-PMN}
p_n[h_k]\,s_{\la/\eta}
=\sum_{\rho,\mu}(-1)^{|\eta/\mu|}\,\sgn(\rho/\la)\,\sgn(\eta^{t}/\mu^{t})\,s_{\rho/\mu},
\end{equation}
where the sum ranges over all partitions $\rho,\mu$ such that both $\rho/\la$ and $\eta^{t}/\mu^{t}$ are horizontal $n$-ribbons and the sum of their weights is $k$.
\end{thm}

\begin{proof}
Taking $(q,t,Y)=(0,0,z\Omega_n)$ in \eqref{e:ge-skew-Mac} gives
\[
\exp\!\left(\sum_{m\ge1}\frac{1}{m}\,p_m[X]p_m[z\Omega_n]\right)s_{\la/\eta}
=\sum_{\rho,\mu} z^{|\rho/\la|+|\eta/\mu|}\,s_{\rho/\la}[\Omega_n]\;s_{\eta/\mu}[-\Omega_n]\;s_{\rho/\mu}.
\]
Using
\[
\sum_{m\ge0}h_m[\Omega_n X]\,z^m
=\exp\!\left(\sum_{r\ge1}\frac{1}{r}p_r[\Omega_n X]\,z^r\right)
\qquad\text{and}\qquad
p_n[h_k]=h_{kn}[\Omega_n X],
\]
and extracting the coefficient of $z^{kn}$ yields
\[
\begin{aligned}
p_n[h_k]\,s_{\la/\eta}
&=\sum_{\substack{\rho,\mu\\ |\rho/\la|+|\eta/\mu|=kn}}
s_{\rho/\la}[\Omega_n]\;s_{\eta/\mu}[-\Omega_n]\;s_{\rho/\mu}
\\
&=\sum_{\substack{\rho,\mu\\ |\rho/\la|+|\eta/\mu|=kn}}
(-1)^{|\eta/\mu|}\,s_{\rho/\la}[\Omega_n]\;s_{\eta^{t}/\mu^{t}}[\Omega_n]\;s_{\rho/\mu},
\end{aligned}
\]
where we used $s_{\alpha/\beta}[-\Omega_n]=(-1)^{|\alpha/\beta|}s_{\alpha^{t}/\beta^{t}}[\Omega_n]$. Applying Lemma~\ref{l:spec-omega} to both $s_{\rho/\la}[\Omega_n]$ and $s_{\eta^{t}/\mu^{t}}[\Omega_n]$ gives \eqref{e:skew-PMN}.
\end{proof}

\section{Skew Pieri-like rule and expansion of modular Schur functions }\label{ss:skew-Pieri-like}
Petrie symmetric functions were introduced independently by Doty--Walker~\cite{DW92} 
and Bazeniar--Ahmia--Belbachir~\cite{BAB18}. In~\cite{DW92} they appear in the study of 
certain truncated tensor products of representations of the general linear group, 
whereas in~\cite{BAB18} they arise from a generalization of Pascal triangles. 
More recently, Fu--Mei~\cite{FM22} and Grinberg~\cite{Gr22} have, independently of each 
other, established a number of interesting properties of these functions. 
The work of Fu and Mei is driven by the observation that Petrie symmetric functions 
provide a natural common framework for elementary symmetric functions and complete 
homogeneous symmetric functions, while Grinberg’s approach is motivated by Liu and 
Polo’s results on the cohomology of line bundles over a flag scheme~\cite{Liu21,LP21}.

\phantomsection\label{def:petrie-functions}
For a fixed positive integer $k$, the {\em Petrie symmetric functions} $G(k,m)$ of 
degree $m$ are defined by the generating series
\begin{align}\label{e:defG}
\sum_{m=0}^{\infty} G(k,m)(X)\, z^m
  &= \prod_{i=1}^{\infty}\left(\sum_{j=0}^{k-1}(x_i z)^j\right) \notag\\
  &= \prod_{i=1}^{\infty}\frac{1-(x_i z)^k}{1-x_i z}
   = \prod_{i=1}^{\infty}\prod_{j=1}^{k-1}\bigl(1-\omega_k^j x_i z\bigr),
\end{align}
where $\omega_k = e^{2\pi i/k}$ is a primitive $k$-th root of unity.

Combining the first identity in~\eqref{e:defG} with the standard generating function 
for integer partitions, one obtains an equivalent description of $G(k,m)(X)$ in the 
basis of monomial symmetric functions:
\begin{align}
G(k,m)(X)
  = \sum_{\substack{\lambda \vdash m \\ \lambda_1 < k}} m_{\lambda}(X).
\end{align}
In particular, this yields the following special cases:
\begin{align*}
G(k,m)(X)=
\begin{cases}
e_m(X), & k=2,\\
h_m(X), & k>m,\\
h_m(X) - p_m(X), &k=m.  
\end{cases}
\end{align*}

The {\em Petrie numbers} ${\rm Pet}_k(\la,\mu)$ are defined as the transition coefficients of Schur functions in the product of Petrie symmetric functions and Schur functions, i.e.,
\begin{align}\label{e:Pieri1}
G(k,m)(X)s_{\mu}(X)=\sum_{\la\vdash m+|\mu|}{\rm Pet}_k(\la,\mu)s_{\la}(X).
\end{align}
Grinberg \cite{Gr22} showed that ${\rm Pet}_k(\la,\mu)$ take values in $\{0,-1,1\}$.
Very recently, Wu et al.\ \cite[Theorem 1.4]{WEK+25} gave a combinatorial interpretation of ${\rm Pet}_k(\la,\mu)$, extending the corresponding interpretation of ${\rm Pet}_k(\la,\varnothing)$ \cite[Theorem 1.3]{CCE+23}.

In this subsection, we present a skew version of \eqref{e:Pieri1}.
Namely, we expand
\[
G(k,m)(X)s_{\la/\mu}(X)
\]
in the skew Schur basis using our skew-MN framework.
To proceed, we need to introduce the combinatorial models that govern ${\rm Pet}_k(\la,\mu)$.

We say a skew diagram $\K$ is a {\em vertical $k$-ribbon} if $\K^{t}$ is a horizontal $k$-ribbon, as recalled in the last subsection.
For a vertical $k$-ribbon $\K$ consisting of $m$ $k$-ribbons $\xi_1,\ldots,\xi_m$, we define 
\begin{align}\label{e:sgn'}
    \sgn^{'}(\K):=\prod_{i=1}^m(-1)^{r(\xi_i)}.
\end{align}
Here $r(\xi_i)$ denotes the number of rows occupied by $\xi_i$.

\phantomsection\label{def:GPR}
A skew diagram $\la/\mu$ is called a {\em proper $k$-ribbon} if there exists a $\nu$ such that $\mu\subset\nu\subset\la$ satisfying (i) $\nu/\mu$ is a horizontal strip; (ii) $\la/\nu$ is a vertical $k$-ribbon. It is clear that $\nu$ is not unique for a given $\la/\mu$. A proper $k$-ribbon $\la/\mu$ is {\em good} if $\nu$ is unique. In this case, we define $\sgn^{'}(\la/\mu):=\sgn^{'}(\la/\nu)$. Figure \ref{fig:proper-tiling} gives two decompositions of the same proper $6$-ribbon.
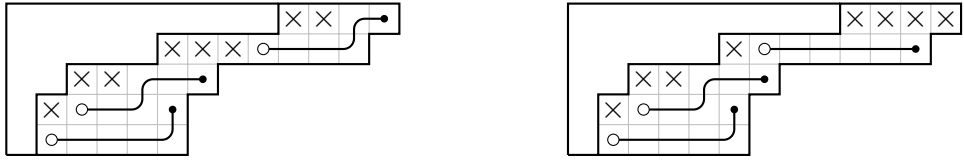
\begin{figure}
    \centering
\begin{tikzpicture}[scale = 0.4]
  \begin{scope}
    \clip (0,0) -| (6,2) -| (7,3) -| (12,4) -| (13,5)  -| (9,4) -| (5,3) -| (2,2) -| (1,0);
    \draw [color=black!25] (0,0) grid (13,5);
  \end{scope}
  \draw [thick] (0,0) -| (6,2) -| (7,3) -| (12,4) -| (13,5)  -| (9,4) -| (5,3) -| (2,2) -| (1,0);
  \draw[thick] (0,0) -- (0,5) -- (9,5);

  \draw [thick, rounded corners] (1.5,0.5) -- (5.5,0.5) -- (5.5,1.5);
  \draw [color=black,fill=black,thick] (5.5,1.5) circle (.5ex);
  \node [draw, circle, fill = white, inner sep = 1.5pt] at (1.5,0.5) { };

  \draw [thick, rounded corners] (2.5,1.5) -- (4.5,1.5) -- (4.5,2.5) -- (6.5,2.5);
  \draw [color=black,fill=black,thick] (6.5,2.5) circle (.5ex);
  \node [draw, circle, fill = white, inner sep = 1.5pt] at (2.5,1.5) { };

  \draw [thick, rounded corners] (8.5,3.5) -- (11.5,3.5) -- (11.5,4.5) -- (12.5,4.5);
  \draw [color=black,fill=black,thick] (12.5,4.5) circle (.5ex);
  \node [draw, circle, fill = white, inner sep = 1.5pt] at (8.5,3.5) { };

 \node at (1.5,1.5) {$\times$};
 \node at (2.5,2.5) {$\times$};
  \node at (3.5,2.5) {$\times$};
   \node at (5.5,3.5) {$\times$};
    \node at (6.5,3.5) {$\times$};
     \node at (7.5,3.5) {$\times$};
      \node at (9.5,4.5) {$\times$};
       \node at (10.5,4.5) {$\times$};
\end{tikzpicture}\hspace{5em}
\begin{tikzpicture}[scale = 0.4]
  \begin{scope}
    \clip (0,0) -| (6,2) -| (7,3) -| (12,4) -| (13,5)  -| (9,4) -| (5,3) -| (2,2) -| (1,0);
    \draw [color=black!25] (0,0) grid (13,5);
  \end{scope}
  \draw [thick] (0,0) -| (6,2) -| (7,3) -| (12,4) -| (13,5)  -| (9,4) -| (5,3) -| (2,2) -| (1,0);
  \draw[thick] (0,0) -- (0,5) -- (9,5);

  \draw [thick, rounded corners] (1.5,0.5) -- (5.5,0.5) -- (5.5,1.5);
  \draw [color=black,fill=black,thick] (5.5,1.5) circle (.5ex);
  \node [draw, circle, fill = white, inner sep = 1.5pt] at (1.5,0.5) { };

  \draw [thick, rounded corners] (2.5,1.5) -- (4.5,1.5) -- (4.5,2.5) -- (6.5,2.5);
  \draw [color=black,fill=black,thick] (6.5,2.5) circle (.5ex);
  \node [draw, circle, fill = white, inner sep = 1.5pt] at (2.5,1.5) { };

  \draw [thick, rounded corners] (6.5,3.5) -- (11.5,3.5);
  \draw [color=black,fill=black,thick] (11.5,3.5) circle (.5ex);
  \node [draw, circle, fill = white, inner sep = 1.5pt] at (6.5,3.5) { };

 \node at (1.5,1.5) {$\times$};
 \node at (2.5,2.5) {$\times$};
  \node at (3.5,2.5) {$\times$};
   \node at (5.5,3.5) {$\times$};
      \node at (9.5,4.5) {$\times$};
       \node at (10.5,4.5) {$\times$};
        \node at (11.5,4.5) {$\times$};
     \node at (12.5,4.5) {$\times$};
\end{tikzpicture}
\caption{Let $\la=(13,12,7,6,6)$ and $\mu=(9,5,2,1,1)$. Then $\la/\mu$ is a proper $6$-ribbon with two decompositions corresponding to $\nu=(11,8,4,2,1)$ (left) and $\nu^{'}=(13,6,4,2,1)$ (right), where every box of $\nu/\mu$ and $\nu^{'}/\mu$ is marked by a cross $\times$.}
\label{fig:proper-tiling}
\end{figure}

The Pieri-like rule for Petrie symmetric functions is
\begin{align}
    G(k,m)(X)s_{\mu}(X)=\sum_{\la\vdash m+|\mu|}{\rm Pet}_k(\la,\mu)s_{\la}(X),
\end{align}
where ${\rm Pet}_k(\la,\mu)$ admits the combinatorial interpretation
\begin{align}\label{e:combin-sgn'}
    {\rm Pet}_k(\la,\mu)=
    \begin{cases}
      \sgn^{'}(\la/\mu) &\text{if $\la/\mu$ is a good proper $k$-ribbon with $\la_i-\mu_i<k$ for all $i$,}\\
      0 &\text{otherwise.}
    \end{cases}
\end{align}

\begin{lem}\label{l:1-Theta}
    For $\mu\subset\la$, we have
    \begin{align}\label{e:1-Theta}
        s_{\la/\mu}[1-\Omega_k]={\rm Pet}_k(\la,\mu),
    \end{align}
    \phantomsection\label{def:omega-k}
    where $\Omega_k=1+\omega_k+\cdots+\omega_k^{k-1}$ and $\omega_k$ is a primitive $k$-th root of unity.
\end{lem}
\begin{proof}
  By \eqref{e:defG}, $G(k,m)(X)=h_{m}[(1-\Omega_k)X]$. Taking the specialization of \eqref{e:ge-skew-Mac} at $(q,t,\eta,Y)=(0,0,\varnothing,(1-\Omega_k)z)$ yields
  \begin{align}\label{e:gene-1-Theta}
   \exp\left(\sum_{n=1}^{\infty}\frac{1}{n}p_n[X]p_n[(1-\Omega_k)z]\right)s_{\la}(X)=\sum_{\rho}z^{|\rho/\la|}s_{\rho/\la}[1-\Omega_k]s_{\rho}(X).  
  \end{align}
  Note that 
  \begin{align}
      \exp\left(\sum_{n=1}^{\infty}\frac{1}{n}p_n[X]p_n[(1-\Omega_k)z]\right)
      =\sum_{m\geq 0}h_{m}[(1-\Omega_k)X]z^m
      =\sum_{m\geq 0}G(k,m)(X)z^m.
  \end{align}
  Taking the coefficient of $z^m$ on both sides of \eqref{e:gene-1-Theta} implies
  \begin{align}\label{e:Gs}
      G(k,m)(X)s_{\la}(X)=\sum_{\rho}s_{\rho/\la}[1-\Omega_k]s_{\rho}(X).
  \end{align}
  Comparing the coefficients in \eqref{e:Pieri1} and \eqref{e:Gs} completes the proof. 
\end{proof}

Now we are ready to give the skew Pieri-like rule for Petrie symmetric functions.
\begin{thm}\label{t:Pet}
   Let $\eta\subset\la$. Then
   \begin{align}
       G(k,m)(X)s_{\la/\eta}(X)=\sum_{\rho,\mu}(-1)^{|\eta/\mu|}\sgn^{'}(\rho/\la)\sgn^{'}(\eta^t/\mu^t)s_{\rho/\mu}(X)
   \end{align}
   where the sum ranges over all partitions $\rho$ and $\mu$ such that $|\rho/\la|+|\eta/\mu|=m$ and both $\rho/\la$ and $\eta^t/\mu^t$ are good proper $k$-ribbons with $\rho_i-\la_i<k$ and $\eta^t_i-\mu^t_i<k$ for all $i$. 
\end{thm}
\begin{proof}
  Specializing \eqref{e:ge-skew-Mac} at $(q,t,Y)=(0,0,(1-\Omega_k)z)$ yields
  \begin{multline}\label{e:pps}
   \exp\left(\sum_{n=1}^{\infty}\frac{1}{n}p_n[X]p_n[(1-\Omega_k)z]\right)s_{\la/\eta}(X)\\
   =\sum_{\rho,\mu}z^{|\rho/\la|+|\eta/\mu|}s_{\rho/\la}[1-\Omega_k]s_{\eta/\mu}[\Omega_k-1]s_{\rho/\mu}(X).  
  \end{multline}
  By \eqref{e:defG}, we have
  \begin{align}\label{e:gene-G}
    \exp\left(\sum_{n=1}^{\infty}\frac{1}{n}p_n[X]p_n[(1-\Omega_k)z]\right)=\sum_{m=0}^{\infty}G(k,m)(X)z^m.  
  \end{align}
  Plugging \eqref{e:gene-G} into \eqref{e:pps} and taking the coefficient of $z^m$ on both sides imply that
  \begin{multline}\label{e:Gssss}
      G(k,m)(X)s_{\la/\eta}(X)=\sum_{\rho,\mu}s_{\rho/\la}[1-\Omega_k]s_{\eta/\mu}[\Omega_k-1]s_{\rho/\mu}(X)\\
      =\sum_{\rho,\mu}(-1)^{|\eta/\mu|}s_{\rho/\la}[1-\Omega_k]s_{\eta^t/\mu^t}[1-\Omega_k]s_{\rho/\mu}(X),
  \end{multline}
  summed over all partitions $\rho$ and $\mu$ such that $|\rho/\la|+|\eta/\mu|=m$. The proof is completed by combining \eqref{e:Gssss}, Lemma \ref{l:1-Theta} and \eqref{e:combin-sgn'}.
\end{proof}

\begin{defn}\label{def:skew-modular-schur}
    The \emph{skew modular Schur function} $\mathfrak s^{(k)}_{\la/\mu}(X)$ is defined by 
    \begin{align}
        \mathfrak s^{(k)}_{\la/\mu}(X):=\det_{1\leq i,j\leq\ell(\la)}\bigl(G(k,\la_i-\mu_j-i+j)(X)\bigr),
    \end{align}
    where, as usual, we pad $\mu$ with zeros so that it has length $\ell(\la)$.
    Here $G(k,m)(X)$ is the Petrie symmetric function defined in \eqref{e:defG}.
    In particular, the \emph{straight} modular Schur function $\mathfrak s^{(k)}_{\la}(X)$ is $\mathfrak s^{(k)}_{\la/\varnothing}(X)$.
\end{defn}

The transition coefficients between modular Schur functions and classical Schur functions,
defined by
\begin{align}
    \mathfrak s^{(k)}_{\la}(X)=\sum_{\mu}m_{\la,\mu}\,s_{\mu}(X),
\end{align}
form a central combinatorial bridge between characteristic $0$ and characteristic $p$ representation theories of $GL_n(\mathbb K)$.
Following Doty--Walker’s work on truncated symmetric powers \cite{DW92} and Walker’s formalization of related modular Schur functions \cite{Wal94}, the matrix $M(\mathfrak s^{(k)},s)$ plays a key role in describing how Weyl characters behave under truncation and in extracting modular structural information (e.g., via Carter-type criteria \cite{Wal94}).
More broadly, these coefficients provide an explicit combinatorial route toward understanding decomposition phenomena for $GL_n(\mathbb K)$.

More generally, define skew coefficients $m_{\la/\mu,\nu}$ by
\begin{align}\label{eq:def-skew-m-coeff}
   \mathfrak s_{\la/\mu}^{(k)}(X)=\sum_{\nu}m_{\la/\mu,\nu}\,s_{\nu}(X).
\end{align}
This section gives a combinatorial interpretation of $m_{\la/\mu,\nu}$ using the results from the previous sections.

As mentioned earlier, $G(k,m)(X)=h_m[(1-\Omega_k)X]$. By the Jacobi--Trudi definition, this implies
\begin{align}\label{e:modular-plethysm}
\mathfrak s^{(k)}_{\la/\mu}(X)=s_{\la/\mu}[(1-\Omega_k)X].
\end{align}
Substituting $X\mapsto (1-\Omega_k)X$ into \eqref{e:SqS} and \eqref{e:decom-skew-s} yields the following.

\begin{lem}\label{l:decom-modular}
    Let $\mu\subset\la$. Then
    \begin{align}\label{e:decom-modular1}
        \mathfrak s^{(k)}_{\la/\mu}(X)=\sum_{\xi}(-1)^{\rht((\la/\mu)/\xi)}\,G\!\left(k,|\la|-|\mu|-|\xi|\right)(X)\,\mathfrak s^{(k)}_{\xi}(X),
    \end{align}
    where the sum runs over all skew diagrams $\xi$ obtained from $\la/\mu$ by removing a special ribbon of $\la/\mu$.
    Consequently,
   \begin{equation}\label{e:decom-modular2}
    \mathfrak s^{(k)}_{\la/\mu}(X)=\sum_{\nu\models |\la|-|\mu|}\ \sum_{\T}\sgn(\T)\,G(k,\nu)(X),
\end{equation}
where the inner sum is over all SRTs $\T$ of shape $\la/\mu$ and type $\nu$.
Here $G(k,\nu):=G(k,\nu_1)\,G(k,\nu_2)\cdots$.
\end{lem}

Recall the notion of a good proper $k$-ribbon (GPR) from Subsection~\ref{ss:skew-Pieri-like}.
Let $\nu$ be a partition and $\tau$ a composition with $\tau\models |\nu|$.
\phantomsection\label{def:GPRT}
A \emph{good proper $k$-ribbon tableau} (GPRT) $\mathscr{T}$ of \emph{shape} $\nu$ and \emph{type} $\tau$ is a filling of the boxes of $\nu$ by positive integers such that:
\begin{enumerate}
  \item entries are weakly increasing from left to right along each row and from top to bottom down each column;
  \item for each $i\ge 1$, the set of boxes labeled $i$ consists of $\tau_i$ boxes forming a good proper $k$-ribbon,
  and each row of this ribbon contains fewer than $k$ boxes.
\end{enumerate}
For a GPRT $\mathscr T$, define
\begin{align}
    \sgn^{'}(\mathscr T):=\prod_{i=1}^{\ell(\tau)}\sgn^{'}\!\left(\mathscr T^{(i)}/\mathscr T^{(i-1)}\right),
\end{align}
where $\mathscr T^{(i)}$ denotes the subdiagram consisting of boxes labeled by numbers $\le i$ (with $\mathscr T^{(0)}=\varnothing$), and $\sgn^{'}$ is as in \eqref{e:sgn'}.

The next lemma is the dual form of the Pieri-like rule for Petrie symmetric functions, since the Schur basis is orthonormal with respect to the Hall inner product.

\begin{lem}\label{l:G*s}
    Let $m\geq 1$ and let $\nu$ be a partition. Then
    \begin{align}\label{e:G*s}
        G^{\perp_{0,0}}(k,m)\,s_{\nu}
        =\sum_{\rho\vdash|\nu|-m}\sgn^{'}(\nu/\rho)\,s_{\rho},
    \end{align}
    where the sum runs over all partitions $\rho$ such that $\nu/\rho$ is a good proper $k$-ribbon (GPR) whose every row has length $<k$.
    Moreover, if $\tau\models |\nu|$, then
    \begin{align}
        G^{\perp_{0,0}}(k,\tau)\,s_{\nu}
        =\sum_{\mathscr T}\sgn^{'}(\mathscr T),
    \end{align}
    summed over all GPRTs $\mathscr T$ of shape $\nu$ and type $\tau$, where
    \[
    G^{\perp_{0,0}}(k,\tau):=G^{\perp_{0,0}}(k,\tau_1)G^{\perp_{0,0}}(k,\tau_2)\cdots.
    \]
\end{lem}

With Lemmas \ref{l:decom-modular} and \ref{l:G*s} in hand, we can now interpret the coefficients $m_{\la/\mu,\nu}$ combinatorially.

\begin{thm}\label{t:combin-m}
    Let $\mu\subset\la$ and $\nu$ be partitions. Then
    \begin{align}\label{e:combin-m1}
        m_{\la/\mu,\nu}
        =\sum_{|\xi|=|\rho|}
        (-1)^{\rht((\la/\mu)/\xi)}\,\sgn^{'}(\nu/\rho)\,m_{\xi,\rho},
    \end{align}
    where $\xi$ ranges over skew diagrams obtained from $\la/\mu$ by removing a special ribbon, and $\rho$ ranges over partitions obtained from $\nu$ by removing a GPR whose every row has length $<k$.
    Moreover,
    \begin{align}\label{e:combin-m2}
      m_{\la/\mu,\nu}
      =\sum_{(\T,\mathscr T)}\sgn(\T)\,\sgn^{'}(\mathscr T),
    \end{align}
    summed over all pairs consisting of an SRT $\T$ of shape $\la/\mu$ and a GPRT $\mathscr T$ of shape $\nu$ that have the same type.
\end{thm}

\begin{proof}
    The identity \eqref{e:combin-m2} follows by iterating \eqref{e:combin-m1}. Thus it suffices to prove \eqref{e:combin-m1}.
    Using the Hall inner product $\langle\cdot,\cdot\rangle_{0,0}$, we compute
    \begin{align*}
        m_{\la/\mu,\nu}
        &=\left\langle \mathfrak s^{(k)}_{\la/\mu}(X),\, s_{\nu}(X) \right\rangle_{0,0}\\
        &=\left\langle \sum_{\xi}(-1)^{\rht((\la/\mu)/\xi)}
        G\!\left(k,|\la|-|\mu|-|\xi|\right)(X)\,\mathfrak s^{(k)}_{\xi}(X),\, s_{\nu}(X) \right\rangle_{0,0}
        \quad\text{(by \eqref{e:decom-modular1})}\\
        &=\sum_{\xi}(-1)^{\rht((\la/\mu)/\xi)}
        \left\langle \mathfrak s^{(k)}_{\xi}(X),\, G^{\perp_{0,0}}\!\left(k,|\la|-|\mu|-|\xi|\right)s_{\nu}(X) \right\rangle_{0,0}\\
        &=\sum_{\xi,\rho}(-1)^{\rht((\la/\mu)/\xi)}\sgn^{'}(\nu/\rho)
        \left\langle \mathfrak s^{(k)}_{\xi}(X),\, s_{\rho}(X) \right\rangle_{0,0}
        \quad\text{(by \eqref{e:G*s})}\\
        &=\sum_{\xi,\rho}(-1)^{\rht((\la/\mu)/\xi)}\sgn^{'}(\nu/\rho)\,m_{\xi,\rho},
    \end{align*}
    as required.
\end{proof}

\begin{rem}
    In \cite{CCE+23}, the authors gave a combinatorial interpretation of $m_{\la/\mu,\nu}$ in the special case where $\la$ is a single row and $\mu=\varnothing$.
    They also asked for a combinatorial method to complete the transition matrix by expanding $\mathfrak s^{(k)}_{\la}$ in the Schur basis.
    Theorem~\ref{t:combin-m} provides such an interpretation in a broader skew setting.
\end{rem}


\section{Proof of Walker's conjecture surrounding \texorpdfstring{$m_{\la,\nu}$}{m(lambda,nu)}}
In this section, we restrict $k$ to be an odd prime. As before, we let $m_{\la,\nu}$ denote the transition coefficient from $\mathfrak s^{(k)}_{\la}$ to $s_{\nu}$, i.e.,
\begin{align}
    \mathfrak s^{(k)}_{\la}=\sum_{\nu\vdash|\la|}m_{\la,\nu}\,s_{\nu}.
\end{align}

\begin{defn}
Let $\lambda\vdash d$.
We write $\epsilon^\lambda$ for the ordinary irreducible character of the symmetric group $\mathfrak S_d$ corresponding to $\lambda$.

A partition $\lambda$ is called a \emph{$k$-core} if no $k$-ribbon can be removed from $\lambda$.
Equivalently, no hook length of $\lambda$ is divisible by $k$.
\end{defn}

At the end of \cite{Wal94}, Walker proposed a conjecture about $m_{\la,\nu}$, which reads as follows:
\begin{thm}{\rm (\cite[Conjecture 4.7]{Wal94}).}\label{thm:conj47}
Assume that $k$ is an odd prime. For a given partition $\la$, if
\[
m_{\la,\nu}=\delta_{\la,\nu}\quad\text{for all $\nu\vdash|\la|$},
\]
then $\lambda$ is a $k$-core.
\end{thm}

\begin{rem}
The case $k=2$ is not included in this statement. Indeed, for $k=2$ one has
$\mathfrak s^{(2)}_\lambda=s_{\lambda^t}$; hence the implication above would
fail, for example, for the self-conjugate partition $(2,2)$, which is not a
$2$-core.
\end{rem}

Before confirming this conjecture, we make some preparation. 

For a positive integer $n$, we write $n_k$ for the $k$-part of $n$, i.e. the largest power of $k$ dividing $n$. 
\begin{defn}
Let $G$ be a finite group. 
\begin{enumerate}
    \item An element $g\in G$ is called \emph{$k$-singular} if $k\mid |g|$, where $|g|$ denotes the order of $g$.
A conjugacy class of $G$ is called \emph{$k$-singular} if it consists of $k$-singular elements.
\item Let $\chi$ be an ordinary irreducible character of $G$.
The \emph{$k$-defect} of $\chi$ is the integer $d_k(\chi)\geq 0$ defined by
\[
\chi(1)_k=\frac{|G|_k}{k^{d_k(\chi)}}.
\]
Equivalently, $\chi$ has \emph{defect zero} if and only if
\[
\chi(1)_k=|G|_k.
\]
\end{enumerate}

\end{defn}

Now we concentrate on the symmetric group $\mathfrak S_d$.

\begin{rem}
If $\mu\vdash d$, then the conjugacy class of $\mathfrak S_d$ of cycle type $\mu$ is $k$-singular if and only if at least one part of $\mu$ is divisible by $k$.
Indeed, the order of an element of cycle type $\mu=(\mu_1,\dots,\mu_\ell)$ is the least common multiple of $(\mu_1,\dots,\mu_\ell)$.
\end{rem}

Brauer and Nesbitt \cite[Theorem 1]{BN41} proved the following classical result. Let $\chi$ be an irreducible character of the finite group $G$ and has $k$-defect zero, then $\chi$ vanishes on all those elements of $G$ whose order is divisible by $k$. For the symmetric group case, we give the converse of this result.

\begin{lem}\label{lem:vanishing-implies-defect-zero}
Let $\epsilon$ be an irreducible character of $\mathfrak S_d$.
Assume that
\[
\epsilon(g)=0
\qquad
\text{for every $k$-singular element } g\in \mathfrak S_d.
\]
Then $\epsilon$ has $k$-defect zero.
\end{lem}

\begin{proof}
Let $P$ be a sylow $k$-subgroup of $\mathfrak S_d$.
Consider the restriction $\epsilon_P$ of $\epsilon$ to $P$.
By taking the usual inner product on the space of class functions with the trivial character $1_P$, we obtain
\[
\langle \epsilon_P,1_P\rangle
=
\frac{1}{|P|}\sum_{x\in P}\epsilon(x).
\]
Every nonidentity element of $P$ is $k$-singular, so by assumption $\epsilon(x)=0$ for all $x\in P\setminus\{1\}$.
Hence
\[
\langle \epsilon_P,1_P\rangle
=
\frac{1}{|P|}\epsilon(1).
\]
Since $\langle \epsilon_P,1_P\rangle$ is a nonnegative integer, it follows that
\[
|P|\mid \epsilon(1).
\]
Therefore
\[
|P|=|\mathfrak S_d|_k=(d!)_k \leq \epsilon(1)_k.
\]

On the other hand, since $\epsilon=\epsilon^\lambda$ for some partition $\lambda\vdash d$, the hook-length formula gives
\[
\epsilon^\lambda(1)=\frac{d!}{\prod_{u\in[\lambda]} h(u)},
\]
where $h(u)$ denotes the hook length at the node $u$.
In particular, $\epsilon^\lambda(1)$ divides $d!$, and hence
\[
\epsilon(1)_k\leq (d!)_k=|\mathfrak S_d|_k.
\]
Combining the two inequalities yields
\[
\epsilon(1)_k=|\mathfrak S_d|_k.
\]
Thus $\epsilon$ has $k$-defect zero.
\end{proof}

Now we are ready to prove the Walker's conjecture.
\medskip

\noindent {\em Proof of Theorem \ref{thm:conj47}}. 
We first compute the plethystic action of $(1-\Omega_k)X$ on the power-sum basis. Recall that $\Omega_k=1+\omega+\cdots+\omega^{k-1}$ and $\omega$ is the $k$th root of unity.

For every integer $r\geq 1$,
\[
p_r\bigl[(1-\Omega_k)X\bigr]
=
p_r[X]-p_r[\Omega_k X].
\]
Since plethysm is multiplicative on alphabets,
\[
p_r[\Omega_k X]
=
p_r[\Omega_k]\cdot p_r[X].
\]
Moreover,
\[
p_r[\Omega_k]
=
1+\omega^r+\omega^{2r}+\cdots+\omega^{(k-1)r}
=
\begin{cases}
0,& k\nmid r,\\[1mm]
k,& k\mid r.
\end{cases}
\]
Therefore
\[
p_r\bigl[(1-\Omega_k)X\bigr]
=
\begin{cases}
p_r[X],& k\nmid r,\\[1mm]
(1-k)p_r[X],& k\mid r.
\end{cases}
\]
Equivalently,
\[
p_r\bigl[(1-\Omega_k)X\bigr]=\alpha_r p_r[X],
\qquad
\alpha_r=
\begin{cases}
1,& k\nmid r,\\
1-k,& k\mid r.
\end{cases}
\]

Now let $\mu=(\mu_1,\dots,\mu_\ell)\vdash d$, and define
\[
a_k(\mu):=\#\{\,i: k\mid \mu_i\,\}.
\]
By multiplicativity,
\[
p_\mu\bigl[(1-\Omega_k)X\bigr]
=
\prod_{i=1}^\ell p_{\mu_i}\bigl[(1-\Omega_k)X\bigr]
=
(1-k)^{a_k(\mu)}p_\mu[X].
\]

Let us recall the Frobenius character formula
\[
s_\lambda
=
\sum_{\mu\vdash d} z_\mu^{-1}\epsilon^\lambda(\mu)\,p_\mu,
\]
where
\[
p_\mu:=p_{\mu_1}p_{\mu_2}\cdots p_{\mu_\ell}
\qquad
\text{for }
\mu=(\mu_1,\dots,\mu_\ell),
\]
and
\[
z_\mu=\prod_{i\geq 1} i^{m_i(\mu)}m_i(\mu)!.
\]
This together with \eqref{e:modular-plethysm} gives
\[
\mathfrak s^{(k)}_\lambda
=
s_\lambda\bigl[(1-\Omega_k)X\bigr]
=
\sum_{\mu\vdash d} z_\mu^{-1}\epsilon^\lambda(\mu)(1-k)^{a_k(\mu)}\,p_\mu.
\]

Assume now that $m_{\la,\nu}=\delta_{\la,\nu}$ for any $\nu$, which is equivalent to $\mathfrak s^{(k)}_\lambda=s_\lambda$.
Subtracting the Frobenius expansion of $s_\lambda$, we obtain
\[
0
=
\mathfrak s^{(k)}_\lambda-s_\lambda
=
\sum_{\mu\vdash d}
z_\mu^{-1}\Bigl((1-k)^{a_k(\mu)}-1\Bigr)\epsilon^\lambda(\mu)\,p_\mu.
\]
Since the family $\{p_\mu\}_{\mu\vdash d}$ is a basis of $\Lambda^d_{\mathbb Q(\omega)}$, every coefficient must vanish:
\[
\Bigl((1-k)^{a_k(\mu)}-1\Bigr)\epsilon^\lambda(\mu)=0
\qquad
\text{for all }\mu\vdash d.
\]

If $a_k(\mu)>0$, then since $k>2$ we have $(1-k)^{a_k(\mu)}\neq 1$.
Therefore
\[
\epsilon^\lambda(\mu)=0
\qquad
\text{for every }\mu\vdash d\text{ with }a_k(\mu)>0.
\]
However, as observed above, the condition $a_k(\mu)>0$ is equivalent to saying that the conjugacy class of cycle type $\mu$ in $\mathfrak S_d$ is $k$-singular.
Hence $\epsilon^\lambda$ vanishes on all $k$-singular elements of $\mathfrak S_d$.

By Lemma~\ref{lem:vanishing-implies-defect-zero}, the character $\epsilon^\lambda$ has $k$-defect zero.
For symmetric groups, this is equivalent to saying that no hook length of $\lambda$ is divisible by $k$ \cite[Corollary 2]{FRT54}.
Equivalently, $\lambda$ is a $k$-core.
This proves the theorem.$\hfill \Box$

\begin{rem}
Theorem~\ref{thm:conj47} also gives a partial confirmation of Walker's Conjecture~4.6.
In the paragraph following \cite[Conjecture~4.6]{Wal94}, Walker explains that, in the column-regular case, the remaining possibility for the modular Schur function $s'_\lambda$ to be simple would force the equality $s'_\lambda=s_\lambda$.
Translated into our notation, this equality is precisely
\[
\mathfrak s^{(k)}_\lambda=s_\lambda,
\]
or equivalently $m_{\lambda,\nu}=\delta_{\lambda,\nu}$ for all $\nu\vdash|\lambda|$.
Thus, for odd prime $k$, Theorem~\ref{thm:conj47} rules out this remaining possibility whenever $\lambda$ is not a $k$-core.
Since $k$-cores satisfy the Carter condition, as recalled in \cite[Remark~4.5]{Wal94}, this proves the column-regular case of \cite[Conjecture~4.6]{Wal94} for odd primes.
The full conjecture involves additional row-singular/highest-weight issues; these remaining cases will be studied separately in future work and will not be pursued further here.
\end{rem}



\vskip30pt \centerline{\bf Acknowledgments}
N.J. is partially supported by the Simons Foundation (no. MP-TSM-00002518) and National Natural Science Foundation of China (no. 12171303). 
N.L. is supported by China Postdoctoral Science Foundation (No. 2025M773076, No. GZC20252018) and Beijing Natural Science Foundation (No. 1264053). We thank Jennifer L. Morse for reading an earlier version of this manuscript and for helpful comments, especially for pointing out the relevance of the \(k\)-Schur Murnaghan--Nakayama rule to the framework developed here. N.L. would like to thank Zhicheng Feng at Southern University of Science and Technology for providing an excellent working environment during the writing of this paper. 
\bigskip

\appendix
\chapter{A combinatorial identity}\label{A:identity}
\begin{lem}
Let $\nu$ be a partition of $n$ with length $\ell(\nu)=k$. Let $m(\nu) = (m_1(\nu), m_2(\nu), \dots)$ denote the sequence of multiplicities of parts in $\nu$. The following identity holds:
\begin{align}
 \sum_{r,\rho} (-1)^{n-\ell(\rho)-1} t^{n-r-\ell(\rho)} t^{n(\rho)-\binom{\ell(\rho)}{2}} \psi_{\nu/\rho}(t) \qbinomial{\ell(\rho)}{m(\rho)}_t 
    = (-t)^{n-k}t^{n(\nu)-\binom{k}{2}}\qbinomial{k}{m(\nu)}_t
    \label{eq:main_identity}
\end{align}
where the sum runs over $1 \leq r \leq n$ and all partitions $\rho$ such that $\nu/\rho$ is an $r$-horizontal strip. 
\end{lem}

\begin{proof}
Divide both sides of \eqref{eq:main_identity} by the RHS and denote the normalized LHS by $S$.
We aim to prove $S=1$ and proceed in several steps.
\medskip

\noindent 1) {\em Encoding horizontal strips}.
If $\nu/\rho$ is a (possibly empty) horizontal strip, set
\[
a_j:=\nu_j^t-\rho_j^t\in\{0,1\}\qquad (j\ge1),
\]
so that $r=|\nu|-|\rho|=\sum_{j\ge1} a_j$ and $\ell(\rho)=\rho_1^t=k-a_1$.
Also $n(\lambda)=\sum_{j\ge1}\binom{\lambda_j^t}{2}$.
\medskip

\noindent 2) {\em Normalized sign and $t$-power}.
The sign ratio is
\[
\frac{(-1)^{n-\ell(\rho)-1}}{(-1)^{n-k}}=(-1)^{k-\ell(\rho)-1}=(-1)^{a_1-1}.
\]
For the $t$-exponent, we compute
\[
(n-r-\ell(\rho))-(n-k)=k-r-(k-a_1)=a_1-r,
\]
and
\[
n(\rho)-n(\nu)=\sum_{j\ge1}\left(\binom{\nu_j^t-a_j}{2}-\binom{\nu_j^t}{2}\right)
=\sum_{j\ge1} a_j(1-\nu_j^t)=r-\sum_{j\ge1} a_j \nu_j^t,
\]
as $a_j\in\{0,1\}$.
Moreover,
\[
\binom{\ell(\rho)}2-\binom{k}{2}=\binom{k-a_1}{2}-\binom{k}{2}=-a_1(k-1).
\]
Hence the normalized $t$-power is
\[
\Delta E=(a_1-r)+\Bigl(r-\sum_{j\ge1}a_j\nu_j^t\Bigr)+a_1(k-1)
=a_1k-\sum_{j\ge1}a_j\nu_j^t
=-\sum_{j\ge2} a_j \nu_j^t
=-\sum_{j\ge1} a_{j+1}\nu_{j+1}^t.
\]
\medskip

\noindent 3) {\em Multinomial ratio and local weights}.
Using $m_j(\rho)=m_j(\nu)-a_j+a_{j+1}$, we have
\[
\frac{\qbinomial{\ell(\rho)}{m(\rho)}_t}{\qbinomial{k}{m(\nu)}_t}
=\frac{[k-a_1]_t!}{[k]_t!}\prod_{j\ge1}\frac{[m_j(\nu)]_t!}{[m_j(\rho)]_t!}.
\]
Also, $\psi_{\nu/\rho}(t)=\prod_{j:\,a_j=0,a_{j+1}=1}(1-t^{m_j(\rho)})$.
Define for $u,v\in\{0,1\}$ the local weight (with $m_j:=m_j(\nu)$)
\[
w_j(u,v):=\frac{[m_j]_t!}{[m_j-u+v]_t!}\times
\begin{cases}
1-t^{m_j-u+v}, & (u,v)=(0,1),\\
1, & \text{otherwise}.
\end{cases}
\]
A direct check gives the four cases:
\[
w_j(0,0)=1,\quad w_j(1,0)=[m_j]_t,\quad w_j(0,1)=1-t,\quad w_j(1,1)=1.
\]

Note that not every binary sequence $(a_j)$ yields a partition $\rho$, but any \emph{illegal} drop
$\nu_j^t=\nu_{j+1}^t$ together with $(a_j,a_{j+1})=(1,0)$ forces $m_j(\nu)=0$ and hence
$w_j(1,0)=[0]_t=0$. Therefore all non-partition sequences contribute weight $0$, and we may
sum over all $(a_j)\in\{0,1\}^{\mathbb N}$ without changing the value.
\medskip

\noindent 4) {\em Transfer matrices}.
Set $N_j:=\nu_{j+1}^t$.
By 2) the factor $t^{\Delta E}$ contributes $t^{-a_{j+1}N_j}$ at step $j$.
Thus the total normalized weight of a sequence $(a_j)$ is
\[
(-1)^{a_1-1}\frac{[k-a_1]_t!}{[k]_t!}\prod_{j\ge1}\Bigl(w_j(a_j,a_{j+1})\,t^{-a_{j+1}N_j}\Bigr).
\]
Define the $2\times2$ transfer matrix $M_j$ by $M_j(u,v):=w_j(u,v)\,t^{-vN_j}$.
Explicitly,
\[
M_j=\begin{pmatrix}
1 & (1-t)t^{-N_j}\\
{[m_j]_t} & t^{-N_j}
\end{pmatrix}.
\]
Choose $L>\nu_1$ so that $\nu_{L+1}^t=0$; then for $j\ge L$ we have $m_j=0$ and $N_j=0$,
and $M_j\binom10=\binom10$. Hence the infinite product reduces to a finite product:
let $V_{L+1}:=\binom10$ and define $V_j:=M_jV_{j+1}$ for $1\le j\le L$.
Write $V_j=\binom{A_j}{B_j}$.
\medskip

\noindent 5) {\em A one-step invariant}.
We claim that for all $j$, 
\[
B_j=[\nu_j^t]_t\,A_j.
\]
For $j=L+1$, $\nu_{L+1}^t=0$ and $(A_{L+1},B_{L+1})=(1,0)$, so it holds.
Assume $B_{j+1}=[N_j]_tA_{j+1}$. From $V_j=M_jV_{j+1}$,
\[
A_j=A_{j+1}+(1-t)t^{-N_j}B_{j+1}
=A_{j+1}\Bigl(1+(1-t)t^{-N_j}[N_j]_t\Bigr)=t^{-N_j}A_{j+1},
\]
since $1+(1-t)t^{-N}[N]_t=t^{-N}$.
Also
\[
B_j=[m_j]_tA_{j+1}+t^{-N_j}B_{j+1}
=A_{j+1}\Bigl([m_j]_t+t^{-N_j}[N_j]_t\Bigr)
=t^{-N_j}[m_j+N_j]_tA_{j+1}.
\]
Because $\nu_j^t=m_j+N_j$, we get $B_j=[\nu_j^t]_tA_j$.
\medskip

\noindent 6) {\em Evaluate the total sum and remove the empty strip}.
Including the empty strip ($r=0$), the sum over all sequences equals the contribution of
initial state $a_1=0$ plus that of $a_1=1$:
\[
S_{\mathrm{tot}}=\Bigl((-1)\cdot 1\Bigr)A_1+\Bigl(1\cdot \frac{[k-1]_t!}{[k]_t!}\Bigr)B_1
=-A_1+\frac{1}{[k]_t}B_1.
\]
By 5) with $\nu_1^t=k$, $B_1=[k]_tA_1$, hence $S_{\mathrm{tot}}=0$.

The excluded case $r=0$ corresponds to the unique sequence $a_j\equiv0$, whose normalized
weight is $(-1)^{0-1}=-1$. Therefore
\[
S=S_{\mathrm{tot}}-(-1)=1.
\]
This proves \eqref{eq:main_identity}.
\end{proof}

\end{document}